\documentclass[11pt, letterpaper,reqno]{amsart}
\usepackage[left=1in,right=1in,bottom=0.5in,top=0.8in]{geometry}

\usepackage{amsmath, amsthm, amsfonts, amssymb, bm,bbm}

\usepackage{booktabs}
\usepackage{natbib}
\usepackage[title]{appendix}
\usepackage{setspace}
\usepackage{graphicx}
\usepackage{enumitem}
\usepackage[labelfont=bf, font={small,it}]{caption}
\usepackage{float}
\usepackage{mathtools}
\usepackage{array}
\usepackage{xcolor}
\usepackage{hyperref}
\usepackage{float}
\usepackage{subcaption}
\usepackage{fancyhdr}
\usepackage{mdframed}
\usepackage{listings}
\usepackage{upgreek, bm}
\usepackage{tikz} 
\usetikzlibrary{patterns} 
\usetikzlibrary{arrows}
\usetikzlibrary{decorations.markings} 
\usetikzlibrary{shapes} 
\usetikzlibrary{decorations.pathreplacing}

\newtheorem{thm}{Theorem}[section]
\newtheorem{cor}[thm]{Corollary}
\newtheorem{lemma}[thm]{Lemma}
\newtheorem{prop}[thm]{Proposition}
\newtheorem{rmk}[thm]{Remark}
\newtheorem{conj}[thm]{Conjecture}

\newtheorem{define}[thm]{Definition}

\newcommand{\Pb}{\mathbb{P}}
\newcommand{\E}{\mathbbm{E}}
\newcommand{\Id}{\mathbbm{1}}

\newcommand{\I}{{\rm i}}

\newcommand{\Gamm}{\text{Gamma}^{-1}}

\newcommand{\K}{\mathcal{K}}
\newcommand{\B}{\mathcal{B}}
\newcommand{\A}{\mathcal{A}}
\newcommand{\M}{\mathcal{M}}

\newcommand{\BBB}{\mathsf{B}}
\newcommand{\AAA}{\mathsf{A}}
\newcommand{\DDD}{\mathsf{D}}
\newcommand{\CCC}{\mathsf{C}}

\newcommand{\ttt}{\tilde{t}}
\newcommand{\tU}{\tilde{U}}
\newcommand{\tV}{\tilde{V}}
\newcommand{\tbeta}{\tilde{\beta}}

\newcommand{\tW}{\tilde{W}}
\newcommand{\tZ}{\tilde{Z}}
\newcommand{\CC}{\mathbb{C}}

\newcommand{\X}{\mathcal{X}}
\newcommand{\Y}{\mathcal{Y}}

\newcommand{\G}{\overline{G}}
\newcommand{\D}{\mathcal{D}}
\newcommand{\C}{\mathcal{C}}

\newcommand{\R}{\mathbb{R}}
\newcommand{\N}{\mathbb{N}}

\DeclareMathOperator{\Z}{\mathbb{Z}}
\DeclareMathOperator{\PP}{\mathbb{P}}
\newcommand{\Nsigma}{(\sigma N)^{1/3}}
\newcommand{\Nsig}{(\sigma N)^{-1/3}}

\numberwithin{equation}{section}

\DeclareMathOperator{\NN}{\mathcal{N}}
\DeclareMathOperator{\Pf}{\mathrm{Pf}}

\newcommand{\bra}[1]{\left\langle #1 \right|}
\newcommand{\ket}[1]{\left| #1 \right\rangle}
\newcommand{\braket}[2]{\left\langle #1  \left| #2  \right\rangle \right.}
\newcommand{\brabarket}[3]{\left\langle #1 \left| #2 \right| #3 \right\rangle}
\newcommand{\ketbra}[2]{\left| #1 \right\rangle \left\langle #2 \right|}

\title{Stationary log-gamma Polymer in half-space}
\author{Jiyue Zeng, Xinyi Zhang}
\thanks{Email: jz3524@columbia.edu} 
\thanks{Email: xz3272@columbia.edu} 
\date{\today}

\begin{document}

\maketitle

\begin{abstract}
    We study the half-space log-gamma polymer model with stationary initial conditions. We derive exact formulas for the distribution of the partition function along the diagonal across the entire High density phase and Low density phase. We obtain asymptotics of these distributions under the critical scaling. We also prove the first exponential upper bounds for the upper and lower tail of the scaled free energy for these half-space stationary log-gamma models. 
\end{abstract}

\section{Introduction}
Log-gamma polymer model, introduced in \cite{Seppalainen2012}, is an exactly solvable model in the KPZ universality class. It can access KPZ equation and zero temperature models such as Last Passage Percolation (LPP), Totally Asymmetric Simple Exclusion Process (TASEP) under certain limits. Exact Laplace transform formulas for the full-space log-gamma free energy were derived using tropical combinatorics and Whittaker function identities \cite{CorwinOConnellSeppalainenZygouras2014}, and their asymptotics were later analyzed via nested contour integral techniques introduced in \cite{borodin2013log}.

Restricting to the half-space geometry introduces boundary interactions
that affects the overall behavior and fluctuations. In \cite{BarraquandShouda}, the authors established a distributional identity between half- and full-space log-gamma free energies. Based on the distributional identity \cite{BarraquandShouda} and half-space Whittaker process, \cite{barraquand2025kpz} proved the KPZ exponents for the half-space log-gamma without access to exact formulas. Explicit Fredholm Pfaffian formulas for the point-to-point partition function of half-space log-gamma polymer were later obtained in \cite{imamura2022solvablemodelskpzclass}.

The stationary measures of these models are of particular interests.
For instance, stationary measures have been characterized for the open KPZ equation \cite{corwin2024stationary,2023stationary}, for LPP and log-gamma polymer on a strip \cite{stationarystrip}, and for the half-space KPZ equation \cite{LogGammaStationary,barraquand2021steady}. Moreover, \cite{das2025convergence} proves the convergence of half-space log-gamma polymer to the stationary measure along anti-diagonal path. In particular, \cite{LogGammaStationary} proved that the half-space log-gamma polymer model admits a two-parameter family of stationary measures, one boundary parameter and one parameter that describes the drift of the initial condition; this family of stationary measures is the main focus of the present work.

The same two-parameter stationary structure appears in half-space LPP models, where under a special choice of parameters it degenerates to Brownian motion.
Fluctuations of half-space LPP with stationary initial condition have been studied in several settings: \cite{PatrikStatExp} derived the limiting one-point distribution for exponential LPP with Brownian stationary initial condition; \cite{zeng2025stationary} obtained the one-point distribution and critical limits for the two-parameter stationary half-space geometric LPP. It is then natural to further investigate limiting distributions of the half-space log-gamma model with the two-parameter family of stationary initial conditions.

In this work, we derive explicit formulas for
the one-point distribution of the free energy and obtain their critical scaling limits, which should describe the one-point distribution of conjectural half-space KPZ fixed point with stationary initial condition.
We introduce the model and stationary measures explicitly in the following section.

\subsection{Half-space log-gamma polymer with initial conditions}
Given the initial condition $(Z(N,1))_{N\in \Z_{\geq 1}}$, we define the partition function $Z(N,M)$ of the half-space log-gamma polymer via the following recurrence relation:
\begin{equation}\label{def:recurrence}
    \begin{cases}
        Z(N,M) = \omega_{N,M} (Z(N-1,M) + Z(N,M-1)), &\text{ for }N>M\geq 2,\\
        Z(N,N) = \omega_{N,N}Z(N,N-1), &\text{ for } N\geq 2
    \end{cases}
\end{equation}
where the weights $\omega_{N,M}$ are all independent and distributed as
\begin{equation}\label{def:weights}
\begin{cases}
    \omega_{N,M} \sim \Gamm(2\alpha), &\text{ for }N>M\geq 2,\\
    \omega_{N,N} \sim  \Gamm(\beta + \alpha), &\text{ for }N\geq 2.\\
\end{cases}
\end{equation}
for $\alpha >0$ and $\beta + \alpha >0.$ We call $\log Z(N,M)$ the free energy of the polymer.

\begin{define}
    We say a process $(\mathcal{T}(x))_{x \in \Z_{\geq 0}}$ is stationary for the half-space log-gamma polymer if the solution to \eqref{def:recurrence} with the initial condition $Z(1+\cdot,1)/Z(1,1) \stackrel{(d)}{=}\mathcal{T}(\cdot)$ has the property that the law of $(Z(N+x, N) / Z(N,N))_{x\in\Z_{\geq 0}}$ is the same as the law of $(\mathcal{T}(x))_{x \in \Z_{\geq 0}}$ for all $N\geq 1$.
\end{define}

\begin{define}
    For $\alpha \in \R_{\geq 0}$ and $\beta, t\in \R$ with $\alpha>t>0$ and $\beta > -t$, we define a process
    \begin{equation}\label{def: stationary process}
        \mathcal{T}_{\beta, t} (x) := R_2(x) + \frac{1}{\omega}\sum_{\ell = 1}^x R_1(\ell)\frac{R_2(x)}{R_2(\ell-1)} \quad \text{ for } x \in \Z_{\geq 0},
    \end{equation}
    where $R_1$ and $R_2$ are independent $\Gamm (\alpha -t)$ and $\Gamm(\alpha + t)$ multiplicative random walks with $R_1(0) = R_2(0) =1,$ and $\omega \sim \Gamm(\beta + t)$ is independent of $R_1,R_2$. When $\beta = -t,$ we set $\frac{1}{\omega} = 0.$ We call $\mathcal{T}_{\beta, t}(\cdot)$ the \emph{two-parameter stationary measure}. 
\end{define}
It is clear that $\mathcal{T}_{-t,t}$ has the law of a multiplicative $\Gamm(\alpha +t)$ random walk. When $\beta = t$, $\mathcal{T}_{t, t}$ is also a multiplicative $\Gamm(\alpha +t)$ random walk. The proof of this fact can be found in \cite[Lemma 2.9]{LogGammaStationary}. We call $\mathcal{T}_{-t,t}$ the \emph{product stationary measure}.

We state \cite[Theorem 1.8]{LogGammaStationary} in the setting of our parameters.
\begin{thm}
    For any $\alpha \in \R_{>0}$ and $\beta,t \in \R$ such that $\alpha > t >0$ and $\beta > \max\{-t,-\alpha\}$, the process $\mathcal{T}_{\beta,t}(\cdot)$ in \eqref{def: stationary process} is stationary for the half-space log-gamma polymer model \eqref{def:recurrence}.
\end{thm}

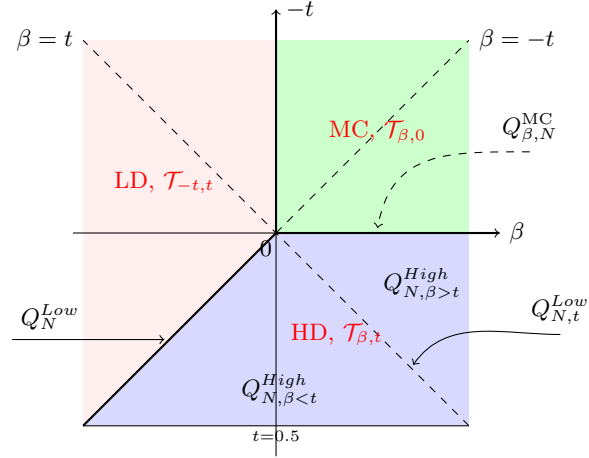
\begin{figure}
\centering
    \begin{tikzpicture}[scale=1.35]
        \fill[green!20] (0,0) -- (1.9,0) -- (1.9,1.9) -- (0,1.9) -- (0,0) -- cycle;
        \fill[red!20,opacity=0.3] (0,0) -- (0,1.9) -- (-1.9,1.9) -- (-1.9,-1.9) -- (0,0) -- cycle;
        \fill[blue!15] (0,0) -- (1.9,0) -- (1.9,-1.9) -- (0,0) -- cycle;
        \fill[blue!15] (0,0) -- (1.9,-1.9) -- (-1.9,-1.9) -- (0,0) -- cycle;

        \draw[thin] (-1.9,-1.9) -- (1.9,-1.9);
        \node  at (0,-2) {\footnotesize{$\scriptstyle t = 0.5$}};

        \draw[thin,->] (-2,0) -- (2.2,0)node[right] {\footnotesize{$\beta$}};
        \draw[thin,->] (0,-2.2) -- (0,2.2)node[right] 
        {\footnotesize{$-t$}};

        \draw[thick] (0,0) -- (0,2.2);
        \draw[thick] (0,0) -- (2.2,0);
        
        \draw[thick] (-1.9,-1.9) -- (0,0);
        \draw[thin, dashed] (0,0) -- (1.9,1.9) node[right] {\footnotesize{$\beta = -t$}};
        \draw[thin,dashed] (1.9,-1.9) -- (-1.9,1.9) node[left] {\footnotesize{$\beta = t$}};

        \node at (-0.1,-0.15) {\footnotesize{$0$}};
        
        \node[text = {red}] at (0.6,-1){\footnotesize{HD, $\mathcal{T}_{\beta,t}$}};
        \node[text = {red}] at (1,1){\footnotesize{MC, $\mathcal{T}_{\beta,0}$}};
        \node[text = {red}] at (-1.1,0.5){\footnotesize{LD, $\mathcal{T}_{-t,t}$}};

        \draw[dashed,->]
        (2.5,0.8) node[above]{\footnotesize $Q^{\mathrm{MC}}_{\beta,N}$}to[out=180,in=80,looseness=1.2] (1,0.05);
    
      \draw[thin,->]
        (-2.6,-1.05) -- (-1.1,-1.05)
        node[near start,above]{\footnotesize $Q_N^{Low}$};
    
      \draw[thin,->]
        (2.8,-1) node[above]{\footnotesize $Q_{N,t}^{Low}$}to[out=180,in=55] (1.35,-1.3);

    \node[text = {black}] at (1.4,-0.5){\footnotesize{ $Q_{N,\beta>t}^{High}$}};
    \node[text = {black}] at (0,-1.55){\footnotesize{ ${Q}_{N,\beta<t}^{High}$}};
        
    \end{tikzpicture}
    \captionof{figure}{This picture describes the phase diagram of the half-space log-gamma polymer model given the boundary parameter $\beta$ and the parameter $t$ that describes the drift. Depending on where $(\beta,t)$ lies in the diagram, the process $Z(\cdot,N) / Z(N,N)$ converges to one of the three spatial processes, $\mathcal{T}_{\beta,t},$ $\mathcal{T}_{\beta,0},$ or $\mathcal{T}_{-t,t},$ which is claimed in Conjecture \ref{conj: stationary measures}. On the full line $\beta = t$ and the half line $\beta = -t$, $\beta \leq 0$, $t\geq 0$, the process converges to the $\text{Gamma}^{-1}$ $(\alpha+t)$ multiplicative random walk. \textbf{MC}, \textbf{HD}, and \textbf{LD} represent Maximal current (green), High density (blue), and Low density (red) phases. We characterize the distribution of $Z(N,N)$ in the High density phase under different conditions of $\beta$ and $t$. When $\beta > t,$ we obtain $Q_{N,\beta >t}^{High}$ and when $\beta < t$, we obtain $Q_{N,\beta <t}^{High}$. Along the Low density boundary, we obtain $Q_{N,t}^{Low}$. In particular, we have ${Q}_{N,t}^{Low}$ on the line $\beta = t$, $t>0$ in High density phase. We obtain these distributions under the condition $t< 0.5$. We do not have the distribution for the Maximal current phase, which is drawn as a dashed line.}
    \label{phaseDiagram}
\end{figure}

\subsection*{Phase diagram}
We define the digamma function and polygamma functions. Let
\begin{equation}
    \psi(z) := \frac{1}{\Gamma(z)} \frac{\mathrm{d}}{\mathrm{d}z}\Gamma(z),
\qquad \qquad
\psi^{(n)}(z) = \frac{\mathrm{d}^n}{\mathrm{d}z^n}\psi(z).
\end{equation}

\begin{conj}\label{conj: stationary measures}
    For given $\alpha \in \R_{>0}$ and $\beta, t \in \R$ with $\alpha>t>0$ and $\beta > -t$, $\{\mathcal{T}_{\beta, t}\}$ constitutes all extremal stationary measures for \eqref{def:recurrence}. For any initial data $Z(x,1)$ such that $\lim_{x\rightarrow \infty} \log Z(x,1)/x = -d \in \R$, we have the following. Let
    \begin{equation}\label{eq: relation}
     t = \alpha - \psi^{-1}(d).        
    \end{equation}
    \begin{itemize}
        \item(High density) For $\beta \geq -t$ and $t \geq 0,$ the process $\left(Z(x+N, N) / Z(N,N)\right)_{x\in \Z_{\geq 0}}$ converges weakly as $N \rightarrow \infty$ to $(\mathcal{T}_{\beta,t}(x))_{x\in \Z_{\geq 0}}$.
        \item(Maximal current) For $\beta \geq 0$ and $t \leq 0$, the process $\left(Z(x+N, N)/Z(N,N)\right)_{x\in \Z_{\geq 0}}$ 
        converges weakly as $N \rightarrow \infty$ to $(\mathcal{T}_{\beta,0}(x))_{x\in \Z_{\geq0}}$.
        \item(Low density) For $-t \geq \beta$ and $\beta \leq 0,$ the process $\left(Z(x+N, N)/ Z(N,N)\right)_{x\in \Z_{\geq 0}}$ converges weakly as $N \rightarrow \infty$ to $(\mathcal{T}_{-t,t}(x))_{x\in \Z_{\geq 0}}$, i.e., a $\Gamm(\alpha + t)$ multiplicative random walk.
    \end{itemize}
\end{conj}

The parameter $t$ describes the drift of the initial condition through the relation \eqref{eq: relation}. Depending on the $(\beta,t)$ values, the system exhibits three different
regimes. We give a heuristic interpretation of each regime: 
\begin{enumerate}
    \item \textbf{High-density regime.}
    The free energy is dominated by contributions of weights from the first row (the initial condition).

    \item \textbf{Low-density regime.}
    By symmetry with the High density regime, the dominant contribution to the free energy
    comes from weights on the diagonal, whereas the influence of
    the initial condition becomes negligible.

    \item \textbf{Maximal-current regime.}
    In this regime, weights from the initial condition and the diagonal have comparable
    strength, leading to a competition between the two. Thus, the dominant contribution comes from the bulk weights.
\end{enumerate}

\subsection{Main results}
We define the diagamma function and polygamma functions. Let
\begin{equation}
    \psi(z) := \frac{1}{\Gamma(z)} \frac{\mathrm{d}}{\mathrm{d}z}\Gamma(z),
\qquad \qquad
\psi^{(n)}(z) = \frac{\mathrm{d}^n}{\mathrm{d}z^n}\psi(z).
\end{equation}
Then we define two constants for the following theorems. Let
\begin{equation}\label{fSconstant}
    f = 2 \psi(\alpha),\quad \sigma = - \psi^{(2)}(\alpha).
\end{equation}
\subsubsection{Main results for Low density phase}
\begin{define}\label{def: product stationary Z}
    Consider the product stationary process $\mathcal{T}_{-t,t}$. We use $\mathcal{Z}^t(N,N)$ to denote the partition function defined by \eqref{def:recurrence} and \eqref{def:weights} with $\beta = -t$ and set the initial condition $\mathcal{Z}^t(1+\cdot,1) \stackrel{(d)}{=} \mathcal{T}_{-t,t}.$
\end{define}

\begin{thm}[Low density phase]\label{thm: one_param_finite}
    Fix any two parameters $t,\alpha \in \R_{>0}$ satisfying $t \in (0,0.5)$ and $\alpha > t.$ Let $\mathcal{Z}^t(N,N)$ be defined as in definition \ref{def: product stationary Z}. Choose $N_0 \in \Z_{>0}$ such that $t> \frac{1}{N_0}$. Then for all $N\geq N_0,$ we have \begin{equation}\label{eq: One_param_formula}
        \begin{aligned}
             &\E\bigg[ 2K_0\left( 2
            {e^{\left(\log{{\mathcal{Z}}^t}(N,N)-\tau\right)/2} }\right)  \bigg]=Q_{N,t}^{Low}(\tau)
        \end{aligned}
    \end{equation}
    where $Q_{N,t}^{Low}(\tau)$ is defined in Definition \ref{def: one_param_finite}. 
\end{thm}

\begin{thm}[Low density phase]\label{thm: One param asymptotic}
    Fix any $\tilde{t},\alpha \in \R_{>0}$ and $s \in \R$. Consider the scaling
    \begin{equation}
        t = \frac{\tilde{t}}{(\sigma N)^{1/3}},  \quad \tau = -fN + (\sigma N)^{1/3} s.
    \end{equation} Let $\mathcal{Z}^t(N,N)$ be defined as in definition \ref{def: product stationary Z}. There exists $N_0$ such that for all $N\geq N_0,$ the identity \eqref{eq: One_param_formula} holds. Then we have the following asymptotic limit:
    \begin{equation}\label{eq: one param asymptotic formula}
        \begin{aligned}
             \lim_{N\rightarrow \infty} (\sigma N)^{-1/3} \E\bigg[ 2K_0\left( 2
            {e^{\left(\log \mathcal{Z}^t(N,N)+Nf-(\sigma N)^{1/3}s\right)/2} }\right)  \bigg] = \widetilde{Q}_{\ttt}^{Low}(s)   
        \end{aligned}
    \end{equation}
    where $\widetilde{Q}_{\ttt}^{Low}(s)$ is defined in Definition \ref{def: One param asymptotic}.
\end{thm}

\begin{thm}[Low density phase]\label{thm: One param probability}
    Fix any $\alpha,\ttt \in \R_{>0}$ and $s \in \R$. Consider the scaling
    \begin{equation}
        t = \frac{\tilde{t}}{\Nsigma}, \quad \tau = -Nf + (\sigma N)^{1/3} s.
    \end{equation} Let $\mathcal{Z}^t(N,N)$ be defined as in definition \ref{def: product stationary Z}. Let $\widetilde{{Q}}_{\ttt}^{Low}$ be defined as in Definition \ref{def: One param asymptotic} with a choice of $0<\delta <\ttt$ for the contour. Then for any $2\ttt>A>\ttt - \delta$, we have
    \begin{equation}\label{eq: Low final prob}
    \begin{aligned}
        & \Pb\left(\lim_{N\rightarrow \infty}\frac{\log \mathcal{Z}^{t}(N,N) +Nf}{(\sigma N)^{1/3}} \leq s\right) =\int_{\R}\int_{\R} (-A+\I y)e^{-Ax} e^{\I x y} \left(\widetilde{{Q}}_{\ttt}^{Low}(x+s)\right) dxdy.
    \end{aligned}
    \end{equation}
\end{thm}

\subsubsection{Main results for High density phase}

\begin{define}\label{def: two param stationary condition}
    Let $X$ be an independent $\Gamm(\beta + t)$ random variable. Let $R_1, R_2$ be independent $\Gamm(\alpha -t)$ and $\Gamm(\alpha +t)$ random walks with $R_1(0)= R_2(0) = 1.$ Let $\mathcal{T}_{\beta,t}$ be the two parameter stationary process constructed from $R_1$, $R_2$ and $\omega = X$. We use $\mathcal{Z}^{t,\beta}(N,N)$ to denote the partition function defined by \eqref{def:recurrence} and \eqref{def:weights} with $\mathcal{Z}^{t,\beta}(1,1) = X$ and $\mathcal{Z}^{t,\beta}(1+\cdot,1)/\mathcal{Z}^{t,\beta}(1,1) = \mathcal{T}_{\beta,t}(\cdot).$
\end{define}

\begin{thm}[High density phase with $\beta > t$]\label{thm: two param beta > t}
    Fix any three parameters $t, \beta,\alpha \in \R_{>0}$ satisfying $t \in (0,0.5)$ and $\alpha,\beta > t$. Let $\mathcal{Z}^{t,\beta}$ be defined as in definition \ref{def: two param stationary condition}. Choose $N_0 \in \Z_{>0}$ such that $t> \frac{1}{N_0}$. Then for all $N\geq N_0$, we have
\begin{equation}\label{eq: Two_param_formula_positive}
        \begin{aligned}
            \frac{1}{\Gamma(\beta - t)}\E \bigg[ \int_{0}^{\infty} 2K_0\left( 2\sqrt{e^{-\tau} \mathcal{Z}^{t,\beta}(N-1,N-1) w} \right) w^{t-\beta-1} e^{-w^{-1}}dw\bigg]= {Q}_{N,\beta >t}^{High}(\tau)
        \end{aligned}
    \end{equation}
    where ${Q}_{N,\beta >t}^{High}(\tau)$ is defined in Definition \ref{def: two param beta > t}.
\end{thm}

\begin{thm}[High density with $\tbeta > \ttt$]\label{thm: High density beta large asymptotic}
    Fix any $\alpha \in \R_{>0}.$
    Fix any three parameters $\tilde{t},\tilde{\beta}, s\in \R$ satisfying $ 0< \tilde{t}<\tilde{\beta}.$ Consider the scaling
    \begin{equation}
        \beta = \frac{\tilde{\beta}}{\Nsigma}, \quad t = \frac{\tilde{t}}{\Nsigma}, \quad \tau = -Nf + (\sigma N)^{1/3}s.
    \end{equation} Let $\mathcal{Z}^{t,\beta}$ be defined as in definition \ref{def: two param stationary condition}. There exists $N_0 \in \Z_{\geq 0}$ such that for all $N \geq N_0$,
    the identity \eqref{eq: Two_param_formula_positive} holds. Then we have the following asymptotic limit:
    \begin{equation}\label{eq: 2 param beta large asymptotic final}
        \begin{aligned}
            \lim_{N\rightarrow \infty} \frac{(\sigma N)^{-1/3}}{\Gamma(\beta - t)}\E \bigg[ \int_{0}^{\infty} 2K_0\left( 2\sqrt{e^{Nf - (\sigma N)^{1/3}s} \mathcal{Z}^{t,\beta}(N-1,N-1) w} \right) w^{t-\beta-1} e^{-w^{-1}}dw\bigg] = \widetilde{Q}_{\tbeta >\ttt}^{High}(s)
        \end{aligned}
    \end{equation}
    where $\widetilde{Q}_{\tbeta >\ttt}^{High}(s)$ is defined in Definition \ref{def: High density beta large asymptotic}
\end{thm}

\begin{thm}[High density with $\tbeta > \ttt$]\label{thm: High density beta large probability}
    Fix any $\alpha \in \R_{>0}.$
    Fix any three parameters $\tilde{t},\tilde{\beta},s\in \R$ satisfying $ \tbeta > \ttt.$ Consider the scaling
    \begin{equation}
        \beta = \frac{\tilde{\beta}}{\Nsigma}, \quad t = \frac{\tilde{t}}{\Nsigma}, \quad \tau = -Nf + \sigma N^{1/3} s.
    \end{equation} Let $Z^{t,\beta}$ be defined as in definition \ref{def: two param stationary condition} and $\widetilde{{Q}}_{\tbeta > \ttt}^{High}$ be defined as in Definition \ref{def: High density beta large asymptotic} with a choice of $0<\delta < \ttt$ for the contour. Then for any $2\ttt>A>\ttt-\delta$, we have  
    \begin{equation}\label{eq: beta large probability final}
    \begin{aligned}
        \Pb\bigg(\lim_{N\rightarrow \infty}&\frac{\log \mathcal{Z}^{t,\beta}(N-1,N-1) +Nf}{(\sigma N)^{1/3}} \leq s\bigg)\\
        &=\left(1 + \frac{1}{\tbeta - \ttt} \partial_{s}\right)\frac{1}{2\pi}\int_{\R}\int_{\R} (-A+\I y)e^{-Ax} e^{\I x y} \left(\widetilde{{Q}}_{\tbeta > \ttt}^{High}(x+s)\right) dxdy.
    \end{aligned}
    \end{equation}
\end{thm}

\begin{thm}[High density phase with $-t<\beta <t$]\label{thm: two param beta < t finite}
    Fix any three parameters $t,\alpha,\beta \in \R$ satisfying $t \in (0,0.5),$ $\beta \in (-t,t)$, and $\alpha > t.$ Choose any $N_0 \in \Z_{\geq 2}$ such that $t>\frac{1}{N_0}.$ Let $\mathcal{Z}^{t,\beta}$ be defined as in definition \ref{def: two param stationary condition}. Then for all $N\geq N_0$, we have
    \begin{equation}\label{Two_param_formula_negative}
        \begin{aligned}
            \frac{1}{\Gamma(\beta - t)}\E \bigg[ \int_{0}^{\infty} 2K_0\left( 2\sqrt{e^{-\tau} {\mathcal{Z}}^{t,\beta}(N-1,N-1) w} \right) w^{t-\beta-1} e^{-w^{-1}}dw\bigg] ={Q}_{N,\beta <t}^{High}(\tau) 
        \end{aligned}
    \end{equation}
    where ${Q}_{N,\beta <t}^{High}(\tau) $ is defined in Definition \ref{def: two param beta < t finite}.
\end{thm}

\begin{thm}[High density phase with $-\ttt<\tbeta <\ttt$]\label{thm: High density beta small asymptotic}
    Fix any $\alpha \in \R$ and any three parameters $\tilde{t}, \tbeta, s \in \R$ satisfying $\tilde{t}>0,$ $\ttt> \tilde{\beta} >-\ttt.$ Consider the scaling
    \begin{equation}
        t = \frac{\tilde{t}}{\Nsigma},\quad \beta = \frac{\tilde{\beta}}{\Nsigma}, \quad \tau = -Nf + (\sigma N)^{1/3}s
    \end{equation} Let $\mathcal{Z}^{t,\beta}$ be defined as in definition \ref{def: two param stationary condition}. There exists $N_0 \in \Z_{>0}$ such that for all $N\geq N_0$ the identity in \eqref{Two_param_formula_negative} holds. Then we have the following asymptotic limit:
    \begin{equation}\label{eq: K_0 version 2 param}
    \begin{aligned}
        \lim_{N\rightarrow \infty}\frac{(\sigma N)^{-1/3}}{\Gamma(\beta - t)}\E \bigg[ \int_{0}^{\infty} 2K_0\left( 2\sqrt{e^{Nf - (\sigma N)^{1/3}s} {\mathcal{Z}}^{t,\beta}(N-1,N-1) w} \right) w^{t-\beta-1} e^{-w^{-1}}dw\bigg] =\widetilde{{Q}}_{\tbeta < \ttt}^{High}(s)
    \end{aligned}
    \end{equation}
  where $\widetilde{{Q}}_{\tbeta < \ttt}^{High}(s)$ is defined in Definition \ref{def: High density beta small asymptotic}.
\end{thm}

\begin{thm}[High density phase with $-\ttt<\tbeta <\ttt$]\label{thm: High density beta small probability}
    Fix any $\alpha \in \R$ and any two parameters $\tilde{t}, \tbeta, s\in \R$ satisfying $\tilde{t}>0,$ $\ttt>\tilde{\beta} >-\tilde{t}.$ Consider the scaling
    \begin{equation}
        t = \frac{\tilde{t}}{\Nsigma},\quad \beta = \frac{\tilde{\beta}}{\Nsigma}, \quad \tau = -Nf + (\sigma N)^{1/3}s
    \end{equation} Let $\mathcal{Z}^{t,\beta}$ be defined as in definition \ref{def: two param stationary condition} and $\widetilde{{Q}}_{\tbeta < \ttt}^{High}$ be defined as in Definition \ref{def: High density beta small asymptotic} with a choice of $\max\{\tbeta, 0\}<\delta <\ttt$. Then for any $2\ttt>A>\ttt-\tbeta$, we have
    \begin{equation}\label{eq: beta small final probability}
    \begin{aligned}
        & \Pb\left(\lim_{N\rightarrow \infty}\frac{\log \mathcal{Z}^{t,\beta}(N-1,N-1) +Nf}{(\sigma N)^{1/3}} \leq s\right)
        =\left(1 + \frac{1}{\tbeta - \ttt} \partial_s\right)\int_{\R}\int_{\R} \frac{(A-\I y)}{2\pi}e^{-Ax} e^{\I x y} \left(\widetilde{{Q}}_{\tbeta < \ttt}^{High}(x+s)\right) dxdy.
    \end{aligned}
    \end{equation}
\end{thm}

\subsection{Connection to existing distributions}
It is known that positive-temperature models such as the log-gamma polymer can degenerate to zero-temperature models such as exponential last-passage percolation (ELPP). There has been work on the one point distribution of the exponential LPP model with product stationary initial condition, see \cite[Theorem 2.7]{PatrikStatExp}. By the KPZ universality, we expect that the asymptotic distribution \eqref{eq: Low final prob} in Theorem \ref{thm: One param probability} should be the same as \cite[Theorem 2.7]{PatrikStatExp}. 
However, the central kernel $\widetilde{\bold{K}}_L$ in \eqref{eq: one param asymptotic formula} is the GSE kernel whereas the central kernel $\overline{\mathcal{A}}$ in \cite[Theorem 2.7]{PatrikStatExp} is the cross-over kernel. It is not clear why these two formulas should match.

For the two-parameter stationary case, we also expect that the distributions in Theorem \ref{thm: High density beta large probability} and Theorem \ref{thm: High density beta small probability} should match the one in \cite[Theorem 1.6 (1.13)]{zeng2025stationary}. This matching would be significantly harder since formulas in either case are very involved. The differences between kernel decompositions of the log-gamma model and geometric LPP (GLPP) are explained in Section \ref{section: difference in models}.

It is clear that one may first take the degeneration from the log-gamma model to the exponential LPP model and then tune parameters to obtain stationary ELPP. Alternatively, one may pass from GLPP to ELPP. In either case, it is reasonable to expect that using the shift argument in \cite[Lemma 3.3]{PatrikStatExp} and \cite[section 6.3.1]{zeng2025stationary} rather than Fourier transform and inversion can equivalently remove extraneous weights at $(1,1)$ (for LD) and $(2,1)$ (for HD). Consequently, we expect that 
\begin{equation}\label{eq: change to shift low}
    \Pb\left(\lim_{N\rightarrow \infty}\frac{\log \mathcal{Z}^{t}(N,N) +Nf}{(\sigma N)^{1/3}} \leq s\right) = \partial_{s}\left(\widetilde{{Q}}_{\ttt}^{Low}(s)\right),
\end{equation}
and similarly, 
\begin{equation}\label{eq: change to shift high}
\begin{aligned}
     \Pb\left(\lim_{N\rightarrow \infty}\frac{\log \mathcal{Z}^{t,\beta}(N-1,N-1) +Nf}{(\sigma N)^{1/3}} \leq r\right)  = \left(\partial_{s} + \frac{1}{\tbeta - \ttt}\partial_s^2\right)\widetilde{{Q}}_{\tbeta > \ttt}^{High}(s) \quad \text{ when }\tbeta > \ttt,\\
     \Pb\left(\lim_{N\rightarrow \infty}\frac{\log \mathcal{Z}^{t,\beta}(N-1,N-1) +Nf}{(\sigma N)^{1/3}} \leq r\right)  = \left(\partial_{s} + \frac{1}{\tbeta - \ttt}\partial_s^2\right)\widetilde{{Q}}_{\tbeta < \ttt}^{High}(s) \quad \text{ when }\tbeta < \ttt,\\
\end{aligned}   
\end{equation}
where $\widetilde{{Q}}_{\ttt}^{Low},$ $\widetilde{{Q}}_{\tbeta > \ttt}^{High},$ $\widetilde{{Q}}_{\tbeta < \ttt}^{High}$
are defined in Definition \ref{def: One param asymptotic}, Definition \ref{def: High density beta large asymptotic}, and Definition \ref{def: High density beta small asymptotic}.

We formulate the above observation as the following conjecture.
\begin{conj} Let $\tbeta,$ $\ttt$, $\alpha,$ $A$ be defined as in Theorem \ref{thm: High density beta small probability} and let $\widetilde{{Q}}_{\tbeta < \ttt}^{High}$ be defined as in Definition \ref{def: High density beta small asymptotic}.
    We expect the following identities of functions:
    \begin{equation}
        \left(1 + \frac{1}{\tbeta - \ttt} \partial_s\right)\frac{1}{2\pi}\int_{\R}\int_{\R} (A-\I y)e^{-Ax} e^{\I x y} \left(\widetilde{{Q}}_{\tbeta < \ttt}^{High}(x+v)\right) dxdy = \widetilde{F}_{-\tbeta, -\ttt}^{HD}(v),
    \end{equation}
    where $\widetilde{F}_{-\tbeta, -\ttt}^{HD}$ is defined in \cite[Theorem 1.6 (1.13)]{zeng2025stationary}.
\end{conj}

\begin{conj}
Let $\tbeta,$ $\ttt$, $\alpha,$ $A$ be defined as in Theorem \ref{thm: High density beta large probability} and let $\widetilde{{Q}}_{\tbeta > \ttt}^{High}$ be defined as in Definition \ref{def: High density beta large asymptotic}. We expect the following identities of functions:
    \begin{equation}
        \left(1 + \frac{1}{\tbeta - \ttt} \partial_s\right)\frac{1}{2\pi}\int_{\R}\int_{\R} (A-\I y)e^{-Ax} e^{\I x y} \left(\widetilde{{Q}}_{\tbeta > \ttt}^{High}(x+v)\right) dxdy = \widetilde{F}_{-\tbeta, -\ttt}^{HD}(v),
    \end{equation}
    where $\widetilde{F}_{-\tbeta, -\ttt}^{HD}$ is defined in \cite[Theorem 1.6 (1.13)]{zeng2025stationary}.
\end{conj}

\begin{conj}
Let $\tbeta,$ $\ttt$, $\alpha,$ $A$ be defined as in Theorem \ref{thm: One param probability} and let $\widetilde{{Q}}_{\ttt}^{Low}$ be defined as in Definition \ref{def: One param asymptotic}. We expect the following identities of functions:
    \begin{equation}
        \int_{\R}\int_{\R} (-A+\I y)e^{-Ax} e^{\I x y} \left(\widetilde{{Q}}_{\ttt}^{Low}(x+v)\right) dxdy = \widetilde{F}_{-\tbeta}^{LD}(v) = F_{0,half}^{(\tbeta,0)}(v),
    \end{equation}
    where $\widetilde{F}_{-\tbeta}^{LD}$ is defined in \cite[Theorem 1.6 (1.15)]{zeng2025stationary} and $F_{0,half}^{(\tbeta,0)}$ is defined in \cite[Theorem 2.7]{PatrikStatExp}.
\end{conj}

\subsection{Our method and novelty}
Our goal is to derive the one-point asymptotic distribution of the half-space log-gamma polymer with stationary initial condition. 
The two-parameter stationary measure of the log-gamma polymer can be scaled to obtain the stationary measure of the half-space KPZ equation, as proved in \cite[Theorem~1.4]{LogGammaStationary}. 
Applying scaling to this stationary measure then yields the stationary measure for the half-space KPZ fixed point, see \cite[Section~4.2]{barraquand2022half}. Our result should imply the one-point distribution of half-space KPZ fixed point with stationary initial condition.

To access this universal distribution, we start with a Fredholm Pfaffian algebraic structure of the half-space log-gamma polymer described in \cite[Theorem~5.10]{imamura2022solvablemodelskpzclass}. 
We specialize the diagonal, first-row, and second-row parameters so that the first two rows reproduce the exact two-parameter stationary measure, and then apply the critical scaling to obtain asymptotics. The main technical challenge is the tuning of parameters. While half-space log-gamma models are typically defined with positive parameters, achieving stationary initial data requires tuning certain parameters to negative values, which causes the Fredholm Pfaffian to diverge. To address this, we decompose the kernel to extract the divergent contributions.

Such kernel decompositions have appeared previously; see \cite[Section~3.2.3]{PatrikStatExp} and \cite[Section~6.6]{zeng2025stationary}. 
Our decomposition in the product stationary case follows similar ideas (see Section~\ref{section: kernel decompose}). 
However, in the two-parameter setting the decomposition is substantially more involved, essentially because the log-gamma kernel in \cite[Theorem 5.10]{imamura2022solvablemodelskpzclass} is symmetric in the diagonal, first-row, and second-row parameters. 
In particular, in the High density regime $0>\beta>-t$, both the diagonal parameter $\beta$ and the first-row parameter must be tuned to negative values. 
We therefore need to extract these poles simultaneously and reorganize them so that the parameters can be analytically continued while the Fredholm Pfaffian of the remaining kernel still converges. 
A detailed comparison between our decomposition and earlier ones is given in Section~\ref{section: difference in models}.
This new decomposition is necessary for recovering the full High density regime.

Our next goal is to translate our limiting formulas \eqref{eq: One_param_formula}, \eqref{eq: Two_param_formula_positive}, \eqref{eq: K_0 version 2 param}, which involve expectations of Bessel function integrals, into exact distribution formulas. In the LPP setting, this issue is resolved by performing a shift argument as described in \eqref{eq: change to shift low} and \eqref{eq: change to shift high}. In the log-gamma case, we need to use the shift argument (for High density regime) together with the fourier inversion. This step requires uniform control over both the upper and lower tails of the scaled free energy. We note that these tail estimates may be of independent probabilistic interest. Accordingly, we provide the first exponential upper bounds for the upper and lower tail of the scaled free energy for the half-space stationary log-gamma model. 

Our strategy is to transport existing tail estimates from the full-space to the half-space geometry. This is facilitated by a distributional correspondence \cite{BarraquandShouda} that identifies the half-space point-to-line partition function with the full-space point-to-point partition function under specific parameter correspondence rules. To employ the distributional identity, one first need to embed the half-space model to a trapezoid model with specific parameter choices. Because stationary models have negative parameters that will be sent to zero under critical scaling, if one naively maps the half-space stationary model to the trapezoid model, either the trapezoid model becomes undefined or no longer exhibits the desired lower tails. To resolve this, we use the monotonicity of the inverse-gamma random variable to construct a sequence of intermediate models to transport the existing tails. However, the derivation of these bounds still entails distinct challenges:
\begin{enumerate}
    \item \textbf{The upper tail}: The stationary models are stochastically larger than its non-stationary counterparts, with the two-parameter stationary model being even larger than the product-stationary model. This hierarchy and the negativity of the parameters make it difficult to identify a dominating model that can be properly embedded to the trapezoid model. Instead, we separate the two-parameter stationary model into two components and establish control over each part. Furthermore, tail bounds for the inhomogeneous full-space log-gamma polymer are not available in the literature. While one could theoretically extract exponential upper bound of order $e^{-cx^{3/2}}$ for upper tails of inhomogeneous full-space model using the same involved steepest descent analysis as in \cite{barraquand2021fluctuations}, such approach will yield an exponentially decaying error term like $e^{-cN}$. In contrast, we provide a more direct steepest descent analysis using the kernel from \cite{imamura2022solvablemodelskpzclass} with much nicer contours and obtain a clean exponential bound that avoids this extra error term. One might expect that the asymptotic formulas \eqref{eq: one param asymptotic formula}, \eqref{eq: 2 param beta large asymptotic final}, and \eqref{Two_param_formula_negative} would provide information about the upper tail of the model. 
In the LPP setting, the upper tail can indeed be read off from such formulas, since the shift argument reduces to differentiation, and differentiating an exponentially decaying tail preserves exponential decay. 
In contrast, in the log-gamma case, applying Fourier inversion to the asymptotics obscures the known exponential upper-tail decay of the formula.
    \item \textbf{The lower tail}: Unlike the upper tail, using the homogeneous half-space model as an intermediate model for the lower tail is straightforward. The difficulty here arises from the nature of the correspondence between half-space and full-space models. Since we want the lower tail of the point-to-point partition function, knowing the point-to-line partition function, which is larger than the point-to-point partition function, seems insufficient. We resolve this by leveraging the stationarity of the model, which induces a random walk behavior with a negative drift along the down-right line. This drift ensures that the point-to-line sum is dominated by the partition function near the diagonal. Using refined random walk estimates, we are able to translate the lower tail of full-space homogeneous point-to-point partition function to the lower tail of half-space partition function on the diagonal. Our method can also derive the lower tail of point-to-point partition function of stationary model away from the diagonal. 

\end{enumerate}

While we expect the optimal upper and lower tail decay to be of the form $e^{-cx^{3/2}}$ and $e^{-cx^{3}}$ as in the case of the full-space model, we focus here on establishing the exponential bounds (linear in $x$ in the exponent) with no $N$-dependent error terms.

\subsection{Organization of Paper}
Sections~2-5 introduce the notations, define the functions appearing in the main results, and introduce the models and their Pfaffian structures. Sections~6-8 are devoted to the reformulation of the kernel, the analytic continuation of the Fredholm Pfaffian formula, and the asymptotic analysis in the product-stationary and two-parameter stationary cases when $\tbeta>\ttt$. Analogously, Sections~9, 11, and~12 treat the two-parameter stationary case when $\tbeta<\ttt$, including the kernel decomposition, analytic continuation, and asymptotic analysis. Section~10 highlights the differences and new features in the kernel decomposition introduced in Section~9. Sections~13--14 establish lower and upper tail estimates for the half-space stationary models. Sections~15--16 provide the lower and upper tail bounds required for the Fourier inversion argument. Finally, Section~17 presents numerical evaluations of the stationary log-gamma polymer and exponential LPP.

\subsection{Acknowledgment}

The authors sincerely thank their advisor, Ivan Corwin, for suggesting this problem and insightful guidance. The authors are especially grateful to Matteo Mucciconi for all the valuable discussions. This research was partially supported by Ivan Corwin’s National Science Foundation grant DMS:2246576 and Simons Investigator in Mathematics award MPS-SIM-00929852.

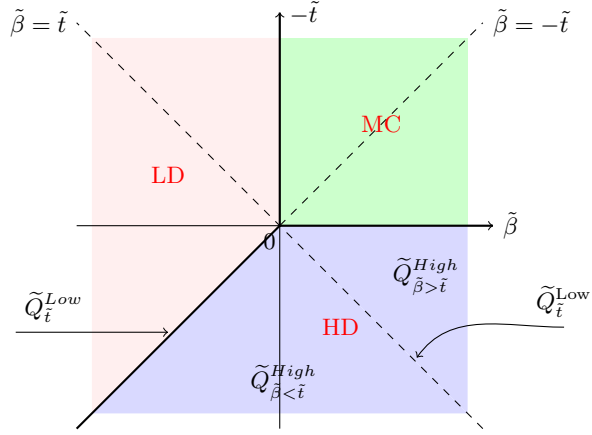
\begin{figure}
\centering
    \begin{tikzpicture}[scale=1.35]
        \fill[green!20] (0,0) -- (1.85,0) -- (1.85,1.85) -- (0,1.85) -- (0,0) -- cycle;
        \fill[red!20,opacity=0.3] (0,0) -- (0,1.85) -- (-1.85,1.85) -- (-1.85,-1.85) -- (0,0) -- cycle;
        \fill[blue!15] (0,0) -- (1.85,0) -- (1.85,-1.85) -- (0,0) -- cycle;
        \fill[blue!15] (0,0) -- (1.85,-1.85) -- (-1.85,-1.85) -- (0,0) -- cycle;
        
        \draw[thin,->] (-2,0) -- (2.1,0)node[right] {\footnotesize{$\tbeta$}};
        \draw[thin,->] (0,-2) -- (0,2.1)node[right] 
        {\footnotesize{$-\ttt$}};

        \draw[thick] (0,0) -- (0,2.1);
        \draw[thick] (0,0) -- (2.1,0);
        
        \draw[thick] (-2,-2) -- (0,0);
        \draw[thin, dashed] (0,0) -- (2,2) node[right] {\footnotesize{$\tbeta = -\ttt$}};
        \draw[thin,dashed] (2,-2) -- (-2,2) node[left] {\footnotesize{$\tbeta = \ttt$}};

        \node at (-0.1,-0.15) {\footnotesize{$0$}};
        
        \node[text = {red}] at (0.6,-1){\footnotesize{HD}};
        \node[text = {red}] at (1,1){\footnotesize{MC}};
        \node[text = {red}] at (-1.1,0.5){\footnotesize{LD}};

    
      \draw[thin,->]
        (-2.6,-1.05) -- (-1.1,-1.05)
        node[near start,above]{\footnotesize $\widetilde{Q}^{Low}_{\ttt}$};
    
      \draw[thin,->]
        (2.8,-1) node[above]{\footnotesize $\widetilde{Q}^{\mathrm{Low}}_{\ttt}$}to[out=180,in=55] (1.35,-1.3);

    \node[text = {black}] at (1.4,-0.5){\footnotesize{ $\widetilde{Q}_{\tbeta>\ttt}^{High}$}};
    \node[text = {black}] at (0,-1.55){\footnotesize{ $\widetilde{Q}_{\tbeta<\ttt}^{High}$}};
        
    \end{tikzpicture}
    \captionof{figure}{This picture describes the phase diagram of the half-space stationary log-gamma polymer model under critical scaling, which sets $t = \ttt/(\sigma N)^{1/3}$ and $\beta = \tbeta/(\sigma N)^{1/3}.$ We characterize the distribution of $\lim_{N\rightarrow \infty}\frac{\log Z(N,N) + Nf}{(\sigma N)^{1/3}}$ under three different cases. Under the High density phase, when $\tbeta > \ttt$, we have $\widetilde{Q}_{\tbeta > \ttt}^{High}$ and when $-\ttt<\tbeta < \ttt$, we have $\widetilde{Q}_{\tbeta < \ttt}^{High}$ as limiting distributions. Under the Low density phase, we obtain $\widetilde{Q}_{\ttt}^{Low}$ as the limiting distribution. In particular, we have $\widetilde{Q}_{\ttt}^{Low}$ on the line $\tbeta = \ttt$, $\ttt>0$ in the High density regime.}
    \label{phaseDiagram}
\end{figure}

\section{Basic notations}

\label{section:notation}
In this paper, we will have two kinds of limits. We consistently use $\widehat{F}$, ${F}^L$(respectively) to denote the limit $\lim_{u\rightarrow -t}F(t)$ in High density and Low density regimes, respectively. We consistently use $\widetilde{F}$, $\widetilde{F}^L$ to denote the asymptotic limit of $\widehat{F}$, ${F}^L$ under the critical scaling in the High density and Low density regimes, respectively.

We use the bracket notation to denote the scalar product on $\mathbb{L}^2((\tau,\infty))$, that is
\begin{equation}\label{ellInnerProduct}
\begin{aligned}
    &\braket{f}{g}=\braket{f(x)}{g(x)} = \int_{\tau}^{\infty} f(x)g(x)dx,\\
    &\brabarket{f}{h}{g} = \brabarket{f(x)}{h(x,y)}{g(y)} = \int_{\tau}^{\infty}\int_{\tau}^{\infty}f(x)h(x,y)g(y)dx dy,\\
    &K\ket{g}(x) = \int_{\tau}^{\infty} K(x,y) g(y)dy.
\end{aligned}
\end{equation}
We use $\ketbra{f}{g}$ to denote the outer product kernel
\begin{equation}
    \ketbra{f}{g}(x,y) = f(x)g(y).
\end{equation}
We will remove the parameters $(x,y)$ when the context is clear.
We define the contour
\begin{equation}
\begin{aligned}
    C(\delta;\theta) = \{\delta + re^{\text{sign}(r)\theta} : r\in \R\},\\
\end{aligned}
\end{equation}
where the choice of $\delta$ will be specified for each function.
We assume the positive orientation of the contour.

Throughout this paper, we define $K_{\nu}(z)$ to be the modified Bessel function as in \cite{abramowitz1965handbook}.

\section{Definition of functions}
\begin{define}[Low density phase]\label{def: one_param_finite}
    Fix any two parameters $t,\alpha \in \R_{>0}$ satisfying $t \in (0,0.5)$ and $\alpha > t.$ We define
    \begin{equation}
        \begin{aligned}
             &Q_{N,t}^{Low}(\tau):=\Pf\left( J - \widehat{\bold{K}}_L \right)\left(\widehat{C}_L -1\right)+\Pf\left(J - \widehat{\bold{K}}_L - \ketbra{\begin{array}{c} \widehat{\phi}_2^L \\
        -\widehat{\phi}_1^L
        \end{array}}{\widehat{B}_L \quad -\widehat{B}_L^{\prime}} - \ketbra{\begin{array}{c}
            \widehat{B}_L\\
            -\widehat{B}_L^{\prime}
        \end{array}}{-\widehat{\phi}_2^L \quad \widehat{\phi}_1^L }\right),
        \end{aligned}
    \end{equation}
    where Fredholm Pfaffian is taken over $\mathbb{L}^2((\tau,\infty))$. 
\end{define}

 Let $\I\R - d$ denote a vertical line going from $-d- \I\infty$ to $-d + \I\infty$ with $0<d< t$ and $\frac{1}{N}<d$. We omit the dependence of $Q_{N,t}^{Low}(\tau)$ on $d$.
We define all functions used in the above definition.

\begin{equation}
    \begin{aligned}
        G_{\alpha}(X)&:= \left(\frac{\Gamma(\alpha + X)}{\Gamma(\alpha - X)}\right)^{N-1},\quad Q(Z,W) := \Gamma(-2Z)\Gamma(-2W)\frac{\sin(\pi(Z-W))}{\sin(\pi(Z+W))},\\
        \widehat{\phi}_1^L(X)
        &:= -\Gamma(2t)G_{\alpha}(t)
        \int \limits_{\I\R - d}\frac{dZ}{2\pi\I} \int\limits_{\I\R-d} \frac{dW}{2\pi\I} Ze^{XZ}e^{\tau(t+W)}\\
        &\frac{\Gamma(-t+Z)}{\Gamma(-t-Z)}\frac{\Gamma(t+Z)}{\Gamma(t-Z)}\frac{\Gamma(-t+W+1)}{\Gamma(-t-W+1)}\frac{\Gamma(t+W)}{\Gamma(t-W)}G_\alpha(Z)G_\alpha(W)
        Q(Z,W),\\
        \widehat{\phi}_2^L(X)&:= -\Gamma(2t)G_{\alpha}(t)
        \int \limits_{\I\R - d}\frac{dZ}{2\pi\I} \int\limits_{\I\R-d} \frac{dW}{2\pi\I} e^{XZ}e^{\tau(t+W)}\\
        &\frac{\Gamma(-t+Z)}{\Gamma(-t-Z)}\frac{\Gamma(t+Z)}{\Gamma(t-Z)}
        \frac{\Gamma(-t+W+1)}{\Gamma(-t-W+1)}\frac{\Gamma(t+W)}{\Gamma(t-W)}G_\alpha(Z)G_\alpha(W)
        Q(Z,W),\\
    \widehat{B}_L(X) &:= \int \limits_{\I\R - d} \frac{dW}{2\pi\I} e^{XW}\frac{\Gamma(t+W)}{\Gamma(t-W)}\frac{\Gamma(1+t+W)}{\Gamma(1+t-W)}G_{\alpha}(W)\Gamma(-2W),\\
    \widehat{B}'_L(X) 
    &:= \int \limits_{\I\R - d} \frac{dW}{2\pi\I} We^{XW}\frac{\Gamma(t+W)}{\Gamma(t-W)}\frac{\Gamma(1+t+W)}{\Gamma(1+t-W)}G_{\alpha}(W)\Gamma(-2W),\\
    \widehat{C}_L &:= G_{\alpha}(t)\Gamma(2t)\int \limits_{\I\R -d} \frac{dW}{2\pi\I} e^{\tau(t+W)}\frac{\Gamma(t+W)^2}{\Gamma(t-W)^2} G_\alpha(W)\Gamma(-2W).
    \end{aligned}
\end{equation}

    \begin{equation}
        \widehat{\bold{K}}_L(X,Y) := \begin{pmatrix}
            \widehat{K}_L(X,Y) & -\partial_Y \widehat{K}_L(X,Y)\\
            - \partial_X \widehat{K}_L(X,Y) & \partial_X\partial_Y \widehat{K}_L(X,Y)
        \end{pmatrix}
    \end{equation} with
    \begin{equation}
    \begin{split}
        {\widehat{K}_L}(X,Y) := &\int\limits_{\mathrm{i}\mathbb{R}-d_0} \frac{dZ}{2\pi i}
\int\limits_{\mathrm{i}\mathbb{R}-d_0} \frac{dW}{2\pi i}\;
e^{XZ+YW}\frac{\Gamma(-t+Z)}{\Gamma(-t-Z)}\frac{\Gamma(t+Z)}{\Gamma(t-Z)}\frac{\Gamma(-t+W)}{\Gamma(-t-W)}\frac{\Gamma(t+W)}{\Gamma(t-W)}\\
&
G_{\alpha}(Z)\,
G_{\alpha}(W)\,
Q(Z,W).
\end{split}
\end{equation}

\begin{define}[Low density phase]\label{def: One param asymptotic}
    Fix any $\alpha,\tilde{t} \in \R_{>0}$ and $s \in \R$. We define
    \begin{equation}
        \begin{aligned}
            \widetilde{Q}_{\ttt}^{Low}(s) &:= \Pf\left( J - \widetilde{\bold{K}}_L \right)\left(\widetilde{C}_L -1\right)
            +\Pf\left( J - \widetilde{\bold{K}}_L - \ketbra{\begin{array}{c} \widetilde{\phi}_2^L \\
        -\widetilde{\phi}_1^L
        \end{array}}{\widetilde{B}_L \quad -\widetilde{B}^{\prime}_L} - \ketbra{\begin{array}{c}
            \widetilde{B}_L\\
            -\widetilde{B}'_L
        \end{array}}{-\widetilde{\phi}_2^L \quad \widetilde{\phi}_1^L }\right)
        \end{aligned}
    \end{equation}
    where Fredholm Pfaffian is taken over $\mathbb{L}^2((s, \infty)).$
\end{define}
Choose $\ttt>\delta >0$. We omit the dependence of $\widetilde{Q}_{\ttt}^{Low}$ on $\delta$.
We define all functions used in the above definition.
\begin{equation}
    \begin{aligned}
        \widetilde{C}_L := &-\int\limits_{C(\delta;\pi/3)} \frac{d\tW}{2\pi \I} \left(\frac{1}{\tW} + \frac{1}{\ttt}\right)\frac{(\tW + \ttt)}{4(\tW-\ttt )^2} e^{-\frac{\ttt^3}{3} + s\ttt + \frac{\tW^3}{3} - s\tW},\\
        \widetilde{B}_L(v) := & -\int\limits_{C(\delta;\pi/3)} \frac{d\tW}{2\pi \I} \frac{(\ttt+\tW)}{2\tW(\ttt-\tW)} e^{ \frac{\tW^3}{3} - v\tW}, \quad  \widetilde{B}'_L(v) =  \int\limits_{C(\delta;\pi/3)} \frac{d\tW}{2\pi \I}  \frac{(\ttt+\tW)}{2(\ttt-\tW)} e^{ \frac{\tW^3}{3} - v\tW}.
    \end{aligned}
\end{equation}
\begin{equation}
    \widetilde{\bold{K}}_L(u,v) := \begin{pmatrix}
        \widetilde{K}_L(u,v) & -\partial_v \widetilde{K}_L(u,v)\\
        -\partial_u\widetilde{K}_L(u,v) & \partial_u\partial_v \widetilde{K}_L(u,v),
    \end{pmatrix}
\end{equation}
where
\begin{equation}
    \begin{aligned}
        \widetilde{K}_L(u,v):=  \int\limits_{C(\delta;\pi/3)} \frac{d\tZ}{2\pi \I }\int\limits_{C(\delta;\pi/3)} \frac{d\tW}{2\pi \I} \frac{(\tZ - \tW)}{4\tZ \tW(\tZ + \tW)} e^{\frac{\tZ^3}{3} - u\tZ + \frac{\tW^3}{3} - v\tW }.
    \end{aligned}
\end{equation}

\begin{equation}
\begin{aligned}
     \widetilde{\phi}_1^L(u) := & \int\limits_{C(\delta;\pi/3)} \frac{d\tZ}{2\pi \I }\int\limits_{C(\delta;\pi/3)} \frac{d\tW}{2\pi \I} \frac{(\tZ - \tW)}{8\ttt \tW(\tZ + \tW)}\frac{(\ttt+\tW)}{(\ttt-\tW)} e^{\frac{\tZ^3}{3} - u\tZ + \frac{\tW^3}{3} - s\tW - \frac{\ttt^3}{3} + s\ttt },\\
     \widetilde{\phi}_2^L(u) := & -\int\limits_{C(\delta;\pi/3)} \frac{d\tZ}{2\pi \I }\int\limits_{C(\delta;\pi/3)} \frac{d\tW}{2\pi \I} \frac{(\tZ - \tW)}{8\ttt\tZ \tW(\tZ + \tW)}\frac{(\ttt+\tW)}{(\ttt-\tW)} e^{\frac{\tZ^3}{3} - u\tZ + \frac{\tW^3}{3} - s\tW - \frac{\ttt^3}{3} + s\ttt }.
\end{aligned}
\end{equation}

\begin{define}[High density phase with $\beta > t$]\label{def: two param beta > t}
    Fix any three parameters $t, \beta,\alpha \in \R_{>0}$ satisfying $t \in (0,0.5)$ and $\alpha,\beta > t$. We define
\begin{equation}
        \begin{aligned}
            {Q}_{N,\beta >t}^{High}(\tau) &:= \Pf\left( J - \widehat{\bold{K}} \right)\left(\widehat{C} -1\right)+\Pf\left(J - \widehat{\bold{K}} - \ketbra{\begin{array}{c} \widehat{\phi}_2 \\
        -\widehat{\phi}_1
        \end{array}}{\widehat{B} \quad -\widehat{B}^{\prime}} - \ketbra{\begin{array}{c}
            \widehat{B}\\
            -\widehat{B}^{\prime}
        \end{array}}{-\widehat{\phi}_2 \quad \widehat{\phi}_1 }\right)
        \end{aligned}
    \end{equation}
    where Fredholm Pfaffian is taken over $\mathbb{L}^2((\tau,\infty))$.
\end{define}
Let $\I\R - d$ denote a vertical line going from $-d- \I\infty$ to $-d + \I\infty$ with $0<d< t$ and $\frac{1}{N}<d$. We omit the dependence of ${Q}_{N,\beta >t}^{High}$ on $\delta$.
We define all functions used in the above theorem.

\begin{equation}
    \begin{aligned}
        G_{\alpha,\beta}(X)&:= \left(\frac{\Gamma(\beta + X)}{\Gamma(\beta - X)}\right)\left(\frac{\Gamma(\alpha + X)}{\Gamma(\alpha - X)}\right)^{N-2},\quad Q(Z,W) := \Gamma(-2Z)\Gamma(-2W)\frac{\sin(\pi(Z-W))}{\sin(\pi(Z+W))},\\
        \widehat{\phi}_1(X)
        &:= -\Gamma(2t)G_{\alpha,\beta}(t)
        \int \limits_{\I\R - d}\frac{dZ}{2\pi\I} \int\limits_{\I\R-d} \frac{dW}{2\pi\I} Ze^{XZ}e^{\tau(t+W)}\\
        &\frac{\Gamma(-t+Z)}{\Gamma(-t-Z)}\frac{\Gamma(t+Z)}{\Gamma(t-Z)}\frac{\Gamma(-t+W+1)}{\Gamma(-t-W+1)}\frac{\Gamma(t+W)}{\Gamma(t-W)}G_{\alpha,\beta}(Z)G_{\alpha,\beta}(W)
        Q(Z,W),\\
        \widehat{\phi}_2(X)&:= -\Gamma(2t)G_{\alpha,\beta}(t)
        \int \limits_{\I\R - d}\frac{dZ}{2\pi\I} \int\limits_{\I\R-d} \frac{dW}{2\pi\I} e^{XZ}e^{\tau(t+W)}\\
        &\frac{\Gamma(-t+Z)}{\Gamma(-t-Z)}\frac{\Gamma(t+Z)}{\Gamma(t-Z)}
        \frac{\Gamma(-t+W+1)}{\Gamma(-t-W+1)}\frac{\Gamma(t+W)}{\Gamma(t-W)}G_{\alpha,\beta}(Z)G_{\alpha,\beta}(W)
        Q(Z,W).
        \end{aligned}
        \end{equation}
    \begin{equation}
    \begin{aligned}
    \widehat{B}(X) &:= \int \limits_{\I\R - d} \frac{dW}{2\pi\I} e^{XW}\frac{\Gamma(t+W)}{\Gamma(t-W)}\frac{\Gamma(1+t+W)}{\Gamma(1+t-W)}G_{\alpha,\beta}(W)\Gamma(-2W),\\
    \widehat{B}'(X) 
    &:= \int \limits_{\I\R - d} \frac{dW}{2\pi\I} We^{XW}\frac{\Gamma(t+W)}{\Gamma(t-W)}\frac{\Gamma(1+t+W)}{\Gamma(1+t-W)}G_{\alpha,\beta}(W)\Gamma(-2W),\\
    \widehat{C} &:= G_{\alpha,\beta}(t)\Gamma(2t)\int \limits_{\I\R -d} \frac{dW}{2\pi\I}e^{\tau(t+W)}\frac{\Gamma(t+W)^2}{\Gamma(t-W)^2} G_{\alpha,\beta}(W)\Gamma(-2W).\\
    \end{aligned}
\end{equation}

    \begin{equation}
        \widehat{\bold{K}}(X,Y) := \begin{pmatrix}
            \widehat{K}(X,Y) & -\partial_Y \widehat{K}(X,Y)\\
            - \partial_X \widehat{K}(X,Y) & \partial_X\partial_Y \widehat{K}(X,Y)
        \end{pmatrix}
    \end{equation} with
    \begin{equation}
    \begin{split}
        {\widehat{K}}(X,Y) := &\int\limits_{\mathrm{i}\mathbb{R}-d_0} \frac{dZ}{2\pi i}
\int\limits_{\mathrm{i}\mathbb{R}-d_0} \frac{dW}{2\pi i}\;
e^{XZ+YW}\frac{\Gamma(-t+Z)}{\Gamma(-t-Z)}\frac{\Gamma(t+Z)}{\Gamma(t-Z)}\frac{\Gamma(-t+W)}{\Gamma(-t-W)}\frac{\Gamma(t+W)}{\Gamma(t-W)}\\
&
G_{\alpha,\beta}(Z)\,
G_{\alpha,\beta}(W)\,
Q(Z,W).
\end{split}
\end{equation}

\begin{rmk}
    One may ask whether the product stationary case can be obtained by taking the limit $\beta \to t$ in the above formula, since the two-parameter stationary initial condition then degenerates to the product stationary one. However, two obstacles arise. 
First, Definition~\ref{def: two param stationary condition} introduces a random shift at $(2,2)$ through $\mathcal{Z}^{t,\beta}(1,1)=X\stackrel{(d)}{=}\Gamma(\beta+t)$; under critical scaling this contribution remains nontrivial, leading to a worse result than the direct Low density formula $\widehat{Q}_{N,t}^{Low}$. 
Second, setting $\beta=t$ at finite $N$ forces
\[
A = e^{-Xu}\left(\frac{\Gamma(t-u)}{\Gamma(t+u)}\right)^2
    \left(\frac{\Gamma(\alpha-u)}{\Gamma(\alpha+u)}\right)^{N-2},
\] which is defined in \eqref{eq: low def of A}.
The denominator of $A$ contains $\Gamma(t+u)^2$. This factor cannot be simply compensated by removing the $(1,1)$ weight (which produces only one $\Gamma(t+u)$ factor), making the computation substantially more difficult. 
Finally, the Maximal current phase cannot be recovered by taking $t\to 1$, essentially due to the symmetry of the kernel.
\end{rmk}

\begin{define}[High density with $\tbeta > \ttt$]\label{def: High density beta large asymptotic}
    Fix any $\alpha \in \R_{>0}.$
    Fix any three parameters $\tilde{t},\tilde{\beta}, s\in \R$ satisfying $ 0< \tilde{t}<\tilde{\beta}.$ We define
    \begin{equation}
        \begin{aligned}
            \widetilde{Q}_{\tbeta >\ttt}^{High}(s) &:= \Pf\left( J - \widetilde{\bold{K}} \right) \left(\widetilde{C} -1\right)
            +\Pf\left( J - \widetilde{\bold{K}} - \ketbra{\begin{array}{c} \widetilde{\phi}_2 \\
        -\widetilde{\phi}_1
        \end{array}}{\widetilde{B} \quad -\widetilde{B}^{\prime}} - \ketbra{\begin{array}{c}
            \widetilde{B}\\
            -\widetilde{B}^{\prime}
        \end{array}}{-\widetilde{\phi}_2 \quad \widetilde{\phi}_1 }\right)
        \end{aligned}
    \end{equation}
    where Fredholm Pfaffian is taken over $\mathbb{L}^2((s, \infty)).$
\end{define}

We define all functions used in the above theorem. The dependence of $\widetilde{Q}_{\tbeta >\ttt}^{High}$ on $\delta$ is omitted.
\begin{equation}\label{eq: some equations beta>t 2 param}
    \begin{aligned}
        \widetilde{C} := &- \frac{\tbeta - \ttt}{\tbeta+\ttt}\int\limits_{C(\delta;\pi/3)} \frac{d\tW}{2\pi \I} \left(\frac{1}{\tW} + \frac{1}{\ttt}\right)\frac{(\tW + \ttt)}{4(\tW-\ttt )^2}\frac{(\tbeta+\tW)}{(\tbeta-\tW)} e^{-\frac{\ttt^3}{3} + s\ttt + \frac{\tW^3}{3} - s\tW},\\
        \widetilde{B}(v) := & -\int\limits_{C(\delta;\pi/3)} \frac{d\tW}{2\pi \I} \frac{(\ttt+\tW)}{2\tW(\ttt-\tW)} \frac{(\tbeta+\tW)}{(\tbeta-\tW)} e^{ \frac{\tW^3}{3} - v\tW}\\
        \widetilde{B}'(v) := &  \int\limits_{C(\delta;\pi/3)} \frac{d\tW}{2\pi \I}  \frac{(\ttt+\tW)}{2(\ttt-\tW)}\frac{(\tbeta+\tW)}{(\tbeta-\tW)} e^{ \frac{\tW^3}{3} - v\tW}.
    \end{aligned}
\end{equation}
\begin{equation}
    \widetilde{\bold{K}}(u,v) := \begin{pmatrix}
        \widetilde{K}(u,v) & -\partial_v \widetilde{K}(u,v)\\
        -\partial_u\widetilde{K}(u,v) & \partial_u\partial_v \widetilde{K}(u,v)
    \end{pmatrix},
\end{equation}
where
\begin{equation}
    \begin{aligned}
        \widetilde{K}(u,v):=  \int\limits_{C(\delta;\pi/3)} \frac{d\tZ}{2\pi \I }\int\limits_{C(\delta;\pi/3)} \frac{d\tW}{2\pi \I} \frac{(\tZ - \tW)}{4\tZ \tW(\tZ + \tW)}\frac{(\tbeta+\tW)}{(\tbeta-\tW)}\frac{(\tbeta+\tZ)}{(\tbeta-\tZ)} e^{\frac{\tZ^3}{3} - u\tZ + \frac{\tW^3}{3} - v\tW }.
    \end{aligned}
\end{equation}

\begin{equation}
\begin{aligned}
     \widetilde{\phi}_1(u) := & \frac{\tbeta - \ttt}{\tbeta+\ttt}\int\limits_{C(\delta;\pi/3)} \frac{d\tZ}{2\pi \I }\int\limits_{C(\delta;\pi/3)} \frac{d\tW}{2\pi \I} \frac{(\tZ - \tW)}{8\ttt \tW(\tZ + \tW)}\frac{(\ttt+\tW)}{(\ttt-\tW)} \frac{(\tbeta+\tW)}{(\tbeta-\tW)}\frac{(\tbeta+\tZ)}{(\tbeta-\tZ)}  \\
     &\quad e^{\frac{\tZ^3}{3} - v\tZ + \frac{\tW^3}{3} - s\tW - \frac{\ttt^3}{3} + s\ttt },\\
     \widetilde{\phi}_2(u) := & - \frac{\tbeta - \ttt}{\tbeta+\ttt}\int\limits_{C(\delta;\pi/3)} \frac{d\tZ}{2\pi \I }\int\limits_{C(\delta;\pi/3)} \frac{d\tW}{2\pi \I} \frac{(\tZ - \tW)}{8\ttt\tZ \tW(\tZ + \tW)}\frac{(\ttt+\tW)}{(\ttt-\tW)}\frac{(\tbeta+\tW)}{(\tbeta-\tW)}\frac{(\tbeta+\tZ)}{(\tbeta-\tZ)}  \\
     & \quad e^{\frac{\tZ^3}{3} - v\tZ + \frac{\tW^3}{3} - s\tW - \frac{\ttt^3}{3} + s\ttt }.
\end{aligned}
\end{equation}

\begin{define}[High density phase with $-t<\beta <t$]\label{def: two param beta < t finite}
    Fix any three parameters $t,\alpha,\beta \in \R$ satisfying $t \in (0,0.5),$ $\beta \in (-t,t)$, and $\alpha > t.$ We define
    \begin{equation}
        \begin{aligned}
            {Q}_{N,\beta <t}^{High}(\tau) &:= \frac{1}{\Pf\left(J - \widehat{\boldsymbol{\mathcal{K}}}\right)}\left(\widehat{\M}\widehat{\NN} - \widehat{\A}\widehat{\D} + \widehat{\B}\widehat{\C}\right).
        \end{aligned}
    \end{equation}
    The Fredholm Pfaffian is taken over $\mathbb{L}^2((\tau,\infty))$.
\end{define}

Let $\I\R - d$ denote a vertical line going from $-d- \I\infty$ to $-d + \I\infty$ with $\beta<d< t$ and $\frac{1}{N}<d$. We omit the dependence of ${Q}_{N,\beta <t}^{High}$ on $\delta$.
We define each function used in the above theorem. The well-definedness of $\frac{1}{\Pf\left(J - \widehat{\boldsymbol{\mathcal{K}}}\right)}$ is explained in the proof of Theorem~\ref{thm: two param beta < t finite}.

\begin{equation}\label{eq: definitions of H,F,Q}
    H_{\alpha}(Z) := \left(\frac{\Gamma(\alpha + Z)}{\Gamma(\alpha - Z)}\right)^{N-2}\!\!\!\!,\quad F(W):= \frac{\Gamma(t+W)}{\Gamma(t-W)}\frac{\Gamma(\beta + W)}{\Gamma(\beta - W)}, \quad Q(Z,W) := \Gamma(-2Z)\Gamma(-2W)\frac{\sin(\pi(Z-W))}{\sin(\pi(Z+W))}.
\end{equation}
\begin{equation}\label{eq: def of A}
\begin{aligned}
    c_0 := &\frac{\Gamma(-t-\beta)}{\Gamma(-t+\beta)}\frac{\Gamma(t-\beta)}{\Gamma(t+\beta)}H_{\alpha}(-\beta),\\
    \A_1 := &c_0 \int \limits_{\I\R - d} \frac{dW}{2\pi\I} \frac{-(\beta+W)}{(\beta - W)}e^{-\tau(\beta - W)}\frac{\Gamma(1+t+W)}{\Gamma(1+t-W)}\Gamma(-2W)F(W)H_{\alpha}(W),\\
    \A_2  :=&c_0 \int\limits_{\I\R-d}\frac{dU}{2\pi\I}\int\limits_{\I\R-d}\frac{dZ}{2\pi\I}\int\limits_{\I\R-d}\frac{dW}{2\pi\I}\frac{(Z-U)(W+\beta)}{(Z+U)(\beta - W)}e^{\tau(U+Z+W-\beta)}\frac{\Gamma(-t+Z)\Gamma(-t+W)}{\Gamma(-t-Z)\Gamma(-t-W)}\\&
    F(U)H_{\alpha}(U)\Gamma(-2U)\frac{\Gamma(1+t+U)}{\Gamma(1+t-U)}
    F(Z)H_{\alpha}(Z)F(W)H_{\alpha}(W)Q(Z,W).
\end{aligned}    
\end{equation}

\begin{equation}\label{eq: def of B}
    \begin{aligned}
        \B_1 :=& \int \limits_{\I\R - d}\frac{dZ}{2\pi\I}\int \limits_{\I\R - d}\frac{dW}{2\pi\I} \frac{(W-Z)}{(Z+W)}e^{\tau(Z+W)}\frac{\Gamma(t+Z)}{\Gamma(t-Z)}\frac{\Gamma(-t+Z)}{\Gamma(-t-Z)}\frac{\Gamma(1-\beta + Z)}{\Gamma(1-\beta - Z)}\frac{\Gamma(1+t+W)}{\Gamma(1+t-W)}\\
        &\Gamma(-2Z)\Gamma(-2W)F(W)H_{\alpha}(W)H_{\alpha}(Z) + \frac{\Gamma(\beta - t)}{\Gamma(\beta + t)}\frac{\Gamma(1-\beta + t)}{\Gamma(1-\beta - t)},\\
        \B_2 := & \int \limits_{\I\R - d}\frac{dU}{2\pi\I}\int \limits_{\I\R - d}\frac{dZ}{2\pi\I}\int \limits_{\I\R - d}\frac{dW}{2\pi\I}\int \limits_{\I\R - d}\frac{dV}{2\pi\I}\frac{(Z-U)(V-W)}{(Z+U)(V+W)}e^{\tau(Z+U+V+W)}F(Z)F(W)H_{\alpha}(Z)H_{\alpha}(W)Q(Z,W)\\
        &\frac{\Gamma(-t+Z)}{\Gamma(-t-Z)}\frac{\Gamma(-t+W)}{\Gamma(-t-W)}
        \frac{\Gamma(t+V)}{\Gamma(t-V)}\frac{\Gamma(-t+V)}{\Gamma(-t-V)}\frac{\Gamma(1-\beta + V)}{\Gamma(1-\beta - V)}\frac{\Gamma(1+t+U)}{\Gamma(1+t-U)}\Gamma(-2V)\Gamma(-2U)F(U)H_{\alpha}(U)H_{\alpha}(V).
    \end{aligned}
\end{equation}

\begin{equation}\label{eq: def of C}
    \begin{aligned}
        c_1 := & \frac{\Gamma(-t-\beta)}{\Gamma(-t+\beta)}\frac{\Gamma(t-\beta)}{\Gamma(\beta - t)}\Gamma(2t)H_{\alpha}(-\beta) H_{\alpha}(t),\quad
        \C_1 := c_1\frac{(\beta + t)}{(\beta - t)}e^{-\tau(\beta - t)},\\
        \C_2 := & c_1 \int\limits_{\I\R - d}\frac{dZ}{2\pi\I}\int\limits_{\I\R - d}\frac{dW}{2\pi\I}\frac{(Z-t)(\beta + W)}{(t+Z)(W-\beta)}e^{\tau(t+Z-\beta +W)}\frac{\Gamma(-t+Z)}{\Gamma(-t-Z)}\frac{\Gamma(-t+W)}{\Gamma(-t-W)}F(Z)F(W)H_{\alpha}(Z)H_{\alpha}(W)Q(Z,W).
    \end{aligned}
\end{equation}

\begin{equation}\label{eq: def of D}
    \begin{aligned}
        \D_1 := & \int\limits_{\I \R -d} \frac{dW}{2\pi\I}\frac{W-t}{W+t}e^{\tau(t+W)}
        \frac{\Gamma(-t+W)}{\Gamma(-t-W)}\frac{\Gamma(t+W)}{\Gamma(t-W)}\Gamma(-2W)\frac{\Gamma(1-\beta+W)}{\Gamma(1-\beta-W)}\Gamma(2t)\frac{\Gamma(\beta+t)}{\Gamma(\beta-t)}H_\alpha(t)H_\alpha(W),\\
        \D_2 := & \int\limits_{\I \R -d} \frac{dW}{2\pi \I} \int\limits_{\I \R -d} \frac{dZ}{2\pi \I} \int\limits_{\I \R -d} \frac{dV}{2\pi\I}
        \frac{(Z-t)(W-V)}{(Z+t)(V+W)}e^{\tau(V+W+Z+t)}
        \frac{\Gamma(-t+Z)}{\Gamma(-t-Z)}\frac{\Gamma(-t+W)}{\Gamma(-t-W)}
        \frac{\Gamma(t+V)}{\Gamma(t-V)}\\&\frac{\Gamma(-t+V)}{\Gamma(-t-V)}H_\alpha(V)\Gamma(-2V)
        \frac{\Gamma(1-\beta+V)}{\Gamma(1-\beta-V)}\Gamma(2t)\frac{\Gamma(\beta+t)}{\Gamma(\beta-t)}H_\alpha(t)F(Z)H_\alpha(Z)F(W)H_\alpha(W)Q(Z,W).
    \end{aligned}
\end{equation}

\begin{equation}\label{eq: def of M}
    \begin{aligned}
        \M_1 := & \Gamma(2t)\frac{\Gamma(\beta + t)}{\Gamma(\beta - t)}H_\alpha(t)\int\limits_{\I \R -d} \frac{dW}{2\pi\I}\frac{W-t}{-(W+t)}e^{\tau(t+W)}
        \frac{\Gamma(1+t+W)}{\Gamma(1+t-W)}\Gamma(-2W)F(W)H_\alpha(W),\\
        \M_2 := &\Gamma(2t)\frac{\Gamma(\beta + t)}{\Gamma(\beta - t)}H_{\alpha}(t)\int \limits_{\I\R-d} \frac{dZ}{2\pi\I}\int \limits_{\I\R-d} \frac{dW}{2\pi\I} \int \limits_{\I\R-d} \frac{dV}{2\pi\I} \frac{(Z-V)(W-t)}{(W+t)(V+Z)}e^{\tau(t + W + V+Z)}\frac{\Gamma(-t+Z)}{\Gamma(-t-Z)}\frac{\Gamma(-t+W)}{\Gamma(-t-W)}\\
        & \Gamma(-2V)\frac{\Gamma(1+t+V)}{\Gamma(1-t-V)} F(V)H_\alpha(V)F(Z)F(W)H_{\alpha}(Z)H_{\alpha}(W)Q(Z,W).
    \end{aligned}
\end{equation}

\begin{equation}\label{eq: def of N}
    \begin{aligned}
        \NN_1 :=&1+ c_0 \int\limits_{\I \R -d} \frac{dW}{2\pi\I} \frac{\Gamma(-t+W)}{\Gamma(-t-W)}\frac{\Gamma(t+W)}{\Gamma(t-W)}H_\alpha(W)\Gamma(-2W)\frac{\Gamma(1-\beta+W)}{\Gamma(1-\beta-W)}\frac{\beta + W}{\beta - W} e^{-\tau(\beta-W)},\\
        \NN_2 := &c_0 \int\limits_{\I \R -d} \frac{dW}{2\pi\I}
        \int\limits_{\I \R -d} \frac{dZ}{2\pi\I}
        \int\limits_{\I \R -d} \frac{dV}{2\pi\I} \frac{(\beta  + W)(Z -V)}{(W-\beta)(V+Z)}e^{-\tau(\beta - W - V- Z)}
        \frac{\Gamma(-t+V)}{\Gamma(-t-V)}\frac{\Gamma(t+V)}{\Gamma(t-V)}H_\alpha(V)\Gamma(-2V)
\\&\frac{\Gamma(1-\beta+V)}{\Gamma(1-\beta-V)}\frac{\Gamma(-t+W)}{\Gamma(-t-W)}\frac{\Gamma(-t+Z)}{\Gamma(-t-Z)}F(Z)F(W)H_\alpha(Z)H_\alpha(W)Q(Z,W).
    \end{aligned}
\end{equation}

\begin{equation}
    \begin{aligned}
        \widehat{\aleph}_1(X):= &\int \limits_{\I\R - d} \frac{dU}{2\pi\I}\int \limits_{\I\R - d} \frac{dZ}{2\pi\I}\int \limits_{\I\R - d} \frac{dW}{2\pi\I} \frac{(U-Z)}{(U+Z)}e^{\tau(Z+U)+XW} F(U)H_{\alpha}(U)\Gamma(-2U)\frac{\Gamma(1+t+U)}{\Gamma(1+t-U)}\\
        &\frac{\Gamma(-t+Z)}{\Gamma(-t-Z)}\frac{\Gamma(-t+W)}{\Gamma(-t-W)}F(Z)F(W)H_{\alpha}(Z)H_{\alpha}(W)Q(Z,W),\\
        \widehat{\aleph}_2(X):=&\int \limits_{\I\R - d} \frac{dU}{2\pi\I}\int \limits_{\I\R - d} \frac{dZ}{2\pi\I}\int \limits_{\I\R - d} \frac{dW}{2\pi\I} \frac{-(U-Z)W}{(U+Z)}e^{\tau(Z+U)+XW} F(U)H_{\alpha}(U)\Gamma(-2U)\frac{\Gamma(1+t+U)}{\Gamma(1+t-U)}\\
        &\frac{\Gamma(-t+Z)}{\Gamma(-t-Z)}\frac{\Gamma(-t+W)}{\Gamma(-t-W)}F(Z)F(W)H_{\alpha}(Z)H_{\alpha}(W)Q(Z,W),\\
    \end{aligned}
\end{equation}

\begin{equation}
    \begin{aligned}
        \widehat{\psi}_1(X) := &\Gamma(2t)\frac{\Gamma(\beta + t)}{\Gamma(\beta - t)}H_{\alpha}(t)\int \limits_{\I\R-d} \frac{dZ}{2\pi\I}\int \limits_{\I\R-d} \frac{dW}{2\pi\I} \frac{(Z-t)}{(Z+t)}e^{\tau(Z+t) + XW}\frac{\Gamma(-t+Z)}{\Gamma(-t-Z)}\frac{\Gamma(-t+W)}{\Gamma(-t-W)}\\
        &F(Z)F(W)H_{\alpha}(Z)H_{\alpha}(W)Q(Z,W),\\
        \widehat{\psi}_2(X) := &\Gamma(2t)\frac{\Gamma(\beta + t)}{\Gamma(\beta - t)}H_{\alpha}(t)\int \limits_{\I\R-d} \frac{dZ}{2\pi\I}\int \limits_{\I\R-d} \frac{dW}{2\pi\I} \frac{-(Z-t)W}{(Z+t)}e^{\tau(Z+t) + XW}\frac{\Gamma(-t+Z)}{\Gamma(-t-Z)}\frac{\Gamma(-t+W)}{\Gamma(-t-W)}\\
        &F(Z)F(W)H_{\alpha}(Z)H_{\alpha}(W)Q(Z,W).
    \end{aligned}
\end{equation}

\begin{equation}
    \begin{aligned}
        \widehat{\Theta}_1(X):= &\int \limits_{\I\R - d} \frac{dV}{2\pi\I}\int \limits_{\I\R - d} \frac{dZ}{2\pi\I}\int \limits_{\I\R - d} \frac{dW}{2\pi\I}\frac{(V-Z)}{(V+Z)} e^{\tau(Z+V)+XW} \frac{\Gamma(-t+V)}{\Gamma(-t-V)}\frac{\Gamma(t+V)}{\Gamma(t-V)}H_\alpha(V)\Gamma(-2V)\\
        & \frac{\Gamma(1-\beta + V)}{\Gamma(1-\beta - V)}\frac{\Gamma(-t+Z)}{\Gamma(-t-Z)}\frac{\Gamma(-t+W)}{\Gamma(-t-W)}F(Z)H_\alpha(Z)F(W)H_\alpha(W)Q(Z,W),\\
         \widehat{\Theta}_2(X):=&\int \limits_{\I\R - d} \frac{dV}{2\pi\I}\int \limits_{\I\R - d} \frac{dZ}{2\pi\I}\int \limits_{\I\R - d} \frac{dW}{2\pi\I}  \frac{W(Z- V)}{V+Z}e^{\tau(Z+V)+XW} \frac{\Gamma(-t+V)}{\Gamma(-t-V)}\frac{\Gamma(t+V)}{\Gamma(t-V)}H_\alpha(V)\Gamma(-2V)\\
        & \frac{\Gamma(1-\beta + V)}{\Gamma(1-\beta - V)}\frac{\Gamma(-t+Z)}{\Gamma(-t-Z)}\frac{\Gamma(-t+W)}{\Gamma(-t-W)}F(Z)H_\alpha(Z)F(W)H_\alpha(W)Q(Z,W).
    \end{aligned}
\end{equation}

\begin{equation}
    \begin{aligned}
        \widehat{\zeta}_1(Y) := &\Gamma(2t)\frac{\Gamma(\beta + t)}{\Gamma(\beta - t)}H_{\alpha}(t)\int \limits_{\I\R-d} \frac{dZ}{2\pi\I}\int \limits_{\I\R-d} \frac{dW}{2\pi\I} \frac{Z(W-t)}{(W+t)}e^{\tau(t + W)+YZ}\frac{\Gamma(-t+Z)}{\Gamma(-t-Z)}\frac{\Gamma(-t+W)}{\Gamma(-t-W)}\\
        &F(Z)F(W)H_{\alpha}(Z)H_{\alpha}(W)Q(Z,W),\\
        \widehat{\zeta}_2(Y) := &\Gamma(2t)\frac{\Gamma(\beta + t)}{\Gamma(\beta - t)}H_{\alpha}(t)\int \limits_{\I\R-d} \frac{dZ}{2\pi\I}\int \limits_{\I\R-d} \frac{dW}{2\pi\I} \frac{(W-t)}{(W+t)}e^{\tau(t + W)+YZ}\frac{\Gamma(-t+Z)}{\Gamma(-t-Z)}\frac{\Gamma(-t+W)}{\Gamma(-t-W)}\\
        &F(Z)F(W)H_{\alpha}(Z)H_{\alpha}(W)Q(Z,W).\\
    \end{aligned}
\end{equation}

\begin{equation}
    \begin{aligned}
        \widehat{\eta}_1(Y) := &\frac{\Gamma(t-\beta )}{\Gamma(t+\beta)}\frac{\Gamma(-t-\beta )}{\Gamma(-t+\beta)}H_{\alpha}(-\beta)\int \limits_{\I\R-d} \frac{dZ}{2\pi\I}\int \limits_{\I\R-d} \frac{dW}{2\pi\I} \frac{Z(W+\beta)}{(W-\beta)}e^{-\tau(\beta -  W)+YZ}\frac{\Gamma(-t+Z)}{\Gamma(-t-Z)}\frac{\Gamma(-t+W)}{\Gamma(-t-W)}\\
        &F(Z)F(W)H_{\alpha}(Z)H_{\alpha}(W)Q(Z,W),\\
        \widehat{\eta}_2(Y) := &\frac{\Gamma(t-\beta )}{\Gamma(t+\beta)}\frac{\Gamma(-t-\beta )}{\Gamma(-t+\beta)}H_{\alpha}(-\beta)\int \limits_{\I\R-d} \frac{dZ}{2\pi\I}\int \limits_{\I\R-d} \frac{dW}{2\pi\I} \frac{(W+\beta)}{(W-\beta)}e^{-\tau(\beta -  W)+YZ}\frac{\Gamma(-t+Z)}{\Gamma(-t-Z)}\frac{\Gamma(-t+W)}{\Gamma(-t-W)}\\
        &F(Z)F(W)H_{\alpha}(Z)H_{\alpha}(W)Q(Z,W).
    \end{aligned}
\end{equation}

\begin{equation}
    \begin{aligned}
        \widehat{\theta}_1(Y):=\!\! &\int \limits_{\I\R-d}\!\! \frac{dZ}{2\pi\I}\int \limits_{\I\R-d}\!\! \frac{dW}{2\pi\I}\!\!\int \limits_{\I\R-d} \!\!\frac{dV}{2\pi\I} \frac{(W-V)Z}{(V+W)}e^{\tau(V+W)+YZ}\frac{\Gamma(-t+Z)}{\Gamma(-t-Z)}\frac{\Gamma(-t+W)}{\Gamma(-t-W)}F(Z)F(W)H_{\alpha}(Z)H_{\alpha}(W)Q(Z,W)\\
        &\frac{\Gamma(t+V)}{\Gamma(t-V)}\frac{\Gamma(-t+V)}{\Gamma(-t-V)}\frac{\Gamma(1-\beta + V)}{\Gamma(1-\beta - V)}\Gamma(-2V)H_{\alpha}(V),\\
        \widehat{\theta}_2(Y):=\!\! &\int \limits_{\I\R-d}\!\! \frac{dZ}{2\pi\I}\int \limits_{\I\R-d}\!\! \frac{dW}{2\pi\I}\!\!\int \limits_{\I\R-d} \!\!\frac{dV}{2\pi\I} \frac{(W-V)}{(V+W)}e^{\tau(V+W)+YZ}\frac{\Gamma(-t+Z)}{\Gamma(-t-Z)}\frac{\Gamma(-t+W)}{\Gamma(-t-W)}F(Z)F(W)H_{\alpha}(Z)H_{\alpha}(W)Q(Z,W)\\
        &\frac{\Gamma(t+V)}{\Gamma(t-V)}\frac{\Gamma(-t+V)}{\Gamma(-t-V)}\frac{\Gamma(1-\beta + V)}{\Gamma(1-\beta - V)}\Gamma(-2V)H_{\alpha}(V).\\
    \end{aligned}
\end{equation}

Define the kernel
\begin{equation}
    \widehat{\boldsymbol{\mathcal{K}}}(X,Y) := \begin{pmatrix}
        \K(X,Y) & -\partial_X \K(X,Y)\\
        -\partial_Y \K(X,Y) & \partial_X\partial_Y \K(X,Y)
    \end{pmatrix},
\end{equation}
with
\begin{equation}
    \begin{aligned}
{\K}(X,Y) := &\int\limits_{\mathrm{i}\mathbb{R}-d_0} \frac{dZ}{2\pi i}
\int\limits_{\mathrm{i}\mathbb{R}-d_0} \frac{dW}{2\pi i}\;
e^{XZ+YW}\frac{\Gamma(-t+Z)}{\Gamma(-t-Z)}\frac{\Gamma(t+Z)}{\Gamma(t-Z)}\frac{\Gamma(-t+W)}{\Gamma(-t-W)}\frac{\Gamma(t+W)}{\Gamma(t-W)}\\
&
\frac{\Gamma(\beta + Z)}{\Gamma(\beta - Z)}\frac{\Gamma(\beta + W)}{\Gamma(\beta - W)}H_{\alpha}(Z)\,
H_{\alpha}(W)\,
Q(Z,W).
    \end{aligned}
\end{equation}

\begin{equation}
    \begin{aligned}
        \widehat{\A} :=&\Pf\left(J - \widehat{\boldsymbol{\mathcal{K}}}\right) \left(\A_1 + \A_2 + 1\right)- \Pf\left(J - \widehat{\boldsymbol{\mathcal{K}}} - \ketbra{\begin{array}{c} \widehat{\eta}_2 \\
        -\widehat{\eta}_1
        \end{array}}{\widehat{\aleph}_1 \quad \widehat{\aleph}_2} - \ketbra{\begin{array}{c}
            \widehat{\aleph}_1\\
            \widehat{\aleph}_2
        \end{array}}{-\widehat{\eta}_2 \quad \widehat{\eta}_1 }\right),\\
        \widehat{\B} :=&\Pf\left(J - \widehat{\boldsymbol{\mathcal{K}}}\right) \left(\B_1 + \B_2 + 1\right)- \Pf\left(J - \widehat{\boldsymbol{\mathcal{K}}} - \ketbra{\begin{array}{c} \widehat{\theta}_2 \\
        -\widehat{\theta}_1
        \end{array}}{\widehat{\aleph}_1 \quad \widehat{\aleph}_2} - \ketbra{\begin{array}{c}
            \widehat{\aleph}_1\\
            \widehat{\aleph}_2
        \end{array}}{-\widehat{\theta}_2 \quad \widehat{\theta}_1 }\right),\\
        \widehat{\C} :=&\Pf\left(J - \widehat{\boldsymbol{\mathcal{K}}}\right) \left(\C_1 + \C_2 + 1\right)- \Pf\left(J - \widehat{\boldsymbol{\mathcal{K}}} - \ketbra{\begin{array}{c} \widehat{\eta}_2 \\
        -\widehat{\eta}_1
        \end{array}}{\widehat{\psi}_1 \quad \widehat{\psi}_2} - \ketbra{\begin{array}{c}
            \widehat{\psi}_1\\
            \widehat{\psi}_2
        \end{array}}{-\widehat{\eta}_2 \quad \widehat{\eta}_1 }\right),\\
        \end{aligned}
        \end{equation}
        \begin{equation}
        \begin{aligned}
         \widehat{\D} :=&\Pf\left(J - \widehat{\boldsymbol{\mathcal{K}}}\right) \left(\D_1 + \D_2 + 1\right)- \Pf\left(J - \widehat{\boldsymbol{\mathcal{K}}} - \ketbra{\begin{array}{c} \widehat{\theta}_2 \\
        -\widehat{\theta}_1
        \end{array}}{\widehat{\psi}_1 \quad \widehat{\psi}_2} - \ketbra{\begin{array}{c}
            \widehat{\psi}_1\\
            \widehat{\psi}_2
        \end{array}}{-\widehat{\theta}_2 \quad \widehat{\theta}_1 }\right),\\
    \end{aligned}
\end{equation}

\begin{equation}
    \begin{aligned}
        \widehat{\M} :=&\Pf\left(J - \widehat{\boldsymbol{\mathcal{K}}}\right) \left(\M_1 + \M_2 - 1\right) + \Pf\left(J - \widehat{\boldsymbol{\mathcal{K}}} - \ketbra{\begin{array}{c} \widehat{\zeta}_2 \\
        -\widehat{\zeta}_1
        \end{array}}{\widehat{\aleph}_1 \quad \widehat{\aleph}_2} - \ketbra{\begin{array}{c}
            \widehat{\aleph}_1\\
            \widehat{\aleph}_2
        \end{array}}{-\widehat{\zeta}_2 \quad \widehat{\zeta}_1 }\right),\\
        \widehat{\NN} :=&\Pf\left(J - \widehat{\boldsymbol{\mathcal{K}}}\right) \left(\NN_1 + \NN_2 -1\right)+ \Pf\left(J - \widehat{\boldsymbol{\mathcal{K}}} - \ketbra{\begin{array}{c} \widehat{\eta}_2 \\
        -\widehat{\eta}_1
        \end{array}}{\widehat{\Theta}_1 \quad \widehat{\Theta}_2} - \ketbra{\begin{array}{c}
            \widehat{\Theta}_1\\
            \widehat{\Theta}_2
        \end{array}}{-\widehat{\eta}_2 \quad \widehat{\eta}_1 }\right).\\
    \end{aligned}
\end{equation}

\begin{rmk}
One may also ask what happens to our formula in the special cases $\beta = t$ and $\beta = -t$. We discuss them separately.

When $\beta = -t$, the log-gamma polymer model explodes due to the inverse-Gamma random weight at $(2,2)$. Thus, our formula also diverges.
When $\beta = t$, our kernel decomposition breaks down. In the kernel decomposition in section \ref{section: 2 param decompose}, we removed the poles at $Z=-\beta$ and $W=-\beta$, but if $\beta=t$ these poles should instead remain inside the double contours and merge into a double pole at $t$. Consequently, the decomposition of $\boldsymbol{\check{\mathcal{K}}}$ into $X_i$ and $Y_i$ for $i\in\{1,2,3,4\}$ is no longer valid, and all terms must be recomputed and regrouped to extract the limit.
Moreover, there remains a random shift at $(2,2)$. Although the two-parameter stationary condition degenerates to the product stationary case at $\beta=t$, the partition function still retains this random shift, whose contribution is nontrivial under critical scaling. As a result, the outcome is worse than the Low density formula $\widehat{Q}_{N,t}^{\mathrm{Low}}$.
\end{rmk}

\begin{define}[High density phase with $-\ttt<\tbeta <\ttt$]\label{def: High density beta small asymptotic}
    Fix any $\alpha \in \R$ and any three parameters $\tilde{t}, \tbeta, s \in \R$ satisfying $\tilde{t}>0,$ $\ttt> \tilde{\beta} >-\ttt.$ We define
    \begin{equation}
    \begin{aligned}
        \widetilde{{Q}}_{\tbeta < \ttt}^{High}(s)&:= \frac{1}{\Pf\left(J - \widetilde{\boldsymbol{\mathcal{K}}}\right)}\left(\widetilde{\M}\widetilde{\NN} - \widetilde{\A}\widetilde{\D} + \widetilde{\B}\widetilde{\C}\right).
    \end{aligned}
    \end{equation}
   The Fredholm Pfaffian is taken over $\mathbb{L}^2((s,\infty))$.
\end{define} 
We choose $\delta$ satisfying $\tbeta<\delta <\ttt$. The contours for all functions defined below are $C(\delta;\pi/3)$. The dependence of $\widetilde{{Q}}_{\tbeta < \ttt}^{High}$ on $\delta$ is omitted. The well-definedness of $\frac{1}{\Pf\left(J - \widetilde{\boldsymbol{\mathcal{K}}}\right)}$ is explained in the proof of Theorem~\ref{thm: High density beta small asymptotic}.

\begin{equation}
\begin{aligned}
    \widetilde{\A}_1  :&=\int\limits_{C(\delta;\pi/3)} \frac{d\tW}{2\pi \I} \frac{( \ttt + \tW)}{2\tW(\ttt-\tW )} e^{\frac{\tbeta^3}{3} - s\tbeta + \frac{\tW^3}{3} - s\tW},\\
    \widetilde{\A}_2  :&=\int\limits_{C(\delta;\pi/3)} \frac{d\tU}{2\pi \I}\int\limits_{C(\delta;\pi/3)} \frac{d\tW}{2\pi \I}\int\limits_{C(\delta;\pi/3)} \frac{d\tZ}{2\pi \I} \frac{( \tU - \tZ)(\ttt + \tU)(\tbeta + \tU)}{(\tU + \tZ)(\ttt - \tU)(\tbeta - \tU)}\\
    &\frac{( \tZ - \tW)(\tbeta + \tZ)}{8\tU \tZ \tW( \tZ + \tW)(\tbeta - \tZ)}e^{\frac{\tbeta^3}{3} - s\tbeta + \frac{\tU^3}{3} - s\tU +\frac{\tZ^3}{3} - s\tZ+ \frac{\tW^3}{3} - s\tW}.
\end{aligned}    
\end{equation}

\begin{equation}
\begin{aligned}
    \widetilde{\B}_1  :&=\int\limits_{C(\delta;\pi/3)} \frac{d\tW}{2\pi \I}\int\limits_{C(\delta;\pi/3)} \frac{d\tZ}{2\pi \I} \frac{( \ttt + \tW)( \tbeta + \tW)(\tW-\tZ)}{4\tW\tZ(\ttt-\tW )( \tbeta - \tW)(\tW+\tZ)} e^{\frac{\tZ^3}{3} - s\tZ + \frac{\tW^3}{3} - s\tW} + \frac{\tbeta + \ttt}{\tbeta - \ttt}\\
    \widetilde{\B}_2  :&=\int\limits_{C(\delta;\pi/3)} \frac{d\tU}{2\pi \I}\int\limits_{C(\delta;\pi/3)} \frac{d\tW}{2\pi \I}\int\limits_{C(\delta;\pi/3)} \frac{d\tV}{2\pi \I}\int\limits_{C(\delta;\pi/3)} \frac{d\tZ}{2\pi \I} \frac{( \tZ - \tU)(\ttt + \tU)(\tbeta + \tU)}{( \tZ + \tU)(\ttt - \tU)(\tbeta - \tU)}\\
    &\frac{( \tZ - \tW)(\tV - \tW)(\tbeta + \tZ)( \tbeta + \tW)}{16 \tV \tU \tZ \tW( \tZ + \tW)(\tV + \tW)(\tbeta - \tZ)( \tbeta - \tW)}e^{\frac{\tU^3}{3} - s\tU+ \frac{\tV^3}{3} - s\tV +\frac{\tZ^3}{3} - s\tZ+ \frac{\tW^3}{3} - s\tW}
\end{aligned}    
\end{equation}

\begin{equation}
\begin{aligned}
    \widetilde{\C}_1  :&=\frac{1}{2\ttt}
 e^{\frac{\tbeta^3}{3} - s\tbeta -\frac{\ttt^3}{3} + s\ttt }\\
     \widetilde{\C}_2 :&=\frac{1}{2\ttt}\frac{\tbeta -\ttt}{\tbeta + \ttt} \int\limits_{C(\delta;\pi/3)} \frac{d\tW}{2\pi \I}\int\limits_{C(\delta;\pi/3)} \frac{d\tZ}{2\pi \I} \frac{( \ttt + \tZ)(\tbeta + \tZ)(\tZ - \tW)}{4\tZ\tW(\ttt - \tZ)(\tbeta - \tZ)(\tZ + \tW)}e^{\frac{\tbeta^3}{3} - s\tbeta  +\frac{\tZ^3}{3} - s\tZ+ \frac{\tW^3}{3} - s\tW- \frac{\ttt^3}{3} + s\ttt}
     \end{aligned}    
\end{equation}

\begin{equation}
\begin{aligned}
    \widetilde{\D}_1  :&=\int\limits_{C(\delta;\pi/3)} \frac{d\tW}{2\pi \I} \frac{( \ttt + \tW)( \tbeta - \ttt)}{4\tW\ttt(\ttt-\tW )(\tbeta + \ttt)} e^{-\frac{\ttt^3}{3} + s\ttt + \frac{\tW^3}{3} - s\tW}\\
    \widetilde{\D}_2  :&=\int\limits_{C(\delta;\pi/3)} \frac{d\tV}{2\pi \I}\int\limits_{C(\delta;\pi/3)} \frac{d\tW}{2\pi \I}\int\limits_{C(\delta;\pi/3)} \frac{d\tZ}{2\pi \I} \frac{( \ttt+ \tZ)(\tW - \tV)(\tbeta + \tZ)(\tbeta + \tW)}{(\ttt - \tZ)(\tW + \tV)(\tbeta - \tZ)(\tbeta - \tW)}\\
    &\frac{( \tZ - \tW)(\tbeta -\ttt )}{16 \ttt \tV \tZ \tW( \tZ + \tW)(\tbeta +\ttt) }e^{\frac{\tV^3}{3} - s\tV - \frac{\ttt^3}{3} + s\ttt +\frac{\tZ^3}{3} - s\tZ+ \frac{\tW^3}{3} - s\tW}
\end{aligned}    
\end{equation}

\begin{equation}
    \begin{aligned}
        \widetilde{\M}_1 := &-\frac{(\tbeta - \ttt)}{2\ttt(\tbeta + \ttt)}e^{-\frac{\ttt^3}{3} + s\ttt} \int\limits_{C(\delta;\pi/3)} \frac{d\tW}{2\pi\I} e^{\frac{\tW^3}{3} - s\tW}\frac{(\tW + \ttt)^2}{(\tW - \ttt)^2}\frac{(\tbeta +\tW)}{(\tbeta - \tW)}\frac{1}{2\tW},\\
        \widetilde{\M}_2 := & -\frac{(\tbeta - \ttt)}{2\ttt(\tbeta + \ttt)}e^{-\frac{\ttt^3}{3} + s\ttt}\int\limits_{C(\delta;\pi/3)} \frac{d\tV}{2\pi \I}\int\limits_{C(\delta;\pi/3)} \frac{d\tZ}{2\pi \I}\int\limits_{C(\delta;\pi/3)} \frac{d\tW}{2\pi \I} \frac{(\tV - \tZ)(\tW + \ttt)}{(\ttt -\tW)(\tV + \tZ)}\\
        &e^{\frac{\tV^3}{3} - s\tV + \frac{\tZ^3}{3} - s\tZ + \frac{\tW^3}{3}-s\tW}
        \frac{(\tbeta + \tV)}{(\tbeta-\tV)}\frac{(\tbeta + \tW)}{(\tbeta-\tW)}\frac{(\tbeta + \tZ)}{(\tbeta-\tZ)}
        \frac{(\ttt+\tV)}{(\ttt - \tV)}\frac{(\tZ-\tW)}{(\tZ+\tW)}\frac{1}{8\tV\tZ\tW}
    \end{aligned}
\end{equation}

\begin{equation}
    \begin{aligned}
        \widetilde{\NN}_1 := &1-e^{\frac{\tbeta^3}{3} - s\tbeta} \int\limits_{C(\delta;\pi/3)} \frac{d\tW}{2\pi \I}e^{\frac{\tW^3}{3} - s\tW} \frac{(\tbeta - \tW)}{(\tbeta + \tW)}\frac{1}{2\tW},\\
        \widetilde{\NN}_2 := & \,e^{\frac{\tbeta^3}{3} - s\tbeta} \int\limits_{C(\delta;\pi/3)} \frac{d\tV}{2\pi \I}\int\limits_{C(\delta;\pi/3)} \frac{d\tZ}{2\pi \I}\int\limits_{C(\delta;\pi/3)} \frac{d\tW}{2\pi \I} e^{\frac{\tZ^3}{3} - s\tZ + \frac{\tW^3}{3} - s\tW + \frac{\tV^3}{3} - s\tV}\\
        &\frac{(\tbeta + \tZ)}{(\tbeta - \tZ)}\frac{(\tZ - \tV)}{(\tZ+ \tV)} \frac{(\tZ-\tW)}{(\tZ+ \tW)}\frac{1}{8\tZ\tV\tW}.
    \end{aligned}
\end{equation}

The following:

\begin{equation}
    \begin{aligned}
        \widetilde{\aleph}_1(u):= -&\int\limits_{C(\delta;\pi/3)} \frac{d\tU}{2\pi \I}\int\limits_{C(\delta;\pi/3)} \frac{d\tW}{2\pi \I}\int\limits_{C(\delta;\pi/3)} \frac{d\tZ}{2\pi \I} \frac{(\tU-\tZ)(\ttt+\tU)(\tbeta + \tU)(\tbeta + \tZ)}{(\tU+\tZ)(\ttt-\tU)(\tbeta - \tU)(\tbeta - \tZ)}\\
    &\frac{(\tbeta + \tW)( \tZ - \tW)}{8  \tU \tZ \tW(\tbeta - \tW) ( \tZ + \tW)}e^{\frac{\tU^3}{3} - s\tU  +\frac{\tZ^3}{3} - s\tZ+ \frac{\tW^3}{3} - u\tW}\\
        \widetilde{\aleph}_2(u):= -&\int\limits_{C(\delta;\pi/3)} \frac{d\tU}{2\pi \I}\int\limits_{C(\delta;\pi/3)} \frac{d\tW}{2\pi \I}\int\limits_{C(\delta;\pi/3)} \frac{d\tZ}{2\pi \I} \frac{(\tU-\tZ)(\ttt+\tU)(\tbeta + \tU)(\tbeta + \tZ)}{(\tU+\tZ)(\ttt-\tU)(\tbeta - \tU)(\tbeta - \tZ)}\\
    &\frac{(\tbeta + \tW)( \tZ - \tW)}{8  \tU \tZ(\tbeta - \tW) ( \tZ + \tW)}e^{\frac{\tU^3}{3} - s\tU  +\frac{\tZ^3}{3} - s\tZ+ \frac{\tW^3}{3} - u\tW}
    \end{aligned}
\end{equation}

\begin{equation}
    \begin{aligned}
        \widetilde{\psi}_1(u):= &\int\limits_{C(\delta;\pi/3)} \frac{d\tW}{2\pi \I}\int\limits_{C(\delta;\pi/3)} \frac{d\tZ}{2\pi \I} \frac{(\tbeta-\ttt)(\tZ + \ttt)(\tbeta + \tW)(\tbeta + \tZ)(\tZ - \tW)}{8\ttt \tZ \tW(\tbeta+\ttt)(\tZ - \ttt)(\tbeta - \tW)(\tbeta - \tZ)(\tZ + \tW)}\\
    & e^{\frac{\tZ^3}{3} - s\tZ  - \frac{\ttt^3}{3} + s\ttt+ \frac{\tW^3}{3} - u\tW}\\
        \widetilde{\psi}_2(u):= &\int\limits_{C(\delta;\pi/3)} \frac{d\tW}{2\pi \I}\int\limits_{C(\delta;\pi/3)} \frac{d\tZ}{2\pi \I} \frac{(\tbeta-\ttt)(\tZ + \ttt)(\tbeta + \tW)(\tbeta + \tZ)(\tZ - \tW)}{8\ttt \tZ (\tbeta+\ttt)(\tZ - \ttt)(\tbeta - \tW)(\tbeta - \tZ)(\tZ + \tW)}\\
    & e^{\frac{\tZ^3}{3} - s\tZ  - \frac{\ttt^3}{3} + s\ttt+ \frac{\tW^3}{3} - u\tW}
    \end{aligned}
\end{equation}

\begin{equation}
    \begin{aligned}
        \widetilde{\Theta}_1(u):= -&\int\limits_{C(\delta;\pi/3)} \frac{d\tV}{2\pi \I}\int\limits_{C(\delta;\pi/3)} \frac{d\tW}{2\pi \I}\int\limits_{C(\delta;\pi/3)} \frac{d\tZ}{2\pi \I} \frac{(\tV-\tZ)(\tbeta + \tW)(\tbeta + \tZ)(\tZ - \tW)}{8\tV \tZ \tW(\tV+\tZ)(\tbeta - \tW)(\tbeta - \tZ)(\tZ + \tW)}\\
    &e^{\frac{\tV^3}{3} - s\tV  +\frac{\tZ^3}{3} - s\tZ+ \frac{\tW^3}{3} - u\tW}\\
        \widetilde{\Theta}_2(u):= -&\int\limits_{C(\delta;\pi/3)} \frac{d\tV}{2\pi \I}\int\limits_{C(\delta;\pi/3)} \frac{d\tW}{2\pi \I}\int\limits_{C(\delta;\pi/3)} \frac{d\tZ}{2\pi \I} \frac{(\tV-\tZ)(\tbeta + \tW)(\tbeta + \tZ)(\tZ - \tW)}{8\tV \tZ (\tV+\tZ)(\tbeta - \tW)(\tbeta - \tZ)(\tZ + \tW)}\\
    &e^{\frac{\tV^3}{3} - s\tV  +\frac{\tZ^3}{3} - s\tZ+ \frac{\tW^3}{3} - u\tW}\\
    \end{aligned}
\end{equation}

\begin{equation}
    \begin{aligned}
        \widetilde{\zeta}_1(v):= &\frac{\tbeta - \ttt}{\tbeta + \ttt}\int\limits_{C(\delta;\pi/3)} \frac{d\tW}{2\pi \I}\int\limits_{C(\delta;\pi/3)} \frac{d\tZ}{2\pi \I} \frac{(\ttt+\tW)(\tbeta + \tW)(\tbeta + \tZ)(\tZ - \tW)}{8\ttt  \tW(\ttt-\tW)(\tbeta - \tW)(\tbeta - \tZ)(\tZ + \tW)}\\
    & e^{\frac{\tZ^3}{3} - v\tZ  - \frac{\ttt^3}{3} + s\ttt+ \frac{\tW^3}{3} - s\tW}\\
        \widetilde{\zeta}_2(v):= -&\frac{\tbeta - \ttt}{\tbeta + \ttt}\int\limits_{C(\delta;\pi/3)} \frac{d\tW}{2\pi \I}\int\limits_{C(\delta;\pi/3)} \frac{d\tZ}{2\pi \I} \frac{(\ttt+\tW)(\tbeta + \tW)(\tbeta + \tZ)(\tZ - \tW)}{8\ttt \tZ \tW(\ttt-\tW)(\tbeta - \tW)(\tbeta - \tZ)(\tZ + \tW)}\\
    & e^{\frac{\tZ^3}{3} - v\tZ  - \frac{\ttt^3}{3} + s\ttt+ \frac{\tW^3}{3} - s\tW}\\
    \end{aligned}
\end{equation}

\begin{equation}
    \begin{aligned}
        \widetilde{\eta}_1(v) := & e^{\frac{\tbeta^3}{3} - s\tbeta} \int\limits_{C(\delta;\pi/3)} \frac{d\tZ}{2\pi \I}\int\limits_{C(\delta;\pi/3)} \frac{d\tW}{2\pi \I} \frac{\tZ(\tbeta - \tW)}{(\tW + \tbeta)}e^{\frac{\tZ^3}{3}-v\tZ + \frac{\tW^3}{3}-s\tW}\frac{(\tbeta+\tZ)}{(\tbeta - \tZ)}\frac{(\tbeta+\tW)}{(\tbeta - \tW)}\frac{(\tZ-\tW)}{(\tZ+\tW)}\frac{1}{4\tZ\tW},\\
        \widetilde{\eta}_2(v) := & -e^{\frac{\tbeta^3}{3} - s\tbeta} \int\limits_{C(\delta;\pi/3)} \frac{d\tZ}{2\pi \I}\int\limits_{C(\delta;\pi/3)} \frac{d\tW}{2\pi \I} \frac{(\tbeta - \tW)}{(\tW + \tbeta)}e^{\frac{\tZ^3}{3}-v\tZ + \frac{\tW^3}{3}-s\tW}\frac{(\tbeta+\tZ)}{(\tbeta - \tZ)}\frac{(\tbeta+\tW)}{(\tbeta - \tW)}\frac{(\tZ-\tW)}{(\tZ+\tW)}\frac{1}{4\tZ\tW},\\
    \end{aligned}
\end{equation}

\begin{equation}
    \begin{aligned}
        \widetilde{\theta}_1(v) := & -\int\limits_{C(\delta;\pi/3)} \frac{d\tV}{2\pi \I}\int\limits_{C(\delta;\pi/3)} \frac{d\tZ}{2\pi \I}\int\limits_{C(\delta;\pi/3)} \frac{d\tW}{2\pi \I}\frac{\tZ(\tV-\tW)}{(\tV+\tW)}e^{\frac{\tV^3}{3}-s\tV + \frac{\tW^3}{3} - s\tW + \frac{\tZ^3}{3} - v\tZ}\\
        &\frac{(\tbeta + \tZ)}{(\tbeta - \tZ)}\frac{(\tbeta + \tW)}{(\tbeta - \tW)}\frac{(\tZ-\tW)}{(\tZ+\tW)}\frac{1}{8\tV\tZ\tW},\\
        \widetilde{\theta}_2(v) := & \int\limits_{C(\delta;\pi/3)} \frac{d\tV}{2\pi \I}\int\limits_{C(\delta;\pi/3)} \frac{d\tZ}{2\pi \I}\int\limits_{C(\delta;\pi/3)} \frac{d\tW}{2\pi \I}\frac{(\tV-\tW)}{(\tV+\tW)}e^{\frac{\tV^3}{3}-s\tV + \frac{\tW^3}{3} - s\tW + \frac{\tZ^3}{3} - v\tZ}\\
        &\frac{(\tbeta + \tZ)}{(\tbeta - \tZ)}\frac{(\tbeta + \tW)}{(\tbeta - \tW)}\frac{(\tZ-\tW)}{(\tZ+\tW)}\frac{1}{8\tV\tZ\tW}.
    \end{aligned}
\end{equation}

Define the kernel
\begin{equation}
    \widetilde{\boldsymbol{\mathcal{K}}}(u,v) := \begin{pmatrix}
        \widetilde{\K}(u,v) & -\partial_v \widetilde{\K}(u,v)\\
        -\partial_u \widetilde{\K}(u,v) & \partial_u\partial_v \widetilde{\K}(u,v)
    \end{pmatrix},
\end{equation}
with
\begin{equation}
    \begin{aligned}
{\widetilde{K}}(u,v) := &\int\limits_{C(\delta;\pi/3)} \frac{d\tZ}{2\pi i}
\int\limits_{C(\delta;\pi/3)} \frac{d\tW}{2\pi i}e^{\frac{\tZ^3}{3} - u\tZ + \frac{\tW^3}{3} - v\tW}
\frac{(\tbeta + \tZ)}{(\tbeta - \tZ)}\frac{(\tbeta + \tW)}{(\tbeta - \tW)}\frac{(\tZ - \tW)}{(\tZ+ \tW)}\frac{1}{4\tZ \tW}.
    \end{aligned}
\end{equation}

\begin{equation}
    \begin{aligned}
        \widetilde{\A} :=&\Pf\left(J - \widetilde{\boldsymbol{\mathcal{K}}}\right) \left(\widetilde{\A}_1 + \widetilde{\A}_2 + 1\right)- \Pf\left(J - \widetilde{\boldsymbol{\mathcal{K}}} - \ketbra{\begin{array}{c} \widetilde{\eta}_2 \\
        -\widetilde{\eta}_1
        \end{array}}{\widetilde{\aleph}_1 \quad \widetilde{\aleph}_2} - \ketbra{\begin{array}{c}
            \widetilde{\aleph}_1\\
            \widetilde{\aleph}_2
        \end{array}}{-\widetilde{\eta}_2 \quad \widetilde{\eta}_1 }\right),\\
        \widetilde{\B} :=&\Pf\left(J - \widetilde{\boldsymbol{\mathcal{K}}}\right) \left(\widetilde{\B}_1 + \widetilde{\B}_2 + 1\right)- \Pf\left(J - \widetilde{\boldsymbol{\mathcal{K}}} - \ketbra{\begin{array}{c} \widetilde{\theta}_2 \\
        -\widetilde{\theta}_1
        \end{array}}{\widetilde{\aleph}_1 \quad \widetilde{\aleph}_2} - \ketbra{\begin{array}{c}
            \widetilde{\aleph}_1\\
            \widetilde{\aleph}_2
        \end{array}}{-\widetilde{\theta}_2 \quad \widetilde{\theta}_1 }\right),\\
        \widetilde{\C} :=&\Pf\left(J - \widetilde{\boldsymbol{\mathcal{K}}}\right) \left(\widetilde{\C}_1 + \widetilde{\C}_2 + 1\right)- \Pf\left(J - \widetilde{\boldsymbol{\mathcal{K}}} - \ketbra{\begin{array}{c} \widetilde{\eta}_2 \\
        -\widetilde{\eta}_1
        \end{array}}{\widetilde{\psi}_1 \quad \widetilde{\psi}_2} - \ketbra{\begin{array}{c}
            \widetilde{\psi}_1\\
            \widetilde{\psi}_2
        \end{array}}{-\widetilde{\eta}_2 \quad \widetilde{\eta}_1 }\right),\\
         \widetilde{\D} :=&\Pf\left(J - \widetilde{\boldsymbol{\mathcal{K}}}\right) \left(\widetilde{\D}_1 + \widetilde{\D}_2 + 1\right)- \Pf\left(J - \widetilde{\boldsymbol{\mathcal{K}}} - \ketbra{\begin{array}{c} \widetilde{\theta}_2 \\
        -\widetilde{\theta}_1
        \end{array}}{\widetilde{\psi}_1 \quad \widetilde{\psi}_2} - \ketbra{\begin{array}{c}
            \widetilde{\psi}_1\\
            \widetilde{\psi}_2
        \end{array}}{-\widetilde{\theta}_2 \quad \widetilde{\theta}_1 }\right),\\
    \end{aligned}
\end{equation}

\begin{equation}
    \begin{aligned}
        \widetilde{\M} :=&\Pf\left(J - \widetilde{\boldsymbol{\mathcal{K}}}\right) \left(\widetilde{\M}_1 + \widetilde{\M}_2 + 1\right)- \Pf\left(J - \widetilde{\boldsymbol{\mathcal{K}}} - \ketbra{\begin{array}{c} \widetilde{\zeta}_2 \\
        -\widetilde{\zeta}_1
        \end{array}}{\widetilde{\aleph}_1 \quad \widetilde{\aleph}_2} - \ketbra{\begin{array}{c}
            \widetilde{\aleph}_1\\
            \widetilde{\aleph}_2
        \end{array}}{-\widetilde{\zeta}_2 \quad \widetilde{\zeta}_1 }\right),\\
        \widetilde{\NN} :=&-\Pf\left(J - \widetilde{\boldsymbol{\mathcal{K}}}\right) \left(\widetilde{\NN}_1 + \widetilde{\NN}_2\right)+ \Pf\left(J - \widetilde{\boldsymbol{\mathcal{K}}} - \ketbra{\begin{array}{c} \widetilde{\eta}_2 \\
        -\widetilde{\eta}_1
        \end{array}}{\widetilde{\Theta}_1 \quad \widetilde{\Theta}_2} - \ketbra{\begin{array}{c}
            \widetilde{\Theta}_1\\
            \widetilde{\Theta}_2
        \end{array}}{-\widetilde{\eta}_2 \quad \widetilde{\eta}_1 }\right),\\
    \end{aligned}
\end{equation}
where Fredholm Pfaffian is taken over $\mathbb{L}^2((s,\infty)).$

\section{Product stationary and Two-parameter stationary half-space log-gamma polymer models}
Recall the partition function defined via the recurrence relation in \eqref{def:recurrence}. We introduce another equivalent definition of the partition function via weight collection along all the up-right paths.
\begin{define}
    Let $A_0,A_1,A_2,\dots \in \R$ be a sequence of parameters such that $A_0 + A_i >0$ for all $i\geq 1$ and $A_i + A_j>0$ for all $i > j \geq 1.$ Define the octant $\D := \{ (i,j) \in \Z^2 | 1 \leq j \leq i\}$. Let $(\omega_{i,j})_{(i,j)  \in \D}$ be a family of independent random variables such that $\omega_{i,j} \sim \text{Gamma}^{-1}(A_i + A_j)$ for $i > j$ and $\omega_{i,i} \sim\text{Gamma}^{-1}(A_0 + A_i)$. The partition function of the half-space log-gamma polymer is defined as
    \begin{equation}
        Z(N,M) = \sum_{\pi:(1,1)\to (N,M)} \prod_{(i,j)\in\pi} \omega_{i,j},
    \end{equation}
where the sum is taken over up-right paths from $(1,1)$ to $(N,M)$ in $\D.$ We call $\log Z(N,M)$ the free energy of the model.
\end{define}

We invoke the symmetry property of the half-space log-gamma polymer model in the specific setting of our model. This symmetry arises from the fact that the law of the partition function is governed by the the Whittaker process. A more general statement and its proof can be found in \cite[Lemma~2.7]{LogGammaStationary}.
\begin{lemma}\label{lem: symmetry}
    For any $N \in \Z_{\geq 2},$ the law of $Z(N,N)$ is invariant with respect to permutations of the parameters $A_0,A_1,\dots, A_N$.
\end{lemma}

Special realizations of parameters lead to product stationary model and two-parameter stationary model.

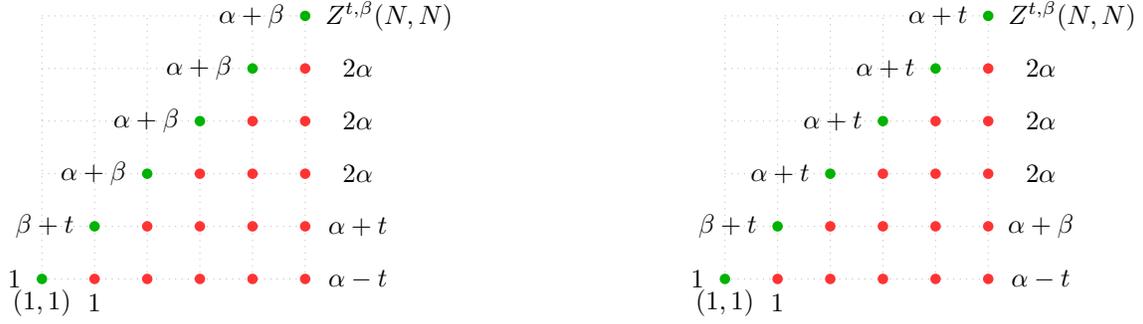
\begin{figure}
    \centering

\begin{minipage}{0.45\textwidth}
\centering

\begin{tikzpicture}[
  scale=0.7,
  grid/.style={gray!50, dotted},
  redpt/.style={circle, fill=red!80, draw=none, inner sep=1.4pt},
  grnpt/.style={circle, fill=green!70!black, draw=none, inner sep=1.4pt},
  baseline=(current bounding box.center),
  every node/.style={font=\small}
]

\def\N{5} 

\foreach \x in {0,...,\N} {
  \draw[grid] (\x,0) -- (\x,\N);
}
\foreach \y in {0,...,\N} {
  \draw[grid] (0,\y) -- (\N,\y);
}

\foreach \x in {1,...,\N} {
  \foreach \y in {0,...,\numexpr\x-1\relax} {
    \node[redpt] at (\x,\y) {};
  }
}

\foreach \k in {0,...,\N} {
  \node[grnpt] at (\k,\k) {};
}

\node[anchor=east] at (-0.2,0) {$1$};
\node[anchor=south] at (0,-0.9) {$(1,1)$};

\node[anchor=north] at (1,-0.1) {$1$};

\node[anchor=east] at (8,5) {$Z^{t,\beta}(N,N)$};

\foreach \k in {2,...,\numexpr\N\relax} {
  \node[anchor=east] at (\k-0.2,\k) {$\alpha+\beta$};
}

\node[anchor=east] at (0.8,1) {$\beta + t$};
\node at (6,0) {$\alpha - t$};
\node at (6,1) {$\alpha + t$};
\node at (6,2) {$2\alpha$};
\node at (6,3) {$2\alpha$};
\node at (6,4) {$2\alpha$};

\end{tikzpicture}
\end{minipage}
\hfill
\begin{minipage}{0.45\textwidth}
\centering
\begin{tikzpicture}[
  scale=0.7,
  grid/.style={gray!50, dotted},
  redpt/.style={circle, fill=red!80, draw=none, inner sep=1.4pt},
  grnpt/.style={circle, fill=green!70!black, draw=none, inner sep=1.4pt},
  baseline=(current bounding box.center),
  every node/.style={font=\small}
]

\def\N{5} 

\foreach \x in {0,...,\N} {
  \draw[grid] (\x,0) -- (\x,\N);
}
\foreach \y in {0,...,\N} {
  \draw[grid] (0,\y) -- (\N,\y);
}

\foreach \x in {1,...,\N} {
  \foreach \y in {0,...,\numexpr\x-1\relax} {
    \node[redpt] at (\x,\y) {};
  }
}

\foreach \k in {0,...,\N} {
  \node[grnpt] at (\k,\k) {};
}

\node[anchor=east] at (-0.2,0) {$1$};

\node[anchor=south] at (0,-0.9) {$(1,1)$};
\node[anchor=north] at (1,-0.1) {$1$};
\node[anchor=east] at (8,5) {$Z^{t,\beta}(N,N)$};

\foreach \k in {2,...,\numexpr\N\relax} {
  \node[anchor=east] at (\k-0.2,\k) {$\alpha+t$};
}

\node[anchor=east] at (0.8,1) {$\beta + t$};
\node at (6,0) {$\alpha - t$};
\node at (6,1) {$\alpha + \beta$};
\node at (6,2) {$2\alpha$};
\node at (6,3) {$2\alpha$};
\node at (6,4) {$2\alpha$};

\end{tikzpicture}
\end{minipage}
\caption{Both figures are half-space stationary two-parameter log-gamma polymer model with weights described in Definition \ref{def: two-param stationary model}. The equality in distribution is justified by Lemma \ref{lem: symmetry}. }
\label{fig: two-parameter stationary figure}
\end{figure}

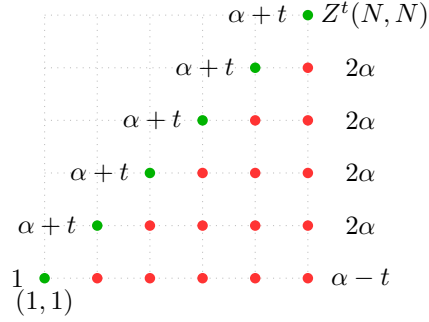
\begin{figure}
    \centering
\begin{tikzpicture}[
  scale=0.7,
  grid/.style={gray!50, dotted},
  redpt/.style={circle, fill=red!80, draw=none, inner sep=1.4pt},
  grnpt/.style={circle, fill=green!70!black, draw=none, inner sep=1.4pt},
  baseline=(current bounding box.center),
  every node/.style={font=\small}
]

\def\N{5} 

\foreach \x in {0,...,\N} {
  \draw[grid] (\x,0) -- (\x,\N);
}
\foreach \y in {0,...,\N} {
  \draw[grid] (0,\y) -- (\N,\y);
}

\foreach \x in {1,...,\N} {
  \foreach \y in {0,...,\numexpr\x-1\relax} {
    \node[redpt] at (\x,\y) {};
  }
}

\foreach \k in {0,...,\N} {
  \node[grnpt] at (\k,\k) {};
}

\node[anchor=east] at (-0.2,0) {$1$};

\node[anchor=south] at (0,-0.9) {$(1,1)$};

\node[anchor=east] at (7.5,5) {$Z^{t}(N,N)$};

\foreach \k in {2,...,\numexpr\N\relax} {
  \node[anchor=east] at (\k-0.2,\k) {$\alpha+t$};
}

\node[anchor=east] at (0.8,1) {$\alpha + t$};
\node at (6,0) {$\alpha - t$};
\node at (6,1) {$2\alpha$};
\node at (6,2) {$2\alpha$};
\node at (6,3) {$2\alpha$};
\node at (6,4) {$2\alpha$};

\end{tikzpicture}
\caption{Half-space product stationary log-gamma polymer model with weights described in Definition \ref{def: product stationary model}.}
\label{fig: product stationary figure}
\end{figure}

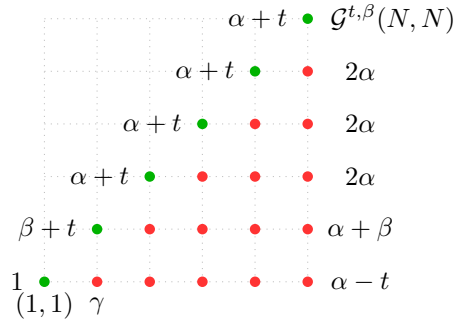
\begin{figure}
    \centering
\begin{tikzpicture}[
  scale=0.7,
  grid/.style={gray!50, dotted},
  redpt/.style={circle, fill=red!80, draw=none, inner sep=1.4pt},
  grnpt/.style={circle, fill=green!70!black, draw=none, inner sep=1.4pt},
  baseline=(current bounding box.center),
  every node/.style={font=\small}
]

\def\N{5} 

\foreach \x in {0,...,\N} {
  \draw[grid] (\x,0) -- (\x,\N);
}
\foreach \y in {0,...,\N} {
  \draw[grid] (0,\y) -- (\N,\y);
}

\foreach \x in {1,...,\N} {
  \foreach \y in {0,...,\numexpr\x-1\relax} {
    \node[redpt] at (\x,\y) {};
  }
}

\foreach \k in {0,...,\N} {
  \node[grnpt] at (\k,\k) {};
}

\node[anchor=east] at (-0.2,0) {$1$};

\node[anchor=south] at (0,-0.9) {$(1,1)$};
\node[anchor=north] at (1,-0.1) {$\gamma$};

\node[anchor=east] at (8,5) {$\mathcal{G}^{t,\beta}(N,N)$};

\foreach \k in {2,...,\numexpr\N\relax} {
  \node[anchor=east] at (\k-0.2,\k) {$\alpha+t$};
}

\node[anchor=east] at (0.8,1) {$\beta + t$};
\node at (6,0) {$\alpha - t$};
\node at (6,1) {$\alpha + \beta$};
\node at (6,2) {$2\alpha$};
\node at (6,3) {$2\alpha$};
\node at (6,4) {$2\alpha$};

\end{tikzpicture}
\caption{This model connects the product stationary case and the two-parameter stationary case. Let $\gamma$ denote the weight at $(2,1)$. If $\gamma = 1,$ i.e., a constant weight $1$, then this model becomes the two-parameter stationary model. If $\gamma = \Gamm(\alpha -t)$ and $\beta = \alpha$, then this model becomes the product stationary model.}
\label{fig: intermediate model}
\end{figure}

\begin{define}\label{def: product stationary model}
    Consider the half-space geometry $\mathcal{D}$. We define the following weights:
    \begin{equation} \label{OneParamApprox}
  \omega_{i, j} = \begin{cases}
1, & \textrm{if } i=j=1, \\

    \mathrm{Gamma}^{-1}\left( \alpha + t \right), &\textrm{if } i=j\geq 2,\\
    \mathrm{Gamma}^{-1}\left( \alpha -t \right), &\textrm{if } j=1, i\geq 2, \\
    \mathrm{Gamma}^{-1}(2\alpha), &\textrm{otherwise}.
  \end{cases}
\end{equation} Let $Z^{t}(N,M)$ denote the partition function from $(1,1)$ to $(N,M)$. We refer to this as the half-space log-gamma polymer model with product stationary initial condition (Figure \ref{fig: product stationary figure}). 
\end{define}
One can easily check that
\begin{equation}
    Z^{t}(1+\cdot,1) / Z^{t}(1,1) \stackrel{(d)}{=} \mathcal{T}_{-t, t}(\cdot)
\end{equation}
where $\mathcal{T}_{-t, t}$ is defined in \eqref{def: stationary process}. Moreover, $Z^t(N,M)$ satisfies the same recurrence relations as in \eqref{def:recurrence}. Hence, $Z^t(N,N) \stackrel{(d)}{=} \mathcal{Z}^t(N,N)$ in Definition \ref{def: product stationary Z}.

\begin{define}\label{def: two-param stationary model}
    Consider the half-space geometry $\mathcal{D}$. We define the following weights:
    \begin{equation} \label{OneParamApprox}
  \omega_{i, j} = \begin{cases}
1, & \textrm{if } i=j=1, \\
1, &\textrm{if } i = 2,j=1,\\
\mathrm{Gamma}^{-1}\left( \beta + t \right), &\textrm{if } i=j= 2,\\
\mathrm{Gamma}^{-1}\left(\alpha +t \right), &\textrm{if } i>j= 2,\\
    \mathrm{Gamma}^{-1}\left( \alpha + \beta \right), &\textrm{if } i=j\geq 3,\\
    \mathrm{Gamma}^{-1}\left( \alpha -t \right), &\textrm{if } j=1, i\geq 3, \\
    \mathrm{Gamma}^{-1}(2\alpha), &\textrm{otherwise},
  \end{cases}
\end{equation}
where $\alpha >0,$ $t\in (0,\alpha)$, $\beta \in (-t,\infty)$. Let $Z^{t,\beta}(N,M)$ denote the partition function from $(1,1)$ to $(N,M)$. We refer to this as the half-space log-gamma polymer model with two-parameter stationary initial condition (Figure \ref{fig: two-parameter stationary figure}). 
\end{define}
One can also check that
\begin{equation}
    Z^{t,\beta}(2+\cdot,2) / Z^{t,\beta}(2,2) \stackrel{(d)}{=} \mathcal{T}_{\beta, t}(\cdot)
\end{equation}
where $\mathcal{T}_{\beta, t}$ is defined in \eqref{def: stationary process}.
The fact that the first two rows encode the two-parameter stationary initial condition results in an index shift from $Z^{t,\beta}(N,N)$ to $\mathcal{Z}^{t,\beta}(N-1,N-1)$ which is introduced in Definition \ref{def: two param stationary condition}.

\begin{define}\label{def: most general stationary model}
    Consider the half-space geometry $\mathcal{D}$. We define the following weights:
    \begin{equation} \label{OneParamApprox}
  \omega_{i, j} = \begin{cases}
1, & \textrm{if } i=j=1, \\
\gamma, &\textrm{if } i = 2,j=1,\\
\mathrm{Gamma}^{-1}\left( \beta + t \right), &\textrm{if } i=j= 2,\\
\mathrm{Gamma}^{-1}\left( \beta +\alpha  \right), &\textrm{if } i>j= 2,\\
    \mathrm{Gamma}^{-1}\left( \alpha + t \right), &\textrm{if } i=j\geq 3,\\
    \mathrm{Gamma}^{-1}\left( \alpha -t \right), &\textrm{if } j=1, i\geq 3, \\
    \mathrm{Gamma}^{-1}(2\alpha), &\textrm{otherwise},
  \end{cases}
\end{equation} 
where $\alpha >0,$ $t\in (0,\alpha)$, $\beta \in (-t,\infty)$. Let $\mathcal{G}^{t,\beta}(N,M)$ denote the partition function from $(1,1)$ to $(N,M)$ under the above model. This model is shown in Figure \ref{fig: intermediate model}.
\end{define}
 If $\gamma = 1$, then by Lemma \ref{lem: symmetry}, $\mathcal{G}^{t,\beta}(N,N) \stackrel{(d)}{=} Z^{t,\beta}(N,N)$ for all $N \geq 1$. If instead, $\gamma = \mathrm{Gamma}^{-1}\left( \alpha - t \right)$ and $\beta = \alpha$, then $\mathcal{G}^{t,\beta}(N,N) \stackrel{(d)}{=} Z^{t}(N,N)$ for all $N \geq 1$.

By \cite[Remark 1.7]{LogGammaStationary}, it is known that if $\beta = t$, $\mathcal{T}_{\beta, t} \stackrel{(d)}{=}\mathcal{T}_{-t,t}$, i.e., the product stationary measure. In terms of half-space log-gamma polymer model, we see that $Z^{t,t}(2+N,2)/Z^{t,t}(2,2) \stackrel{(d)}{=}\mathcal{T}_{-t,t}$. However, we make it clear that $Z^{t,t}(N,N)$ and $\mathcal{Z}^t(N,N)$ are not equal in distribution due to the fact that the weight at $(2,2)$ contributes non-trivially to $Z^{t,t}(N,N)$ (even in the limit) while $\mathcal{Z}^t(N,N)$ has no such issue as the weight at $(1,1)$ is just $1$. Therefore, it is better to only compute the distribution of $Z^{t}(N,N)$. On the other hand, when $\beta = -t,$ the model $Z^{t,\beta}$ is not defined because of the random variable $\Gamm(\beta + t)$ at $(2,2)$. As a result, we will only  analyze distributions of $Z^{t,\beta}$ when $\beta \neq t,-t$. 
To preview the analysis, $Z^{t,\beta}$ with $\beta >t$ and  $Z^t$ will be analyzed together. The regime $-t<\beta <t$ for $Z^{t,\beta}$, however, is handled separately and requires a different method.

To obtain the distribution of $Z^{t,\beta}(N,N)$, we start with the half-space log-gamma polymer model with the introduction of a new first-row parameter $u$.
\begin{define}\label{def: approximation model}
Consider the half-space geometry $\mathcal{D}$. We define the following weights:
\begin{equation} \label{eq: two param beta large approx}
  \omega_{i, j} = \begin{cases}
\mathrm{Gamma}^{-1}\left( \beta + t \right), &\textrm{if } i=j= 2,\\
\mathrm{Gamma}^{-1}\left( \beta + u \right), &\textrm{if } i = 2,j=1,\\
\mathrm{Gamma}^{-1}\left( \beta +\alpha  \right), &\textrm{if } i>j= 2,\\
    \mathrm{Gamma}^{-1}\left( \alpha + t \right), &\textrm{if } i=j\geq 3,\\
    \mathrm{Gamma}^{-1}\left( \alpha +u \right), &\textrm{if } j=1, i\geq 3, \\
    \mathrm{Gamma}^{-1}\left( u + t \right), & \textrm{if } i=j=1, \\
    \mathrm{Gamma}^{-1}(2\alpha), &\textrm{otherwise}.
  \end{cases}
\end{equation}
Let $Z^{u}_{\beta > t}(N,M)$ ($Z^{u}_{\beta < t}(N,M)$) denote the partition function from $(1,1)$ to $(N,M)$ under the above model when $\beta >t$ (respectively, $\beta < t$). This is shown in Figure \ref{fig: Stationary Model Approx}.
\end{define}

\section{Pfaffian Structure of the half-space log-gamma model}
Our goal is to
identify the exact distribution of the polymer free energy under the stationary
initial condition. The following theorem, established in \cite[Theorem 5.10]{imamura2022solvablemodelskpzclass}, gives a
Fredholm Pfaffian expression for the Laplace transform of the partition function. We adapt it to our model \eqref{eq: two param beta large approx}.
This representation is the starting point for our analysis.

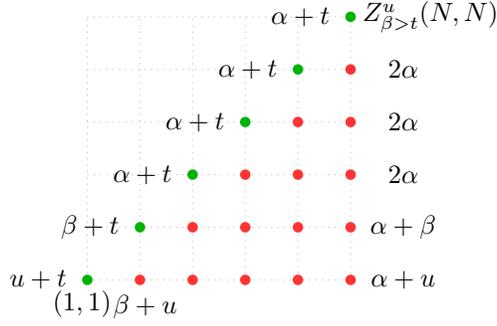
\begin{figure}
\centering
\begin{tikzpicture}[
  scale=0.7,
  grid/.style={gray!50, dotted},
  redpt/.style={circle, fill=red!80, draw=none, inner sep=1.4pt},
  grnpt/.style={circle, fill=green!70!black, draw=none, inner sep=1.4pt},
  baseline=(current bounding box.center),
  every node/.style={font=\small}
]

\def\N{5} 

\foreach \x in {0,...,\N} {
  \draw[grid] (\x,0) -- (\x,\N);
}
\foreach \y in {0,...,\N} {
  \draw[grid] (0,\y) -- (\N,\y);
}

\foreach \x in {1,...,\N} {
  \foreach \y in {0,...,\numexpr\x-1\relax} {
    \node[redpt] at (\x,\y) {};
  }
}

\foreach \k in {0,...,\N} {
  \node[grnpt] at (\k,\k) {};
}

\node[anchor=east] at (-0.2,0) {$u+t$};

\node[anchor=south] at (-0.1,-0.9) {$(1,1)$};
\node[anchor=north] at (1.1,-0.1) {$\beta + u$};

\node[anchor=east] at (8,5) {$Z^u_{\beta >t}(N,N)$};

\foreach \k in {2,...,\numexpr\N\relax} {
  \node[anchor=east] at (\k-0.2,\k) {$\alpha+t$};
}

\node[anchor=east] at (0.8,1) {$\beta + t$};
\node at (6,0) {$\alpha +u$};
\node at (6,1) {$\alpha + \beta$};
\node at (6,2) {$2\alpha$};
\node at (6,3) {$2\alpha$};
\node at (6,4) {$2\alpha$};

\end{tikzpicture}
\caption{log-gamma polymer model as described in $\eqref{eq: two param beta large approx}$.}
\label{fig: Stationary Model Approx}
\end{figure}

\begin{thm}\label{thm: Pfaffian structure}
Fix $N \in \Z_{\geq 3}.$ Let $t\in (\frac{1}{N},\frac12),$ $\alpha> t$, $\beta>\frac{1}{N}$ and $t>u > \frac{1}{N}.$ 
Consider the half-space log-gamma polymer partition function
\(
Z(N,N)
\)
with boundary parameter $\beta$, first row parameter $u$, second row parameter $t$, and bulk parameters \(\alpha\).
Choose a real number \(d>0\) satisfying
\begin{equation}
\frac1{N}< d < \min\{1/2,u,\beta,t,\alpha\}.
\end{equation}
Then for every \(\tau\in\mathbb{R}\) one has the identity
\begin{equation}\label{eq: starting point pfaffian}
\mathbb{E}\left[e^{-e^{-\tau+\log Z(N,N)}}\right]
= 
\Pf(J-{\bold{\check{K}}})_{\mathbb{L}^2(\tau,\infty)},
\end{equation}
where  $\bold{\check{K}}$ is a $2 \times 2$ matrix kernel
\begin{equation}
\bold{\check{K}}(X,Y)=
\begin{pmatrix}
\check{K}(X,Y) & -\partial_Y \check{K}(X,Y) \\
-\partial_X \check{K}(X,Y) & \partial_X\partial_Y \check{K}(X,Y)
\end{pmatrix}.
\end{equation}
The scalar function \(\check{K}\) is given by the double contour integral
\begin{equation}\label{eq: central kernel K}
\check{K}(X,Y)
=
\int\limits_{\mathrm{i}\mathbb{R}-d} \frac{dZ}{2\pi i}
\int\limits_{\mathrm{i}\mathbb{R}-d} \frac{dW}{2\pi i}\;
e^{XZ+YW}\;
\frac{\Gamma(u+Z)}{\Gamma(u-Z)}
\frac{\Gamma(u+W)}{\Gamma(u-W)}
\frac{\Gamma(t+Z)}{\Gamma(t-Z)}
\frac{\Gamma(t+W)}{\Gamma(t-W)}
G_{\alpha,\beta}(Z)\,
G_{\alpha,\beta}(W)\,
Q(Z,W),
\end{equation}
with
\begin{equation}\label{eq: def of G}
G_{\alpha, \beta}(Z)
=
\left(
\frac{\Gamma(\beta+Z)}{\Gamma(\beta-Z)}\right)\left(
\frac{\Gamma(\alpha+Z)}{\Gamma(\alpha-Z)}\right)^{N-2},\quad Q(Z,W) = \Gamma(-2Z)\Gamma(-2W)\,
\frac{\sin \bigl(\pi(Z-W)\bigr)}{\sin \bigl(\pi(Z+W)\bigr)}.
\end{equation}
\end{thm}

\section{Reformulation of kernel for High density phase when $\beta >t$}
To extract the exact distribution of the half-space log-gamma free energy under the stationary initial conditions, we need to remove the effect of the random variable at $(1,1)$ and then take the limit of $u\rightarrow -t$. 

\subsection{Removing the weight at $(1,1)$}

Recall that $\omega_{1,1} \stackrel{(d)}{=} \text{Gamma}^{-1}(u+t)$. For $\beta >t$, we let  $Z^{u,-}_{\beta >t}$ be the partition function with the weight at $(1,1)$ being replaced by a constant $1$. Similarly, we define $Z^{u,-}_{\beta < t}$ for $\beta <t$. Since every polymer path from $(1,1)$ to $(N,N)$ must go through $(1,1)$, it follows that
\begin{equation}
Z^{u}_{\beta >t} = \omega_{1,1}Z^{u,-}_{\beta >t}, \quad Z^{u}_{\beta <t} = \omega_{1,1}Z^{u,-}_{\beta <t}.
\end{equation}
For any $x>0,$ we rewrite the Laplace transform of $Z^{u}_{\beta >t}(N,N)$ as
\begin{equation}
\E[e^{-xZ^{u}_{\beta >t}(N,N)}] = \E[e^{-xZ^{u,-}_{\beta >t}(N,N) \omega_{1,1}}].
\end{equation}
Given the independence of $\omega_{1,1}$ and $Z^{u,-}_{\beta >t}(N,N)$, we first integrate $\omega_{1,1}$ to get:
\begin{equation}\label{eq:shift}
\Gamma(u+t)\E[e^{-xZ^{u,-}_{\beta >t}(N,N) \omega_{1,1}}] = \E\bigg[2(xZ_{\beta >t}^{u,-}(N,N))^{\frac{u+t}{2}}K_{-u-t}\left(2\sqrt{x Z_{\beta >t}^{u,-}(N,N)}\right)\bigg]
\end{equation}
where $K_{-\nu}$ is the modified Bessel function \cite{abramowitz1965handbook} defined to be 
\begin{equation}
 2c^{\nu/2}K_{-\nu}(2\sqrt{c}) = \int\limits_0^{\infty} e^{-cy-y^{-1}}y^{-\nu -1} dy \quad \text{ for } c>0.
\end{equation}
\begin{rmk}
Theorem \ref{thm: Pfaffian structure} provides a Fredholm Pfaffian representation for the LHS of \eqref{eq:shift}. When we take the $u \rightarrow -t$ limit, the factor $\Gamma(u+t)$ on the RHS of \eqref{eq:shift} diverges to infinity, while the expectation on the RHS of \eqref{eq:shift} admits a finite and nontrivial limit from which the exact distribution of the stationary half-space log-gamma free energy can be extracted. Consequently, the Fredholm Pfaffian on the LHS of \eqref{eq:shift} must vanish at a compensating rate so that its product with $\Gamma(u+t)$ converges to a well-defined limit. This observation motivates the kernel decomposition introduced below.
\end{rmk}

\subsection{Kernel Decomposition}\label{section: kernel decompose}
Recall the definition of $G_{\alpha,\beta}$ in \eqref{eq: def of G}.
We define the following functions for further decomposition of the kernel.
\begin{equation}\label{eq: low def of A}
    \begin{aligned}
        A(X) := e^{-Xu}\frac{\Gamma(t-u)}{\Gamma(t+u)}G_{\alpha, \beta}(-u),\quad
        B(Y) := \int \limits_{\I\R - d} \frac{dW}{2\pi\I} e^{YW}\frac{\Gamma(t+W)}{\Gamma(t-W)}G_{\alpha,\beta}(W)\Gamma(-2W)\frac{\Gamma(1-u+W)}{\Gamma(1-u-W)}.
    \end{aligned}
\end{equation}

\begin{lemma}
For any $N\in\Z_{\geq 3},$ $t\in (\frac{1}{N},\frac{1}{2}),$ $t >u > \frac{1}{N}$, and $\alpha,\beta>t$,
we have the following decomposition of kernel:
    \begin{equation}
        {\bold{\check{K}}}(X,Y) = \bold{K}(X,Y) + {R}(X,Y).
    \end{equation}
We define
    \begin{equation}
        \bold{K}(X,Y) := \begin{pmatrix}
            K(X,Y) & -\partial_Y K(X,Y)\\
            - \partial_X K(X,Y) & \partial_X\partial_Y K(X,Y)
        \end{pmatrix}
    \end{equation} with
    \begin{equation}
        {K}(X,Y) := \int\limits_{\mathrm{i}\mathbb{R}-d} \frac{dZ}{2\pi i}
\int\limits_{\mathrm{i}\mathbb{R}-d} \frac{dW}{2\pi i}\;
e^{XZ+YW}\;
G_{\alpha,\beta}(Z)\,
G_{\alpha,\beta}(W)\,
Q(Z,W)
    \end{equation}
where $d$ satisfies $u<d <t$,
and define
\begin{equation}
    \begin{aligned}
        {R}(X,Y) = \begin{pmatrix}
A(X)B(Y)-A(Y)B(X) & -A(X)B'(Y)+A'(Y)B(X) \\
-A'(X)B(Y)+A(Y)B'(X) & A'(X)B'(Y)-A'(Y)B'(X)
\end{pmatrix}
    \end{aligned}
\end{equation}
where ${R}$ captures poles at $Z = -u$ and $W = -u.$
\end{lemma}
\begin{proof}
    We show the $11-$entry as an example. The overall idea is that since $u$ will be sent to $-t<0$, the poles $z = -u$ and $w= -u$ have to cross the contour. Therefore, we isolate these two poles via direct evaluation of the residue and store them in the matrix ${R}$.
    We see that
    \begin{equation}
        \begin{aligned}
            &\text{Res}(Z=-u, \text{ integrand of }\eqref{eq: central kernel K}) = A(X)B(Y)\\
            &\text{Res}(W=-u, \text{ integrand of }\eqref{eq: central kernel K}) = -A(Y)B(X).
        \end{aligned}
    \end{equation}
    The remaining kernel is written as $K(X,Y)$ which contains poles $\{-t - i, -\alpha -i, -\beta -i, -u -1 -i : i \in \Z_{\geq 0}\}$ for both $Z$ and $W$, which is the reason for choosing $t>d_0 >u$.
\end{proof}

Recall that $J(x,y)= \Id_{x=y}\begin{pmatrix}
    0 & 1\\
    -1 & 0
\end{pmatrix}.$
We have that
\begin{equation}
    \begin{aligned}
        J^{-1}{R}= \ketbra{\X_1}{\Y_1} +  \ketbra{\X_2}{\Y_2},
    \end{aligned}
\end{equation}
where
\begin{equation}
    \X_1 = \begin{pmatrix}
        A^{\prime} \\
        A
    \end{pmatrix}, \quad \Y_1 = \begin{pmatrix}
        B & -B^{\prime}
    \end{pmatrix}, \quad \X_2 = \begin{pmatrix}
        B^{\prime} \\
        B
    \end{pmatrix}, \quad \Y_2 = \begin{pmatrix}
        -A & A^{\prime}
    \end{pmatrix}.
\end{equation}
We define $\G = J^{-1}\bold{K}$ and perform the following reformulation of Fredholm Pfaffian
\begin{equation}\label{eq: PfSimple}
    \begin{aligned}
        \mathrm{Pf}(J-{\bold{\check{K}}})^2 &= \det(\Id - J^{-1}{\bold{\check{K}}})\\
        &= \det\left(\Id - \G -\ketbra{\X_1}{\Y_1} - \ketbra{\X_2}{\Y_2}\right)\\
        &= \det(\Id - \G)\det\left(\Id -(\Id - \G)^{-1}\ketbra{\X_1}{\Y_1} - (\Id - \G)^{-1}\ketbra{\X_2}{\Y_2}\right)\\
        &= \det(\Id - \G)\det\left(\Id - \begin{pmatrix}
           (\Id -\G)^{-1}\ket{\X_1} & (\Id -\G)^{-1}\ket{\X_2}
        \end{pmatrix}\begin{pmatrix}
            \Y_1\\
            \Y_2
        \end{pmatrix}\right)\\
        &= \det(\Id - \overline{G})\det\left(\Id - \begin{pmatrix}
            \bra{\Y_1}(\Id -\G)^{-1}\ket{\X_1} & \bra{\Y_1}(\Id -\G)^{-1}\ket{\X_2}  \\
            \bra{\Y_2}(\Id -\G)^{-1}\ket{\X_1} & \bra{\Y_2}(\Id -\G)^{-1}\ket{\X_2} 
        \end{pmatrix}\right)\\
        &= \det(\Id - \G)(1- \bra{\Y_1}(\Id - \G)^{-1}\ket{\X_1})^2.
    \end{aligned}
\end{equation}
Since $\bold{K}$ is antisymmetric, we see that
\begin{equation}
    \begin{aligned}
        \bra{\Y_1}(\Id - \overline{G})^{-1}\ket{\X_1} = \bra{\Y_2}(\Id - \overline{G})^{-1}\ket{\X_2}, \quad \bra{\Y_1}(\Id - \overline{G})^{-1}\ket{\X_2} = \bra{\Y_2}(\Id - \overline{G})^{-1}\ket{\X_1} = 0,
    \end{aligned}
\end{equation}
which justifies the last equality in \eqref{eq: PfSimple}. The fifth equality follows from $\det(\Id - AB) = \det(\Id - BA).$
To simplify our notation, we define $$\phi_1 = (\G\X_1)_1,\quad \phi_2 = (\G\X_1)_2.$$

\begin{lemma}\label{lem: difference of pfaffians}
Fix any $t\in (\frac{1}{N},\frac12),$ $\alpha,\beta > t,$ and $ t>u > \frac{1}{N}.$
The Fredholm Pfaffian can be reformulated into
    \begin{equation}\label{eq: Pfreformulate}
    \begin{aligned}
         &\Pf(J-{\bold{\check{K}}}) = \Pf(J-\bold{K})\left( 1- \braket{B}{A^{\prime}} + \braket{B^{\prime}}{A}\right) - \Pf(J - \bold{K})\\
        &+ \mathrm{Pf}\left(J - \bold{K} - \ketbra{\begin{array}{c}
            \phi_2\\
            -\phi_1 
        \end{array}}
        {B \quad -B^{\prime}} - \ketbra{\begin{array}{c}
            B\\
            -B^{\prime} 
        \end{array}}{-\phi_2 \quad \phi_1}\right).
    \end{aligned}
    \end{equation}
\end{lemma}

\begin{proof}
By \cite[Lemma $5.13$]{imamura2022solvablemodelskpzclass} and \eqref{eq: upper bound of K}, we know that $||\bold{K}|| < 1$ when $N$ is sufficiently large, which implies that $(\Id - J^{-1}\bold{K})^{-1}$ can be viewed as a series expansion. We choose $N_0 \in \Z_{\geq 3}$ large enough such that for all $N\geq N_0$, $||\bold{K}|| <1$. 
By further decomposing the formula in \eqref{eq: PfSimple}, we get that
    \begin{equation}\label{eq: analytic 1}
    \begin{aligned}
        &\Pf(J-{\bold{\check{K}}})
         = \Pf(J-\bold{K})\left( 1 - \braket{\Y_1}{\X_1} - \brabarket{\Y_1}{(\Id-\G)^{-1}}{\G\X_1}\right)\\
         &=\Pf(J-\bold{K})\left( 1 - \braket{B}{A^{\prime}} + \braket{B^{\prime}}{A}  - \brabarket{(B,\, -B^{\prime})}{(\Id - J^{-1}\bold{K})^{-1}}{\begin{pmatrix}
            \phi_1\\\phi_2
        \end{pmatrix}}\right):= \mu(N)\\
    \end{aligned}
    \end{equation}
    for $N \in \Z_{\geq N_0}$. Now, we view $\mu(N)$ as a function of $N\in \R_{>N_0}$. We want to show that $\mu(N)$ is a real-analytic function for $N\in (N_0,\infty)$.
    It suffices to justify that we may differentiate with respect to $N$ under the contour integrals. We use ${K}$ as an example. Let $H(Z,W) = \frac{\Gamma(u+Z)}{\Gamma(u-Z)}
            \frac{\Gamma(u+W)}{\Gamma(u-W)}
            \frac{\Gamma(t+Z)}{\Gamma(t-Z)}
            \frac{\Gamma(t+W)}{\Gamma(t-W)}            \frac{\Gamma(\alpha-Z)}{\Gamma(\alpha+Z)}\frac{\Gamma(\alpha-W)}{\Gamma(\alpha+W)},$ and $T_0(\alpha, Z,W) = \frac{\Gamma(\alpha + Z)}{\Gamma(\alpha - Z)}\frac{\Gamma(\alpha + W)}{\Gamma(\alpha - W)}.$ Then 
    \begin{equation}\label{eq: kernel temp}
        \begin{aligned}
            \frac{\partial}{\partial N}{K}(X,Y)
            =
            \int\limits_{\mathrm{i}\mathbb{R}-d} \frac{dZ}{2\pi i}
            \int\limits_{\mathrm{i}\mathbb{R}-d} \frac{dW}{2\pi i}\;
            e^{XZ+YW}H(Z,W)
            \frac{\partial}{\partial N}\left(T_0(\alpha, Z,W)\right)^N\,
            Q(Z,W),
        \end{aligned}
    \end{equation}
    where 
    \begin{equation}
        \frac{\partial}{\partial N}\left(T_0(\alpha,Z,W)\right)^N = \log\left(T_0(\alpha, Z,W)\right)e^{N\log (T_0(\alpha, Z, W))}.
    \end{equation}
    Consider the change of variables $W = -d + \I w,$ and $Z = -d + \I v$. Then \eqref{eq: bound gamma} implies that for any $m \in\Z_{>0}$,
    \begin{equation}
        \bigg|\frac{\partial^m}{\partial N^m}\left(T_0(\alpha,Z,W)\right)^N \bigg|\leq \frac{2d\left(\log((1+|v|)(1+|w|))\right)^m}{((1+|v|)(1+|w|))^{2dN}},
    \end{equation}
    hence the integrand in \eqref{eq: kernel temp} is bounded by an $N$-independent integrable function of $(v,w)$:
    \begin{equation}
        \begin{aligned}
            Ce^{-d(X+Y)-4d}e^{-|v+w|\pi}\frac{2d\left(\log((1+|v|)(1+|w|))\right)^m}{((1+|v|)(1+|w|))^{2d(N_0+1)}}(1+|2v|)^{2d-\frac{1}{2}}(1+|2w|)^{2d-\frac{1}{2}},
        \end{aligned}
    \end{equation}
    for some constant $C>0$. Since $\int_{\R} \left(\log x\right)^m x^{-2} dx < \infty$ for any fixed $m$, this bound is integrable in $(v,w)$. This bound uses the same argument as in \eqref{eq: UBofKernel}. Therefore, by dominated convergence, we may pass $\partial/\partial N$ inside the $Z,W$-integrals in \eqref{eq: kernel temp}. This proves that ${K}(X,Y)$ (hence $\mu(N)$) depends real-analytically on $N\in \R_{>N_0}$.
    By \cite[Lemma 6.10]{zeng2025stationary}, we can rewrite the bracket term as the difference of two Fredholm Pfaffians as follows:
    \begin{equation}\label{eq: analytic 2}
        \begin{aligned}
            &\mu(N)= \Pf(J-\bold{K})\left( 1- \braket{B}{A^{\prime}} + \braket{B^{\prime}}{A}\right) - \Pf(J-\bold{K})\brabarket{(B,\, -B^{\prime})}{(\Id - J^{-1}\bold{K})^{-1}}{\begin{pmatrix}
            \phi_1\\\phi_2
        \end{pmatrix}} \\
        &= \text{RHS of }\eqref{eq: Pfreformulate} =: \nu(N).\\
        \end{aligned}
    \end{equation}
Since $\eqref{eq: Pfreformulate}$ does not require the well-definedness of $(\Id - J^{-1}K)^{-1}$, we have $\nu(N)$ as the natural analytic continuation of $\mu(N)$ to $N \in \R_{\geq 3}.$ Combining \eqref{eq: analytic 1} and \eqref{eq: analytic 2} gives the equality of two formulas in \eqref{eq: Pfreformulate} for all $N\in \Z_{\geq 3}$.
\end{proof}

\begin{rmk}\label{modifed}
    Lemma \ref{lem: difference of pfaffians} is designed to remove the resolvent $(\Id - \G)^{-1}$ because it is difficult to show its limit under the critical scaling exists. We also have a slightly modified version of \eqref{eq: Pfreformulate}. By \cite[Remark 6.13]{zeng2025stationary}, we know that for any constant $c_0$
    \begin{equation}
        \begin{aligned}
            \mathrm{Pf}(J - \bold{K})&c_0 \brabarket{(B,\, -B^{\prime})}{(\Id - J^{-1}\bold{K})^{-1}}{\begin{pmatrix}
            \phi_1\\\phi_2
        \end{pmatrix}}\\
            &=
            \mathrm{Pf}(J-\bold{K}) - \mathrm{Pf}\left(J - \bold{K} - c_0\ket{\begin{array}{c}
                 \phi_2 \\
                 -\phi_1
            \end{array}}\bra{B \quad -B^{\prime}} - c_0\ket{\begin{array}{c}
                 B \\
                 -B^{\prime}
            \end{array}}\bra{-\phi_2 \quad \phi_1}\right).
        \end{aligned}
    \end{equation}
\end{rmk}

\section{Analytic continuation of two parameter stationary formula when $\beta > t$}
We first analyze the $u\rightarrow -t$ limit of the central kernel $\mathbf{K}$ and its Fredholm Pfaffian.
\begin{lemma}\label{lemma: ptws_convergence_one_param}
Fix any $\varepsilon \ll 1.$
 For any $\alpha> t,$ $t \in (\frac{1}{N},\frac{1}{2}),$ and $\frac{1}{2}>\beta > t$, $X,Y \in \R,$ the kernel $\bold{K}(X,Y)$ is analytic for $u \in (-t-\varepsilon, -t+\varepsilon)$ with the following limit:
    \begin{equation}
        \lim_{u\rightarrow -t} \bold{K}(X,Y) = \widehat{\bold{K}}(X,Y).
    \end{equation}
\end{lemma}

\begin{proof}
    It suffices to discuss the case of $K(X,Y)$ as the convergence of other entries follow analogously. Because of the kernel decomposition, the poles $Z = -u$ and $W=-u$ that were present in the kernel ${\bold{\check{K}}}$ no longer exist in the kernel $\bold{K}$. The remaining $u$-dependent poles are $Z,W = -u-i$ for $i \geq 1$, which stay to the left of the contour $-d+\I \R$ for $u \in (-t-\varepsilon, -t+\varepsilon)$.

    The integrand is analytic in $u$ as $\Gamma$ functions are analytic on $\mathbb{C} \setminus \{0,-1, \cdots \}$. To apply dominated convergence theorem, we need to bound the integrand of $K(X,Y)$ uniformly for all $u \in (-t-\varepsilon, -t+\varepsilon)$. Recall the Stirling approximation for $\Gamma$ function
    \begin{equation}
    |\Gamma(a+iu)| = \sqrt{2\pi}|u|^{a-\frac12}e^{-a-\frac{|u|\pi}{2}}(1+\mathcal{O}(|u|^{-1})).
    \end{equation}
    Let $W = -d + \I w$ and $Z = -d + \I v$. Then there exists some constant $C>0$ such that for all $u\in (-t-\varepsilon,-t+\varepsilon)$,
    \begin{equation}\label{eq: bound gamma}
\left| \frac{\Gamma(u+Z)}{\Gamma(u-Z)}\right| = \left| \frac{\Gamma(u-d+\I v)}{\Gamma(u+d-\I v)}\right| \leq C(1+ |v|)^{-2d}.
\end{equation}
Moreover, since $ \left| \sin(\pi z)\right| \leq \frac{1}{2} |e^{-\I \pi z}|(|e^{2\I\pi z}|+1)$, we have that
\begin{equation}
\label{eq:sin-bound}
\begin{aligned}
\bigl|\sin\bigl(\pi \I(v-w)\bigr)\bigr|
\le \tfrac12\, e^{\pi(v-w)}\left(e^{-2\pi(v-w)} + 1\right)  \le C_1\left(e^{-\pi(v-w)} + e^{\pi(v-w)}\right).
\end{aligned}
\end{equation}
for some $C_1 > 0$. Additionally, for some $C_2 >0,$
\begin{equation}
    \left| \frac{1}{\sin(\pi( -2d+\I(v+w)))}\right| \leq C_2 e^{-\pi|v+w|}.
\end{equation} Thus, we know that
\begin{equation}
    \left| \frac{\sin(\pi(Z-W))}{\sin(\pi(Z+W))} \right|\leq C_3e^{-\pi|v+w|+\pi|v-w|}
\end{equation}
for some $C_3 > 0$.
Therefore, there exists some constant $C_4 >0$ such that the integrand is upper bounded by
\begin{equation}\label{eq: UBofKernel}
\begin{aligned}
         C_4 e^{-d(X+Y)} (1+|v|)^{-2d(N+1)}(1+|w|)^{-2d(N+1)}(1+|2v|)^{2d - \frac12 }(1+|2w|)^{2d - \frac12 }e^{-4d-(|v+w|)\pi}
\end{aligned}
\end{equation}
for all $u \in (-t-\varepsilon,-t+\varepsilon)$. In the cases of $\partial_XK(X,Y), \partial_YK(X,Y), \partial_X\partial_YK(X,Y)$, the upper bound will be multiplied by additional $|v|,|w|,|vw|$. But since $d > \frac{1}{N}$, all the upper bounds are integrable in $u$ and $v$.
\end{proof}

\begin{lemma}\label{lemma:pfaffian_convergence}
Fix any $\tau \in \R$ and any $\varepsilon \ll 1.$  For any $t \in (\frac{1}{N},\frac{1}{2}),$ $\alpha> t,$ and $\beta > t$, we have that $\Pf\left(J-\bold{K}\right)$ is analytic for $u \in (-t-\varepsilon, -t + \varepsilon)$ with the following well-defined limit:
    \begin{equation}
        \lim_{u\rightarrow -t} \Pf(J-\bold{K})_{\mathbb{L}^2(\tau,\infty)} = \Pf(J-\widehat{\bold{K}})_{\mathbb{L}^2(\tau,\infty)}.
    \end{equation}
\end{lemma}

\begin{proof}
    From \eqref{eq: UBofKernel}, we see that for all $u \in (-t-\varepsilon, -t+\varepsilon),$
    \begin{equation}\label{eq: upper bound of K}
        |K(X,Y)| \leq \frac{C}{2Nd}e^{-d(X+Y)}
    \end{equation}
    where $C>0$ is a constant independent of $X,Y$. Using the same argument, we get that for all $X,Y \geq \tau,$
    \begin{equation}
    |K(X,Y)|, |\partial_XK(X,Y)|,|\partial_YK(X,Y)|,|\partial_X\partial_YK(X,Y)|\leq \frac{C}{2Nd}e^{-d(X+Y)}.
    \end{equation}
    By Hadamard's bound and dominated convergence theorem, we conclude the limit.
\end{proof}

\begin{lemma}\label{lem: limit of the constant for one param}
Fix any $\tau \in \R$ and any $\varepsilon \ll 1.$  For any $t \in (\frac{1}{N},\frac{1}{2}),$ $\alpha >t,$ and $\beta > t$, the following expression is analytic for $u \in (-t-\varepsilon, -t + \varepsilon)$ with a well-defined limit:
    \begin{equation}\label{eq: special_C}
    \begin{split}
        &\lim_{u\rightarrow -t} \Gamma(u+t)\left(1 - \braket{B}{A^{\prime}} + \braket{B^{\prime}}{A}\right) = \widehat{C}.\\
    \end{split}
    \end{equation}
\end{lemma}

\begin{proof}
    We have
    \begin{equation}
    \begin{split}
            &- \braket{B}{A'} + \braket{B'}{A} \\
            &= \int\limits_{\I \R - d} \frac{dW}{2\pi\I}  e^{-\tau(u-W)} \frac{u+W}{u-W}\frac{\Gamma(t-u)}{\Gamma(t+u)} G_{\alpha,\beta}(-u)G_{\alpha,\beta}(W)\frac{\Gamma(t+W)}{\Gamma(t-W)} \Gamma(-2W)\frac{\Gamma(1-u+W)}{\Gamma(1-u-W)}.
    \end{split}
    \end{equation}
    Notice that a new pole $W = u$ is introduced due to the inner product. The expression is not analytic when we take $u \rightarrow -t$ as the pole will cross the contour. Thus, we move the contour to enclose $u$ and then subtract the residue of the extra pole at $W=u$ so that the expression is analytic in $u \in (-t-\varepsilon,-t+\varepsilon)$. By simple computation, we find that $\text{Res}(W=u) = 1$. Let $W = -d + \I w$. The integrand is upper bounded by the following
    \begin{equation}\label{eq: special_C upper bound}
    Ce^{-d \tau} (1+|w|)^{-2d(N+1)}(1+|2w|)^{2d-\frac12}e^{-2d - |w|\pi}
    \end{equation}
    for all $u \in (-t-\varepsilon,-t+\varepsilon)$ and for some constant $C>0$. By dominated convergence theorem, we get \eqref{eq: special_C}.
\end{proof}

\begin{lemma}\label{lem: limits of function for one param}
Fix any $\tau \in \R$ and any $\varepsilon \ll 1.$ For any $t \in (\frac{1}{N},\frac{1}{2}),$ $\alpha >t$, and $\beta > t$, the following expression is analytic for $u \in (-t-\varepsilon, -t + \varepsilon)$ with a well-defined limit:
    \begin{equation}
        \begin{aligned}
            \lim_{u\rightarrow -t} \Gamma(u+t)\phi_1(X) = \widehat{\phi}_1(X)&, \quad \lim_{u\rightarrow -t} \Gamma(u+t)\phi_2(X) = \widehat{\phi}_2(X),\\
            \quad \lim_{u\rightarrow -t}B(X) = \widehat{B}(X)&, \quad \lim_{u\rightarrow -t}B'(X) = \widehat{B}'(X).
        \end{aligned}
    \end{equation}
\end{lemma}

\begin{proof}
    Let us first prove $\lim_{u\rightarrow -t} \Gamma(u+t)\phi_1(X) = \widehat{\phi}_1(X)$ and the rest follows analogously. We have
\begin{equation}
    \begin{aligned}
        \Gamma(u+t){\phi}_1(X)
        &= \Gamma(t-u)G_{\alpha,\beta}(-u)
        \int \limits_{\I\R - d}\frac{dZ}{2\pi\I} \int\limits_{\I\R-d} \frac{dW}{2\pi\I} Z(t-W)e^{XZ}e^{-\tau(u-W)}\frac{\Gamma(u+Z)}{\Gamma(u-Z)}\frac{\Gamma(t+Z)}{\Gamma(t-Z)}\\&G_{\alpha,\beta}(Z)
        \frac{\Gamma(u+W)}{\Gamma(u-W+1)}\frac{\Gamma(t+W)}{\Gamma(t-W)}G_{\alpha,\beta}(W)
        \Gamma(-2Z) \Gamma(-2W)\frac{\sin(\pi(Z-W))}{\sin(\pi(Z+W))}.
        \end{aligned}
        \end{equation}
Unlike Lemma \ref{lem: limit of the constant for one param} where a new pole $W = u$ is created by taking inner product with $e^{-uY}$, the pole $W = u$ in this case gets absorbed into the term $\Gamma(u-W)$ and thus no new pole is introduced. Let $W = -d+\I w$, $Z = -d + \I v$ and do the same analysis of upper bounds as in \eqref{eq: special_C upper bound}. The integrand is bounded by
\begin{equation}
\begin{aligned}
         C e^{-dX} |vw|(1+|v|)^{-2d(N+1)}(1+|w|)^{-2d(N+1)-1}(1+|2v|)^{2d - \frac12 }(1+|2w|)^{2d - \frac12 }e^{-4d-(|v+w|)\pi}
\end{aligned}
\end{equation} for all $u \in (-t-\varepsilon, -t+ \varepsilon)$, which allows us to apply dominated convergence theorem to get desired limiting functions.
\end{proof}

Recall remark \ref{modifed} for the reason of placing $\Gamma(u+t)$ inside the Pfaffian in the next lemma.
\begin{lemma}\label{lem: diff of Pfaffian converge}
Fix any $\tau \in \R$ and any $\varepsilon \ll 1.$ For any $t \in (\frac{1}{N},\frac{1}{2}),$ $\alpha >t,$ and $\beta >t$, the following expression is analytic for $u \in (-t-\varepsilon, -t + \varepsilon)$ with a well-defined limit:
    \begin{equation}\label{eq: diff of Pfaffian equation}
    \begin{split}
        &\lim_{u\rightarrow -t}- \mathrm{Pf}\left(J - \bold{K} - \Gamma(u+t)\ketbra{\begin{array}{c}
        \phi_2\\
        -\phi_1 
        \end{array}}
        {B \quad -B^{\prime}} - \Gamma(u+t)\ketbra{\begin{array}{c}
            B\\
            -B^{\prime} 
        \end{array}}{-\phi_2 \quad \phi_1}\right)\\
        &=-\mathrm{Pf}\left(J - \widehat{\bold{K}} - \ketbra{\begin{array}{c}
            \widehat{\phi}_2\\
            -\widehat{\phi}_1 
        \end{array}}
        {\widehat{B} \quad -\widehat{B}^{\prime}} - \ketbra{\begin{array}{c}
            \widehat{B}\\
            -\widehat{B}^{\prime} 
        \end{array}}{-\widehat{\phi}_2 \quad \widehat{\phi}_1}\right)
    \end{split}
    \end{equation}
    where the Fredholm Pfaffian is taken over $\mathbb{L}^2((\tau,\infty))$.
\end{lemma}

\begin{proof}
    From the proof of Lemma \ref{lemma: ptws_convergence_one_param}, we have that there exists $C>0$ such that for all $u \in (-t-\varepsilon, -t+\varepsilon)$,
    \begin{equation}
    |K(X,Y)|, |\partial_XK(X,Y)|,|\partial_YK(X,Y)|,|\partial_X\partial_YK(X,Y)|\leq \frac{C}{2Nd}e^{-d(X+Y)}.
    \end{equation}
    for some $C > 0$ independent of $X,Y$. To apply the Hadamard's bound and dominated convergence theorem, we need the same type of upper bounds for
    \begin{equation}
    \phi_2(U)B(V), \phi_2(U)B'(V), \phi_1(U)B(V),\phi_1(U)B'(V)
    \end{equation}
    where $(U,V) = (X,Y)$ or $(Y,X)$. From the proof of Lemma \ref{lemma: ptws_convergence_one_param}, we have the following upper bounds for all $u \in (-t-\varepsilon, -t + \varepsilon)$,
    \begin{equation}
    \phi_1(U), \phi_2(U) \leq \frac{C_1}{2Nd}e^{-dU}, \quad
    B(V), B'(V)\leq \frac{C_2}{2Nd}e^{-dV}
    \end{equation}
    for some $C_1, C_2 > 0$ independent of $U,V$. Lastly, by Lemma \ref{lemma: ptws_convergence_one_param} and Lemma \ref{lem: limits of function for one param}, we know that the Pfaffian kernels converges pointwise to their desired limits.
\end{proof}

\begin{proof}[Proof of Theorem~\ref{thm: one_param_finite}]
We set $\beta  = \alpha$ and $\gamma = \Gamm(\alpha - u)$. We will use the notation $Z^{u,-}_{\alpha >t}(N,N)$. Under this case, we have $Z^{u,-}_{\alpha >t}(N,N) \stackrel{(d)}{\Rightarrow} Z^t(N,N)$ when $u\rightarrow -t$. Reproving Lemmas \ref{lem: difference of pfaffians}, \ref{lemma: ptws_convergence_one_param}, \ref{lemma:pfaffian_convergence}, \ref{lem: limit of the constant for one param}, \ref{lem: limits of function for one param}, and \ref{lem: diff of Pfaffian converge} with $\beta = \alpha$ shows that RHS of \eqref{eq: One_param_formula} is the analytic continuation of the LHS of \eqref{eq:shift}.
To prove the theorem, we simply need to show the following limit for the RHS of \eqref{eq:shift}:
\begin{equation}\label{eq: final limit for low density}
    \begin{aligned}
        \lim_{u\rightarrow -t}\E\bigg[ 2\left(e^{-\tau}Z^{u,-}_{\alpha >t}(N,N)\right)^{\frac{u+t}{2}}K_{-u-t}\left( 2
            \sqrt{{Z^{u,-}_{\alpha >t}(N,N)e^{-\tau}}}\right) \bigg] = \E\bigg[ 2K_0\left( 2
            {e^{\left(\log Z^{t}(N,N)-\tau\right)/2} }\right)  \bigg].
    \end{aligned}
\end{equation}
Let $\xi = e^{-\tau}.$
Then using $e^{-y} \leq y^{-1}$, we get
\begin{equation}\label{eq: u to -t step 1}
    \begin{aligned}
        &\E\left[ \int_0^\infty e^{-\xi Z^{u,-}_{\alpha >t}(N,N)x}x^{-(u+t)-1}e^{-x^{-1}}dx\right]\\
        &\leq \E\left[ \int_{0}^{\infty}(\xi Z^{u,-}_{\alpha >t}(N,N)x)^{-1}x^{-(u+t)-1}e^{-x^{-1}} dx\right]
     = \E\left[ (\xi Z^{u,-}_{\alpha >t}(N,N))^{-1}\right]\int_{0}^{\infty} x^{-2-(u+t)}e^{-x^{-1}} dx.
\end{aligned}
\end{equation} for all $u\in (-t-\varepsilon,-t+\varepsilon)$. Let $Y_1 \sim \Gamm(\alpha -t)$, $Y_2 \sim \Gamm(\alpha + t)$, $Y_i\sim \Gamm(2\alpha)$ for $i = 3,\dots 2N-4$. For all $u \in (-t-\varepsilon,-t+\varepsilon),$ we have
\begin{equation}\label{eq: u to -t step 3}
    \begin{aligned}
        \E\left[ (\xi Z^{u,-}_{\alpha >t}(N,N))^{-1}\right] \leq \xi^{-1}\E\left[  Y_1^{-1}\right]\E\left[  Y_2^{-1}\right]\E\left[  Y_3^{-1}\right]^{4-2N} < \infty
    \end{aligned}
\end{equation}
because $Z^{u,-}_{\alpha >t}(N,N) \geq \prod_{i = 1}^{2N-2} Y_i$, which is product of weights over a single up-right path. Hence, 
we can apply the dominated convergence theorem to get \eqref{eq: final limit for low density}.
\end{proof}

\begin{proof}[Proof of Theorem~\ref{thm: two param beta > t}]
We consider the case when $\gamma = 1$ and $Z^{u,-}_{\beta >t} \stackrel{(d)}{\Rightarrow } Z^{t,\beta}$ when $u\rightarrow -t.$ Again, by Lemmas \ref{lem: difference of pfaffians}, \ref{lemma: ptws_convergence_one_param}, \ref{lemma:pfaffian_convergence}, \ref{lem: limit of the constant for one param}, \ref{lem: limits of function for one param}, and \ref{lem: diff of Pfaffian converge}, we found the analytic continuation of $\Gamma(u+t)\Pf(J-\check{\bold{K}})_{\mathbb{L}^2(\tau,\infty)}$ and prove the $u\rightarrow -t$ limit is RHS of \eqref{eq: Two_param_formula_positive}.
Since the weight at $(2,1)$ is $1$, we isolate the random variable $\omega_{2,1}$ from the expectation and get
\begin{equation}\label{eq: two param beta large pre limit}
    \begin{aligned}
        &\E\bigg[ 2(e^{-\tau}Z^{t,\beta}(N,N)\omega_{2,1})^{\frac{u+t}{2}}K_{-u-t}\left( 2
        {e^{\left(\log Z^{u,-}_{\beta >t}(N,N)\omega_{2,1}-\tau\right)/2} }\right)  \bigg]\\
        &= \frac{1}{\Gamma(\beta - t)}\int_{0}^{\infty} \E\bigg[ \int_{0}^{\infty} e^{-e^{-\tau} Z^{u,-}_{\beta >t}(N,N)wx}e^{-x^{-1}}x^{-(u+t)-1} dx \bigg]w^{t-\beta-1} e^{-w^{-1}}dw.
    \end{aligned}
\end{equation}
We simply need to show that
\begin{equation}
    \lim_{u\rightarrow -t} \text{RHS of }\eqref{eq: two param beta large pre limit} = \frac{1}{\Gamma(\beta - t)}\E \bigg[ \int_{0}^{\infty} 2K_0\left( 2\sqrt{e^{-\tau} Z^{t,\beta}(N,N) w} \right) w^{t-\beta-1} e^{-w^{-1}}dw\bigg].
\end{equation}
This can be justified by the same arguments as in \eqref{eq: u to -t step 1} and \eqref{eq: u to -t step 3}. The only difference is that 
\begin{equation}
    \begin{aligned}
        &\int_{0}^{\infty} \E\bigg[ \int_{0}^{\infty} e^{-\xi Z^{u,-}_{\beta >t}(N,N)wx}e^{-x^{-1}}x^{-(u+t)-1} dx \bigg]w^{t-\beta-1} e^{-w^{-1}}dw\\
        &\leq \int_{0}^{\infty} \E\bigg[{(-\xi Z^{u,-}_{\beta >t}(N,N))^{-1}}  \bigg]\left(\int_{0}^{\infty} {e^{-x^{-1}}x^{-(u+t)-2}} dx\right) w^{t-\beta-2} e^{-w^{-1}}dw<\infty.
    \end{aligned}
\end{equation}
\end{proof}

\section{Asymptotic analysis for product stationary and two-parameter stationary formulas when $\tbeta > \ttt$}
In this section, we perform steepest descent method for the asymptotic limit of \eqref{eq: Two_param_formula_positive} with $\beta$ parameter involved. The case when $\beta = \alpha$ with no scaling is easier and can be proved by nearly identical arguments. Therefore, we omit the detailed proof for \eqref{eq: one param asymptotic formula}.
We recall the notations for the digamma function and polygamma functions:
\begin{equation}
    \psi(z) = \frac{1}{\Gamma(z)} \frac{\mathrm{d}}{\mathrm{d}z}\Gamma(z),
\qquad \qquad
\psi^{(n)}(z) = \frac{\mathrm{d}^n}{\mathrm{d}z^n}\psi(z).
\end{equation}
We define the following two constants:
\begin{equation}
    f = 2\psi(\alpha), \quad \sigma =  -\psi^{(2)}(\alpha),
\end{equation}
Given the mean $-Nf$ and the fluctuation order $N^{-1/3},$ we consider the following scaling of parameters near the critical value $0$:
\begin{equation}\label{eq: essential scaling}
\begin{aligned}
    \beta = \frac{\tbeta}{\Nsigma}, \quad t = \frac{\ttt}{(\sigma N)^{1/3}}, &\quad Z = \frac{-\tZ}{(\sigma N)^{1/3}}, \quad W =  \frac{-\tW}{(\sigma N)^{1/3}},\\
    X = -Nf + (\sigma N)^{1/3}u, \quad Y &= -Nf + (\sigma N)^{1/3}v, \quad \tau = -Nf + (\sigma N)^{1/3}s.
\end{aligned}
\end{equation}

\begin{define}\label{def:contour}
    For any $a \in \R$ and $\theta \in [0,2\pi)$, we define the following contours on the complex plane:
    \begin{equation}
    V(a;\theta;d) =\left( \{a+re^{\I\theta}: r > 0\} \cup \{a-re^{-\I\theta}: r \leq 0\}\right) \cap \{z \in \CC: |\text{Re}(z)| \leq d\},
    \end{equation}
    \begin{equation}
    L(a;b) = \{a + k\I: |k| \geq b\}.
    \end{equation}
\end{define}

In the following lemmas, we start to use $(u,v)$ as new variables for the asymptotics. We write
\begin{equation}
    \begin{aligned}
        &\partial_u\partial_v\widehat{K}(u,v) = (\sigma N)^{2/3}\partial_X \partial_Y\widehat{K}(X,Y),\\
        &\widehat{K}(u,v) = \widehat{K}(X,Y) = \widehat{K}(-Nf+(\sigma N)^{1/3}u, -Nf+(\sigma N)^{1/3}v).
    \end{aligned}
\end{equation}

\begin{lemma}\label{lem: limit and upper tail of kernel}
    Choose any $r>0.$ Fix any $\alpha > 0$, $\tbeta > \ttt,$ $\ttt >0.$ The following limit holds uniformly over $u,v \in [-r,r]$:
    \begin{equation}
    \begin{aligned}
        &\lim_{N\rightarrow \infty} \widehat{K}(X,Y)  = \widetilde{K}(u,v), \quad \lim_{N\rightarrow \infty} -(\sigma N)^{1/3}\partial_X\widehat{K}(X,Y)  = -\partial_u\widetilde{K}(u,v)\\
        &\lim_{N\rightarrow \infty} -(\sigma N)^{1/3}\partial_Y\widehat{K}(X,Y)  = -\partial_v\widetilde{K}(u,v), \quad \lim_{N\rightarrow \infty} (\sigma N)^{2/3}\partial_X\partial_Y\widehat{K}(X,Y)  = \partial_u\partial_v\widetilde{K}(u,v)
    \end{aligned}
    \end{equation}
    Moreover, fix any $\tilde{\delta}\in(0,\ttt)$, there exists a constant $C$ independent of $u,v$ such that the following upper bounds hold for all $u,v\geq -r$:
    \begin{equation}
        \begin{aligned}
            |\widehat{K}(u,v)|, \,|\partial_u\widehat{K}(u,v)|,\, |\partial_v\widehat{K}(u,v)|, \,|\partial_u\partial_v\widehat{K}(u,v)| \leq Ce^{-\tilde{\delta}(u+v)}.
        \end{aligned}
    \end{equation}
\end{lemma}

\begin{proof}
    After change of variables, we have that 
    \begin{equation}
    \begin{aligned}
        \widehat{K}(u,v) =
        &\int\limits_{\I\R - d} \frac{dZ}{2\pi\I} \int\limits_{\I\R - d} \frac{dW}{2\pi\I} e^{N(h(Z)+h(W)) + \sigma N^{1/3}(uZ + vW)} H(t,\alpha,Z,W)
    \end{aligned}
\end{equation}
where $h$ is defined as the following function on the complex plane:
\begin{equation}
    h(Z) = \log \Gamma(\alpha+Z) - \log \Gamma(\alpha-Z) -fZ,
\end{equation}
and 
\begin{equation}
\begin{aligned}
    H(t,\alpha,Z,W)=&\frac{\Gamma(-t+Z)}{\Gamma(-t-Z)}\frac{\Gamma(t+Z)}{\Gamma(t-Z)} \frac{\Gamma(\beta+Z)}{\Gamma(\beta-Z)}\frac{\Gamma(-t+W)}{\Gamma(-t-W)}\frac{\Gamma(t+W)}{\Gamma(t-W)}\frac{\Gamma(\beta + W)}{\Gamma(\beta-W)}\\
    &\Gamma(-2Z) \Gamma(-2W)\frac{\sin(\pi(Z-W))}{\sin(\pi(Z+W))}\left(\frac{\Gamma(\alpha-Z)}{\Gamma(\alpha+Z)}\frac{\Gamma(\alpha-W)}{\Gamma(\alpha+W)}\right)^{2}.
\end{aligned}
\end{equation}
Notice that 
\begin{equation}
h(0) = h'(0) = h''(0) = 0 \quad \text{ and } \quad h'''(0) = 2\psi^{(2)}(\alpha).
\end{equation} 
Recall that if $z \in \C$, then $\text{Re} \{ \log(z)\} = \log|z|$. Thus
\begin{equation}
\text{Re}\{h(Z)\} = \log\left|\frac{\Gamma(\alpha+Z)}{\Gamma(\alpha-Z)}\right| - f\text{Re}\{Z\}.
\end{equation}
Given the product representation of the Gamma function
\begin{equation}
\Gamma(z) = \frac{1}{z} \prod_{n=1}^\infty (1+\frac{z}{n})^{-1}(1+\frac{1}{n})^{z},
\end{equation}
we have that
\begin{equation}
\begin{split}
    \text{Re}\{h(-d+iy)\}& = \log \left|\frac{(\alpha+d-iy)}{(\alpha-d+iy)}\right| \left| \prod_{n=1}^\infty \frac{(n + \alpha+d-iy)}{(n+\alpha-d+iy)}\right|  \left|\left(\frac{n+1}{n}\right)^{-2d}\right| + fd.
\end{split}
\end{equation}
Now we take the derivative with respect to $y$ and get that
\begin{equation}
\begin{split}
    \frac{d}{dy}\text{Re}\{h(-d+iy)\}& = \frac{d}{dy}\log \left|\frac{(\alpha+d-iy)}{(\alpha-d+iy)} \prod_{n=1}^\infty \frac{(n + \alpha+d-iy)}{(n+\alpha-d+iy)}\right| \\
    &= \frac{d}{dy} \sum_{n=0}^\infty \log|n+\alpha+d-iy| - \log|n+\alpha-d+iy|\\
    &= \sum_{n=0}^\infty \frac{-4d(n+\alpha)y}{((n+\alpha-d)^2 + y^2)((n+\alpha+d)^2 + y^2)}.
\end{split}
\end{equation}
Because the above derivative is positive when $y < 0$ and negative when $y>0$, we deduce that $\text{Re}\{h(Z)\}$ attains its maximum at $Z = -d$ on the contour $-d + i\R$.
Additionally, for $x \in \R$, 
\begin{equation}
\frac{d}{dx} h(x) = \psi(\alpha+x) + \psi(\alpha - x) - 2\psi(\alpha). 
\end{equation}
Because $\psi(x)$ is strictly concave for $x > 0$, we know that $\frac{d}{dx}h(x) < 0$. We see that for all $x \in (-d,0),$ $h(x) < 0$ and therefore, for all $z = x+\I\R$, $\text{Re}(h(z)) < h(x)$. Then we can see that for $a,b \in \R_{+}$ small satisfying $b>a$, the integration contour $V(-a;2\pi/3;b) \cup L(-b; \sqrt{3}(b-a))$ is a steepest descent path for $h(Z)$.  

Thus, fix any $\delta \in (0, \ttt)$, $b > \ttt,$ and set $a = \frac{\delta}{(\sigma N)^{1/3}}$. By Cauchy's theorem, we can
deform the contour $\I\R - d$ into $V(\frac{-\delta}{(\sigma N)^{1/3}};2\pi/3;b) \cup L(-b; \sqrt{3}(b-\frac{\delta}{(\sigma N)^{1/3}}))$ without crossing any poles. Given the descent path, we know that for $(Z,W)$ on $L(-b; \sqrt{3}(b-\frac{\delta}{(\sigma N)^{1/3}}))$, there exists some $\varepsilon>0$ such that $h(Z)+h(W) < -\varepsilon < 0$. Let $L = L(-b; \sqrt{3}(b-\frac{\delta}{(\sigma N)^{1/3}}))$ for shorter notation. Then we have
\begin{equation}
\begin{split}
    &\left|\int\limits_{L} \frac{dZ}{2\pi\I} \int\limits_{L} \frac{dW}{2\pi\I} e^{N\left(h(Z) + h(W)\right) + (\sigma N)^{1/3}(uZ + vW) }H(t,\alpha, Z,W)\right|\\
    &\leq Ce^{-\varepsilon N}\int\limits_{L} \frac{dZ}{2\pi\I} \int\limits_{L} \frac{dW}{2\pi\I} e^{N\left(h(Z) + h(W) + \varepsilon\right) + (\sigma N)^{1/3}(uZ + vW) }H(t,\alpha, Z,W),
\end{split}
\end{equation} 
where the double integrals are bounded due to the exponential decay in $\exp(N(h(Z)+h(W)+\varepsilon))$ on $L\times L$. So only the integrals over $V(-\delta/{(\sigma N)^{1/3}};2\pi/3;b)$ contribute to the limit.

We apply the change of variables
$Z = -\frac{\tZ}{(\sigma N)^{1/3}},$ $W = -\frac{\tW}{(\sigma N)^{1/3}}$. Since $e^{{\tZ}^{3}/3+{\tW}^{3}/3-u\tZ-v\tW}$ decays exponentially 
along the directions $\pm\pi/3$ and $H(t,\alpha,Z,W)$ exhibits at most polynomial 
growth in $\tZ,\tW$, the absolute value of the integrand is bounded by an integrable 
function, allowing the contours to be extended to infinity. Hence, we see that
\begin{equation}
\begin{aligned}
    &\widehat{K}(u,v) = \widetilde{K}(u,v)\left(1 + \mathcal{O}\!\left((\sigma N)^{-1/3}\right)\right),\quad 
        (\sigma N)^{1/3}\partial_X\widehat{K}(u,v) = \partial_u\widetilde{K}(u,v)\left(1 + \mathcal{O}\!\left((\sigma N)^{-1/3}\right)\right),\\
        &(\sigma N)^{1/3}\partial_Y\widehat{K}(u,v) = \partial_v\widetilde{K}(u,v)\left(1 + \mathcal{O}\!\left((\sigma N)^{-1/3}\right)\right),\,\,\,(\sigma N)^{2/3}\partial_X\partial_Y\widehat{K}(u,v) = \partial_u\widetilde{K}(u,v)\left(1 + \mathcal{O}\!\left((\sigma N)^{-1/3}\right)\right).
    \end{aligned}
\end{equation}
\end{proof}

\begin{lemma}\label{lem: converge of Pfaffian 2param beta > t}
    Choose any $\alpha>0,$ $\ttt >0,$ $\tbeta > \ttt$, and $s \in \R$ and set $\tau = -Nf + (\sigma N)^{1/3}s.$ We have the following limit of the Fredholm Pfaffian:
    \begin{equation}\label{lem: limit of pf}
        \lim_{N\rightarrow \infty} \Pf\left( J - \widehat{\bold{K}} \right)_{\mathbb{L}^2(\tau, \infty)} = \Pf\left( J - \widetilde{\bold{K}} \right)_{\mathbb{L}^2(s,\infty)}.
    \end{equation}
\end{lemma}

\begin{proof}
Let $\rho = (\sigma N)^{-1/3}.$ We have
    \begin{equation}
        \begin{aligned}
            &\Pf\left( J - \widehat{\bold{K}} \right)_{\mathbb{L}^2(\tau, \infty)}\\
            &= \sum_{n = 0}^{\infty} \frac{(-1)^n}{n!}\int_{s}^{\infty} \cdots \int_{s}^{\infty} \Pf\left(\begin{pmatrix}
                \widehat{K} & -\rho^{-1}\partial_Y \widehat{K}\\
                -\rho^{-1}\partial_X \widehat{K} & \rho^{-2}\partial_X\partial_Y \widehat{K}
            \end{pmatrix}(u_i,u_j)\right)_{i,j = 1}^n \prod_{i = 1}^n du_i\\
            &\leq \sum_{n = 0}^{\infty} \frac{(-1)^n}{n!}\int_{s}^{\infty} \cdots \int_{s}^{\infty} (2n)^{n/2}C^n\prod_{i = 1}^n e^{-2\delta u_i} \prod_{i = 1}^n du_i < \infty.\\
        \end{aligned}
    \end{equation}
    By dominated convergence theorem and the limit of the kernel in Lemma \ref{lem: limit and upper tail of kernel}, we get the convergence of Fredholm Pfaffian.
\end{proof}

\begin{lemma}\label{lem: limit of functions 2param beta >t}
    Choose any $r >0.$ For any fixed $\alpha,\beta >0,$ $\ttt>0,$ the following limit holds uniformly over $u \in [-r,r],$
    \begin{equation}
        \begin{aligned}
            &\lim_{N\rightarrow \infty} (\sigma N)^{-1/3}\widehat{C} = \widetilde{C},\quad
            \lim_{N\rightarrow \infty} \widehat{B}(X) = \widetilde{B}(u), \quad \lim_{N\rightarrow \infty} (\sigma N)^{1/3}\widehat{B}^{\prime}(X) = \widetilde{B}^{\prime}(u), \\ &\lim_{N\rightarrow \infty} \widehat{\phi}_1(X) = \widetilde{\phi}_1(u), \quad \lim_{N\rightarrow \infty} (\sigma N)^{-1/3}\widehat{\phi}_2(X) = \widetilde{\phi}_2(u).
        \end{aligned}
    \end{equation}
\end{lemma}

\begin{proof}
    The convergences follow the same steepest descent method as in the proof of Lemma \ref{lem: limit and upper tail of kernel}.
\end{proof}

\begin{lemma}\label{lem: 1param pfaffian diff}
    Choose any $\alpha > 0,$ $\ttt>0$, $\tbeta > \ttt$, and any $s\in \R$ with $\tau = -Nf + (\sigma N)^{1/3}s.$ We have the following limit:
    \begin{equation}
        \begin{aligned}
            \lim_{N\rightarrow \infty} &-\mathrm{Pf}\left(J - \widehat{\bold{K}} - (\sigma N)^{-1/3}\ketbra{\begin{array}{c}
            \widehat{\phi}_2\\
            -\widehat{\phi}_1 
        \end{array}}
        {\widehat{B} \quad -\widehat{B}^{\prime}} - (\sigma N)^{-1/3}\ketbra{\begin{array}{c}
            \widehat{B}\\
            -\widehat{B}^{\prime} 
        \end{array}}{-\widehat{\phi}_2 \quad \widehat{\phi}_1}\right)\\
        &=-\mathrm{Pf}\left(J - \widetilde{\bold{K}} - \ketbra{\begin{array}{c}
            \widetilde{\phi}_2\\
            -\widetilde{\phi}_1 
        \end{array}}
        {\widetilde{B} \quad -\widetilde{B}^{\prime}} - \ketbra{\begin{array}{c}
            \widetilde{B}\\
            -\widetilde{B}^{\prime} 
        \end{array}}{-\widetilde{\phi}_2 \quad \widetilde{\phi}_1}\right).
        \end{aligned}
    \end{equation}
\end{lemma}

\begin{proof}
    By Lemma \ref{lem: limit and upper tail of kernel}, Lemma \ref{lem: limit of functions 2param beta >t} and Lemma \ref{lem: upper tail for phi and C}, we get the pointwise limits and upper bounds of each function and the matrix kernel $\widehat{\mathbf{K}}$. To upgrade to the convergence of Fredholm Pfaffian, we use upper bounds $e^{-\delta u_i}$, Hadamard's bound on kernels and the dominated convergence theorem as explained in Lemma \ref{lem: converge of Pfaffian 2param beta > t}.
\end{proof}

\begin{proof}[Proof of Theorem~\ref{thm: One param asymptotic}]
We first set $\beta = \alpha$ and do not apply any scaling to $\alpha$. We then reprove Lemmas~\ref{lem: limit and upper tail of kernel}, \ref{lem: converge of Pfaffian 2param beta > t}, \ref{lem: limit of functions 2param beta >t}, and \ref{lem: 1param pfaffian diff} in this setting, which yields the corresponding limiting objects $\widetilde{C}_L$, $\widetilde{\phi}_1^L$, $\widetilde{\phi}_2^L$, $\widetilde{B}_L$,
$\widetilde{\mathbf{K}}_L$.
\end{proof}

We will take the scaling limits of the following proposition \cite[Proposition E.1.]{borodin2015height} to uniquely determine the distributions of $\frac{\log Z^{t}(N,N)+Nf}{(\sigma N)^{1/3}}$ and $\frac{\log Z^{t,\beta}(N,N)+Nf}{(\sigma N)^{1/3}}$.
\begin{prop}
    Let $R$ be a random variable and $\sigma$ be a strictly positive constant such that there exists some $\delta > 0$ and $\E[e^{-{\delta R/\sigma}}] < \infty$. Let
    \begin{equation}
    Q(x,\sigma) = \E[2\sigma K_0(2e^{(R-x)/(2\sigma)})].
    \end{equation}
    Then, we can express the cumulative distribution function of $R$ as
    \begin{equation}
    \PP(R \leq r) = \frac{1}{2\sigma^2\pi \I} \int\limits_{-\delta + i\R} \frac{d \xi}{\Gamma(-\xi)\Gamma(-\xi + 1)} \int\limits_{\R} dx e^{x \xi /\sigma} Q(x+r,\sigma).
    \end{equation}
\end{prop}

\begin{proof}[Proof of Theorem~\ref{thm: One param probability}]
For simpler notation, let $\X_N = \frac{\log Z^{t}(N,N) + Nf}{(\sigma N)^{1/3}}.$ We define $F_N(r) := \Pb\left(\X_N \leq r\right).$
By Theorem \ref{thm:lower_tail_prodstat}, we know that for any $\epsilon>0,$ there exists $N_0 \in \mathbb{N}$, $C>0$ such that for all $N\geq N_0$ and $x>0$,
\begin{equation}
    \Pb\left( \X_N\leq x \right) \leq Ce^{-(1-\epsilon)2\ttt x}
\end{equation}
By \cite[(E.2)]{borodin2015height}, we get that for any fixed $N$ and for all $\xi\in\CC$ satisfying $\text{Re}(w) <0,$
\begin{equation}
    \begin{aligned}
        \int\limits_{\R} \, e^{v w}Q_{N,t}^{Low}(v+\tau) dv = \int\limits_{\R} F_N(u+\tau)e^{u w}\Gamma(-w +1)\Gamma(-w) du,
    \end{aligned}
\end{equation}
where $F(u) = \Pb\left( \log(Z^t(N,N)) \leq u\right)$. Then applying the scaling
\begin{equation}
    \tau = r(\sigma N)^{1/3} - Nf, \quad v =x\sigma N^{1/3}, \quad w = \xi/\sigma N^{1/3},\quad u = z\sigma N^{1/3}.
\end{equation}
with $\xi = -A+\I y$ gives
\begin{equation}\label{eq: fourier two param beta large}
\begin{aligned}
\int\limits_{\R}e^{-Ax}e^{\I x y}\frac{ {Q}_{N,t}^{Low}((x+r)\sigma N^{1/3} - Nf)}{\Gamma(-\xi/(\sigma N^{1/3})) \Gamma(-\xi/(\sigma N^{1/3}) + 1)} dx &= \int\limits_{\R} F_N(z+r) e^{-Az} e^{\I z y} dz.\\
\end{aligned}
\end{equation}
Since the sequence of random variables $\left(\X_N\right)_{N \in \mathbb{N}}$ is tight by Theorem \ref{thm:lower_tail_prodstat} and Theorem \ref{thm:two_para_upper_tail}, we know that by Prohorov's theorem, 
there exists a subsequence ${N_k}$ and a limiting random variable $\X$ with CDF $\widetilde{F}$ such that 
\begin{equation}
    \widetilde{F}(z) = \lim_{N\rightarrow \infty} F_{N_k}(z).
\end{equation}
for all continuity points of $\widetilde{F}$.
The limiting distribution also admits the same lower tail behavior, i.e., there exists $C>0,$ such that for all $x>0,$
\begin{equation}
    \Pb\left(\X \leq -x\right)\leq Ce^{-2\ttt x}.
\end{equation}
We recall that $\delta$ is chosen with $0<\delta<\ttt$ for the steepest-descent contours in Lemma \ref{lem: limit and upper tail of kernel}.
By the lower tail and upper tail of $Q_{N,t}^{Low}$ provided in Corollary \ref{cor: 1param and 2param beta large lower tails} and Lemma \ref{lem: 1param 2param beta large upper tail Q}, we see that for any $\epsilon >0,$ there exists $c,C >0$, $N_0$ and $s_0 >0$ such that for all $N\geq N_0$,
\begin{equation}
    \begin{aligned}
        &(\sigma N)^{-1/3} Q_{N,t}^{Low}((x+r)(\sigma N)^{1/3} - Nf)\leq C|x|e^{-(1-\epsilon)2\ttt |x|} \quad \text{for all } x<0 ,\\
        &(\sigma N)^{-1/3} Q_{N,t}^{Low}((x+r)(\sigma N)^{1/3} - Nf)\leq Ce^{(\ttt-\delta) |x|} \quad \text{ for all }x>x_0.
    \end{aligned}
\end{equation}
Then fix some $\epsilon \ll 1$ such that $(1-\epsilon)2\ttt > \ttt - \delta$. We need to pick $A$ such that $(1-\epsilon)2\ttt >A > \ttt - \delta$ because $e^{-Ax}$ needs to deay faster than the right tail $e^{(\ttt - \delta)}|x|$ for $x>x_0$ and grows slower than the left tail $|x|e^{-(1-\epsilon) 2\ttt|x|}$ for $x<0$.
By the dominated convergence theorem, we can take the subsequential limit of \eqref{eq: fourier two param beta large} and obtain
\begin{equation}
    \begin{aligned}
        \int\limits_{\R}(A-\I y)e^{-Ax}e^{\I x y} \widetilde{{Q}}_{ \ttt}^{Low}(x+r) dx = \int\limits_{\R}  \widetilde{F}(z+r) e^{-Az} e^{\I z y} dz.
    \end{aligned}
\end{equation}
Since $\widetilde{F}(z + r)e^{-Az} \in \mathbb{L}^2(\R)$, we can apply the Fourier inversion to obtain
\begin{equation}\label{eq: limit CDF two param beta large}
    \widetilde{F}(r) = \frac{1}{2\pi}\int_{\R}\int_{\R} (A-\I y)e^{-A{x}} e^{\I {x} y} \widetilde{{Q}}_{\ttt}^{Low}(x+r)dxdy \quad \text{a.s.}
\end{equation}
Since the limiting function is a CDF, it is uniquely characterized by the RHS of \eqref{eq: limit CDF two param beta large}. Then by Prohorov's theorem, we see that the full sequence $\left(\frac{\log Z(N,N) + Nf}{(\sigma N)^{1/3}}\right)_{N\geq N_0}$ converges weakly to $\X$.
\end{proof}

\begin{proof}[Proof of Theorem~\ref{thm: High density beta large asymptotic}]
This theorem is a direct consequence of Lemmas \ref{lem: limit and upper tail of kernel}, \ref{lem: converge of Pfaffian 2param beta > t}, \ref{lem: limit of functions 2param beta >t}, \ref{lem: 1param pfaffian diff}.
\end{proof}

\begin{proof}[Proof of Theorem~\ref{thm: High density beta large probability}]

Fix any $\tbeta$, $\ttt$ with $\tbeta > \ttt.$
By the tightness proved in Theorem \ref{thm:lower_2para} and Theorem \ref{thm:two_para_upper_tail}, we know that there exists a limiting random variable $\mathcal{X}_{\tbeta}$ such that $\frac{\log Z^{t,\beta}(N,N) + Nf}{(\sigma N)^{1/3}}$ converges weakly along a subsequence. Define $\widetilde{F}_{\tbeta}(r) = \Pb(\X_{\tbeta} \leq r).$

We can use the same argument as in the proof of Theorem \ref{thm: One param probability} since ${Q}_{N,\beta > t}^{High}$ and ${Q}_{N,t}^{Low}$ have the same lower tail and upper tail by Corollary \ref{cor: 1param and 2param beta large lower tails} and Lemma \ref{lem: 1param 2param beta large upper tail Q}. Fix some $1\gg \epsilon >0$ such that $(1-2\epsilon)2\ttt > \ttt - \delta$. Notice that in the steepest–descent analysis, the parameter $\delta$ can be chosen arbitrarily close to $\ttt$. Then we need to choose $A$ such that $(1-\epsilon)2\ttt>A>\ttt- \delta>0$. By dominated convergence theorem, we obtain the limit
\begin{equation}\label{eq: fourierIdentity}
\begin{aligned}
    &\frac{(A - \I y)}{2\pi}\int_{\R}\widetilde{Q}_{\tbeta > \ttt}^{High}(x + r - u)e^{-Ax}e^{\I x y} dx\\
    &\quad =  \frac{1}{2\pi}\int_{\R} \int_{0}^{\infty} \widetilde{F}_{\tbeta}(x+r-u)(\tbeta - \ttt)e^{-(\tbeta - \ttt)u}e^{-Ax}e^{\I x y}du dx,
\end{aligned}
\end{equation}
Under the scaling, $(\sigma N)^{-1/3}\Gamm((\tbeta - \ttt)/(\sigma N)^{1/3})$ converges to $\text{Exp}(\tbeta - \ttt)$.
Applying the fourier inversion and the shift argument for removing an independent exponential random variable, which can be found in \cite[Lemma 3.3]{PatrikStatExp}, to the RHS of \eqref{eq: fourierIdentity} yields
\begin{equation}\label{eq: twoParamShift}
\begin{aligned}
    &\left(1 + \frac{1}{\tbeta - \ttt} \partial_{r}\right)\frac{1}{2\pi}\int_{\R}\int_{\R} \int_{0}^{\infty} \widetilde{F}_{\tbeta}(x+r-u)(\tbeta - \ttt)e^{-(\tbeta - \ttt)u}e^{-Ax}e^{\I x y}du dx dy\\
    &=\left(1 + \frac{1}{\tbeta - \ttt} \partial_{r}\right)\Pb\left(\X_{\tbeta} + \text{Exp}(\tbeta - \ttt) \leq r\right) = \Pb\left(\X_{\tbeta} \leq r\right).
\end{aligned}
\end{equation}
Since the LHS of \eqref{eq: twoParamShift} uniquely characterizes the limiting distribution, the subsequential weak convergence can be upgraded to full sequence convergence.
By \eqref{eq: fourierIdentity} and \eqref{eq: twoParamShift}, we get the following identity under the condition $\tbeta >\ttt$:
\begin{equation}\label{eq: dis equal}
\begin{aligned}
     \Pb\left(\X_{\tbeta} \leq r\right) = \left(1 + \frac{1}{\tbeta - \ttt} \partial_{r}\right)\frac{1}{2\pi}\int_{\R}\int_{\R}(A - \I y)\widetilde{Q}_{\tbeta > \ttt}^{High}(x + r)e^{-Ax}e^{\I x y} dx dy.
\end{aligned}
\end{equation}
\end{proof}

\section{Two-parameter stationary case when $\beta <t$}
Recall $Z^{t,\beta}(N,N)$ in Definition \ref{def: two-param stationary model} and $Z^{u}_{\beta < t}(N,N)$ in Definition \ref{def: approximation model}.
\subsection{Pfaffian decomposition for the two-parameter stationary model when $\beta <t$ }\label{section: 2 param decompose}
By Theorem \ref{thm: Pfaffian structure}, we have the following Fredholm Pfaffian identity for $\frac{1}{N} < u,\beta < t$:
\begin{equation}\label{eq: two param beta small starting pfaffian}
\mathbb{E}\left[e^{-e^{-\tau }Z^{u}_{\beta < t}(N,N)}\right]
= 
\Pf(J-\boldsymbol{\check{\mathcal{K}}})_{\mathbb{L}^2((\tau,\infty))},
\end{equation}
where  $\boldsymbol{\check{\mathcal{K}}}$ is a $2 \times 2$ matrix kernel
\begin{equation}
\boldsymbol{\check{\mathcal{K}}}(X,Y)=
\begin{pmatrix}
\check{\K}(X,Y) & -\partial_Y \check{\K}(X,Y) \\
-\partial_X \check{\K}(X,Y) & \partial_X\partial_Y \check{\K}(X,Y)
\end{pmatrix}
\end{equation} and $\check{\K}$ is the same function as defined in \eqref{eq: central kernel K}.
Using the same argument as before, our goal is to find the analytic continuation of \eqref{eq: two param beta small starting pfaffian}.
We temporarily set $u \in (1/N,t)$ for the following kernel decomposition. The contour $-d+ \I\R$ is the vertical line from $-d - \I\infty$ to $-d + \I\infty$ such that $\max\{u,\beta,0\} <d <t$. Note that $u,\beta$ will be tuned to negative values. We define the following auxiliary functions.
Recall definitions of $H_{\alpha}(x),$ $F(x),$ $Q(x,y)$ in \eqref{eq: definitions of H,F,Q}.

\begin{equation}
    \begin{aligned}
    \mathsf{A}(X) &:= F(-u)H_{\alpha}(-u)e^{-Xu},\quad\quad 
    \mathsf{A}'(X) := (-u)F(-u)H_{\alpha}(-u)e^{-Xu}.\\
    \mathsf{B}(Y) &:= \!\!\!\int\limits_{\I\R - d} \!\!\!\frac{dW}{2\pi \I } e^{YW} F(W)H_{\alpha}(W)\Gamma(-2W)\frac{\Gamma(1-u+W)}{\Gamma(1-u-W)},\\
    \mathsf{B}'(Y) &:= \!\!\!\int\limits_{\I\R - d}\!\!\! \frac{dW}{2\pi \I }  We^{YW} F(W)H_{\alpha}(W)\Gamma(-2W)\frac{\Gamma(1-u+W)}{\Gamma(1-u-W)}.\\
    \mathsf{C}(X) &:= \frac{\Gamma(t-\beta)}{\Gamma(t+\beta)}\frac{\Gamma(u-\beta)}{\Gamma(u+\beta)}\left( \frac{\Gamma(\alpha-\beta)}{\Gamma(\alpha+\beta)} \right)^{N-2}\!\!\!e^{-X\beta},\quad
    \mathsf{C}'(X) := -\beta\frac{\Gamma(t-\beta)}{\Gamma(t+\beta)}\frac{\Gamma(u-\beta)}{\Gamma(u+\beta)}\left( \frac{\Gamma(\alpha-\beta)}{\Gamma(\alpha+\beta)} \right)^{N-2}\!\!\!e^{-X\beta}.\\
   \mathsf{D}(Y) &:= \!\!\!\int\limits_{\I\R - d}\!\!\! \frac{dW}{2\pi \I }e^{YW} \frac{\Gamma(t+W)}{\Gamma(t-W)}\frac{\Gamma(u+W)}{\Gamma(u-W)}H_{\alpha}(W)\Gamma(-2W)\frac{\Gamma(1-\beta+W)}{\Gamma(1-\beta-W)}+ e^{-u Y}H_{\alpha}(-u)\frac{\Gamma(t-u)}{\Gamma(t+u)}\frac{\Gamma(1-\beta-u)}{\Gamma(1-\beta+u)},\\
    \mathsf{D}'(Y) &:= \int\limits_{\I\R - d} \frac{dW}{2\pi \I }We^{YW} H_{\alpha}(W)\Gamma(-2W)\frac{\Gamma(t+W)}{\Gamma(t-W)}\frac{\Gamma(u+W)}{\Gamma(u-W)} H_{\alpha}(W)\frac{\Gamma(1-\beta+W)}{\Gamma(1-\beta-W)}\\
    &\quad -ue^{-u Y}\frac{\Gamma(t-u)}{\Gamma(t+u)}H_{\alpha}(-u)\frac{\Gamma(1-\beta-u)}{\Gamma(1-\beta+u)}.
    \end{aligned}
\end{equation}
Let
\begin{equation}
\boldsymbol{{\mathcal{K}}}(X,Y)=
\begin{pmatrix}
{\K}(X,Y) & -\partial_Y {\K}(X,Y) \\
-\partial_X {\K}(X,Y) & \partial_X\partial_Y {\K}(X,Y)
\end{pmatrix},
\end{equation}
where
\begin{equation}
    \begin{aligned}
        {\K}(X,Y) = &\int\limits_{\mathrm{i}\mathbb{R}-d} \frac{dZ}{2\pi i}
\int\limits_{\mathrm{i}\mathbb{R}-d} \frac{dW}{2\pi i}\;
e^{XZ+YW}\;
\frac{\Gamma(u+Z)}{\Gamma(u-Z)}
\frac{\Gamma(u+W)}{\Gamma(u-W)}\frac{\Gamma(\beta+Z)}{\Gamma(\beta-Z)}
\frac{\Gamma(\beta+W)}{\Gamma(\beta-W)}\\
&\frac{\Gamma(t+Z)}{\Gamma(t-Z)}
\frac{\Gamma(t+W)}{\Gamma(t-W)}
H_{\alpha}(Z)\,
H_{\alpha}(W)\,
Q(Z,W).
    \end{aligned}
\end{equation}

We define
\begin{equation}\label{eq: two param beta small decomposition}
\begin{aligned}
    &X_1 = \left(\begin{array}{c}
         \mathsf{A}' \\
         \mathsf{A}
    \end{array}\right), \quad Y_1 = \left( \mathsf{B} \quad -\mathsf{B}'\right), \quad X_2 = \left(\begin{array}{c}
         \mathsf{B}'\\
         \mathsf{B}
    \end{array}\right), \quad Y_2 = \left( -\mathsf{A} \quad \mathsf{A}'\right),\\
    &X_3 = \left(\begin{array}{c}
         \mathsf{C}' \\
         \mathsf{C}
    \end{array}\right), \quad Y_3 = \left( \mathsf{D} \quad -\mathsf{D}'\right), \quad X_4 = \left(\begin{array}{c}
         \mathsf{D}' \\
         \mathsf{D}
    \end{array}\right), \quad Y_4 = \left( -\mathsf{C}\quad \mathsf{C}'\right).
\end{aligned}
\end{equation}

We know that
\begin{equation}\label{eq: two param beta small decomposition for kernel}
    \begin{aligned}
        J^{-1}\boldsymbol{\check{\mathcal{K}}} = J^{-1}{\boldsymbol{\mathcal{K}}} + \ketbra{X_1}{Y_1} + \ketbra{X_2}{Y_2} + \ketbra{X_3}{Y_3} + 
        \ketbra{X_4}{Y_4}.
    \end{aligned}
\end{equation}

At this stage of the decomposition, we subtract from $\boldsymbol{\mathcal{K}}$ the contributions of the poles at $Z = -\beta,-u$ and $W = -\beta, -u$. Corresponding residues have already been evaluated and stored in $X_i,Y_j$. Now, $\boldsymbol{\mathcal{K}}$ contain poles $\{-u-1 - i, -\beta -1-i, -t-i, - \alpha - i : i \in \Z_{\geq 0}\}$. Then $\boldsymbol{\mathcal{K}}$ can be analytically extended to $\beta \in (-t,t)$ since $-\beta -1 < t-1 < -d$.
Then we decompose the kernel as follows. Let $\G = J^{-1}{\boldsymbol{\mathcal{K}}}$. We have
\begin{equation}
    \begin{aligned}
        &\Pf\left(J - \boldsymbol{\check{\mathcal{K}}}\right)^2 = \det\left(\Id - J^{-1}\boldsymbol{\check{\mathcal{K}}}\right)=
        \det\left(\Id - J^{-1}{\boldsymbol{\mathcal{K}}} - \ketbra{X_1}{Y_1} - \ketbra{X_2}{Y_2} - \ketbra{X_3}{Y_3} - \ketbra{X_4}{Y_4}\right)\\
        &= \det(\Id - \overline{G})
        \det\left(\Id -  \ketbra{(\Id -\overline{G})^{-1}X_1}{Y_1} -  \ketbra{(\Id -\overline{G})^{-1}X_2}{Y_2} - \ketbra{(\Id -\overline{G})^{-1}X_3}{Y_3} - \ketbra{(\Id -\overline{G})^{-1}X_4}{Y_4}\right)\\
        &= \det(\Id- \overline{G})
        \det\left(\Id - \begin{pmatrix}
            \bra{Y_1}\\
            \bra{Y_2}\\
            \bra{Y_3}\\
            \bra{Y_4}
        \end{pmatrix}\begin{pmatrix}
            \ket{(\Id -\overline{G})^{-1}X_1} & \ket{(\Id -\overline{G})^{-1}X_2}&\ket{(\Id -\overline{G})^{-1}X_3}&\ket{(\Id -\overline{G})^{-1}X_4}        
        \end{pmatrix}\right)\\
        &=  \det(\Id- \overline{G})\\
        &\det\left(\Id - \begin{pmatrix}
            \braket{Y_1}{(\Id - \overline{G})^{-1}X_1} & \braket{Y_1}{(\Id - \overline{G})^{-1}X_2} & \braket{Y_1}{(\Id - \overline{G})^{-1}X_3} & \braket{Y_1}{(\Id - \overline{G})^{-1}X_4}\\
            \braket{Y_2}{(\Id - \overline{G})^{-1}X_1} & \braket{Y_2}{(\Id - \overline{G})^{-1}X_2} & \braket{Y_2}{(\Id - \overline{G})^{-1}X_3} & \braket{Y_2}{(\Id - \overline{G})^{-1}X_4}\\
            \braket{Y_3}{(\Id - \overline{G})^{-1}X_1} & \braket{Y_3}{(\Id - \overline{G})^{-1}X_2} & \braket{Y_3}{(\Id - \overline{G})^{-1}X_3} & \braket{Y_3}{(\Id - \overline{G})^{-1}X_4}\\
            \braket{Y_4}{(\Id - \overline{G})^{-1}X_1} & \braket{Y_4}{(\Id - \overline{G})^{-1}X_2} & \braket{Y_4}{(\Id - \overline{G})^{-1}X_3} & \braket{Y_4}{(\Id - \overline{G})^{-1}X_4}\\
        \end{pmatrix} \right)\\
        &=  \det(\Id- \overline{G})\\
        &\det\left(\Id - \begin{pmatrix}
            \braket{Y_1}{(\Id - \overline{G})^{-1}X_1} & 0 & \braket{Y_1}{(\Id - \overline{G})^{-1}X_3} & \braket{Y_1}{(\Id - \overline{G})^{-1}X_4}\\
            0 & \braket{Y_1}{(\Id - \overline{G})^{-1}X_1} & \braket{Y_2}{(\Id - \overline{G})^{-1}X_3} & \braket{Y_2}{(\Id - \overline{G})^{-1}X_4}\\
            \braket{Y_3}{(\Id - \overline{G})^{-1}X_1} & \braket{Y_3}{(\Id - \overline{G})^{-1}X_2} & \braket{Y_3}{(\Id - \overline{G})^{-1}X_3} & 0\\
            \braket{Y_4}{(\Id - \overline{G})^{-1}X_1} & \braket{Y_4}{(\Id - \overline{G})^{-1}X_2} & 0 & \braket{Y_4}{(\Id - \overline{G})^{-1}X_4}\\
        \end{pmatrix} \right).\\
    \end{aligned}
\end{equation}
By the anti-symmetry of the kernel $\boldsymbol{{\mathcal{K}}}$, we have 
\begin{equation}
    \braket{Y_1}{(\Id - \overline{G})^{-1}X_2} = \braket{Y_2}{(\Id - \overline{G})^{-1}X_1} = \braket{Y_3}{(\Id - \overline{G})^{-1}X_4} = \braket{Y_4}{(\Id - \overline{G})^{-1}X_3} = 0,
\end{equation}
as well as the following identities:
\begin{equation}
    \begin{aligned}
        &\M = 1-\brabarket{Y_1}{(\Id - \overline{G})^{-1}}{X_1} = 1-\brabarket{Y_2}{(\Id - \overline{G})^{-1}}{X_2},\\
        &\A = \brabarket{Y_1}{(\Id - \overline{G})^{-1}}{X_3} = \brabarket{Y_4}{(\Id - \overline{G})^{-1}}{X_2},\\
        &\B = \brabarket{Y_1}{(\Id - \overline{G})^{-1}}{X_4} = -\brabarket{Y_3}{(\Id - \overline{G})^{-1}}{X_2},\\
        &\C = \brabarket{Y_2}{(\Id - \overline{G})^{-1}}{X_3} = -\brabarket{Y_4}{(\Id - \overline{G})^{-1}}{X_1},\\
        &\D = \brabarket{Y_2}{(\Id - \overline{G})^{-1}}{X_4} = \brabarket{Y_3}{(\Id - \overline{G})^{-1}}{X_1},\\
        &\NN = 1-\brabarket{Y_3}{(\Id - \overline{G})^{-1}}{X_3} = 1-\brabarket{Y_4}{(\Id - \overline{G})^{-1}}{X_4}.\\
    \end{aligned}
\end{equation}

Then the previous equation becomes the following:
\begin{equation}
    \begin{aligned}
        &\det(\Id - \overline{G})
        \det\begin{pmatrix}
            \M & 0 & \A & \B\\
            0 & \M & \C & \D\\
            \D & -\B & \NN & 0\\
            -\C & \A & 0 & \NN\\
        \end{pmatrix}\\
        &= \det(\Id - \overline{G})\M^2 \det\left( \begin{pmatrix}
            \NN & 0\\
            0 & \NN
        \end{pmatrix} - \M^{-1}\begin{pmatrix}
            \D & -\B\\
            -\C & \A
        \end{pmatrix}\begin{pmatrix}
            \A & \B \\
            \C & \D
        \end{pmatrix} \right)\\
        &= \det(\Id - \overline{G})\M^2 \det\left( \begin{pmatrix}
            \NN & 0\\
            0 & \NN
        \end{pmatrix} - \M^{-1}\begin{pmatrix}
            \D\A-\B\C & 0\\
            0 & -\C\B+\A\D
        \end{pmatrix}\right)\\
        &= \det(\Id - \overline{G}) \left(\M\NN - \A\D +\B\C\right)^2.\\
    \end{aligned}
\end{equation}
Therefore,
\begin{equation}
\begin{aligned}
    &\Pf(J - \boldsymbol{\check{\mathcal{K}}}) = \Pf(J - {\boldsymbol{\mathcal{K}}})\left(\M\NN - \A\D+\B\C\right).
\end{aligned}
\end{equation}

Our next goal is to find the $u\rightarrow -t$ limit of the following identity:
\begin{equation}
    \begin{aligned}
        \Gamma(u+t)\Pf(J-\boldsymbol{\check{\mathcal{K}}})&=
        \Pf(J - {\boldsymbol{\mathcal{K}}})\Gamma(u+t)\left(\M\NN - \A\D+\B\C\right).\\
    \end{aligned}
\end{equation}

\begin{define}
We define the scalar functions $\aleph_i, \psi_i, \Theta_i,\zeta_i, \eta_i, \theta_i$ for $i \in \{1,2\}$ in terms of the following identities. For $X,Y \in \R$,
\begin{equation}
\begin{aligned}
    &\bra{Y_1}\G(X)= \bra{{\aleph}_1(X) \quad {\aleph}_2(X)}, \quad \bra{Y_2}\G(X) = \bra{{\psi}_1(X) \quad {\psi}_2(X)}, \quad \bra{Y_3}\G(X) = \bra{{\Theta}_1(X) \quad {\Theta}_2(X)},\\
&\G \ket{X_1}(Y) = \ket{\begin{array}{c}
        {\zeta}_1(Y)\\
        {\zeta}_2(Y)
    \end{array}},
    \quad\quad
    \G \ket{X_3}(Y) = \ket{\begin{array}{c}
        {\eta}_1(Y)\\
        {\eta}_2(Y)
    \end{array}}, \quad\quad \G \ket{X_4}(Y) = \ket{\begin{array}{c}
        {\theta}_1(Y)\\
        {\theta}_2(Y)
    \end{array}}.
\end{aligned} 
\end{equation}
\end{define}

The next lemma explains the existence of $\left(\Pf(J - \widehat{\boldsymbol{\mathcal{K}}})\right)^{-1}$ in \eqref{Two_param_formula_negative}. Multiplying each of $\M,$ $\NN,$ $\A$, $\B$, $\C$, $\D$ with $\Pf(J - \widehat{\boldsymbol{\mathcal{K}}})$ converts them into the difference of two Fredholm Pfaffians.
\begin{lemma}\label{lem: prelimit two param formula}
Fix any $N\in\Z_{\geq 3}.$ For any $t\in (\frac{1}{N},\frac12),$ $\alpha > t$, $\beta \in (-t,t),$ and $ t>u > \frac{1}{N},$
we have
    \begin{equation}\label{eq: M and NN}
    \begin{aligned}
        \Pf\left( J - \boldsymbol{\mathcal{K}} \right)\Gamma(u+t)\M
        = &\Pf\left( J - \boldsymbol{\mathcal{K}} \right)\Gamma(u+t)\left(1 - \braket{Y_1}{X_1}- \brabarket{Y_1}{\G}{X_1}\right)\\
        &- \Pf\left( J - \boldsymbol{\mathcal{K}} \right) + \Pf\left( J - {\boldsymbol{\mathcal{K}}} - \Gamma(u+t)\ketbra{\begin{array}{c} \zeta_2 \\
        -\zeta_1
        \end{array}}{\aleph_1 \quad \aleph_2} - \Gamma(u+t)\ketbra{\begin{array}{c}
            \aleph_1\\
            \aleph_2
        \end{array}}{-\zeta_2 \quad \zeta_1 } \right).\\
        \Pf\left( J - \boldsymbol{\mathcal{K}} \right)\NN
        = &\Pf\left( J - \boldsymbol{\mathcal{K}} \right)\left(1 - \braket{Y_3}{X_3}- \brabarket{Y_3}{\G}{X_3}\right)\\
        &- \Pf\left( J - \boldsymbol{\mathcal{K}} \right) + \Pf\left( J - {\boldsymbol{\mathcal{K}}} - \ketbra{\begin{array}{c} \eta_2 \\
        -\eta_1
        \end{array}}{\Theta_1 \quad \Theta_2} - \ketbra{\begin{array}{c}
            \Theta_1\\
            \Theta_2
        \end{array}}{-\eta_2 \quad \eta_1 } \right).
    \end{aligned}
    \end{equation}

    \begin{equation}\label{eq:A,B,C,D}
    \begin{aligned}
        \Pf\left( J - \boldsymbol{\mathcal{K}} \right)\A
        = &\Pf\left( J - \boldsymbol{\mathcal{K}} \right)\left( \braket{Y_1}{X_3}+ \brabarket{Y_1}{\G}{X_3}\right)\\
        &+ \Pf\left( J - \boldsymbol{\mathcal{K}} \right) - \Pf\left( J - {\boldsymbol{\mathcal{K}}} - \ketbra{\begin{array}{c} \eta_2 \\
        -\eta_1
        \end{array}}{\aleph_1 \quad \aleph_2} - \ketbra{\begin{array}{c}
            \aleph_1\\
            \aleph_2
        \end{array}}{-\eta_2 \quad \eta_1 } \right),\\
        \Pf\left( J - \boldsymbol{\mathcal{K}} \right)\B
        = &\Pf\left( J - \boldsymbol{\mathcal{K}} \right)\left(\braket{Y_1}{X_4} + \brabarket{Y_1}{\G}{X_4}\right)\\
        &+ \Pf\left( J - \boldsymbol{\mathcal{K}} \right) - \Pf\left( J - {\boldsymbol{\mathcal{K}}} - \ketbra{\begin{array}{c} \theta_2 \\
        -\theta_1
        \end{array}}{\aleph_1 \quad \aleph_2} - \ketbra{\begin{array}{c}
            \aleph_1\\
            \aleph_2
        \end{array}}{-\theta_2 \quad \theta_1 } \right),\\
        \Pf\left( J - \boldsymbol{\mathcal{K}} \right)\C
        = &\Pf\left( J - \boldsymbol{\mathcal{K}} \right)\Gamma(u+t)\left(\braket{Y_2}{X_3} + \brabarket{Y_2}{\G}{X_3}\right)\\
        &+ \Pf\left( J - \boldsymbol{\mathcal{K}} \right) - \Pf\left( J - {\boldsymbol{\mathcal{K}}} - \Gamma(u+t)\ketbra{\begin{array}{c} \eta_2 \\
        -\eta_1
        \end{array}}{\psi_1 \quad \psi_2} - \Gamma(u+t)\ketbra{\begin{array}{c}
            \psi_1\\
            \psi_2
        \end{array}}{-\eta_2 \quad \eta_1 } \right),\\
        \Pf\left( J - \boldsymbol{\mathcal{K}} \right)\D
        = &\Pf\left( J - \boldsymbol{\mathcal{K}} \right)\Gamma(u+t)\left(\braket{Y_2}{X_4} + \brabarket{Y_2}{\G}{X_4}\right)\\
        &+ \Pf\left( J - \boldsymbol{\mathcal{K}} \right) - \Pf\left( J - {\boldsymbol{\mathcal{K}}} - \Gamma(u+t)\ketbra{\begin{array}{c} \theta_2 \\
        -\theta_1
        \end{array}}{\psi_1 \quad \psi_2} - \Gamma(u+t)\ketbra{\begin{array}{c}
            \psi_1\\
            \psi_2
        \end{array}}{-\theta_2 \quad \theta_1 } \right).
    \end{aligned}
    \end{equation}
  
\end{lemma}
\begin{proof}
    The RHS of the above equations are analytic continuations of the LHS for all $N \in \R_{\geq 3}$. The proof follows the same arguments as in Lemma \ref{lem: difference of pfaffians}. We see that $\Pf\left( J - \boldsymbol{\mathcal{K}} \right)(\M\NN - \A\D + \B\C)$ can be analytically extended to $\beta \in (-t,t)$ because for all $\beta \in (-t,t)$, we have the following upper bounds for each function:
    \begin{equation}
        \begin{aligned}
            &|\K(X,Y)|, \, |\partial_Y\K(X,Y)|,\, |\partial_X\K(X,Y)|, \, |\partial_X\partial_Y\K(X,Y)|Ce^{-d(X+Y)}\\
            &|\aleph_i(X)|, \,|\psi_i(X)|,\, |\Theta_i(X)|,\,|\zeta_i(X)|, \,|\eta_i(X)|, \,|\theta_i(X)| \leq Ce^{-dX}
        \end{aligned}
    \end{equation}
    for all $X,Y \geq \tau$ and for some $C>0$ independent of $X,Y$. Therefore, all Fredholm Pfaffians and $\braket{\cdot}{\cdot}$ (after evaluation) are analytic for $\beta \in (-t,t)$.
\end{proof}

To preview the analysis, section \ref{section: analytic continuation} will be devoted to prove the $u\rightarrow -t$ limit of all formulas on the RHS of \eqref{eq: M and NN} and \eqref{eq:A,B,C,D}.

\section{Difference in kernel decompositions between log-gamma polymer and Geometric Last Passage Percolation}\label{section: difference in models}
We now explain the intuition behind this kernel decomposition and why, in the log–gamma case, a substantially more complicated analysis is required. We begin with the geometric last passage model, where the decomposition is simpler. Let $\mathbf{1}$ denote the matrix kernel with all entries $1$. For example, according to \cite[(6.37)]{zeng2025stationary}, we have
\begin{equation}
    \begin{aligned}
    \begin{pmatrix}
        K_{11}^{geo} & K_{12}^{geo}\\
        K_{21}^{geo} & K_{22}^{geo}
    \end{pmatrix}(k,\ell) \approx \begin{pmatrix}
        s^{k+\ell} & s^k r^\ell\\
        s^\ell r^k & r^{k+\ell}
    \end{pmatrix} + \mathcal{O}((\mathring{t})^{k+\ell})\mathbf{1},
    \end{aligned} 
\end{equation}
where $s<1,$ $r \in (0,1/s),$ and $\mathring{t} \rightarrow 1/s$ to achieve the stationary geometric LPP.
For the log-gamma case,
\begin{equation}
    \begin{aligned}
    \begin{pmatrix}
        \check{\K} & -\partial_Y \check{\K}\\
        -\partial_X \check{\K} & \partial_X\partial_Y \check{\K}
    \end{pmatrix}(X,Y) \approx \begin{pmatrix}
        e^{-t(X+Y)} & -te^{-t(X+Y)}\\
        -te^{-t(X+Y)} & t^2e^{-t(X+Y)}
    \end{pmatrix} + \mathcal{O}(e^{-u(X+Y)})\mathbf{1} + \mathcal{O}(e^{-\beta(X+Y)})\mathbf{1}.
    \end{aligned}
\end{equation}
The parameters are related by
\begin{equation}
    \beta = - \log r, \quad t = -\log s, \quad u = -\log\mathring{t}
\end{equation}
and the ranges correspond via
\begin{equation}
    \begin{aligned}
        s<1 &\quad\Leftrightarrow \quad t > 0, \\
        \frac{1}{s}>r>s &\quad\Leftrightarrow \quad -t< \beta < t.
    \end{aligned}
\end{equation}
In the geometric model we take the limit $\mathring{t}\rightarrow 1/s$, while in the log–gamma model we take $u\rightarrow -t.$ In both cases, the terms $\mathcal{O}((\mathring{t})^{k+\ell})\mathbf{1}$ and $\mathcal{O}(e^{-u(X+Y)})\mathbf{1}$ become exploding terms that obstruct the convergence of the Fredholm Pfaffian.  The key simplifying feature of the geometric kernel is its non-symmetric kernel structure. Specifically, $K_{11}^{geo}$ decays exponentially, while $K_{22}^{geo}$ grows exponentially. Since $r$ cannot exceed $1/s$, the growth in $K_{22}^{geo}$ is precisely compensated by the decay in $K_{11}^{geo}$. This imbalance allows the explosive term $r^{k},r^{\ell}$ within the kernel to be controlled.

The situation changes fundamentally in the log–gamma case. The kernel is fully symmetric: permuting the parameters $t,\beta, u$ does not change the kernel. Consequently, we can no longer use the decay in one entry to compensate for the growth in another. Instead, all four entries of the kernel either decay or grow simultaneously. Thus, the potential exponential growth from $\mathcal{O}(e^{-\beta(X+Y)})\mathbf{1}$ and $\mathcal{O}(e^{-u(X+Y)})\mathbf{1}$ must be isolated explicitly to guarantee convergence of the Fredholm Pfaffian. Concretely, in \eqref{eq: two param beta small decomposition} , $\mathsf{A},\mathsf{A}'$ collect the $u-$poles and $\mathsf{C},\mathsf{C}'$ collect the $\beta-$poles. As a result, The rank-two representations
\begin{equation}
    \mathcal{O}(e^{-u(X+Y)})\mathbf{1} = \ketbra{X_1}{Y_1} + \ketbra{X_2}{Y_2}, \quad \mathcal{O}(e^{-\beta(X+Y)})\mathbf{1} = \ketbra{X_3}{Y_3} + \ketbra{X_4}{Y_4}
\end{equation} make these divergent contributions explicit and allow them to be factored out of the Pfaffian. 
 
This analysis also clarifies why such a decomposition, \eqref{eq: two param beta small decomposition} and \eqref{eq: two param beta small decomposition for kernel}, is unnecessary in the two-parameter stationary case when $\beta > t$ as $\mathcal{O}(e^{-\beta(X+Y)})\mathbf{1}$ decays exponentially and can remain in the kernel without affecting Pfaffian convergence.

\section{Analytic continuation of two parameter stationary formula when $\beta <t$}\label{section: analytic continuation}

\subsection{Analytic continuation of kernel and Fredholm Pfaffian}
The following two lemmas prove limits of the central kernel, i.e., $\boldsymbol{\mathcal{K}}$, and $\Pf\left(J - \boldsymbol{\mathcal{K}}\right)$.
\begin{lemma}\label{lemma:ptws_convergence}
Fix any $\varepsilon \ll 1.$
For any $t \in (\frac{1}{N},\frac{1}{2}),$ $\alpha> t,$ $\beta \in (-t,0],$ $X,Y \in \R,$ the kernel $\boldsymbol{\K}(X,Y)$ is analytic for $u \in (-t-\varepsilon, -t+\varepsilon)$ with the following limit:
    \begin{equation}
        \lim_{u\rightarrow -t} {\boldsymbol{\mathcal{K}}}(X,Y) = \widehat{\boldsymbol{\mathcal{K}}}(X,Y).
    \end{equation}
\end{lemma}
\begin{proof}
    The proof is completely analogous to the proof of Lemma \ref{lemma: ptws_convergence_one_param}.
\end{proof}

\begin{lemma}\label{lem: convergence of two param pfaffian finite}
Fix any $\tau \in \R$ and any $\varepsilon \ll 1.$ For any  $t \in (\frac{1}{N},\frac{1}{2}),$ $\alpha> t,$ $\beta \in (-t,t),$ we have that $\Pf(J-{\boldsymbol{\mathcal{K}}})$ is analytic for $u \in (-t-\varepsilon, -t + \varepsilon)$ with the following well-defined limit:
    \begin{equation}
        \lim_{u\rightarrow -t} \Pf(J-{\boldsymbol{\mathcal{K}}})_{\mathbb{L}^2(\tau,\infty)} = \Pf(J-{\widehat{\boldsymbol{\mathcal{K}}}})_{\mathbb{L}^2(\tau,\infty)}.
    \end{equation}
\end{lemma}

\begin{proof}
The proof is completely analogous to the proof of Lemma \ref{lemma:pfaffian_convergence}. Specifically, there exists $C>0$ independent of $X,Y,N$ such that for all $u \in (-t-\varepsilon, -t+\varepsilon)$,
\begin{equation}
    |\K(X,Y)|, |\partial_X\K(X,Y)|,|\partial_Y\K(X,Y)|,|\partial_X\partial_Y\K(X,Y)|\leq \frac{C}{2Nd}e^{-d(X+Y)}
    \end{equation}
    for all $X,Y \geq \tau.$
    Then we apply Hadamard's bound.
\end{proof}

\subsection{Analytic continuation of $\M$}
For the following Lemma, recall definitions of $H_{\alpha},$ $F$, and $Q$ in \eqref{eq: definitions of H,F,Q} and definitions of $\A_i,\B_i,\C_i,\D_i,\M_i,\NN_i$ for $i=1,2$ in \eqref{eq: def of A}, \eqref{eq: def of B}, \eqref{eq: def of C}, \eqref{eq: def of D}, \eqref{eq: def of M}, \eqref{eq: def of N}.

This lemma proves limits of all constant terms in Lemma \ref{lem: prelimit two param formula}.
\begin{lemma}\label{lem: finite limit for constants}
Fix any $\tau \in \R$ and any $\varepsilon \ll 1.$ For any fixed $t \in (\frac{1}{N},\frac{1}{2}),$ $\alpha> t,$ $\beta \in (-t,t),$ the following expressions are analytic for $u \in (-t-\varepsilon, -t + \varepsilon)$ with the well-defined limits:
    \begin{equation}
    \begin{aligned}
        &\lim_{u\rightarrow -t} \Gamma(u+t)\left(1 - \braket{\mathsf{B}}{\mathsf{A}'} + \braket{\mathsf{B}'}{\mathsf{A}}\right) = \M_1.\\
        &\lim_{u\rightarrow -t} \Gamma(u+t)\left(- \brabarket{\mathsf{B}}{\partial_X \K}{\mathsf{A}'} + \brabarket{\mathsf{B}'}{\K}{\mathsf{A}'} + \brabarket{\mathsf{B}}{\partial_X \partial_Y \K}{\mathsf{A}} - \brabarket{\mathsf{B}'}{\partial_Y \K}{\mathsf{A}}\right) = \M_2.\\
        &\lim_{u\rightarrow -t} 1 - \braket{\DDD}{\CCC'} + \braket{\DDD'}{\CCC} = \NN_1.\\
        &\lim_{u\rightarrow -t} - \brabarket{\DDD}{\partial_X \K}{\CCC'} + \brabarket{\DDD'}{\K}{\CCC'} + \brabarket{\DDD}{\partial_X \partial_Y \K}{\CCC} - \brabarket{\DDD'}{\partial_Y \K}{\CCC} = \NN_2.\\
        &\lim_{u\rightarrow -t} \braket{\BBB}{\CCC'} - \braket{\BBB'}{\CCC} = \A_1.\\
        &\lim_{u\rightarrow -t}  \brabarket{\BBB}{\partial_X \K}{\CCC'} - \brabarket{\BBB'}{\K}{\CCC'} - \brabarket{\BBB}{\partial_X \partial_Y \K}{\CCC} + \brabarket{\BBB'}{\partial_Y \K}{\CCC} = \A_2.\\
        &\lim_{u\rightarrow -t}  \braket{\BBB}{\DDD'} - \braket{\BBB'}{\DDD} = \B_1.\\
        &\lim_{u\rightarrow -t}  \brabarket{\BBB}{\partial_X \K}{\DDD'} - \brabarket{\BBB'}{\K}{\DDD'} - \brabarket{\BBB}{\partial_X \partial_Y \K}{\DDD} + \brabarket{\BBB'}{\partial_Y \K}{\DDD} = \B_2.\\
        &\lim_{u\rightarrow -t} \Gamma(u+t) \left( - \braket{\AAA}{\CCC'} + \braket{\AAA'}{\CCC}\right) = \C_1.\\
        &\lim_{u\rightarrow -t} \Gamma(u+t)\left(- \brabarket{\AAA}{\partial_X \K}{\CCC'} + \brabarket{\AAA'}{\K}{\CCC'} + \brabarket{\AAA}{\partial_X \partial_Y \K}{\CCC} - \brabarket{\AAA'}{\partial_Y \K}{\CCC}\right) = \C_2.\\
        &\lim_{u\rightarrow -t} \Gamma(u+t) \left(- \braket{\AAA}{\DDD'} + \braket{\AAA'}{\DDD}\right) = \D_1.\\
        &\lim_{u\rightarrow -t} \Gamma(u+t) \left(- \brabarket{\AAA}{\partial_X \overline{L}}{\DDD'} + \brabarket{\AAA'}{\overline{L}}{\DDD'} + \brabarket{\AAA}{\partial_X \partial_Y \overline{L}}{\DDD} - \brabarket{\AAA'}{\partial_Y \overline{L}}{\DDD}\right) = \D_2.\\
    \end{aligned}
    \end{equation}
\end{lemma}

\begin{proof}
We need to multiply $\Gamma(u+t)$ to all terms that contain $\mathsf{A}$ and $\mathsf{A}'$.
We prove $\M$ as an example. When $u >0$, direct evaluation gives
    \begin{equation}
    \begin{split}
            &-\braket{Y_1}{X_1} = - \braket{\BBB}{\AAA'} + \braket{\BBB'}{\AAA} \\
            &=\!\!\! \int\limits_{\I \R - d} \!\!\! \frac{dW}{2\pi\I}   \frac{e^{-\tau(u-W)}(u+W)}{(u-W)}\frac{\Gamma(t-u)}{\Gamma(t+u)}\frac{\Gamma(\beta-u)}{\Gamma(\beta+u)} 
            H_\alpha(-u)H_\alpha(W)F(W) \Gamma(-2W)\frac{\Gamma(1-u+W)}{\Gamma(1-u-W)}.
    \end{split}
    \end{equation}
    Similar to the product stationary case in Lemma \ref{lem: limit of the constant for one param}, a new pole $W = u$ is introduced to both integrands due to the inner product. The expression is not analytic when we take $u \rightarrow -t$ as the pole will cross the contour. Thus, we move the contour to enclose $u$ and then subtract the residue of the extra pole at $W=u$ so that the expression is analytic in $u \in (-t-\varepsilon,-t+\varepsilon)$. By direct computation, we find that $\text{Res}(W=u) = 1$. Let $W = -d+\I w$. Using the following upper bound on the integrand
    \begin{equation}\label{eq: some upper bound example}
    Ce^{-d \tau} (1+|w|)^{-2d(N+1)}(1+|2w|)^{2d-\frac12}e^{-2d - |w|\pi}
    \end{equation}
    for all $u \in (-t-\varepsilon,-t+\varepsilon)$, we conclude that
    \begin{equation}
    \begin{split}
            &\lim_{u\rightarrow -t} \Gamma(u+t)\left(1 - \braket{\BBB}{\AAA^{\prime}} + \braket{\BBB^{\prime}}{\AAA}\right)=\M_1.\\
    \end{split}
    \end{equation}
    We remark that $\braket{\BBB}{\DDD'} - \braket{\BBB'}{\DDD}$ also involves the $W=u$ pole that is intentionally added and needs to be removed.
    It is clear that
    \begin{equation}\label{eq: M2 expansion}
        -\brabarket{Y_1}{\G}{X_1} = - \brabarket{\mathsf{B}}{\partial_X \K}{\mathsf{A}'} + \brabarket{\mathsf{B}'}{\K}{\mathsf{A}'} + \brabarket{\mathsf{B}}{\partial_X \partial_Y \K}{\mathsf{A}} - \brabarket{\mathsf{B}'}{\partial_Y \K}{\mathsf{A}}.
    \end{equation}
    We show the limit of the first term in \eqref{eq: M2 expansion}. For simpler notation, we set
    \begin{equation}
        \begin{aligned}
            B(X) = \int\limits_{\I\R -d} \frac{dV}{2\pi\I} e^{XV}\text{Int}_B(V), \quad -\partial_X \K(X,Y) = \!\!\!\int\limits_{\I\R-d}\!\!\! \frac{dZ}{2\pi\I}\!\!\!\int\limits_{\I\R-d} \!\!\!\frac{dW}{2\pi\I} (-Z)e^{XZ}e^{YW}\frac{\Gamma(u+W)}{\Gamma(u-W)}\text{Int}_{\K}(Z,W),
        \end{aligned}
    \end{equation}
    where $\text{Int}_B(V)$ denote the remaining integrand of $B$ and $\text{Int}_{\K}(Z,W)$ denote the remaining integrand of $\K$.
    Then we first assume $u >0$ and compute the integrals:
    \begin{equation}\label{eq: one term in M2}
        \begin{aligned}
            &-\brabarket{\mathsf{B}}{\partial_X \K}{\mathsf{A}'}\\
            &= -uF(-u)H_{\alpha}(-u)\!\!\!\int\limits_{\I\R -d}\!\!\! \frac{dV}{2\pi\I} \!\!\!\int\limits_{\I\R-d}\!\!\! \frac{dZ}{2\pi\I}\!\!\!\int\limits_{\I\R-d} \!\!\!\frac{dW}{2\pi\I}\text{Int}_B(V)(-Z)\int_{0}^{\infty}e^{X(V+Z)}dX\int_{0}^{\infty}e^{Y(W-u)} dY\frac{\Gamma(u+W)}{\Gamma(u-W)}\text{Int}_{\K}(Z,W)\\
            &= -uF(-u)H_{\alpha}(-u)\!\!\!\int\limits_{\I\R -d}\!\!\! \frac{dV}{2\pi\I} \!\!\!\int\limits_{\I\R-d}\!\!\! \frac{dZ}{2\pi\I}\!\!\!\int\limits_{\I\R-d} \!\!\!\frac{dW}{2\pi\I}\text{Int}_B(V)(-Z)\frac{e^{\tau(V+Z)}}{-(V+Z)}\frac{e^{\tau(W-u)}}{(u-W)}\frac{\Gamma(u+W)}{\Gamma(u-W)}\text{Int}_{\K}(Z,W),
        \end{aligned}
    \end{equation}
    where integrals of $X$ and $Y$ are well-defined because $\text{Re}(V+Z) = -2d$ and $\text{Re}(W-u) = -d-u$. Let $V = -d+\I v,$ $Z = -d+\I z$, and $W = -d+\I w$.
    After careful estimation of the integrand as in Lemma \ref{lemma: ptws_convergence_one_param} \eqref{eq: UBofKernel}, we get the following upper bound:
    \begin{equation}\label{eq: M2 one term upper bound}
    \begin{aligned}
        Ce^{-(3d+u)\tau-6d}\left((1+|z|)(1+|w|)(1+|v|)\right)^{-2d(N+1)}
        \left((1+|2z|)(1+|2w|)(1+|2v|)\right)^{2d-\frac{1}{2}}e^{-\pi(|v| + |z+w|)}.
    \end{aligned}
    \end{equation}
    The term $1/(u-W)$ in the last step of \eqref{eq: one term in M2} does not introduce a $W=u$ pole because it is canceled by the Gamma function $(u-W)\Gamma(u-W) = \Gamma(1+u-W)$. Using the upper bound \eqref{eq: M2 one term upper bound} and dominated convergence theorem, we get the $u\rightarrow -t$ limit of one term. The limit $\lim_{u\rightarrow -t} -\Gamma(u+t)\brabarket{Y_1}{\G}{X_1} = \M_2$ can be proved by applying above arguments to the other three terms.
\end{proof}

The following lemma provides limits of all Pfaffian terms in Lemma \ref{lem: prelimit two param formula}.
\begin{lemma}\label{lem: u to -t limits of Pfaffians diff}
Fix any $\tau \in \R$ and any $\varepsilon \ll 1.$ For any $\alpha> t,$ $t \in (\frac{1}{N},\frac{1}{2}),$ and $\beta \in (-t,t)$, the following expression is analytic for $u \in (-t-\varepsilon, -t + \varepsilon)$ with a well-defined limit:
    \begin{equation}\label{eq:1}
    \begin{split}
        &\lim_{u\rightarrow -t} \mathrm{Pf}\left(J - \boldsymbol{\mathcal{K}} - A\right)
        =\mathrm{Pf}\left(J - \widehat{\boldsymbol{\mathcal{K}}} - \widehat{A}\right).
    \end{split}
    \end{equation}
    where $A$ is one of the following matrix kernels:
    \begin{equation}
        \begin{aligned}
        &\ketbra{\begin{array}{c} \Gamma(u+t)\zeta_2 \\
        -\Gamma(u+t)\zeta_1
        \end{array}}{\aleph_1 \quad \aleph_2} + \ketbra{\begin{array}{c}
            \aleph_1\\
            \aleph_2
        \end{array}}{-\Gamma(u+t)\zeta_2 \quad \Gamma(u+t)\zeta_1 },\\
        &\ketbra{\begin{array}{c} \eta_2 \\
        -\eta_1
        \end{array}}{\Theta_1 \quad \Theta_2} + \ketbra{\begin{array}{c}
            \Theta_1\\
            \Theta_2
        \end{array}}{-\eta_2 \quad \eta_1 },\\
        &\ketbra{\begin{array}{c} \eta_2 \\
        -\eta_1
        \end{array}}{\aleph_1 \quad \aleph_2} + \ketbra{\begin{array}{c}
            \aleph_1\\
            \aleph_2
        \end{array}}{-\eta_2 \quad \eta_1 },\\
        &\ketbra{\begin{array}{c} \theta_2 \\
        -\theta_1
        \end{array}}{\aleph_1 \quad \aleph_2} + \ketbra{\begin{array}{c}
            \aleph_1\\
            \aleph_2
        \end{array}}{-\theta_2 \quad \theta_1 },\\
        &\ketbra{\begin{array}{c} \eta_2 \\
        -\eta_1
        \end{array}}{\Gamma(u+t)\psi_1 \quad \Gamma(u+t)\psi_2} + \ketbra{\begin{array}{c}
            \Gamma(u+t)\psi_1\\
            \Gamma(u+t)\psi_2
        \end{array}}{-\eta_2 \quad \eta_1 },\\
        &\ketbra{\begin{array}{c} \theta_2 \\
        -\theta_1
        \end{array}}{\Gamma(u+t)\psi_1 \quad \Gamma(u+t)\psi_2} + \ketbra{\begin{array}{c}
            \Gamma(u+t)\psi_1\\
            \Gamma(u+t)\psi_2
        \end{array}}{-\theta_2 \quad \theta_1 },\\
        \end{aligned}
    \end{equation}
    and $\widehat{A}$ represents its corresponding limit under $u\rightarrow -t$. In particular, $\widehat{\psi_i}, \widehat{\zeta_i}, \widehat{\aleph_i},\widehat{\theta_i},\widehat{\eta_i},\widehat{\Theta_i}$ represents the limit of $\Gamma(u+t)\psi_i,\Gamma(u+t)\zeta_i,{\aleph_i},{\theta_i},{\eta_i},{\Theta_i}$  for $i \in \{1,2\}$.
\end{lemma}

\begin{proof}
    We prove the limit of $\zeta_2$ and limits of all other functions follow similarly. By definition,
    \begin{equation}
        \begin{aligned}
        \G\ket{X_1} = \ket{\begin{array}{c}
                \zeta_1\\
                \zeta_2
            \end{array}} = \begin{pmatrix}
            \partial_X\K & -\partial_X\partial_Y \K\\
            \K & -\partial_Y\K
        \end{pmatrix}\ket{\begin{array}{c}
                \AAA'\\
                \AAA
            \end{array}} = \ket{\begin{array}{c}
                \partial_X \K \AAA' - \partial_X\partial_Y \K \AAA\\
                \K\AAA' - \partial_Y \K \AAA
            \end{array}}.
        \end{aligned}
    \end{equation}
    All four terms $\partial_X \K \AAA',$ $- \partial_X\partial_Y \K \AAA$, $\K\AAA'$, and $- \partial_Y \K \AAA$ can be evaluated in a similar way. We compute one term to give an example. We assume $u>0$ temporarily to compute the inner product $\K \AAA'$ and then take $u\rightarrow -t$.
    \begin{equation}
        \begin{aligned}
            &\Gamma(u+t)\K \AAA' =\Gamma(u+t)\K \ket{\AAA'}\\
            &= \Gamma(u+t)(-u)F(-u)H_{\alpha}(-u)\!\!\!\int\limits_{\I\R-d}\!\!\! \frac{dZ}{2\pi\I}\!\!\!\int\limits_{\I\R-d} \!\!\!\frac{dW}{2\pi\I} e^{XZ}\int_{\tau}^{\infty}e^{Y(W-u)}dY\frac{\Gamma(u+Z)}{\Gamma(u-Z)}\frac{\Gamma(u+W)}{\Gamma(u-W)}\text{Int}_{\K}(Z,W)\\
            &=  \Gamma(u+t)(-u)F(-u)H_{\alpha}(-u)\!\!\!\int\limits_{\I\R-d}\!\!\! \frac{dZ}{2\pi\I}\!\!\!\int\limits_{\I\R-d} \!\!\!\frac{dW}{2\pi\I} e^{XZ}e^{-\tau(u-W)}\frac{\Gamma(u+Z)}{\Gamma(u-Z)}\frac{\Gamma(u+W)}{\Gamma(1+u-W)}\text{Int}_{\K}(Z,W)\\
            &\rightarrow  t\Gamma(2t)\frac{\Gamma(\beta + t)}{\Gamma(\beta - t)}H_{\alpha}(t)\!\!\!\int\limits_{\I\R-d}\!\!\! \frac{dZ}{2\pi\I}\!\!\!\int\limits_{\I\R-d} \!\!\!\frac{dW}{2\pi\I} e^{XZ}e^{\tau(t+W)}\frac{\Gamma(-t+Z)}{\Gamma(-t-Z)}\frac{\Gamma(-t+W)}{\Gamma(1-t-W)}\text{Int}_{\K}(Z,W),
        \end{aligned}
    \end{equation}
    where $\text{Int}_{\K}(Z,W) = F(Z)F(W)H_{\alpha}(Z)H_{\alpha}(W)Q(Z,W)$ and we used $(u-W)\Gamma(u-W) = \Gamma(1+u-W)$ in the second equality. In the last step, the limit is justified by dominated convergence theorem and an upper bound similar to \eqref{eq: UBofKernel}, \eqref{eq: some upper bound example}, and \eqref{eq: M2 one term upper bound}. Combining all four terms, we see that $\lim_{u\rightarrow -t} \Gamma(u+t)\G \ket{X_1} = \ket{\begin{array}{c}
                \widehat{\zeta}_1\\
                \widehat{\zeta}_2
            \end{array}}.$ It is clear that there exists $C>0$ independent of $X$ such that for all $u \in (-t-\epsilon, -t+\epsilon)$ and $X\geq \tau,$
    \begin{equation}
    |\Gamma(u+t)\zeta_1(X)| \leq Ce^{-dX}, \quad |\Gamma(u+t)\zeta_2(X)|\leq Ce^{-dX}.
    \end{equation}
    Without showing the details, we also claim that there exists $C>0$ independent of $X$ such that for all $u \in (-t-\epsilon, -t+\epsilon)$, $X\geq \tau,$ and $i \in \{1,2\}$,
    \begin{equation}\label{eq: u to -t upper bounds}
    \begin{aligned}
    &|\Gamma(u+t)\psi_i(X)|\leq Ce^{-dX},\,\,\,
    |\aleph_i(X)| \leq Ce^{-dX}, \,\,\, |\theta_2(X)|\leq Ce^{-dX},\,\,\, |\eta_i(X)| \leq Ce^{-dX}, \,\,\, |\Theta_2(X)|\leq Ce^{-dX}.\\
    \end{aligned}
    \end{equation}
    Consequently, the convergence of Fredholm Pfaffian follows from \eqref{eq: u to -t upper bounds}, Lemma \ref{lemma:ptws_convergence}, Lemma \ref{lem: convergence of two param pfaffian finite}, dominated convergence theorem and Hadamard's bound.
\end{proof}

\begin{proof}[Proof of Theorem~\ref{thm: two param beta < t finite}]
Recall that $\beta \in (-t,t)$. 
By Lemmas \ref{lem: prelimit two param formula}, \ref{lemma:ptws_convergence}, \ref{lem: convergence of two param pfaffian finite}, \ref{lem: finite limit for constants}, and \ref{lem: u to -t limits of Pfaffians diff}, we showed the following limit exists provided that $\Pf\left(J - \widehat{\boldsymbol{{\mathcal{K}}}}\right) \neq 0,$
\begin{equation}
    \lim_{u\rightarrow -t} \Gamma(u+t)\Pf\left(J - \boldsymbol{\check{\mathcal{K}}}\right) = \frac{1}{\Pf\left(J - \widehat{\boldsymbol{{\mathcal{K}}}}\right)}\left(\widehat{\M}\widehat{\NN} - \widehat{\A}\widehat{\D} + \widehat{\B}\widehat{\C}\right).
\end{equation}
Assume temporarily $\beta > t$. By isolating the random variable at $(1,1)$ and $(2,1)$, we get
\begin{equation}\label{eq: expansion of K0}
\begin{aligned}
    &\Gamma(u+t)\mathbb{E}\left[e^{-e^{-\tau }Z^{u,-}_{\beta < t}(N,N)\omega_{1,1}} \right] = \E\bigg[2(\tau Z_{\beta <t}^{u,-}(N,N))^{\frac{u+t}{2}}K_{-u-t}\left(2\sqrt{\tau Z_{\beta <t}^{u,-}(N,N)}\right)\bigg]\\
    & = \frac{1}{\Gamma(\beta - t)}\int_{0}^{\infty}\E\bigg[2(\tau Z_{\beta <t}^{u,-}(N,N)w)^{\frac{u+t}{2}}K_{-u-t}\left(2\sqrt{\tau Z_{\beta <t}^{u,-}(N,N)}w\right)\bigg] w^{t-\beta-1} e^{-w^{-1}}dw.\\
\end{aligned}
\end{equation}
We want to analytically extend \eqref{eq: expansion of K0} to $\beta \in (-t,t)$ and take $u\rightarrow -t$.
We need to show the following convergence as $u\rightarrow -t$ when $\beta \in (-t,t)$:
\begin{equation}\label{eq: two param beta small prelimit}
    \begin{aligned}
        \frac{1}{\Gamma(\beta - t)}\int_{0}^{\infty} &\E\bigg[ \int_{0}^{\infty} e^{-\tau Z^{u,-}_{\beta<t}(N,N)wx}e^{-x^{-1}}x^{-(u+t)-1} dx \bigg]w^{t-\beta-1} e^{-w^{-1}}dw\\
        &\rightarrow \frac{1}{\Gamma(\beta - t)}\E \bigg[ \int_{0}^{\infty} 2K_0\left( 2\sqrt{e^{-\tau} Z^{t,\beta}(N,N) w} \right) w^{t-\beta-1} e^{-w^{-1}}dw\bigg].
    \end{aligned}
\end{equation}
The difficulty in this case is that $w^{t-\beta -1}e^{-w^{-1}}$ is no longer a probability density function and is itself no longer integrable when $\beta \in (-t,t)$. Let $\xi = e^{-\tau}.$
By using $e^{-x} \leq x^{-1}$, we see that for all $u \in (-t-\varepsilon, -t+\varepsilon),$
\begin{equation}
    \begin{aligned}
        \text{LHS of } \eqref{eq: two param beta small prelimit} \leq  \E\bigg[\frac{1}{\xi Z^{u,-}_{\beta<t}(N,N)} \bigg]\left(\int_{0}^{\infty} e^{-x^{-1}}x^{-(u+t)-2} dx \right)\int_{0}^{\infty}w^{t-\beta-2} e^{-w^{-1}}dw <\infty
    \end{aligned}
\end{equation}
since $t<1/2$ and $\beta\in (-t,t)$ implies that $-2<t-\beta -2 < 2t-2<-1$.
Let $Y_i\sim \Gamm(2\alpha)$ for $i = 1,\dots, 2N-3$. By Lemma \ref{lem:stochastic_dominance}, we get for all $u \in (-t-\varepsilon,-t+\varepsilon)$, $Z^{u,-}_{\beta <t} \geq_{\mathrm{st}} \prod_{i =1}^{2N-3} Y_i$ and
\begin{equation}
    \begin{aligned}
        \E \bigg[ \left(Z^{u,-}_{\beta<t}(N,N)\right)^{-1}\bigg]  \leq \E \bigg[ \frac{1}{ Y_1Y_2\cdots Y_{2N-2}}\bigg]\leq \E \bigg[ Y_1^{-1}\bigg]^{2N-3}.
    \end{aligned}
\end{equation}
Hence, \eqref{eq: two param beta small prelimit} is proved by dominated convergence theorem.

By similar arguments as in Lemma \ref{lem:upper tail for modified kernel 2param}, there exists $\tau_0 >0$ such that for all $\tau \geq \tau_0$, $\Pf(J-\widehat{\boldsymbol{\mathcal{K}}})$ is close to $1$. We know that $\Pf(J-\widehat{\boldsymbol{\mathcal{K}}})$ is analytic on $\R$ and strictly positive on $(\tau_0,\infty)$. Hence $\Pf(J-\widehat{\boldsymbol{\mathcal{K}}})$ can only have isolated zeroes. By the proof above, we have the equality in \eqref{Two_param_formula_negative} provided that $\Pf(J-\widehat{\boldsymbol{\mathcal{K}}})_{\mathbb{L}^2(\tau,\infty)} \neq 0$. Notice that ${Q}_{N,\beta <t}^{High}(\tau)$ is an increasing function ($K_0$ is a decreasing function) in $\tau$ for all $\tau \in \R$ and hence is bounded by some constant on $(-\infty, \tau_0]$. Thus, any singularity of the RHS of \eqref{Two_param_formula_negative} is removable and we can extend the equality of \eqref{Two_param_formula_negative} to $\R$. 
\end{proof}

\section{Asymptotic analysis of two parameter stationary formula when $\tbeta < \ttt$}
\subsection{Asymptotic analysis of kernel and Fredholm Pfaffian}
In this section, we consider the critical scaling as in \ref{eq: essential scaling}.
\begin{lemma}\label{lem: 2 param kernel beta small limit, upper bound and pfaffian}
    Fix any $r>0.$ Fix any $\alpha > 0$, $\ttt >0,$ $\tbeta \in (-\ttt,\ttt)$. The following limits hold uniformly for $u,v \in (-r,r)$:
    \begin{equation}
    \begin{aligned}
        &\lim_{N\rightarrow \infty} \widehat{\K}(X,Y)  = \widetilde{\K}(u,v), \quad \lim_{N\rightarrow \infty} -(\sigma N)^{1/3}\partial_X\widehat{\K}(X,Y)  = -\partial_u\widetilde{\K}(u,v)\\
        &\lim_{N\rightarrow \infty} -(\sigma N)^{1/3}\partial_Y\widehat{\K}(X,Y)  = -\partial_v\widetilde{\K}(u,v), \quad \lim_{N\rightarrow \infty} (\sigma N)^{2/3}\partial_X\partial_Y\widehat{\K}(X,Y)  = \partial_u\partial_v\widetilde{\K}(u,v)
    \end{aligned}
    \end{equation}
    Fix any $\delta \in(\max\{\tbeta,0\},\ttt)$. There exists a constant $C$ independent of $u,v$ such that the following upper bounds hold for all $u,v\geq -r$:
    \begin{equation}\label{eq: 2 param K upper bound}
        \begin{aligned}
            |\widehat{\K}(u,v)|, \,|\partial_u\widehat{\K}(u,v)|,\, |\partial_v\widehat{\K}(u,v)|, \,|\partial_u\partial_v\widehat{\K}(u,v)| \leq Ce^{-\delta(u+v)}.
        \end{aligned}
    \end{equation}
    As a result, we have the convergence of the Fredholm Pfaffian 
    \begin{equation}
        \lim_{N\rightarrow \infty} \Pf\left( J - \widehat{\boldsymbol{\mathcal{K}}} \right)_{\mathbb{L}^2(\tau, \infty)} = \Pf\left( J - \widetilde{\boldsymbol{\mathcal{K}}} \right)_{\mathbb{L}^2(s,\infty)}.
    \end{equation}
    under the scaling $\tau = -Nf + (\sigma N)^{1/3}s$.
\end{lemma}

\begin{proof}
    We apply the same steepest descent analysis as in Lemma~\ref{lem: limit and upper tail of kernel}. The only difference is that the contour has to satisfy $\max\{\tbeta,0\}<\delta < \ttt$ since the pole $W = -\beta$ and $Z = -\beta$ are not included in $\widehat{\K}.$ The convergence of the Fredholm Pfaffian then follows from the bound \eqref{eq: 2 param K upper bound} using arguments in Lemma~\ref{lem: converge of Pfaffian 2param beta > t}.
\end{proof}

\subsection{Asymptotic analysis of ${\M}$, ${\NN}$, ${\A}$, ${\B}$, ${\C}$, ${\D}$, $\widehat{\aleph}$, $\widehat{\psi}$, $\widehat{\Theta}$, $\widehat{\zeta}$, $\widehat{\eta}$, $\widehat{\theta}$}

\begin{lemma}\label{lem: limits of constant terms for 2 param beta small}
    Fix any $s\in \R$ and consider the scaling $\tau = -Nf + (\sigma N)^{1/3}s.$ For any fixed $\alpha >0,$ $\ttt>0,$ $\tbeta \in (-\ttt,\ttt)$, the following limit holds
    \begin{equation}
        \begin{aligned}
            &\lim_{N\rightarrow \infty} (\sigma N)^{-1/3}\M_1 = \widetilde{\M}_1, \quad \lim_{N\rightarrow \infty} (\sigma N)^{-1/3}\M_2 = \widetilde{\M}_2, \quad \lim_{N\rightarrow \infty} \NN_1 = \widetilde{\NN}_1, \quad \lim_{N\rightarrow \infty} \NN_2 = \widetilde{\NN}_2,\\
            &\lim_{N\rightarrow \infty} \A_1 = \widetilde{\A}_1, \quad \lim_{N\rightarrow \infty} \A_2 = \widetilde{\A}_2,\quad
            \lim_{N\rightarrow \infty} \B_1 = \widetilde{\B}_1, \quad \lim_{N\rightarrow \infty} \B_2 = \widetilde{\B}_2,\\
            &\lim_{N\rightarrow \infty} (\sigma N)^{-1/3}\C_1 = \widetilde{\C}_1, \quad \lim_{N\rightarrow \infty} (\sigma N)^{-1/3}\C_2 = \widetilde{\C}_2,\quad
            \lim_{N\rightarrow \infty} (\sigma N)^{-1/3}\D_1 = \widetilde{\D}_1, \quad \lim_{N\rightarrow \infty} (\sigma N)^{-1/3}\D_2 = \widetilde{\D}_2.\\
        \end{aligned}
    \end{equation}
\end{lemma}

\begin{proof}
    For $\M_1$, the term $\Gamma(2t)\Gamma(-2W)dW$ contributes a prefactor of $(\sigma N)^{1/3}$. Consequently, $\M_1$ must be scaled by $(\sigma N)^{-1/3}.$ similar scaling arguments apply to $\C_1,\C_2,\D_1,\D_2$. We then apply the steepest descent method as in Lemma~\ref{lem: limit and upper tail of kernel} to obtain the limits. The contours $C(\delta;\pi/3)$ satisfy $\max\{0,\tbeta\}<\delta < \ttt$.
\end{proof}

\begin{lemma}\label{lem: upper bounds for all functions 2 param beta small}
Fix any $\max\{\tbeta,0\}<\delta < \ttt$ and any $s \in \R.$ There exists $N_0 \in \Z_{>0}$ and $C$ depending on $N_0,\delta$ but independent of $u,s$ such that for all $u\geq s,$ and $N\geq N_0$, we have the following upper bounds:
\begin{equation}\label{eq: scaling for each function two param beta small}
    \begin{aligned}
        &|\widehat{\aleph}_1(s,u)|,|(\sigma N)^{1/3}\widehat{\aleph}_2(s,u)| \leq Ce^{-2\delta s}e^{-u\delta}, \quad |(\sigma N)^{-1/3}\widehat{\psi}_1(s,u)|,|\widehat{\psi}_2(s,u)| \leq Ce^{s(\ttt-\delta)}e^{-u\delta},\\
        &|\widehat{\Theta}_1(s,u)|,|(\sigma N)^{1/3}\widehat{\Theta}_2(s,u)| \leq Ce^{-2\delta s}e^{-u\delta}, \quad 
        |\widehat{\zeta}_1(s,v)|,|(\sigma N)^{-1/3}\widehat{\zeta}_2(s,v)| \leq Ce^{s(\ttt-\delta)}e^{-v\delta}\\
        &|(\sigma N)^{1/3}\widehat{\eta}_1(s,u)|,|\widehat{\eta}_2(s,u)| \leq Ce^{-s(\tbeta + \delta)}e^{-v\delta}, \quad 
        |(\sigma N)^{1/3}\widehat{\theta}_1(s,v)|,|\widehat{\theta}_2(s,v)| \leq Ce^{-2s\delta}e^{-v\delta}.\\
    \end{aligned}
\end{equation}
\end{lemma}

\begin{proof}
    All upper bounds can be proved in the same manner. We illustrate the argument using $\widehat{\psi}_1$ as an example. We define $\text{Int}(\widehat{\psi}_1) = \frac{(\tbeta-\ttt)(\tZ + \ttt)(\tbeta + \tW)(\tbeta + \tZ)(\tZ - \tW)}{8\ttt \tZ \tW(\tbeta+\ttt)(\tZ - \ttt)(\tbeta - \tW)(\tbeta - \tZ)(\tZ + \tW)}.$ Applying the parameter scaling and the change of variables in \eqref{eq: essential scaling} to $\widehat{\psi}_1$ yields
\begin{equation}
    \begin{aligned}
        &|(\sigma N)^{-1/3}\widehat{\psi}_1(s,u)|= e^{s(\ttt - \delta) - u\delta}\mathcal{O}(e^{- \epsilon N})\\
        &+ \!\!\!\!\!\!\int\limits_{C(\delta;\pi/3)} \!\!\!\!\!\!\frac{d\tW}{2\pi \I}\!\!\!\!\!\!\int\limits_{C(\delta;\pi/3)} \!\!\!\!\!\!\frac{d\tZ}{2\pi \I}e^{\frac{\tZ^3}{3} - s\tZ  - \frac{\ttt^3}{3} + s\ttt+ \frac{\tW^3}{3} - u\tW+\mathcal{O}((\sigma N)^{-1/3})} \bigg|\text{Int}(\widehat{\psi}_1) + \mathcal{O}((\sigma N)^{-1/3})\bigg|
        \leq Ce^{s(\ttt - \delta) - u\delta}.
        \end{aligned}
    \end{equation}
    Here, the term $\mathcal{O}(e^{- \epsilon N})$arises from integration over portions of the contour away from the critical point. The final bound follows since the factor $e^{\frac{\tZ^3}{3} + \frac{\tW^3}{3}}$ decays exponentially along the contours, yielding a finite constant $C>0.$
\end{proof}

\begin{lemma}\label{lem: u to -t limits of Pfaffians diff}
Fix any $\max\{\tbeta,0\}<\delta < \ttt$ and any $s \in \R$ such that $\tau = -Nf + (\sigma N)^{1/3}s$. We have the following limit:
    \begin{equation}\label{eq:1}
    \begin{split}
        &\lim_{N\rightarrow \infty} \mathrm{Pf}\left(J - \boldsymbol{\widehat{\mathcal{K}}} - \widehat{A}\right)
        =\mathrm{Pf}\left(J - \widetilde{\boldsymbol{\mathcal{K}}} - \widetilde{A}\right).
    \end{split}
    \end{equation}
    where $\widehat{A}$ is one of the following matrix kernels:
    \begin{equation}\label{eq: all possible prelimit kernels}
        \begin{aligned}
        &(\sigma N)^{-1/3}\ketbra{\begin{array}{c} \widehat{\zeta}_2 \\
        -\widehat{\zeta}_1
        \end{array}}{\widehat{\aleph}_1 \quad \widehat{\aleph}_2} + (\sigma N)^{-1/3}\ketbra{\begin{array}{c}
            \widehat{\aleph}_1\\
            \widehat{\aleph}_2
        \end{array}}{-\widehat{\zeta}_2 \quad \widehat{\zeta}_1 },\\
        &\ketbra{\begin{array}{c} \widehat{\eta}_2 \\
        -\widehat{\eta}_1
        \end{array}}{\widehat{\Theta}_1 \quad \widehat{\Theta}_2} + \ketbra{\begin{array}{c}
            \widehat{\Theta}_1\\
            \widehat{\Theta}_2
        \end{array}}{-\widehat{\eta}_2 \quad \widehat{\eta}_1 },\\
        &\ketbra{\begin{array}{c} \widehat{\eta}_2 \\
        -\widehat{\eta}_1
        \end{array}}{\widehat{\aleph}_1 \quad \widehat{\aleph}_2} + \ketbra{\begin{array}{c}
            \widehat{\aleph}_1\\
            \widehat{\aleph}_2
        \end{array}}{-\widehat{\eta}_2 \quad \widehat{\eta}_1 },\\
        &\ketbra{\begin{array}{c} \widehat{\theta}_2 \\
        -\widehat{\theta}_1
        \end{array}}{\widehat{\aleph}_1 \quad \widehat{\aleph}_2} + \ketbra{\begin{array}{c}
            \widehat{\aleph}_1\\
            \widehat{\aleph}_2
        \end{array}}{-\widehat{\theta}_2 \quad \widehat{\theta}_1 },\\
        &(\sigma N)^{-1/3}\ketbra{\begin{array}{c} \widehat{\eta}_2 \\
        -\widehat{\eta}_1
        \end{array}}{\widehat{\psi}_1 \quad \widehat{\psi}_2} + (\sigma N)^{-1/3}\ketbra{\begin{array}{c}
            \widehat{\psi}_1\\
            \widehat{\psi}_2
        \end{array}}{-\widehat{\eta}_2 \quad \widehat{\eta}_1 },\\
        &(\sigma N)^{-1/3}\ketbra{\begin{array}{c} \widehat{\theta}_2 \\
       -\widehat{\theta}_1
        \end{array}}{\widehat{\psi}_1 \quad \widehat{\psi}_2} + (\sigma N)^{-1/3}\ketbra{\begin{array}{c}
            \widehat{\psi}_1\\
            \widehat{\psi}_2
        \end{array}}{-\widehat{\theta}_2 \quad \widehat{\theta}_1 },\\
        \end{aligned}
    \end{equation}
    and $\widetilde{A}$ denotes the corresponding limit under the critical scaling. In particular, for $i\in\{1,2\},$ the functions $\widetilde{\psi_i}, \widetilde{\zeta_i}, \widetilde{\aleph_i},\widetilde{\theta_i},\widetilde{\eta_i},\widetilde{\Theta_i}$ denote the limits of $\widehat{\psi_i}, \widehat{\zeta_i}, \widehat{\aleph_i},\widehat{\theta_i},\widehat{\eta_i},\widehat{\Theta_i}$ respectively, after applying the appropriate $(\sigma N)^{1/3}$ scaling specified in \eqref{eq: scaling for each function two param beta small}.
\end{lemma}
\begin{proof}
    Since we fix $s\in \R$, we only use the upper bounds in terms of $u$ from Lemma~\ref{lem: upper bounds for all functions 2 param beta small}, treating the dependence on $s$ as a constant. By Lemma~\ref{lem: upper bounds for all functions 2 param beta small} and Lemma~\ref{lem: 2 param kernel beta small limit, upper bound and pfaffian}, we get that for all choices of $\widehat{A}$ and for all $u,v \geq s$, there exists a constant $C>0$ independent of $u,v$ such that
    \begin{equation}
        \bigg|\boldsymbol{\widehat{\mathcal{K}}} + \widehat{A}\bigg|(u,v) \leq \begin{pmatrix}
            Ce^{-\delta(u+v)} & Ce^{-\delta(u+v)}\\
            Ce^{-\delta(u+v)} & Ce^{-\delta(u+v)}
        \end{pmatrix}.
    \end{equation}
    The convergence of the Fredholm Pfaffian then follows from the dominated convergence theorem together with Hadamard’s bound. 
\end{proof}

\begin{proof}[Proof of Theorem~\ref{thm: High density beta small asymptotic}]
Theorem \ref{thm: High density beta small asymptotic} follows from Lemmas \ref{lem: 2 param kernel beta small limit, upper bound and pfaffian}, \ref{lem: limits of constant terms for 2 param beta small}, and \ref{lem: u to -t limits of Pfaffians diff} provided that $\Pf\left(J - \widetilde{\boldsymbol{\mathcal{K}}}\right) \neq 0$. Since the point-wise limit of increasing function is still increasing, $\mathcal{J}(s)=\frac{\widetilde{\M}\widetilde{\NN} - \widetilde{\A}\widetilde{\D} + \widetilde{\B}\widetilde{\C}}{\Pf\left(J - \widetilde{\boldsymbol{\mathcal{K}}}\right)}$ is increasing in $s$. Due to the same reason as in the proof of Theorem \ref{thm: two param beta < t finite}, we see that all zeros of $\Pf\left(J - \widetilde{\boldsymbol{\mathcal{K}}}\right)$ is a removable singularity and the equality in \eqref{eq: K_0 version 2 param} can be extended to $\R$.
\end{proof}

\begin{proof}[Proof of Theorem~\ref{thm: High density beta small probability}]
Fix any $\tbeta$, $\ttt$ with $-\ttt<\tbeta < \ttt.$
By the tightness proved in Theorem \ref{thm:lower_2para} and Theorem \ref{thm:two_para_upper_tail}, we know that there exists a limiting random variable $\mathcal{X}_{\tbeta}$ such that $\frac{\log Z^{t,\beta}(N,N) + Nf}{(\sigma N)^{1/3}}$ converges weakly along a subsequence. Define $F_{\beta, N}(r) = \Pb\left(\frac{\log Z^{t,\beta}(N,N) + Nf}{(\sigma N)^{1/3}} \leq r\right)$ and  $\widetilde{F}_{\tbeta}(r) = \Pb(\X_{\tbeta} \leq r).$ By Theorem \ref{thm:lower_2para}, we have that for any $\epsilon >0,$ there exists $C>0$, $N_0\in \Z_{\geq 0}$ such that for all $N\geq N_0,$ ${F}_{\beta,N}(r) \leq Ce^{(1-\epsilon)2\ttt|r|}$ for all $r<0$. We thus have the same lower tail upper bound for $\widetilde{F}_{\tbeta}$.
We use similar arguments as in the proof of Theorem \ref{thm: One param probability}. By Corollary \ref{cor: 1param and 2param beta large lower tails}, we have for any $1\gg \epsilon>0$, there exists $N_0 >0$, $C>0$ such that for all $N\geq N_0,$
\begin{equation}
    |(\sigma N)^{-1/3}Q_{N,\beta < t}^{High}(-2N\psi(\alpha) + x(\sigma N)^{1/3})| \leq C|x| e^{-(1-\epsilon)2\ttt |x|}\quad \text{ for all } x<0.
\end{equation} 
The difference in this case is the upper tail of ${Q}_{N,\beta < t}^{High}$. By Lemma \ref{lem: two param beta small upper tail Q}, there exists $C>0,$ $s_0>0$ and $N_0$ such that for all $N \geq N_0$ and all $x > s_0$, 
\[|(\sigma N)^{-1/3}Q_{N,\beta <t}^{High}(-2N\psi(\alpha)+ x(\sigma N)^{1/3})|\leq Ce^{x(\ttt-\tbeta)}.
\]
Fix any $\epsilon \ll 1$ such that $(1-\epsilon)2\ttt > \ttt - \tbeta.$ We need to choose $A$ such that $(1-\epsilon)2\ttt>A>\ttt- \tbeta>0$. Then by dominated convergence theorem, we obtain the limit
\begin{equation}\label{eq: fourierIdentity}
\begin{aligned}
    &\frac{(A - \I y)}{2\pi}\int_{\R}\widetilde{Q}_{\tbeta < \ttt}^{High}(x + r )e^{-Ax}e^{\I x y} dx =  \frac{1}{2\pi}\int_{\R} \int_{0}^{\infty} \widetilde{F}_{\tbeta}(x+r-u)(\tbeta - \ttt)e^{-(\tbeta - \ttt)u}e^{-Ax}e^{\I x y}du dx.
\end{aligned}
\end{equation}
We verify that such Fourier transform is invertible by showing that
\[
H(z) := e^{-Az}\int_{0}^{\infty} \widetilde{F}_{\tbeta}(z+r-u)\, e^{-u(\tbeta - \ttt)} \,du
\]
belongs to $\mathbb{L}^2(\R)$. Let $\lambda = (1-\epsilon)2\ttt$. By Theorem \ref{thm:lower_2para}, we have $|\widetilde{F}_{\tbeta}(z)| \leq Ce^{\lambda z}$ for all $z \in \R$. For $z < 0$, recall that we have chosen $\lambda > A>\ttt - \tbeta$,
\begin{equation}
    |H(z)| \leq e^{-Az}\int_{0}^\infty Ce^{\lambda(z+r-u) - u(\tbeta - \ttt)} du = C'e^{(\lambda -A)z}.
\end{equation}
For $z>0,$
\begin{equation}
\begin{aligned}
    \left|\int_{0}^{\infty} \widetilde{F}_{\tbeta}(z+r-u)\, e^{-u(\tbeta - \ttt)} \,du\right| &\leq \int_{0}^{z+r} e^{-u(\tbeta - \ttt)} du + \int_{z+r}^{\infty} Ce^{\lambda (z+r-u)}e^{-u(\tbeta - \ttt)} du\\
    &= \frac{1}{\ttt-\tbeta}\left( e^{-(z+r)(\tbeta - \ttt)}-1\right) + e^{\lambda(z+r)}\int_{z+r}^{\infty} Ce^{-(\lambda+\tbeta - \ttt)u } du\\
    &\leq  C' \left( e^{-(z+r)(\tbeta - \ttt)} +e^{\lambda(z+r)} e^{-(\lambda+\tbeta - \ttt)(z+r)}\right)\\
    &\leq C'' e^{(\ttt-\tbeta)z}
\end{aligned}
\end{equation}
Since $A>\ttt-\tbeta$, we have $H(z) \in \mathbb{L}^2(\R).$ Then we can apply the fourier inversion to \eqref{eq: fourierIdentity} and get
\begin{equation}\label{eq: before shift argument}
 \int_{0}^{\infty} \widetilde{F}_{\tbeta}(r-u)(\tbeta - \ttt)e^{-(\tbeta - \ttt)u}du = \frac{1}{2\pi}\int_{\R}\int_{\R}(A - \I y)\widetilde{Q}_{\tbeta < \ttt}^{High}(x + r)e^{-Ax}e^{\I x y} dx dy.
\end{equation}
Since we know the above equality holds for all $r$, $\widetilde{F}_{\tbeta}$ can be uniquely characterized. Hence, the full sequence converges to $\X_{\tbeta}$. Now, we write $ F_{\beta,N}(x)=F_{x, N}(\tbeta)$ to emphasize the dependence on $\tbeta$ for fixed $x$. Assume $\tbeta >\ttt$. We have $(\tbeta - \ttt)e^{-(\tbeta - \ttt)u}$ is the density function of Exp$(\tbeta - \ttt)$.
By the deconvolution method in \cite[Lemma 3.3]{PatrikStatExp}, we get the following identity for any fixed $r$, $N$, $\tbeta$, $\ttt$,
\begin{equation}\label{eq: two param equal distribution}
        {F}_{r,N}(\tbeta) = \left(1+ \frac{\partial_{r}}{\tbeta - \ttt}\right)\int_{0}^{\infty} {F}_{r-u,N}(\tbeta)(\tbeta - \ttt)e^{-(\tbeta - \ttt)u}du.
\end{equation}
We have that $F_{x, N}(\tbeta)$ is an analytic function of $\tbeta$ because density function of Gamma inverse random variable is an analytic function in its parameters and this integral is well-defined under $\tbeta \in (-\ttt,\ttt)$. Then we can analytically extend \eqref{eq: two param equal distribution} to $\tbeta \in (-\ttt +\varsigma,\ttt - \varsigma)$ for some $0<\varsigma \ll 1$. Finally, taking the limit $N\rightarrow \infty$ of \eqref{eq: two param equal distribution} gives
\begin{equation}\label{eq: after shift}
\begin{aligned}
    &\Pb\left(\X_{\tbeta} \leq r\right) =\widetilde{F}_{\tbeta}(r) = \left(1+ \frac{\partial_{r}}{\tbeta - \ttt}\right)\int_{0}^{\infty} \widetilde{F}_{\tbeta}(r-u)(\tbeta - \ttt)e^{-(\tbeta - \ttt)u}du\\
    &=  \left(1+ \frac{\partial_{r}}{\tbeta - \ttt}\right)\frac{1}{2\pi}\int_{\R}\int_{\R}(A - \I y)\widetilde{Q}_{\tbeta < \ttt}^{High}(x + r)e^{-Ax}e^{\I x y} dx dy.
\end{aligned}
\end{equation}
\end{proof}

\section{Lower tail estimates for half-space stationary models}
In this section, we prove lower tail estimates for the half-space stationary models using three main ingredients: stochastic dominance via coupling, the distribution correspondence between full-space and half-space log-gamma partition functions, and random walk estimates. We first introduce these tools.

\subsection{Stochastic dominance via coupling}
\begin{lemma}[Stochastic dominance with respect to shape parameters]
\label{lem:stochastic_dominance}
Let $D\subset\mathbb Z^2$ be a connected set. Let $\{\alpha_v\}_{v\in D}$ and $\{\tilde{\alpha}_v\}_{v\in D}$ be two sets of shape parameters such that 
\begin{equation}
\alpha_v \ge \tilde\alpha_v
\qquad\text{for all } v\in D.
\end{equation}
For each $v\in D$, let
\begin{equation}
W_v\sim \Gamm(\alpha_v),
\qquad
\widetilde W_v\sim \Gamm(\tilde\alpha_v),
\end{equation}
be independent collections of inverse-gamma random variables with shape parameters
$\{\alpha_v\}_{v\in D}$ and $\{\tilde\alpha_v\}_{v\in D}$, respectively. Let $\Pi_{x\to y}(D)$ denote the set of up-right paths in $D$ from $x$ to $y$, and define the
corresponding partition functions
\begin{equation}
Z(x,y) := \sum_{\pi\in\Pi_{x\to y}(D)} \prod_{v\in\pi} W_v,
\qquad
\widetilde Z(x,y) := \sum_{\pi\in\Pi_{x\to y}(D)} \prod_{v\in\pi} \widetilde W_v.
\end{equation}

Then, there exists a coupling of $Z$ and $\widetilde{Z}$ such that
\[
\log Z(x,y) \le \log \widetilde Z (x,y).
\]
We write it as $\log Z(x,y) \le_{\mathrm{st}}\log \widetilde Z (x,y).$
\end{lemma}

\begin{proof}
Recall that the inverse-gamma random variable with parameter $\alpha$ has the following cumulative distribution function
\begin{equation}
F(\alpha;x) = \int_0^x \frac{1}{\Gamma(\alpha)}x^{-\alpha - 1}e^{-x^{-1}}dx.
\end{equation}
Let $\rho_\alpha(x) = \frac{1}{\Gamma(\alpha)}x^{-\alpha - 1}e^{-x^{-1}}$. Let $H(\alpha;\cdot)$ be the inverse of $F(\alpha; \cdot)$ defined on the interval $(0,1)$. We can calculate the first logarithmic derivative of $H$ with respect to $\alpha$:
\begin{equation}
    \partial_\alpha \log H(\alpha;x) = -\frac{1}{y\rho_\alpha(y)}\mathrm{Cov}(-\log X, \mathbbm{1}[X \leq y]) < 0
\end{equation}
where $y = H(\alpha;x)$ and $X \sim \Gamm(\alpha)$.

Now let $\{U_v\}_{v \in D}$ be a sequence of independent uniform $[0,1]$ random variables. Setting 
\begin{equation}
W_v = H(\alpha_v; U_v), \quad \widetilde{W}_v = H(\tilde{\alpha}_v; U_v)
\end{equation}
yields a coupling under which 
\begin{equation}
W_v\sim \Gamm(\alpha_v),
\qquad
\widetilde W_v\sim \Gamm(\tilde\alpha_v).
\end{equation}
Since $\alpha_v \geq \tilde{\alpha}_v$ for all $v \in D$ and $\partial_\alpha \log H(\alpha,x) < 0$ for all $\alpha > 0$, we have
\begin{equation}
W_v \leq \widetilde{W}_v
\end{equation}
for all $v \in D$ and thus
\begin{equation}
\log Z(x,y) \leq \log \widetilde{Z}(x,y).
\end{equation}
\end{proof}

\subsection{Correspondence between the full-space and half-space log-gamma partition function}
The correspondence is encoded in the following theorem \cite[Theorem 1.4]{BarraquandShouda}:

\begin{thm}\label{thm:BW}
Fix integers $n\ge 1$ and $m\ge 0$. Let $\alpha^\circ\in\R$,
$\alpha=(\alpha_1,\dots,\alpha_n)\in\R^n$, and $\beta=(\beta_1,\dots,\beta_m)\in\R^m$
satisfy
\begin{equation}\label{eq:BW_admissibility}
\alpha_i+\alpha^\circ>0,\qquad
\alpha_i+\alpha_j>0,\qquad
\alpha_i+\beta_k>0.
\end{equation}
for all $1\leq i,j \leq n$ and $1 \leq k \leq m$. 

\smallskip
\noindent\textbf{(A) Full-space point-to-point polymer.}
Consider the rectangular domain
\begin{equation}
\{(i,j): 1\le i\le n+m+1,\ 1\le j\le n\}.
\end{equation}
Let the weights on the rectangular domain be independently distributed as follows:
\begin{equation}\label{eq:BW_full_params}
W_{i,j}\sim
\begin{cases}
\Gamm(\alpha^\circ + \alpha_j), & i=1,\\
\Gamm(\alpha_{i-1}+\alpha_{j}), & 2\le i\le n+1,\\
\Gamm(\beta_{i-n-1}+\alpha_j), & n+2\le i\le n+m+1.
\end{cases}
\end{equation}
Let $Z_{full}(n+m+1,n)$ denote the full-space log-gamma polymer partition function from $(1,1)$ to $(n+m+1,n)$ according to the model above.

\smallskip
\noindent\textbf{(B) Trapezoidal point-to-line polymer.}
Consider the trapezoidal domain
\begin{equation}
\bigl\{(i,j): 1\le j\le n,\ \ j\le i \le 2n+m-j+1\bigr\}.
\end{equation}
Let the weights on the trapezoid domain be independently distributed as follows:
\begin{equation}\label{eq:BW_trap_params}
W_{i,j}\sim
\begin{cases}
\Gamm(\alpha_i+\alpha^\circ), & 1\le i=j\le n,\\
\Gamm(\alpha_j+\alpha_i), & 1\le j<i\le n,\\
\Gamm(\alpha_j+\beta_{i-n}), & 1\le j\le n,\ \ n<i\le n+m,\\
\Gamm(\alpha_j+\alpha_{2n+m-i+1}), & 1\le j\le n,\ \ n+m<i\le 2n+m-j+1.
\end{cases}
\end{equation}

Let $Z_{trap}(n,m)$ denote the log-gamma polymer partition function from $(1,1)$ to $(n,m)$ according to the model above. Let $Z_{line}$ be the trapezoidal point-to-line partition defined as follows:
\begin{equation}
    Z_{line}(n+m+1,n) :=\sum_{k=0}^{n-1} Z_{trap}(n+m+k+1, n-k).
\end{equation}

Then the following identity in distribution holds:
\begin{equation}\label{eq:BW_identity}
Z_{full}(n+m+1,n)\ \stackrel{d}{=}\ Z_{line}(n+m+1,n).
\end{equation}
\end{thm}

Additionally, we need the following lower tail result from the full space model \cite[Proposition 3.8]{basu2024temporal}.

\begin{thm}
Consider the homogeneous full-space model as shown in Figure \ref{fig: Homogeneous Full space polymer model} where the weight at $(1,1)$ is $1$ and the rest of the weights are distributed as independent $\Gamm(2\alpha)$. Let $\tilde{Z}_{full}(N,N)$ denote the point-to-point partition function from $(1,1)$ to $(N,N)$ of this model. Then there exist constants $c$, $N_0 \in \mathbb{N}$ such that for all $N\geq N_0$ and $t\geq 1$, we have
    \begin{equation}\label{eq3}
        \Pb\left(\log \tilde{Z}_{full}(N+1,N+1) + 2N\psi(\alpha) \leq -tN^{1/3} \right) \leq e^{-c \min\{t^{3/2}, tN^{1/3}\}}. 
    \end{equation}
\end{thm}

\subsection{Random Walk Estimates}
\begin{lemma}\label{lem_RW_approximation}
Let $X_1, \cdots, X_{N} \sim \Gamm(\alpha + t)$ and $Y_1, \cdots, Y_N \sim \Gamm(\alpha - t)$ be two sequences of independent random variables. Let $W_1 = \log X_1$, $W_i =\log X_i - \log Y_i$ for $2 \leq i \leq N$. We define the sum $S_k = \sum_{i = 1}^k W_i$ for $1\leq k \leq N$. Take $t = {\ttt}/{\Nsigma}$ for some $\ttt > 0$. There exist a positive constant $C$ and a positive integer $N_0$ such that for all $N \geq N_0$ and $a > 0$, we have
    \begin{equation}
        \Pb\left(\max_{1 \leq  k \leq N} S_k > a(\sigma N)^{1/3}\right) \leq Ce^{-2\ttt a}.
    \end{equation}
\end{lemma}

\begin{proof}
We define the exponential martingale for $1 \leq k \leq N$ as the following
\begin{equation}
    M_k = \exp\left(\lambda S_k - \Lambda_k\right)
\end{equation}
where
\begin{equation}
\begin{aligned}
    \Lambda_k :=& k\log \E[\exp(\lambda \log X_1)] + (k-1)\log \E[\exp(-\lambda Y_1)]
\end{aligned}
\end{equation}
and for $\alpha \pm (t -\lambda) > 0$,
\begin{equation}
\begin{split}
    &\log\E[\exp(\lambda \log X_1)] = \log \Gamma(\alpha + t- \lambda) - \log \Gamma(\alpha + t)\\
    &\log\E[\exp(-\lambda \log Y_1)] = \log \Gamma(\alpha - t+  \lambda) - \log \Gamma(\alpha - t).
\end{split}
\end{equation}

We want to apply Doob's maximal inequality to the exponential martingale:
    \begin{equation}
        \PP \left( \max_{1 \leq k \leq N} M_k \geq K \right) \leq \E[M_N]K^{-1} = K^{-1}. 
    \end{equation}
To do so, we need to upper bound $\Lambda_k$ for all $1\leq k \leq N$. Let $\lambda = \frac{\tilde{\lambda}}{\Nsigma}$ for some $\tilde{\lambda}$ to be chosen later. By taylor expansion,
\begin{equation}
\log \E[\exp(\lambda W_1)] = \frac{(\ttt-\tilde{\lambda})^2- \ttt^2}{(\sigma N)^{2/3}}\psi^{(1)}(\alpha) + O(N^{-4/3}).
\end{equation}
Because $\psi^{(1)}(\alpha) > 0$, as long as $(\ttt-\tilde{\lambda})^2 - \ttt^2 \leq 0$, we know that there exists some positive constant $C$ such that for all positive integer $N$ this holds: $\max_{1 \leq k \leq N}\Lambda_k < C$. Thus, we choose $\tilde{\lambda} = 2\tilde{t}$ and get that
    \begin{equation}
    \begin{aligned}
        \PP \left( \max_{1 \leq k \leq N} S_k > a\Nsigma \right) &\leq \PP \left( \max_{1 \leq k \leq N} M_k > \exp\left(\lambda a\Nsigma - C\right) \right) \leq e^{C-2\tilde{t}a}.
    \end{aligned}
    \end{equation}
\end{proof}

\subsection{Proof of lower tail for half-space product stationary model}
\begin{figure}[h]
\centering
\begin{tikzpicture}[
  scale=0.85,
  grid/.style={gray!50,dotted},
  baseline=(current bounding box.center),
  dot/.style={circle,inner sep=1.4pt}]
\def\Xmax{7}
\def\Ymax{6}

\foreach \x in {1,...,\Xmax}
  \draw[grid] (\x,1) -- (\x,\Ymax);
\foreach \y in {1,...,\Ymax}
  \draw[grid] (1,\y) -- (\Xmax,\y);

\foreach \y in {1,...,\Ymax}
  \node[dot,fill=green!70!black] at (1,\y) {};

\foreach \x in {2,...,7}
  \foreach \y in {1,...,\Ymax}
    \node[dot,fill=red!80!black] at (\x,\y) {};

\node[below left,font=\small] at (0.8,1) {$\footnotesize{1}$};
\node[above right, font=\small] at (\Xmax,\Ymax) {${\widetilde{Z}_{full}(N+1,N)}$};

\node[black,font=\small] at (0.5,2) {$\footnotesize{2\alpha}$};
\node[black,font=\small] at (0.5,3) {$\footnotesize{2\alpha}$};
\node[black,font=\small] at (0.5,4) {$\footnotesize{2\alpha}$};
\node[black,font=\small] at (0.5,5) {$\footnotesize{2\alpha}$};
\node[black,font=\small] at (0.5,6) {$\footnotesize{2\alpha}$};

\node[black,font=\small] at (4.5,3.5) {$\footnotesize{2\alpha}$};

\node[black,font=\small] at (2,0.6) {$\footnotesize{2\alpha}$};
\node[black,font=\small] at (3,0.6) {$\footnotesize{2\alpha}$};
\node[black,font=\small] at (4,0.6) {$\footnotesize{2\alpha}$};
\node[black,font=\small] at (5,0.6) {$\footnotesize{2\alpha}$};
\node[black,font=\small] at (6,0.6) {$\footnotesize{2\alpha}$};
\node[black,font=\small] at (7,0.6) {$\footnotesize{2\alpha}$};

\end{tikzpicture}
\caption{Homogeneous full space polymer model}
\label{fig: Homogeneous Full space polymer model}
\end{figure}
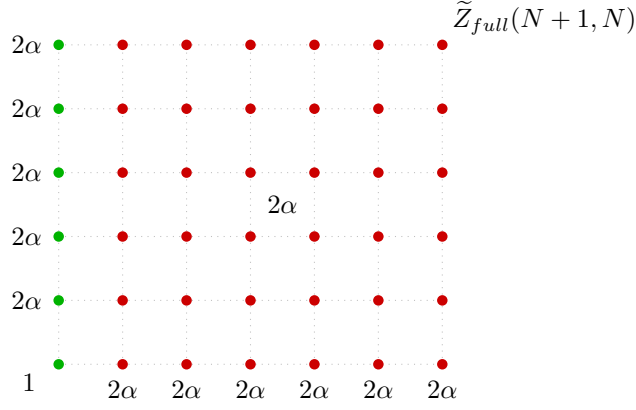

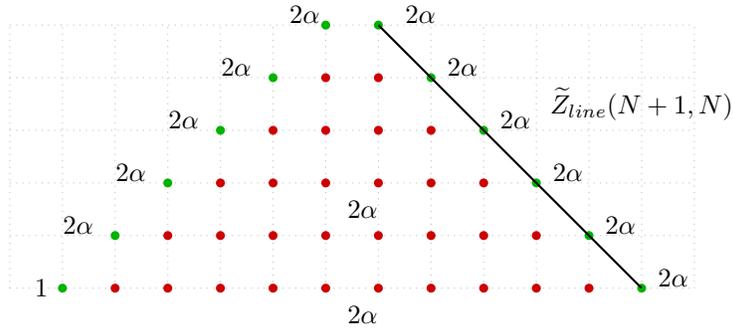
\begin{figure}[h]
\centering
\begin{tikzpicture}[
  scale=0.7,
  grid/.style={gray!50, dotted},
  baseline=(current bounding box.center),
  dot/.style={circle, inner sep=1.2pt},
  every node/.style={font=\small}
]
\def\N{5} 

\foreach \x in {-6,...,7} \draw[grid] (\x,0) -- (\x,\N);
\foreach \y in {0,...,\N} \draw[grid] (-6,\y) -- (7,\y);

\foreach \y in {0,...,\N}{
  \pgfmathsetmacro{\H}{\N - \y + 1}   
  \foreach \x in {-6,...,7}{
    \pgfmathparse{ (\x >= 1-\H) && (\x <= \H) ? 1 : 0 }
    \ifnum\pgfmathresult=1\relax
      \pgfmathparse{ (\x == 1-\H) || (\x == \H) ? 1 : 0 }
      \ifnum\pgfmathresult=1\relax
        \node[dot, fill=green!70!black] at (\x,\y) {};
      \else
        \node[dot, fill=red!80!black]   at (\x,\y) {};
      \fi
    \fi
  }
}

\node[black] at (-3.7,2.2) {$\footnotesize{2\alpha}$};
\node[black] at (-2.7,3.2) {$\footnotesize{2\alpha}$};
\node[black] at (-1.7,4.2) {$\footnotesize{2\alpha}$};
\node[black] at (-4.7,1.2) {$\footnotesize{2\alpha}$};

\node[black] at (-5.4,0) {$\footnotesize{1}$};

\node[black] at (-0.4,5.2) {$\footnotesize{2\alpha}$};
\node[black] at (1.8,5.2) {$\footnotesize{2\alpha}$};
\node[black] at (0.7,-0.5) {$\footnotesize{2\alpha}$};

\node[black] at (0.7,1.5) {$\footnotesize{2\alpha}$};

\node[black] at (4.6,2.2) {$\footnotesize{2\alpha}$};
\node[black] at (5.6,1.2) {$\footnotesize{2\alpha}$};
\node[black] at (3.6,3.2) {$\footnotesize{2\alpha}$};
\node[black] at (2.6,4.2) {$\footnotesize{2\alpha}$};

\node[black] at (6.6,0.2) {$\footnotesize{2\alpha}$};

\draw[black, thick]
(1,5) -- (6,0);
\node[black] at (6,3.5) {$\footnotesize{\widetilde{Z}_{line}(N+1,N)}$};

\end{tikzpicture}
\caption{Trapozoid homogeneous log-gamma polymer}
\label{fig: Homogeneous Trapozoid polymer model}
\end{figure}

\begin{figure}[h]
\centering
\begin{tikzpicture}[
  scale=0.7,
  grid/.style={gray!50, dotted},
  baseline=(current bounding box.center),
  dot/.style={circle, inner sep=1.2pt},
  every node/.style={font=\small}
]
\def\N{5} 

\foreach \x in {-6,...,7} \draw[grid] (\x,0) -- (\x,\N);
\foreach \y in {0,...,\N} \draw[grid] (-6,\y) -- (7,\y);

\foreach \y in {0,...,\N}{
  \pgfmathsetmacro{\H}{\N - \y + 1}   
  \foreach \x in {-6,...,7}{
    \pgfmathparse{ (\x >= 1-\H) && (\x <= \H) ? 1 : 0 }
    \ifnum\pgfmathresult=1\relax
      \pgfmathparse{ (\x == 1-\H) || (\x == \H) ? 1 : 0 }
      \ifnum\pgfmathresult=1\relax
        \node[dot, fill=green!70!black] at (\x,\y) {};
      \else
        \node[dot, fill=red!80!black]   at (\x,\y) {};
      \fi
    \fi
  }
}

\node[black] at (-3.7,2.2) {$\footnotesize{\alpha-t}$};
\node[black] at (-2.7,3.2) {$\footnotesize{\alpha-t}$};
\node[black] at (-1.7,4.2) {$\footnotesize{\alpha-t}$};
\node[black] at (-4.7,1.2) {$\footnotesize{\alpha-t}$};

\node[black] at (-5.4,0) {$\footnotesize{1}$};

\node[black] at (-3.2,5.2) {$\footnotesize{Z^t(N,N)}$};
\draw[->, thin] (-2.1,4.9) to[bend right=30]  (-0.1,4.9);
\node[black] at (3.5,5.2) {$\footnotesize{Z^t(N+1,N)}$};
\draw[->, thin] (3,5.6) to[bend right=30] (1.1,5.1);

\node[black] at (-0.7,5.2) {$\footnotesize{\alpha-t}$};
\node[black] at (0.5,5.2) {$\footnotesize{2\alpha}$};
\node[black] at (0.7,-0.5) {$\footnotesize{\alpha+t}$};

\node[black] at (0.7,1.5) {$\footnotesize{2\alpha}$};

\node[black] at (4.6,2.2) {$\footnotesize{2\alpha}$};
\node[black] at (5.6,1.2) {$\footnotesize{2\alpha}$};
\node[black] at (3.6,3.2) {$\footnotesize{2\alpha}$};
\node[black] at (2.6,4.2) {$\footnotesize{2\alpha}$};

\node[black] at (6.6,0.2) {$\footnotesize{2\alpha}$};

\draw[black, thick]
(1,5) -- (6,0);
\node[black] at (6,3.5) {$\footnotesize{Z^t_{line}(N+1,N)}$};

\end{tikzpicture}
\caption{Trapozoid product stationary log-gamma polymer}
\label{fig: Stationary Trapozoid polymer model}
\end{figure}
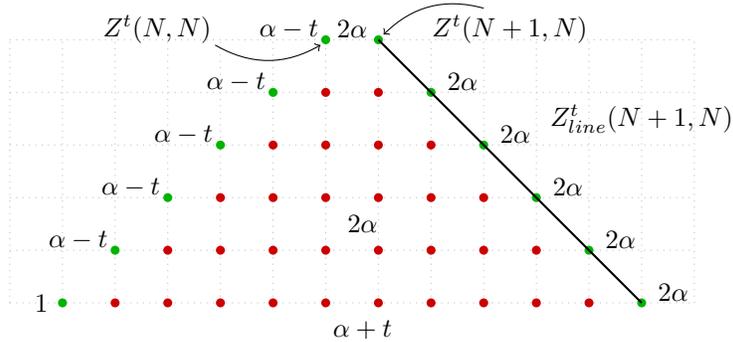

\begin{thm}\label{thm:lower_tail_prodstat}
Let $\log Z^t(N,N)$ be the free energy of the half-space model defined in Definition \ref{def: product stationary model}. Let us scale $t = \frac{\ttt}{\Nsigma}$ for some fixed $\ttt>0$. For any small $\varepsilon > 0$, there exist positive constants $C$ and $N_0$ such that for all $N \geq N_0,$ and $x > 0$, we have
    \begin{equation}
        \Pb\left(\frac{\log Z^t(N,N) + 2N\psi(\alpha)}{(\sigma N)^{1/3}} \leq -x\right) \leq Ce^{-(1-\varepsilon)2\ttt x}.
    \end{equation}
\end{thm}

\begin{proof}
Let $Z_{full}(N+1, N)$ and $Z_{line}(N+1, N)$ denote the full-space point-to-point partition function and trapezoidal point-to-line partition function defined in Theorem \ref{thm:BW} with $\alpha^{\circ} =\alpha_1 = \cdots = \alpha_N = \alpha$. Moreover, let $\widetilde{Z}_{full}(N+1, N)$ and $\widetilde{Z}_{line}(N+1, N)$ denote the same partition functions except the weights at $(1,1)$ in both models are replaced by $1$, as shown in Figure~\ref{fig: Homogeneous Full space polymer model} and Figure~\ref{fig: Homogeneous Trapozoid polymer model} respectively. We know from Theorem \ref{thm:BW} that
\begin{equation}
    Z_{full}(N+1,N) \overset{d}{=} Z_{line}(N+1,N).
\end{equation}
For both models, the weights at $(1,1)$ are independent from the rest and have the same $\Gamm(2\alpha)$ distribution. Since the characteristic function for log-inverse-gamma random variable with parameter $2\alpha$ is never zero, we can conclude
\begin{equation}
    \widetilde{Z}_{full}(N+1, N) \overset{d}{=} \widetilde{Z}_{line}(N+1,N).
\end{equation}

On the same trapezoidal domain
\begin{equation}
\{(i,j): 1 \leq j \leq N, j \leq i \leq 2N+1-j\}
\end{equation}
we define the following independently distributed weights as shown in Figure \ref{fig: Stationary Trapozoid polymer model}
\begin{equation}\label{eq:BW_trap_params}
W^{t}_{i,j}\sim
\begin{cases}
1, & i=j = 1,\\
\Gamm(\alpha - t), & 2\le i=j\le N,\\
\Gamm(\alpha + t), & j=1, 2 \le i\le 2N,\\
\Gamm(2\alpha), & j \geq 2, j<i\le 2N+1-j
\end{cases}
\end{equation}
Let $Z^t(N+k,N+1-k)$ denote the point-to-point partition function from $(1,1)$ to $(N+k, N+1-k)$ for the model defined above. Moreover, let
\begin{equation}
    Z^t_{line}(N+1, N) = \sum_{k=1}^{N} Z^t(N+k, N+1-k).
\end{equation}
By Lemma \ref{lem:stochastic_dominance}, since $\alpha - t, \alpha + t \leq 2\alpha$ for large enough $N$, for any $\varsigma \in \R$ and large enough $N$
\begin{equation}
    \PP\left(Z^t_{line}(N+1, N) \leq \varsigma \right) \leq \PP\left(\tilde{Z}_{line}(N+1,N)  \leq \varsigma \right) = \PP\left(\tilde{Z}_{full}(N+1,N)  \leq \varsigma \right).
\end{equation}

Let $T_k = \log Z^t(N+k, N+1-k)$ for $1 \leq k \leq N$. By the stationarity of the model \eqref{eq:BW_trap_params}, $T_k - \log Z^t(N,N)$ has the same distribution as $S_k$, defined in Lemma \ref{lem_RW_approximation}. Thus, apply Lemma \ref{lem_RW_approximation}, we have
\begin{equation}\label{eq_4}
    \PP \left( \max_{1 \leq k \leq N} T_k - \log Z^t(N,N) \geq a\Nsigma \right) \leq Ce^{-2\ttt a}. 
\end{equation}
Because 
\begin{equation}
    \log Z^t_{line}(N+1,N) \leq \log N + \max_{1 \leq k \leq N} T_k,
\end{equation}
we have for any $\varepsilon \in (0,1)$
\begin{equation}
\begin{aligned}
         \PP \bigg( \log Z^t(N,N) + 2N\psi(\alpha) &\leq -a\Nsigma \bigg)\\
         &\leq \PP \left(\log Z^t(N,N) - \log {Z}^{t}_{line}(N+1,N)  \leq -(1-\varepsilon)a\Nsigma \right)\\
         &\quad + \PP\left(\log {Z}^{t}_{line}(N+1,N)+2N\psi(\alpha) \leq -\varepsilon a \Nsigma\right)\\
         &\leq \PP\left(\max_{1 \leq k \leq N}T_k - \log Z^{t}(N,N)  \geq (1-\varepsilon)a \Nsigma -\log N \right)\\
         & \quad + \PP\left(\log \tilde{Z}_{full}(N+1,N) +2N\psi(\alpha) \leq -\varepsilon a\Nsigma\right)\\
\end{aligned}
\end{equation}
Because of the full-space lower tail (\ref{eq3}) and the random walk estimates (\ref{eq_4}), we know that there exist positive constants $C,N_0$ such that the probability above is bounded by $Ce^{-(1-\varepsilon)2\ttt a}$ for all $N\geq N_0.$
\end{proof}

\subsection{Proof of lower tail for half-space two-parameter stationary model}

\begin{thm}\label{thm:lower_2para}
Let $\log Z^{t,\beta}(N,N)$ be the free energy of the half-space model defined in Definition \ref{def: two-param stationary model} and shown in Figure \ref{Fig:Two_para}. Let us scale $t = \frac{\ttt}{\Nsigma}, \beta = \frac{\tbeta}{\Nsigma}$ for some fixed $\tbeta > -\ttt$. For any small $\varepsilon > 0$, there exist positive constants $C$ and $N_0$ such that for all $N \geq N_0,$ and $x > 0$, we have
    \begin{equation}
        \Pb\left(\frac{\log Z^{t,\beta}(N,N) + 2N\psi(\alpha)}{(\sigma N)^{1/3}} \leq -x\right) \leq Ce^{-(1-\varepsilon)2\ttt x}.
    \end{equation}
\end{thm}

\begin{figure}[htbp]
\centering
\begin{minipage}{0.48\textwidth}
  \centering
\begin{tikzpicture}[
  scale=0.7,
  grid/.style={gray!50, dotted},
  redpt/.style={circle, fill=red!80, draw=none, inner sep=1.4pt},
  grnpt/.style={circle, fill=green!70!black, draw=none, inner sep=1.4pt},
  baseline=(current bounding box.center),
  every node/.style={font=\small}
]

\def\N{5} 

\foreach \x in {0,...,\N} {
  \draw[grid] (\x,0) -- (\x,\N);
}
\foreach \y in {0,...,\N} {
  \draw[grid] (0,\y) -- (\N,\y);
}

\foreach \x in {1,...,\N} {
  \foreach \y in {0,...,\numexpr\x-1\relax} {
    \node[redpt] at (\x,\y) {};
  }
}

\foreach \k in {0,...,\N} {
  \node[grnpt] at (\k,\k) {};
}

\node[anchor=east] at (-0.2,0) {$1$};

\node[anchor=north] at (1,-0.1) {$1$};
\node[anchor=east]  at (0.8,1) {$\beta+t$};

\foreach \k in {2,...,\numexpr\N\relax} {
  \node[anchor=east] at (\k-0.2,\k) {$\alpha+t$};
}

\node[anchor=west] at (\N+0.1,\N+0.1) {$Z^{t,\beta}(N,N)$};

\node at (6,0) {$\alpha - t$};
\node at (6,1) {$\alpha + \beta$};
\node at (6,2) {$2\alpha$};
\node at (6,3) {$2\alpha$};
\node at (6,4) {$2\alpha$};

\end{tikzpicture}
\caption{Two parameter stationary $Z^{t,\beta}$}
\label{Fig:Two_para}
\end{minipage}
\hfill
\begin{minipage}{0.48\textwidth}
  \centering
  \begin{tikzpicture}[
  scale=0.7,
  grid/.style={gray!50, dotted},
  redpt/.style={circle, fill=red!80, draw=none, inner sep=1.4pt},
  grnpt/.style={circle, fill=green!70!black, draw=none, inner sep=1.4pt},
  baseline=(current bounding box.center),
  every node/.style={font=\small}
]

\def\N{5} 

\foreach \x in {0,...,\N} {
  \draw[grid] (\x,0) -- (\x,\N);
}
\foreach \y in {0,...,\N} {
  \draw[grid] (0,\y) -- (\N,\y);
}

\foreach \x in {1,...,\N} {
  \foreach \y in {0,...,\numexpr\x-1\relax} {
    \node[redpt] at (\x,\y) {};
  }
}

\foreach \k in {0,...,\N} {
  \node[grnpt] at (\k,\k) {};
}

\node[anchor=east] at (-0.2,0) {$1$};

\node[anchor=north] at (1,-0.1) {$1$};

\foreach \k in {1,...,\numexpr\N\relax} {
  \node[anchor=east] at (\k-0.2,\k) {$\alpha+t$};
}

\node[anchor=west] at (\N+0.1,\N+0.1) {$\overline{Z}^t(N,N)$};

\node at (6,0) {$\alpha - t$};
\node at (6,1) {$2\alpha$};
\node at (6,2) {$2\alpha$};
\node at (6,3) {$2\alpha$};
\node at (6,4) {$2\alpha$};

\end{tikzpicture}
\caption{Product stationary with the weight at $(2,1)$ replaced by $1$: $\overline{Z}^t$.}
\label{Fig:modifiedStat}
  \end{minipage}
\end{figure}

\begin{proof}
    Let us start with the modified half-space product stationary model as illustrated in Figure \ref{Fig:modifiedStat} where the $\Gamm(\alpha-t)$ weight at $(2,1)$ in the product-stationary model is replaced by constant $1$. Let $\overline{Z}^t$ denote the partition function of this model and $Z^t$ denote the partition function of the product-stationary model. Thus, $\log \overline{Z}^t(N,N) + \log G \overset{d}{=} \log Z^t(N,N)$ for $G \sim \Gamm(\alpha-t)$. For $x > 0$, we have for any $\delta \in (0,1)$
    \begin{equation}
    \Pb\left(\frac{\log \overline{Z}^{t}(N,N) + 2N\psi(\alpha)}{(\sigma N)^{1/3}} \leq -x\right) \leq \Pb\left(\frac{\log {Z}^{t}(N,N) + 2N\psi(\alpha)}{(\sigma N)^{1/3}} \leq -(1-\delta)x\right) + \Pb\left(\frac{\log G}{(\sigma N)^{1/3}} \geq \delta x\right).
    \end{equation}
   By Theorem \ref{thm:lower_tail_prodstat}, the first term is bounded by $Ce^{-(1-\varepsilon)(1-\delta)2\ttt x}$ for all large enough $N$. For the second term, we know it is bounded by 
   \begin{equation}
   \Pb\left(\log G' \geq \Nsigma \delta x \right)
   \end{equation}
   for $G' \sim \Gamm(\alpha/2)$ for $N$ large enough by Lemma \ref{lem:stochastic_dominance}. Since $\log \Gamm(\alpha/2)$ is a sub-exponential random variable \cite[Proposition D.2]{basu2024temporal}, the second term is thus bounded by 
   \[
   C'e^{-c'\Nsigma \delta x}
   \]
   for some $c',C' > 0$. Combining these two bounds, we get that for any $\varepsilon > 0$, there exist positive constants $C$ and $N_0$ such that for all $N \geq N_0$ and $x > 0$, 
   \begin{equation}
   \Pb\left(\frac{\log \overline{Z}^{t}(N,N) + 2N\psi(\alpha)}{(\sigma N)^{1/3}} \leq -x\right) \leq Ce^{-(1-\varepsilon)2\ttt x}.
   \end{equation}
   Again by Lemma \ref{lem:stochastic_dominance}, we see that $ \log \overline{Z}^t(N,N)\leq_{\mathrm{st}}\log Z^{t,\beta}(N,N) $ for large enough $N$ and hence 
   \begin{equation}
   \Pb\left(\frac{\log {Z}^{t,\beta}(N,N) + 2N\psi(\alpha)}{(\sigma N)^{1/3}} \leq -x\right) \leq \left(\frac{\log \overline{Z}^{t}(N,N) + 2N\psi(\alpha)}{(\sigma N)^{1/3}} \leq -x\right) \leq Ce^{-(1-\varepsilon)2\ttt x}.
   \end{equation}
\end{proof}

\section{Upper tail estimates for half-space stationary models}
In this section, we derive upper tail estimates for the half-space stationary models. Our analysis combines analytic continuation and steepest descent analysis of the Fredholm determinant, together with the correspondence between full-space and half-space log-gamma partition functions.

We will first prove the upper tail estimates for the free energy $\log Z^{t,\beta}(N,N)$ of the half-space two-parameter stationary model defined in Definition \ref{def: two-param stationary model}. The upper tail estimates for the free energy $\log Z^t(N,N)$ of the product stationary model defined in Definition \ref{def: product stationary model} will follow as an easy consequence of Lemma \ref{lem:stochastic_dominance}. 

We start by decomposing the partition function $Z^{t,\beta}(N,N)$ into two components as shown in Figure \ref{Fig:decomposition}. Let $\Pi_{x \rightarrow y}$ denote the set of up-right paths in the half-space $\{(i,j) \in \Z^2: i \geq j\}$ from $x$ to $y$. Then, 
\begin{equation}\label{eq:decomposition}
    Z^{t,\beta}(N,N) = \sum_{\pi \in \Pi_{(2,2) \rightarrow (N,N)}} \prod_{v \in \pi} w_{i,j} + \sum_{\pi \in \Pi_{(3,1) \rightarrow (N,N)}} \prod_{v \in \pi} w_{i,j}
\end{equation}
where $w_{i,j}$ is the vertex weight defined according to Definition \ref{def: two-param stationary model}.

\begin{figure}[htbp]
\centering
\begin{tikzpicture}[
  scale=0.7,
  grid/.style={gray!50, dotted},
  redpt/.style={circle, fill=red!80, draw=none, inner sep=1.4pt},
  grnpt/.style={circle, fill=green!70!black, draw=none, inner sep=1.4pt},
  baseline=(current bounding box.center),
  every node/.style={font=\small}
]

\def\N{5} 

\foreach \x in {0,...,\N} {
  \draw[grid] (\x,0) -- (\x,\N);
}
\foreach \y in {0,...,\N} {
  \draw[grid] (0,\y) -- (\N,\y);
}

\foreach \x in {1,...,\N} {
  \foreach \y in {0,...,\numexpr\x-1\relax} {
    \node[redpt] at (\x,\y) {};
  }
}

\foreach \k in {0,...,\N} {
  \node[grnpt] at (\k,\k) {};
}

\node[anchor=east] at (-0.2,0) {$1$};

\node[anchor=north] at (1,-0.1) {$1$};
\node[anchor=east]  at (0.8,1) {$\beta+t$};

\foreach \k in {2,...,\numexpr\N\relax} {
  \node[anchor=east] at (\k-0.2,\k) {$\alpha+t$};
}

\node[anchor=west] at (\N+0.1,\N+0.1) {$Z(N,N)$};

\node at (6,0) {$\alpha - t$};
\node at (6,1) {$\alpha + \beta$};
\node at (6,2) {$2\alpha$};
\node at (6,3) {$2\alpha$};
\node at (6,4) {$2\alpha$};

\draw (1,1) -- (5,1);
\draw (5,1) -- (5,5);
\draw (1,1) -- (5,5);

\draw[blue!60] (2,0) -- (5,0);
\draw[blue!60] (2,0) -- (2,2);
\draw[blue!60] (5,0) -- (5,5);
\draw[blue!60] (2,2) -- (5,5);
\end{tikzpicture}
\caption{This is the half-space log-gamma polymer model with two parameter stationary initial condition specified in Definition \ref{def: most general stationary model} with $\gamma = 1$. The black triangle contains all the paths in the first term of \eqref{eq:decomposition} and the blue trapezoid contains all the paths in the second term of \eqref{eq:decomposition}.}
\label{Fig:decomposition}
\end{figure}
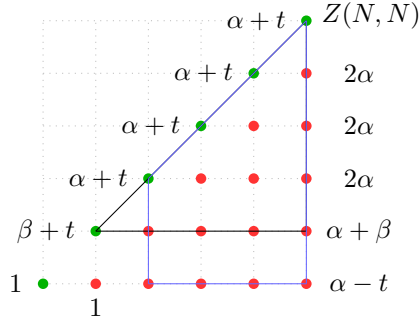

We want to analyze these two terms separately. To this end, we introduce two auxiliary models whose observables reproduce the distributions of the respective terms.

\smallskip
\noindent\textbf{(A)}
Consider the trapezoidal domain
\begin{equation}
D_N = \{(i,j): 1\leq j \leq N-2 \text{ and } j \leq i \leq 2N-1-j\}.
\end{equation}
Let the weights on the trapezoid domain be independently distributed as follows:
\begin{equation}\label{eq:BW_trap_params}
W_{i,j}\sim
\begin{cases}
\Gamm(\alpha + t), & 1\le i=j\le N-2,\\
\Gamm(\alpha + \beta), & i = N-1, \ \ 1\leq j \leq N-2,\\
\Gamm(\alpha - t), & i = N,\ \ 1\leq j \leq N-2,\\
\Gamm(2\alpha), & \text{otherwise}.
\end{cases}
\end{equation}
Let $Z^{t,\beta}_1(n,m)$ denote the log-gamma polymer partition function from $(1,1)$ to $(n,m)$ according to the model above with weights defined in Figure \ref{Fig:Trapezoid1}. By symmetry, we see that
\begin{equation}\label{eq:eq_d_1}
\sum_{\pi \in \Pi_{(3,1) \rightarrow (N,N)}} \prod_{v \in \pi} w_{i,j} \overset{d}{=} Z^{t,\beta}_1(N,N-2).
\end{equation}
Let $Z^{t,\beta}_{1,line}$ be the trapezoidal point-to-line partition defined as follows:
\begin{equation}
Z^{t,\beta}_{1,line} = \sum_{j=1}^{N-2} Z_1^{t,\beta}(2N-1-j,j).
\end{equation}
Then, by the definition of partition function,
\begin{equation}\label{eq:ineq_d_1}
    \log Z_1^{t,\beta}(N, N-2) + \log W_{N+1,N-2} \leq \log Z^{t,\beta}_1(N+1, N-2) \leq \log Z^{t,\beta}_{1, line}. 
\end{equation}
By Theorem \ref{thm:BW}, 
\begin{equation}
Z^{t,\beta}_{1,line} \overset{d}{=} Z^{t,\beta}_{1, full}(N+1, N-2)
\end{equation}
where $Z^{t,\beta}_{1,full}(N+1,N-2)$ is the full-space log-gamma polymer partition function from $(1,1)$ to $(N+1, N-2)$ with weights defined in Figure \ref{Fig:Full1}.

\begin{figure}[htbp]
\centering
\begin{minipage}{0.48\textwidth}
  \centering
  \begin{tikzpicture}[
  scale=0.7,
  grid/.style={black!80, dotted},
  redpt/.style={circle, fill=red!80, draw=none, inner sep=1.4pt},
  grnpt/.style={circle, fill=green!70!black, draw=none, inner sep=1.4pt},
  baseline=(current bounding box.center),
  every node/.style={font=\small}
]

\def\N{3} 
\pgfmathsetmacro{\M}{\N+2}
\def\K{3}
\pgfmathsetmacro{\L}{\M+3}

\foreach \x in {0,...,\K} {
  \draw[grid] (\x,0) -- (\x,\x);
}

\foreach \x in {1,...,\K} {
  \draw[grid] (\M + \x,0) -- (\M +\x,\N-\x+1);
}

\foreach \x in {\K,...,6} {
  \draw[grid] (\x,0) -- (\x,\N);
}

\foreach \y in {0,...,\N} {
  \draw[grid] (\y,\y) -- (\M+\N-\y+1,\y);
}

\foreach \x in {1,...,\N} {
  \foreach \y in {0,...,\numexpr\x-1\relax} {
    \node[redpt] at (\x,\y) {};
  }
}

\foreach \x in {\N+1,\N+2} {
  \foreach \y in {0,...,\N} {
    \node[redpt] at (\x,\y) {};
  }
}

\node[redpt] at (6,2) {};
\node[redpt] at (6,0) {};
\node[redpt] at (7,0) {};
\node[redpt] at (8,0) {};
\node[redpt] at (6,1) {};
\node[redpt] at (7,1) {};

\foreach \k in {0,...,\N} {
  \node[grnpt] at (\k,\k) {};
}

\foreach \k in {0,...,\N} {
  \node[grnpt] at (\M +1 + \k,\N -\k) {};
}

\node[anchor=east] at (-0.2,0) {$\alpha + t$};
\node[anchor=east] at (0.8,1) {$\alpha + t$};
\node[anchor=east] at (1.8,2) {$\alpha + t$};
\node[anchor=east]  at (2.8,3) {$\alpha+t$};

\node[anchor=south] at (1,-0.8) {$2\alpha$};
\node[anchor=south] at (2,-0.8) {$2\alpha$};
\node[anchor=south] at (3,-0.8) {$2\alpha$};
\node[anchor=north] at (4,4) {$\alpha+\beta$};
\node[anchor=north] at (5.2,4) {$\alpha-t$};
\node[anchor=north] at (4,3) {$\alpha+\beta$};
\node[anchor=north] at (5.2,3) {$\alpha-t$};
\node[anchor=north] at (4,2) {$\alpha+\beta$};
\node[anchor=north] at (5.2,2) {$\alpha-t$};
\node[anchor=north] at (4,1) {$\alpha+\beta$};
\node[anchor=north] at (5.2,1) {$\alpha-t$};

\node[anchor=south] at (6,-0.8) {$2\alpha$};
\node[anchor=south] at (7,-0.8) {$2\alpha$};
\node[black,font=\small] at (7.4,3.7) {$(N+1,N-2)$};

\node[anchor=west] at (6.1,3) {$2\alpha $};

\node[anchor=west] at (7.1,2) {$2\alpha $};
\node[anchor=west]  at (8.1,1) {$2\alpha$};
\node[anchor=west]  at (9.1,0) {$2\alpha$};

\end{tikzpicture}
  \caption{}
  \label{Fig:Trapezoid1}
\end{minipage}
\hfill
\begin{minipage}{0.48\textwidth}
  \centering
\begin{tikzpicture}[
  scale=0.85,
  grid/.style={black!80,dotted},
  redpt/.style={circle, fill=red!80, draw=none, inner sep=1.4pt},
  grnpt/.style={circle,fill=green!70!black, draw=none, inner sep=1.4pt},
  baseline=(current bounding box.center),
  dot/.style={circle,inner sep=1.4pt}]
\def\Xmax{7}
\def\Ymax{7}

\foreach \x in {1,...,\Xmax}
  \draw[grid] (\x,1) -- (\x,\Ymax);
\foreach \y in {1,...,\Ymax}
  \draw[grid] (1,\y) -- (\Xmax,\y);

\foreach \y in {1,...,\Ymax}
  \node[grnpt] at (1,\y) {};

\foreach \x in {2,...,7}
  \foreach \y in {1,...,\Ymax}
    \node[redpt] at (\x,\y) {};


\node[black,font=\small] at (0.2,1) {$\footnotesize{\alpha+t}$};
\node[black,font=\small] at (0.2,2) {$\footnotesize{\alpha+t}$};
\node[black,font=\small] at (0.2,3) {$\footnotesize{\alpha + t}$};
\node[black,font=\small] at (0.2,4) {$\footnotesize{\alpha+t}$};
\node[black,font=\small] at (0.2,5) {$\footnotesize{\alpha+t}$};
\node[black,font=\small] at (0.2,6) {$\footnotesize{\alpha+t}$};
\node[black,font=\small] at (0.2,7) {$\footnotesize{\alpha+t}$};

\node[black,font=\small] at (2,0.6) {$\footnotesize{2\alpha}$};
\node[black,font=\small] at (3,0.6) {$\footnotesize{2\alpha}$};
\node[black,font=\small] at (4,0.6) {$\footnotesize{2\alpha}$};
\node[black,font=\small] at (5,0.6) {$\footnotesize{2\alpha}$};
\node[black,font=\small] at (3.5,4.6) {$\footnotesize{2\alpha}$};
\node[black,font=\small] at (6,0.6) {$\footnotesize{\alpha + \beta}$};
\node[black,font=\small] at (6,1.6) {$\footnotesize{\alpha + \beta}$};
\node[black,font=\small] at (6,2.6) {$\footnotesize{\alpha + \beta}$};
\node[black,font=\small] at (6,3.6) {$\footnotesize{\alpha + \beta}$};
\node[black,font=\small] at (6,4.6) {$\footnotesize{\alpha + \beta}$};
\node[black,font=\small] at (6,5.6) {$\footnotesize{\alpha + \beta}$};
\node[black,font=\small] at (6,6.6) {$\footnotesize{\alpha + \beta}$};
\node[black,font=\small] at (7,0.6) {$\footnotesize{\alpha-t}$};
\node[black,font=\small] at (7,1.6) {$\footnotesize{\alpha-t}$};
\node[black,font=\small] at (7,2.6) {$\footnotesize{\alpha-t}$};
\node[black,font=\small] at (7,3.6) {$\footnotesize{\alpha-t}$};
\node[black,font=\small] at (8.4,4) {$(N+1,N-2)$};
\node[black,font=\small] at (7,4.6) {$\footnotesize{\alpha-t}$};
\node[black,font=\small] at (7,5.6) {$\footnotesize{\alpha-t}$};
\node[black,font=\small] at (7,6.6) {$\footnotesize{\alpha-t}$};
\node[black,font=\small] at (8.4,7) {$(N+1,N+1)$};

\end{tikzpicture}
\caption{}
\label{Fig:Full1}
  \end{minipage}
\end{figure}

\smallskip
\noindent\textbf{(B)}
Consider the trapezoidal domain
\begin{equation}
E_N = \{(i,j): 1\leq j \leq N-1 \text{ and } j \leq i \leq 2N-1-j\}.
\end{equation}
Let the weights on the trapezoid domain be independently distributed as follows:
\begin{equation}\label{eq:BW_trap_params}
\widetilde{W}_{i,j}\sim
\begin{cases}
\Gamm(\beta + t), & i=j = 1,\\
\Gamm(\alpha + \beta), & 2\le i=j\le N-1,\\
\Gamm(\alpha + t), & 2 \leq i \leq 2N-3, \ \ j =1,\\
\Gamm(2t), & i=2N-2,\ \ j=1,\\
\Gamm(2\alpha), & \text{otherwise}.
\end{cases}
\end{equation}
Let $Z^{t,\beta}_2(n,m)$ denote the log-gamma polymer partition function from $(1,1)$ to $(n,m)$ according to the model above as shown in Figure \ref{Fig:Trapezoid2}. By the symmetry property described in Lemma \ref{lem: symmetry}, we can exchange the parameter $\beta$ and $t$ and get that
\begin{equation}\label{eq:eq_d_2}
\sum_{\pi \in \Pi_{(2,2) \rightarrow (N,N)}} \prod_{v \in \pi} w_{i,j} \overset{d}{=} Z^{t,\beta}_2(N-1,N-1).
\end{equation}
Let $Z^{t,\beta}_{2,line}$ be the trapezoidal point-to-line partition defined as follows:
\begin{equation}
Z^{t,\beta}_{2,line} = \sum_{j=1}^{N-1} Z_2^{t,\beta}(2N-1-j,j).
\end{equation}
Then, 
\begin{equation}\label{eq:ineq_d_2}
    \log Z_2^{t,\beta}(N-1, N-1) + \log \widetilde{W}_{N,N-1} \leq \log Z^{t,\beta}_2(N, N-1) \leq Z^{t,\beta}_{2, line}. 
\end{equation}
Moreover, by Theorem \ref{thm:BW}, 
\begin{equation}
Z^{t,\beta}_{2,line} \overset{d}{=} Z^{t,\beta}_{2, full}(N, N-1)
\end{equation}
where $Z^{t,\beta}_{2,full}(N,N-1)$ is the full-space log-gamma polymer partition function from $(1,1)$ to $(N, N-1)$ with weights defined in Figure \ref{Fig:Full2}.

\begin{figure}[htbp]
\centering
\begin{minipage}{0.48\textwidth}
  \centering
\begin{tikzpicture}[
  scale=0.7,
  grid/.style={black!80, dotted},
  redpt/.style={circle, fill=red!80, draw=none, inner sep=1.4pt},
  grnpt/.style={circle, fill=green!70!black, draw=none, inner sep=1.4pt},
  baseline=(current bounding box.center),
  every node/.style={font=\small}
]
\def\N{4} 

\foreach \y in {0,...,\N}{
  \pgfmathsetmacro{\H}{\N - \y + 1}
  \foreach \x in {-6,...,7}{
    \pgfmathparse{ (\x >= 1-\H) && (\x <= \H) ? 1 : 0 }
    \ifnum\pgfmathresult=1\relax

      \pgfmathparse{ (\x+1 >= 1-\H) && (\x+1 <= \H) ? 1 : 0 }
      \ifnum\pgfmathresult=1\relax
        \draw[grid] (\x,\y) -- (\x+1,\y);
      \fi

      \pgfmathparse{ (\y+1 <= \N) ? 1 : 0 }
      \ifnum\pgfmathresult=1\relax
        \pgfmathsetmacro{\Hup}{\N-(\y+1)+1}
        \pgfmathparse{ (\x >= 1-\Hup) && (\x <= \Hup) ? 1 : 0 }
        \ifnum\pgfmathresult=1\relax
          \draw[grid] (\x,\y) -- (\x,\y+1);
        \fi
      \fi

    \fi
  }
}

\foreach \y in {0,...,\N}{
  \pgfmathsetmacro{\H}{\N - \y + 1}   
  \foreach \x in {-6,...,7}{
    \pgfmathparse{ (\x >= 1-\H) && (\x <= \H) ? 1 : 0 }
    \ifnum\pgfmathresult=1\relax
      \pgfmathparse{ (\x == 1-\H) || (\x == \H) ? 1 : 0 }
      \ifnum\pgfmathresult=1\relax
        \node[grnpt] at (\x,\y) {};
      \else
        \node[redpt]   at (\x,\y) {};
      \fi
    \fi
  }
}

\node[black] at (-2.7,2.2) {$\footnotesize{\alpha+\beta}$};
\node[black] at (-1.7,3.2) {$\footnotesize{\alpha+\beta}$};
\node[black] at (-0.7,4.2) {$\footnotesize{\alpha+\beta}$};
\node[black] at (-3.7,1.2) {$\footnotesize{\alpha+\beta}$};
\node[black] at (-4.7,0.2) {$\footnotesize{\beta+t}$};

\node[black] at (1.5,4.5) {$(N,N-1)$};

\node[black] at (0.5,1.5) {$\footnotesize{2\alpha}$};

\node[black] at (3.6,2.2) {$\footnotesize{2\alpha}$};
\node[black] at (4.6,1.2) {$\footnotesize{2\alpha}$};
\node[black] at (2.6,3.2) {$\footnotesize{2\alpha}$};
\node[black] at (1.6,4) {$\footnotesize{2\alpha}$};
\node[black] at (5.6,0.2) {$\footnotesize{2t}$};
\draw[decorate, decoration={brace, amplitude=6pt,mirror}]
(-3,-0.2) -- (4,-0.2)
node[midway, below=6pt] {$\footnotesize \alpha+t$};
\end{tikzpicture}
\caption{}\label{Fig:Trapezoid2}
\end{minipage}
\hfill
\begin{minipage}{0.48\textwidth}
  \centering
  \begin{tikzpicture}[
  scale=0.85,
  grid/.style={black!80,dotted},
  redpt/.style={circle, fill=red!80, draw=none, inner sep=1.4pt},
  grnpt/.style={circle,fill=green!70!black, draw=none, inner sep=1.4pt},
  baseline=(current bounding box.center),
  dot/.style={circle,inner sep=1.4pt}]
\def\Xmax{6}
\def\Ymax{6}

\foreach \x in {1,...,\Xmax}
  \draw[grid] (\x,1) -- (\x,\Ymax);
\foreach \y in {1,...,\Ymax}
  \draw[grid] (1,\y) -- (\Xmax,\y);

\foreach \y in {1,...,\Ymax}
  \node[grnpt] at (1,\y) {};

\foreach \x in {2,...,6}
  \foreach \y in {1,...,\Ymax}
    \node[redpt] at (\x,\y) {};


\node[black,font=\small] at (0.2,1) {$\footnotesize{\beta+t}$};
\node[black,font=\small] at (0.2,2) {$\footnotesize{\beta+\alpha}$};
\node[black,font=\small] at (0.2,3) {$\footnotesize{\beta+\alpha}$};
\node[black,font=\small] at (0.2,4) {$\footnotesize{\beta+\alpha}$};
\node[black,font=\small] at (0.2,5) {$\footnotesize{\beta+\alpha}$};
\node[black,font=\small] at (0.2,6) {$\footnotesize{\beta+\alpha}$};

\node[black,font=\small] at (2,0.6) {$\footnotesize{2t}$};
\node[black,font=\small] at (2,1.6) {$\footnotesize{\alpha + t}$};
\node[black,font=\small] at (2,2.6) {$\footnotesize{\alpha + t}$};
\node[black,font=\small] at (2,3.6) {$\footnotesize{\alpha + t}$};
\node[black,font=\small] at (2,4.6) {$\footnotesize{\alpha + t}$};
\node[black,font=\small] at (2,5.6) {$\footnotesize{\alpha + t}$};
\node[black,font=\small] at (3,0.6) {$\footnotesize{2\alpha}$};
\node[black,font=\small] at (4,0.6) {$\footnotesize{2\alpha}$};
\node[black,font=\small] at (5,0.6) {$\footnotesize{2\alpha}$};
\node[black,font=\small] at (6,0.6) {$\footnotesize{2\alpha}$};
\node[black,font=\small] at (4.5,3.6) {$\footnotesize{2\alpha}$};
\node[black,font=\small] at (6,6.3) {$(N,N)$};
\node[black,font=\small] at (6.3,5.3) {$(N,N-1)$};

\end{tikzpicture}
\caption{}
\label{Fig:Full2}
\end{minipage}
\end{figure}

\begin{lemma}\label{lemma:two_para_upper_tail}
There exists some positive constants $c,C$ and $N_0$ such that for all $N \geq N_0$ and $x > 0$,
\begin{equation}
    \begin{aligned}
        \PP \left( \log Z^{t,\beta}(N,N) + 2N\psi(\alpha) \geq x \Nsigma  \right) \leq Ce^{-cx},
    \end{aligned}
\end{equation}
provided that the following two bounds hold (with different positive constants $c, C,N_0$):
\begin{equation}
    \PP \left( \log Z^{t,\beta}_{1,full}(N+1,N-2) + 2N\psi(\alpha) \geq x \Nsigma  \right) \leq Ce^{-cx},
\end{equation}
\begin{equation}
    \PP \left( \log Z^{t,\beta}_{2,full}(N,N-1) + 2N\psi(\alpha) \geq x \Nsigma  \right) \leq Ce^{-cx}.
\end{equation}
\end{lemma}

\begin{proof}
    By the equalities in distribution in \eqref{eq:eq_d_1} and \eqref{eq:eq_d_2} and equation \eqref{eq:decomposition}, for $x > 0$, we have
    \begin{equation}
        \begin{aligned}
        \PP &\left( \log Z^{t,\beta}(N,N)+ 2N\psi(\alpha) \geq x \Nsigma  \right) \\
            \leq  &\PP \left( \log Z_1^{t,\beta}(N,N-2) + 2N\psi(\alpha) \geq x \Nsigma - \log 2  \right) \\&+ \PP \left( \log Z_2^{t,\beta}(N-1,N-1) + 2N\psi(\alpha) \geq x \Nsigma - \log 2  \right).
        \end{aligned}
    \end{equation}
By the inequality in \eqref{eq:ineq_d_1}, we see that
\begin{equation}
\begin{aligned}
    \PP &\left( \log Z_1^{t,\beta}(N,N-2) + 2N\psi(\alpha) \geq x \Nsigma - \log 2  \right) \\ &\leq \PP \left( \log Z_{1,line}^{t,\beta} - \log W_{N+1,N-2} + 2N\psi(\alpha) \geq x \Nsigma  \right)\\
    &\leq \PP \left( \log Z_{1,line}^{t,\beta}  +2N\psi(\alpha) \geq \frac{x}{2} \Nsigma \right) + \PP \left(  \log W_{N+1,N-2}\leq -\frac{x}{2}\Nsigma \right).
\end{aligned}
\end{equation}
Similarly, by the inequality \eqref{eq:ineq_d_2}, we have
\begin{equation}
\begin{aligned}
    \PP &\left( \log Z_2^{t,\beta}(N-1,N-1) + 2N\psi(\alpha) \geq x \Nsigma - \log 2  \right) \\ &\leq \PP \left( \log Z_{2,line}^{t,\beta} - \log \widetilde{W}_{N,N-1} + 2N\psi(\alpha) \geq x \Nsigma\right)\\
    &\leq \PP \left( \log Z_{2,line}^{t,\beta} + 2N\psi(\alpha) \geq \frac{x}{2} \Nsigma \right) + \PP \left( \log \widetilde{W}_{N,N-1} \leq -\frac{x}{2}\Nsigma \right).
\end{aligned}
\end{equation}
Because $\log W_{N+1,N-2}\overset{d}{=}\log \widetilde{W}_{N,N-1} \sim \log \Gamm(2\alpha)$ are sub-exponential random variables \cite[Proposition D.2]{basu2024temporal}, we know that there exists some $C,c > 0$ such that
\[
\PP \left( \log W_{N+1,N-2} \leq -\frac{x}{2}\Nsigma \right) = \PP \left( \log \widetilde{W}_{N,N-1} \leq -\frac{x}{2}\Nsigma \right)  \leq Ce^{-cx\Nsigma}.
\]
This completes the proof.
\end{proof}

To establish the probability bounds required in Lemma~\ref{thm:two_para_upper_tail}, we need the following Fredholm determinant formula for the cumulative distribution function of the free energy from $(1,1)$ to $(N,N)$ in the full-space inhomogeneous log-gamma polymer model. The following theorem is from \cite[Theorem 5.4]{imamura2022solvablemodelskpzclass}.

\begin{thm}\label{thm:1full_space}
Fix integers $N\ge 1$ and parameters $A_1,\dots,A_N,B_1,\cdots, B_N>0$.
Let $Z(N,N)$ denote the point-to-point partition function of the log-gamma polymer
from $(1,1)$ to $(N,N)$ on $\mathbb{Z}_{>0}^2$ with inverse-gamma weights having row/column
parameters $\{A_i\}_{i=1}^N$ and $\{B_j\}_{j=1}^N$. Then, for any $\tau \in \R$,
    \begin{equation}\label{eq:full_space_pf}
        \PP(\log Z(N,N) + \mathcal{G} < \tau) = \det (\Id-\bold{L})_{\mathbb{L}^2(\tau, +\infty)}
    \end{equation}
    where $\mathcal{G}$ is a standard Gumbel random variable and the kernel is defined as 
    \begin{equation}\label{eq:LG-rect-kernel}
        \bold{L}(X,Y) = \int_{\I \R - d_1} \frac{dZ}{2\pi\I} \int_{\I \R + d_2} \frac{dW}{2\pi\I}\frac{\pi}{\sin(\pi(W-Z))}\prod_{i=1}^N \frac{\Gamma(A_i+Z)}{\Gamma(B_i-Z)}\frac{\Gamma(B_i-W)}{\Gamma(A_i+W)}e^{ZX-WY}
    \end{equation}
    for $d_1, d_2>0$ such that $\frac{1}{N} \leq d_1 + d_2 < 1$ and $d_1, d_2 < \min\{A_1, \cdots, A_N, B_1, \cdots, B_N\}$
\end{thm}
Now we are ready to prove upper tails for $\log Z^{t,\beta}_{1,full}(N+1,N-2)$ and $\log Z^{t,\beta}_{2,full}(N,N-1)$.

\begin{thm}\label{thm:full_space_upper_tail}
    Recall that we fix $\alpha >0,$ $\ttt>0$, $\tbeta \in (-\ttt, \infty)$ and scale $\beta = \frac{\tbeta}{\Nsigma}$ and $t = \frac{\ttt}{\Nsigma}$. There exists $N_0 \in \mathbb{N}$ and $C, c > 0$ such that for all $N \geq N_0$ and $x > 0,$
    \begin{equation}
        \Pb\left(\frac{\log Z_{2,full}^{t,\beta}(N,N-1)+2N\psi(\alpha)}{(\sigma N)^{1/3}} \geq x\right) \leq Ce^{-cx}.
    \end{equation}
\end{thm}
\begin{proof}
    Let us start with the equation \eqref{eq:full_space_pf} in Theorem \ref{thm:1full_space}. We observe that $\PP(\log Z(N,N) + \mathcal{G} < \tau)$ is an analytic function in $\beta$ for $\beta > -t$. Recall the notations from Definition \ref{def:contour}
    \begin{equation}
    V(a;\theta;d) =\left( \{a+re^{\I\theta}: r > 0\} \cup \{a-re^{-\I\theta}: r \leq 0\}\right) \cap \{z \in \CC: |\text{Re}(z)| \leq d\},
    \end{equation}
    \begin{equation}
    L(a;b) = \{a + k\I: |k| \geq b\}.
    \end{equation}
    Fix some $\max\{-\tbeta,-\ttt\} < \tilde{\varsigma}_1 < \tilde{\varsigma}_2 < \ttt$. Let us deform the contour $\I \R - d_1$ into 
    \begin{equation}
    C_1 = V(\tilde{\varsigma_1}\Nsig; 2\pi/3;d_1 ) \bigcup L(-d_1;\sqrt{3}(d_1 + \tilde{\varsigma_1}\Nsig))
    \end{equation}
    and $\I \R + d_2$ into
    \begin{equation}
    C_2 = V(\tilde{\varsigma_2}\Nsig; \pi/3;d_2) \bigcup L(d_2;\sqrt{3}(d_2 - \tilde{\varsigma_2}\Nsig))
    \end{equation}
    as shown in the Figure \ref{fig:full_space_contour_2}. Since there is no pole at zero, Cauchy's integral theorem gives that  
    \begin{equation}
    \begin{aligned}
        L(X,Y) := \int_{C_1} \frac{dZ}{2\pi\I} \int_{C_2} \frac{dW}{2\pi\I}&\frac{\pi}{\sin(\pi(W-Z))}\left( \frac{\Gamma(\alpha + Z)}{\Gamma(\alpha + W)}\right)^{N-2}\left( \frac{\Gamma(\alpha - W)}{\Gamma(\alpha - Z)}\right)^{N-1} \\
        &\frac{\Gamma(t+Z)}{\Gamma(t+W)}\frac{\Gamma(\beta+ Z)}{\Gamma(\beta + W)}\frac{\Gamma(t-W)}{\Gamma(t-Z)} e^{ZX-WY}
    \end{aligned}
    \end{equation}
    is an analytic continuation of $\bold{L}(X,Y)$ in \eqref{eq:LG-rect-kernel}. Choose some $\tilde{\varsigma} \in (\tilde{\varsigma}_1, \tilde{\varsigma}_2)$. We define
    \begin{equation}
    \overline{L}(X,Y) = e^{-\frac{\tilde{\varsigma}}{\Nsigma}X + \frac{\tilde{\varsigma}}{\Nsigma}Y}L(X,Y). 
    \end{equation}
    We want to show that $\det (\Id-\overline{L})$ is well-defined for $\beta > -t$ and therefore, is the analytic continuation of $\det(\Id - \bold{L})$. Consider the following scaling and the change of variables
    \begin{equation}
        \begin{aligned}
            \tau = -2N\psi(\alpha) &+ \Nsigma x,\\
            X = -2N\psi(\alpha) + u \Nsigma, &\quad Y = -2N\psi(\alpha) + v \Nsigma.
        \end{aligned}
    \end{equation}
    We can rewrite the kernel $\overline{L}$ as
    \begin{equation}
    \begin{aligned}
            \overline{L}(u,v) =& \int_{C_1} \frac{dZ}{2\pi\I} \int_{C_2} \frac{dW}{2\pi\I}  \frac{\pi}{\sin(\pi(W-Z))} \left( \frac{\Gamma(\alpha + W)}{\Gamma(\alpha + Z)}\right)^{2}\left( \frac{\Gamma(\alpha - Z)}{\Gamma(\alpha - W)}\right) \frac{\Gamma(t+Z)}{\Gamma(t+W)}\\&\frac{\Gamma(\beta+ Z)}{\Gamma(\beta + W)}\frac{\Gamma(t-W)}{\Gamma(t-Z)}\frac{e^{Nh(Z) + Z \Nsigma u - \tilde{\varsigma}u}}{e^{Nh(W) + W \Nsigma v- \tilde{\varsigma}v}}
    \end{aligned}
    \end{equation}
    where $h(Z) = \log \Gamma(\alpha + Z) - \log \Gamma(\alpha-Z) - fZ$. By the same steepest descent analysis used in Lemma \ref{lem: limit and upper tail of kernel}, we see that only the integrals over $V(\tilde{\varsigma_1}\Nsig; 2\pi/3;d_1)$ and $V(\tilde{\varsigma_2}\Nsig; \pi/3;d_2)$ contribute in the limit. We do another change of variables
    \[
    W = \frac{\tW}{\Nsigma}, \quad Z = \frac{\tZ}{\Nsigma}
    \]
    and let 
    $$T_N(Z,W) = \frac{(\ttt+\tW)(\tbeta+\tW)(\ttt-\tZ)}{(\tW - \tZ)(\ttt+\tZ)(\tbeta + \tZ)(\ttt-\tW)} + \mathcal{O}((\sigma N)^{-1/3}).$$
    Thus,
    \[
    \begin{aligned}
    \overline{L}(u,v)
    &=\int\limits_{V_1} \frac{d\tZ}{2\pi\I} \int\limits_{V_2} \frac{d\tW}{2\pi\I} e^{-\frac{\tZ^3}{3} + \frac{\tW^3}{3} + u(\tZ - \tilde{\varsigma}) - v(\tW - \tilde{\varsigma}) + \mathcal{O}((\sigma N)^{-1/3})}T_N(Z,W)\\
    &\quad+ \mathcal{O}(e^{-\varepsilon N-(\tilde{\varsigma}+d_1\Nsigma)u - (d_2\Nsigma - \tilde{\varsigma})v})\\
    \end{aligned}
    \]
    where $V_1 = V(\tilde{\varsigma}_1;2\pi/3;d_1\Nsigma)$and $V_2 = V(\tilde{\varsigma}_2;\pi/3;d_2\Nsigma)$. The first integral converges due to the exponential decay of $e^{-\frac{\tZ^3}{3} + \frac{\tW^3}{3}}$ along the contours and the second term vanishes as $\text{Re}(h(Z) -h(W)) < -\varepsilon$ outside of $V_1$ and $V_2$. Thus, there exists positive constant $C$ and $N_0$ independent of $u,v$ such that following upper bound holds for all $N\geq N_0,$ and for all $u,v\geq 0$,
    \begin{equation}
    \left| \Nsigma \overline{L}(u,v) \right| \leq Ce^{-(\tilde{\varsigma}+\min\{\tbeta,\ttt\})u - (\ttt - \tilde{\varsigma})v}.
    \end{equation}
    Our choice of $\tilde{\varsigma}$ satisfies $\tilde{\varsigma} + \min\{\tbeta , \ttt\} >0$ for all $\tbeta >-\ttt$ and $\ttt - \tilde{\varsigma}>0$. Thus, we see that $\det(\Id-\overline{L})$ is well-defined for all $\tbeta > -\ttt$. By analytic continuation, we have for all $\tbeta >-\ttt$
    \begin{equation}
    \PP(\log Z_{2,full}^{t,\beta}(N,N) + \mathcal{G} < \tau) = \det (\Id-\overline{L})_{\mathbb{L}^2(\tau, +\infty)}.
    \end{equation} Thus, for any $x>0$,
    \begin{equation}\label{eq: pfaffian close to 1}
        \begin{aligned}
            &\Pb\left(\frac{\log Z_{2,full}^{t,\beta}(N,N)+\mathcal{G} +2N\psi(\alpha)}{(\sigma N)^{1/3}} \geq x\right) = 1-\det(1-\overline{L})_{\mathbb{L}^2(-2N\psi(\alpha) + \Nsigma x, +\infty)}\\
            &= \sum_{\ell=1}^\infty \frac{(-1)^{\ell-1}}{\ell!} \int_x^\infty \cdots \int_x^\infty \det[ \Nsigma \overline{L}( x_i,x_i) ] dx_1 \cdots dx_\ell \\
            & \leq \sum_{\ell=1}^\infty  \int_x^\infty \cdots \int_x^\infty \ell^{\ell/2}C^{\ell} e^{-c(x_1 + \cdots + x_{\ell})} dx_1 \cdots dx_\ell\\
            &\leq C'e^{-cx}
        \end{aligned}
    \end{equation}
where we have used the Hadamard's bound in the first inequality and set the constant $c = \min\{ \tilde{\varsigma}+\min\{\tbeta,\ttt\}, \ttt - \tilde{\varsigma}\} > 0$. Lastly, 
\begin{equation}
\begin{aligned}
        &\Pb\left(\frac{\log Z_{2,full}^{t,\beta}(N,N-1) +2N\psi(\alpha)}{(\sigma N)^{1/3}} \geq x\right)\\ 
        &\leq \Pb\left(\frac{\log Z_{2,full}^{t,\beta}(N,N)-\log \Gamm(2\alpha) +2N\psi(\alpha)}{(\sigma N)^{1/3}} \geq x\right)\\
        &\leq \Pb\left(\frac{\log Z_{2,full}^{t,\beta}(N,N-1) + \mathcal{G} +2N\psi(\alpha)}{(\sigma N)^{1/3}} \geq \frac{x}{3}\right) + \Pb\left(\frac{\mathcal{G}}{(\sigma N)^{1/3}} \leq -\frac{x}{3}\right)  +\Pb\left(\frac{\log \Gamm(2\alpha) }{(\sigma N)^{1/3}} \leq -\frac{x}{3}\right) \\
        &\leq C'e^{-cx/3}+e^{-e^{x\Nsigma/3}}+e^{-c'{x\Nsigma/3}}
\end{aligned}
\end{equation}
because $\log \Gamm(2\alpha)$ is sub-exponential. This completes the proof.
\end{proof}

\scalebox{0.8}{
\begin{minipage}{0.6\textwidth}
    \begin{tikzpicture}[thick]

\draw (-4.5,0) -- (4.5,0);

\fill (-4,0) circle (2pt) node[below] {$-\ttt$};
\fill (-1,0) circle (2pt) node[below] {$0$};
\fill (-0.4,0) circle (2pt) node[above] {$-\tilde\beta$};

\fill (0.3,0) circle (2pt) node[below] {$\tilde{\varsigma}_1$};
\fill (1.5,0) circle (2pt) node[below] {$\tilde{\varsigma}_2$};

\fill (4,0) circle (2pt) node[below] {$\ttt$};

\draw (0.3,0) -- (-2, 3.4);
\draw (-2,3.4) -- (-2,4);

\draw (0.3,0) -- (-2,-3.4);
\draw (-2,-3.4) -- (-2,-4) node[below] {$-d_1$};

\draw (1.5,0) -- (3,2);
\draw (3,2) -- (3,4);

\draw (1.5,0) -- (3,-2);
\draw (3,-2) -- (3,-4) node[below] {$d_2$};

\end{tikzpicture}
    \captionof{figure}{The contour of $\overline{L}$ for $Z^{t,\beta}_{2,full}$ when $\tbeta <0$. We choose $-\ttt<-\tbeta<\tilde{\varsigma}_1< \tilde{\varsigma}_2 < \ttt$.}
    \label{fig:full_space_contour_2}
\end{minipage}

\hfill 
\begin{minipage}{0.6\textwidth}
\begin{tikzpicture}[thick]

\draw (-4.5,0) -- (4.5,0);

\fill (-2.7,0) circle (2pt) node[below] {$-\ttt$};
\fill (-1.5,0) circle (2pt) node[below] {$0$};
\fill (-0.8,0) circle (2pt) node[above] {$-\tilde\beta$};

\fill (0.3,0) circle (2pt) node[below] {$\tilde{\varsigma}_1$};
\fill (1.5,0) circle (2pt) node[below] {$\tilde{\varsigma}_2$};

\fill (-0.3,0) circle (2pt) node[below] {$\ttt$};

\draw (0.3,0) -- (-2, 3.4);
\draw (-2,3.4) -- (-2,4);

\draw (0.3,0) -- (-2,-3.4);
\draw (-2,-3.4) -- (-2,-4) node[below] {$-d_1$};

\draw (1.5,0) -- (3,2);
\draw (3,2) -- (3,4);

\draw (1.5,0) -- (3,-2);
\draw (3,-2) -- (3,-4) node[below] {$d_2$};

\end{tikzpicture}
    \captionof{figure}{The contour of $\overline{L}$ for $Z^{t,\beta}_{1,full}$ when $\tbeta <0$. As $-\tbeta < \ttt$, we choose $ \ttt<\tilde{\varsigma}_1< \tilde{\varsigma}_2$.}
    \label{fig:full_space_contour_1}
\end{minipage}
}

\begin{thm}\label{thm:full_space_upper_tail_2}
    There exists $N_0 \in \mathbb{N}$ and $C, c > 0$ such that for all $N \geq N_0$ and $x > 0,$
    \begin{equation}
        \Pb\left(\frac{\log Z_{1,full}^{t,\beta}(N+1,N-2)+2N\psi(\alpha)}{(\sigma N)^{1/3}} \geq x\right) \leq Ce^{-cx}.
    \end{equation}
\end{thm}

\begin{proof}
The upper bound for the upper tail of $\log Z_{1,full}^{t,\beta}(N+1,N+1)$ follows analogously to that of Theorem \ref{thm:full_space_upper_tail} except that we choose $\tilde{\varsigma}_2>\tilde{\varsigma}>\tilde{\varsigma}_1>\ttt$ and the contour for $\overline{L}$ as shown in Figure \ref{fig:full_space_contour_1}. In this case, there exists positive constant $C$ and $N_0$ independent of $u,v$ such that following upper bound holds for all $N\geq N_0,$ and for all $u,v \geq 0$,
    \begin{equation}
    \left| \Nsigma \overline{L}(u,v) \right| \leq Ce^{-(\tilde{\varsigma}-\ttt)u - (\kappa - \tilde{\varsigma})v}
    \end{equation}
for any $\kappa > \tilde{\varsigma}$ since we can freely deform $W-$contour so that $\tilde{\varsigma}_2 > \kappa$. The rest follows from Hadamard's bound and taking the constant $c = \tilde{\varsigma} -\ttt > 0$ in \eqref{eq: pfaffian close to 1}. Once we establish the upper bound for the upper tail of $\log Z_{1,full}^{t,\beta}(N+1,N+1)$, the upper bound for $\log Z_{1,full}^{t,\beta}(N+1,N-2)$ also follows as $\log \Gamm(\alpha-t)$ stochastically dominates $\log \Gamm(2\alpha).$
\end{proof}

\begin{thm}\label{thm:two_para_upper_tail}
Fix any $\alpha >0,$ $\ttt>0$, $\tbeta \in (-\ttt, \infty)$ and scale $\beta = \frac{\tbeta}{\Nsigma}$ and $t = \frac{\ttt}{\Nsigma}$. There exists some positive constants $c,C$ and $N_0$ such that for all $N \geq N_0$ and $x > 0$,
\begin{equation}\label{eq:two_para_upper_tail}
    \begin{aligned}
        \PP \left( \log Z^{t,\beta}(N,N) + 2N\psi(\alpha) \geq x \Nsigma  \right) \leq Ce^{-cx},
    \end{aligned}
\end{equation}
\begin{equation}\label{eq:xxx}
    \PP \left( \log Z^{t}(N,N) + 2N\psi(\alpha) \geq x \Nsigma  \right) \leq Ce^{-cx}.
\end{equation}
\end{thm}

\begin{proof}
    Combining Lemma \ref{lemma:two_para_upper_tail}, Theorem \ref{thm:full_space_upper_tail}, and Theorem \ref{thm:full_space_upper_tail_2}, we obtain \eqref{eq:two_para_upper_tail}. For \eqref{eq:xxx}, it follows by applying same arguments of stochastic domination and exponential tail of log-inverse gamma random variable as in the proof of Theorem \ref{thm:lower_2para} to \eqref{eq:two_para_upper_tail}.
\end{proof}

\section{Lower tails for $Q^{Low}_{N,t}$, $Q^{High}_{N,\beta > t}$, $Q^{High}_{N,\beta < t}$}
Recall that 
\begin{equation}\label{eq:Qlow}
    Q_{N,t}^{Low}(\tau) = \E\left[ 2K_0 (2e^{(\log Z^t(N,N) - \tau)/2})\right]
\end{equation}
\begin{equation}\label{eq:QHigh>}
\begin{aligned}
    Q_{N,\beta > t}^{High} (\tau) =\E \left[ 2K_0 (2e^{(\log Z^{t, \beta}(N,N) - \tau+ \log W)/2}) \right] \text{ where } W \sim \Gamm(\beta -t) \text{ and } \beta > t
\end{aligned}
\end{equation}
\begin{equation}\label{eq:QHigh<}
Q_{N,\beta < t}^{High}(\tau) = \frac{1}{\Gamma(\beta - t)}\E\left[ \int_0^\infty 2K_0 (2e^{(\log Z^{t,\beta}(N,N) - \tau + \log w)/2}) w^{t-\beta - 1}e^{-w^{-1}}dw\right] \text{ for } -t<\beta < t
\end{equation}
where $\beta = \tbeta \Nsig$ and $t = \ttt \Nsig$ for some fixed $\tbeta,\ttt$ such that $\tbeta > -\ttt$ and $\ttt > 0$. We will prove the lower tail for \eqref{eq:Qlow} and \eqref{eq:QHigh>} first. 

\begin{thm}\label{thm:Qlower}
    Let $\tau = (z+r)\Nsigma - 2N\psi(\alpha)$ for some fixed $r \in \R$. Define
    \begin{equation}
    Q_N(z) = \E \left[ 2K_0 (2e^{(\log Z(N) + 2N \psi(\alpha) - (z+r)\Nsigma)/2}) \right]
    \end{equation}
    such that $\log Z(N)$ is a random variable satisfying the following lower tail estimates: there exists some positive constants $c,C$ and $N_0$ such that for all $N \geq N_0$ and $x > 0$, we have
    \begin{equation}\label{eq:lower_tail}
    \Pb \left( \log Z(N) + 2N\psi(\alpha) \leq -x\Nsigma \right) \leq Ce^{-cx}.
    \end{equation}
    Then for any $\varepsilon > 0$, there exists some positive constants $C'$ and $N_0$ such that for all $N \geq N_0$ and all $z < 0$,
    \begin{equation}
    \Nsig Q_N(z) \leq C'|z|e^{-c(1-\varepsilon)|z|}
    \end{equation}
    for the same constant $c> 0$ as in \eqref{eq:lower_tail}.
\end{thm}

\begin{proof}
Because $2K_0(2x)$ is non-negative and decreasing for all $x \in \R$, we have
\begin{equation}
\begin{split}
        \E \left[ 2K_0 \left(2e^{\frac{1}{2}\left(\log Z(N) +2N\psi(\alpha) - (z+r)(\sigma N)^{1/3}  \right)}\right)\right] &= \int_0^\infty \PP \left( 2K_0(2X) > t \right) dt
\end{split}
\end{equation}
with $X = \exp\left(\frac{1}2(\log Z(N) + 2N\psi(\alpha) - (z+r)\Nsigma)\right)$.
By the following power series expansion of $K_0$ near $0$ \cite[10.31.1]{abramowitz1965handbook} and \cite[10.25.2]{abramowitz1965handbook}:
\begin{equation}
K_0(x) = -(\log(x/2) + \gamma)I_0(x) + O(x^2) \text{ where } I_0(x) = \sum_{k=0}^\infty \frac{x^{2k}}{4^k k! \Gamma(k+1)},
\end{equation}
and the asymptotic expansion of $K_0$ near $\infty$ \cite[10.25.3]{abramowitz1965handbook}
\begin{equation}
K_0(x) \sim \sqrt{\pi/(2x)}e^{-x},
\end{equation}
we can upper bound $K_0$ by the following:
\begin{equation}\label{eq:K0Upper}
    \begin{aligned}
       K_0(x)\ \le\ g(x) = 
\begin{cases}
-D\log(x/2), & 0<x\le 2,\\
Dx^{-1}, & x\ge 2.
\end{cases}
    \end{aligned}
\end{equation}
for some constant $D > 0$. Thus, 
\begin{equation}
\begin{aligned}
    &\int_0^\infty \PP \left( 2K_0(2X) > t \right) dt \leq \int_0^\infty \PP \left( g(2X) > \frac{t}{2} \right)dt
\end{aligned}
\end{equation}
and 
\begin{equation}
\begin{split}
    \int_0^\infty \PP \left( g(2X) > \frac{t}{2} \right) dt &= \int_0^\infty  \PP \left( -\log Z(N) -2N\psi(\alpha) + (z+r) (\sigma N)^{1/3}  > tD^{-1} \text{ and } X \leq 1 \right) dt\\
    &+ \int_0^\infty  \PP \left( e^{-\frac{1}{2}\left(\log Z(N) +2N\psi(\alpha) - (z+r)(\sigma N)^{1/3} \right)}  > tD^{-1} \text{ and } X \geq 1 \right) dt.
\end{split}
\end{equation}
Because $r$ is fixed, we can modify \eqref{eq:lower_tail} to
\begin{equation}
    \Pb \left( \log Z(N) + 2N\psi(\alpha) \leq -(x+r)\Nsigma \right) \leq C'e^{-cx}.
\end{equation}
for all $x > 0$ and for some new constant $C'>0$.

For \eqref{eq:lower_tail}, we can in fact extend the range of $x>0$ to $x>-r$ for some fixed $r >0$ by making the constant larger. Thus, for the first integral, we perform the change of variable $t = sN^{1/3}$ and for $z < 0$
\begin{equation}
\begin{split}
    &\int_0^\infty  \PP \left( -\log Z(N) -2N\psi(\alpha) + (z+r) (\sigma N)^{1/3}\ > tD^{-1} \text{ and } X \leq 1 \right) dt\\
    &\leq N^{1/3}\int_0^\infty  \PP \left( \log Z(N)+ 2N\psi(\alpha) < (z+r) (\sigma N)^{1/3} - sD^{-1}N^{1/3} \right) ds\\
    &\leq N^{1/3} \int_0^\infty C'e^{c((z+r)-sD^{-1}\sigma^{-1/3})}ds\\
    &\leq N^{1/3} C'' e^{-c|z|}
\end{split}
\end{equation}
for some new constant $C''>0$.

For the second integral, notice that $X \geq 1$ implies that $\log Z(N) + 2N\psi(\alpha) \geq (z+r) \Nsigma$ and thus if $t > D$, then the integrand probability is zero. Thus, we can change the range of integration from $[0, \infty]$ to $[0, D]$. Let us also do the change of variable that $t = e^{sN^{1/3}}$. Thus, for $z < 0$ and any $\varepsilon >0$, we have
\begin{equation}
\begin{split}
    &\int_0^\infty  \PP \left( e^{-\frac12\left(\log Z(N) +2N\psi(\alpha) - (z+r)\Nsigma \right)}  > tD^{-1} \text{ and } X \geq 1 \right) dt 
    \\& \leq \int_{-\infty}^{\frac{ \log D}{N^{1/3}}} N^{1/3}e^{sN^{1/3}} \PP\left( - \frac12(\log Z (N)+ 2N\psi(\alpha) - (z+r)\Nsigma) > sN^{1/3} - \log D \right) ds\\
    & = \int_{\varepsilon \sigma^{1/3}z}^{\frac{ \log D}{N^{1/3}}} N^{1/3}e^{sN^{1/3}} \PP\left(\log Z(N) + 2N\psi(\alpha) \leq (z+r)\Nsigma- 2sN^{1/3} +2 \log D \right) ds\\
    & + \int_{-\infty}^{\varepsilon \sigma^{1/3} z} N^{1/3}e^{sN^{1/3}} \PP\left(\log Z(N) + 2N\psi(\alpha) \leq (z+r)\Nsigma- 2sN^{1/3} +2 \log D \right) ds\\
    &\leq N^{1/3} \left(\frac{\log D}{N^{1/3}} - \varepsilon \sigma^{1/3} z \right) DC'e^{-c(1-2\varepsilon)|z|} + N^{1/3}\int_{-\infty}^{\varepsilon \sigma^{1/3} z} e^{sN^{1/3}} ds \\
    &\leq N^{1/3}C'''|z|e^{-c(1-2\varepsilon)|z|} + e^{\varepsilon z\Nsigma}
    \end{split}
\end{equation}
for some new constant $C'''>0$. We conclude the proof by combining the upper bounds for these two integrals.
\end{proof}

\begin{cor}\label{cor: 1param and 2param beta large lower tails}
    Let $\tau = (z+r)\Nsigma - 2N\psi(\alpha)$ for some fixed $r \in \R$. Recall that $\beta = \tbeta \Nsig$ and $t = \ttt \Nsig$ for some fixed $\tbeta>\ttt$. For any $\varepsilon > 0$, there exists some positive constants $C$ and $N_0$ such that for all $N \geq N_0$ and all $z < 0$,
    \begin{equation}
    \Nsig Q_{N,t}^{Low}\left((z+r)\Nsigma - 2N\psi(\alpha)\right) \leq C|z|e^{-(1-\varepsilon)2\ttt|z|}
    \end{equation}
    \begin{equation}\label{eq: beta > t Q lower tail}
    \Nsig Q_{N,\beta > t}^{High}\left((z+r)\Nsigma - 2N\psi(\alpha)\right) \leq C|z|e^{-(1-\varepsilon)2\ttt|z|}.
    \end{equation}
\end{cor}

\begin{proof}
    By Theorem \ref{thm:Qlower}, we only need to check the lower tail condition \eqref{eq:lower_tail} for $\log Z^t(N,N)$ and $\log Z^{t,\beta}(N,N) + \log W$ where $W \sim \Gamm(\beta - t)$. The lower tail for $\log Z^t(N,N)$ follows from Theorem \ref{thm:lower_tail_prodstat}. For $\log Z^{t,\beta}(N,N) + \log W$ and $x > 0$, we have
    \begin{equation}
    \begin{aligned}
        &\PP \left( \log Z^{t,\beta}(N,N) + \log W + 2N\psi(\alpha) \leq -x \Nsigma \right)\\
        &\leq \PP \left( \log Z^{t,\beta}(N,N) + 2N\psi(\alpha) \leq -x(1- \delta) \Nsigma \right) + \PP\left(\log W \leq -x\delta \Nsigma\right)
    \end{aligned}
    \end{equation}
for any small $\delta > 0$. Choose $\delta$ such that $(2\delta^{-1}+1)t > \beta - t.$
Notice that by Lemma \ref{lem:stochastic_dominance}, $\log W \geq_{\mathrm{st}} \log \overline{W}$ where $\overline{W} \sim \Gamm((2\delta^{-1}+1)t)$. Set $\lambda = \tilde{\lambda}\Nsig$ for some fixed $\tilde{\lambda} < (2\delta^{-1}+1)\ttt$. Then, 
\begin{equation}
    \E\left[ \overline{W}^\lambda \right] = \frac{\Gamma\left( (2\delta^{-1}\ttt +\ttt-\tilde{\lambda})\Nsig \right)}{\Gamma\left( (2\delta^{-1}\ttt +\ttt)\Nsig \right)}
\end{equation}
which is bounded for all large $N$ as it converges to $\Gamma(2\delta^{-1}\ttt + \ttt - \tilde{\lambda})\Gamma(2\delta^{-1}\ttt + \ttt)^{-1}$. Thus,
\begin{equation}
   \PP\left(\log W \leq -x \delta \Nsigma\right) \leq \PP\left(\log \overline{W} \leq -x \delta \Nsigma\right) \leq \PP\left( W^\lambda \leq \exp(- \lambda x \delta \Nsigma)\right) \leq Ce^{-
   \tilde{\lambda}\delta x}.
\end{equation}
In particular, taking $\tilde{\lambda} = 2\delta^{-1}\ttt$ concludes the proof.
\end{proof}

\begin{thm}\label{thm: lower tail for Q beta small 2param}
     Let $\tau = (z+r)\Nsigma - 2N\psi(\alpha)$ for some fixed $r \in \R$. Recall that $\beta = \tbeta \Nsig$ and $t = \ttt \Nsig$ for some fixed $-\ttt<\tbeta<\ttt$. For any $\epsilon >0,$ there exists some positive constants $C$ and $N_0$ such that for all $N \geq N_0$ and all $z < 0$,
    \begin{equation}\label{eq: beta small lower tail}
    \Nsig Q_{N,\beta < t}^{High}\left((z+r)\Nsigma - 2N\psi(\alpha)\right) \leq C|z|e^{-(1-\epsilon)2\ttt|z|}.
    \end{equation}
\end{thm}

\begin{proof}
By Fubini's theorem, we can write
    \begin{equation}
        \begin{aligned}
            &Q_{N,\beta < t}^{High}\left((z+r)\Nsigma - 2N\psi(\alpha)\right)\\ &= \frac{1}{\Gamma(\beta - t)} \int_0^\infty \E\left[2K_0 (2e^{(\log Z^{t,\beta}(N,N) + 2N\psi(\alpha) - (z+r)\Nsigma + \log w)/2})\right] w^{t-\beta - 1}e^{-w^{-1}}dw.
        \end{aligned}
    \end{equation}
Make the change of variables $y = \Nsig \log w$.
    \begin{equation}
        \begin{aligned}
            &\Nsig Q_{N,\beta < t}^{High}\left((z+r)\Nsigma - 2N\psi(\alpha)\right) = \frac{\Nsigma}{\Gamma\left(\tbeta \Nsig  - \ttt \Nsig\right)}\\
           & \quad \times \int_{-\infty}^\infty \Nsig\E\left[2K_0 (2e^{(\log Z^{t,\beta}(N,N) + 2N\psi(\alpha) - (z+r)\Nsigma + y \Nsigma )/2})\right] e^{y(\ttt-\tbeta)}e^{-e^{-y \Nsigma}}dy.
        \end{aligned}
    \end{equation}
Notice that the factor $\Nsigma \Gamma\left(\tbeta \Nsig  - \ttt \Nsig\right)^{-1}$ is bounded as it converges to $\tbeta - \ttt$. Now we will analyze the integral in two parts, over $(0, \infty)$ and $(-\infty,0)$. For the first part, we apply Theorem \ref{thm:Qlower} using the lower tail for $\log Z^{t,\beta}(N,N)$ obtained in Theorem \ref{thm:lower_2para} where we choose $\varepsilon > 0$ such that $(1-\varepsilon)2\ttt > \ttt-\tbeta$:
\begin{equation}\label{eq: half y positve}
\begin{aligned}
        &\int_{0}^\infty \Nsig\E\left[2K_0 (2e^{(\log Z^{t,\beta}(N,N) + 2N\psi(\alpha) - (z+r)\Nsigma + y \Nsigma )/2})\right] e^{y(\ttt-\tbeta)}e^{-e^{-y \Nsigma}}dy \\
        &\leq \int_0^\infty C'(|z|+y)e^{-(1-\varepsilon)2\ttt(|z|+y)}e^{y(\ttt-\tbeta)} dy \leq C''|z|e^{-(1-\varepsilon)2\ttt|z|}.
\end{aligned}
\end{equation}
for some new constant $C'' > 0$.

For the second part, notice that when $y< 0$, $e^{y(\ttt-\tbeta)}$ increases as $\tbeta$ increases. Moreover, recall that $K_0(x)$ is monotonically decreasing over $(0, \infty)$. Let $\zeta = \tilde{\zeta}\Nsig$ for any fixed $\tilde{\zeta} > \tilde{t}$. By Lemma \ref{lem:stochastic_dominance}, $\log Z^{t,\beta} \geq_{\mathrm{st}} \log Z^{t,\zeta}$. Thus,
\begin{equation}\label{eq: half y negative}
    \begin{aligned}
        &\int_{-\infty}^0 \Nsig\E\left[2K_0 (2e^{(\log Z^{t,\beta}(N,N) + 2N\psi(\alpha) - (z+r)\Nsigma + y \Nsigma) /2})\right] e^{y(\ttt-\tbeta)}e^{-e^{-y \Nsigma}}dy \\
        &\leq \int_{-\infty}^0 \Nsig\E\left[2K_0 (2e^{(\log Z^{t,\zeta}(N,N) + 2N\psi(\alpha) - (z+r)\Nsigma + y \Nsigma /2})\right] e^{y(\ttt-\tilde{\zeta})}e^{-e^{-y \Nsigma}}dy\\
        &\leq \int_{-\infty}^\infty \Nsig\E\left[2K_0 (2e^{(\log Z^{t,\zeta}(N,N) + 2N\psi(\alpha) - (z+r)\Nsigma + y \Nsigma /2})\right] e^{y(\ttt-\tilde{\zeta})}e^{-e^{-y \Nsigma}}dy\\
        &= \frac{\Gamma(\tilde{\zeta}\Nsig - \ttt \Nsig)}{\Nsigma} \Nsig Q_{N,\zeta > t}^{High}\left((z+r)\Nsigma - 2N\psi(\alpha)\right)\\
        &\leq C'|z|e^{-(1-\varepsilon)2\ttt|z|}
\end{aligned}
\end{equation}
for some $C'>0$. The last inequality follow from Corollary \ref{cor: 1param and 2param beta large lower tails}. Combining \eqref{eq: half y positve} and \eqref{eq: half y negative} gives the desired lower tail \eqref{eq: beta small lower tail}.


\end{proof}

\section{Upper tails for $Q^{Low}_{N,t}$, $Q^{High}_{N,\beta > t}$, $Q^{High}_{N,\beta < t}$}

\subsection{Upper tails for $Q^{Low}_{N,t}$, $Q^{High}_{N,\beta > t}$}
Unlike lower tails, which was proved based on the lower tails of the model, the upper tails are obtained from the exact Pfaffian formulas.
We use the new parameter $u,v$ to replace $X,Y$, for example, $\widehat{\phi}_1(u) = \widehat{\phi}_1(X) = \widehat{\phi}_1(-2\psi(\alpha)N + (\sigma N)^{1/3}u).$ We use the notation $C(N_0,\delta)$ to show dependence on $N_0$ and $\delta$.

\begin{lemma}\label{lem: upper tail for phi and C}
    Fix any $0<\delta < \ttt.$ There exists $N_0 \in \Z_{>0}$, $C(N_0,\delta)>0$ independent of $u,s$ such that for all $N\geq N_0$ and $u \geq s$, we have 
    \begin{equation}
        \begin{aligned}
            &|\widehat{\phi}_1(u,s)| \leq Ce^{s(\ttt - \delta)}e^{-\delta u}, \quad |(\sigma N)^{-1/3}\widehat{\phi}_2(u,s)| \leq Ce^{s(\ttt - \delta)}e^{-\delta u}, \quad |(\sigma N)^{-1/3}\widehat{C}(s)| \leq Ce^{s(\ttt - \delta)},\\
            &|\widehat{B}(u)| \leq Ce^{-\delta u}, \quad |(\sigma N)^{1/3}\widehat{B}'(u)| \leq Ce^{-\delta u}.
        \end{aligned}
    \end{equation}
\end{lemma}

\begin{proof}
    We use $(\sigma N)^{1/3}\widehat{\phi}_1$ to give an example. Let $\rho = (\sigma N)^{-1/3}$. We have
    \begin{equation}
        \begin{aligned}
            &\frac{e^{-\frac{\ttt^3}{3}}}{8\ttt}\!\!\!\!\int \limits_{\mathcal{C}(\delta;\pi/3)}\!\!\!\!\frac{d\tZ}{2\pi\I}\!\!\!\! \int \limits_{\mathcal{C}(\delta;\pi/3)}\!\!\!\!\frac{d\tW}{2\pi\I}  e^{-u\tZ}e^{s(\ttt-\tW)}\!\!\left(\frac{(-\ttt + \tZ)(\ttt+\tZ)(\ttt + \tW)}{\tW(-\ttt -\tZ)(\ttt - \tZ)(\ttt - \tW)} +\mathcal{O}(\rho)\right) e^{\frac{\tZ^3}{3}+ \frac{\tW^3}{3} + \mathcal{O}(\rho)}\\
            &\leq e^{-u\delta}e^{s(\ttt - \delta)}\frac{e^{-\frac{\ttt^3}{3}}}{8\ttt}\!\!\!\int \limits_{\mathcal{C}(\delta;\pi/3)}\!\!\!\frac{d\tZ}{2\pi\I}\!\!\! \int \limits_{\mathcal{C}(\delta;\pi/3)}\!\!\!\frac{d\tW}{2\pi\I}  \bigg|\frac{(\ttt + \tW)}{\tW(\ttt - \tW)} + \mathcal{O}(\rho)\bigg| e^{\frac{\tZ^3}{3} + \frac{\tW^3}{3}+ \mathcal{O}(\rho)}\leq Ce^{-u\delta}e^{s(\ttt - \delta)}.
        \end{aligned}
    \end{equation}
    for constant $C>0$ independent of $u$ and uniformly in $N$ because $e^{\frac{\tZ^3}{3} + \frac{\tW^3}{3}}$ has exponential decay along the chosen contours.
    The upper bounds for $\widehat{\phi}_2$, $\widehat{C}$, $\widehat{B}$ and $\widehat{B}'$ are justified by similar arguments. Moreover, the upper bounds for $\widehat{B}$ and $\widehat{B}'$ do not depend on $s$.
\end{proof}

\begin{rmk}
    Since $\delta$ can be chosen arbitrarily close to $\ttt$, we claim that for any $1 \gg c_0 >0,$ there exists $N_0 \in \Z_{>0}$, and $C(N_0,\delta) >0$ independent of $u$ such that 
    \begin{equation}
        |\widehat{\phi}_1(u,s)| \leq Ce^{c_0 s}e^{-(\ttt - c_0) u}, \quad\quad  |(\sigma N)^{-1/3}\widehat{\phi}_2(u,s)| \leq Ce^{c_0 s}e^{-(\ttt-c_0) u}.
    \end{equation}
\end{rmk}

\begin{lemma}\label{lem:upper tail for modified kernel}
    Fix any $\alpha>0,$ $\tbeta > \ttt>0.$ Consider the scaling $\tau = -2\psi(\alpha)N + (\sigma N)^{1/3}s,$ $\beta = \tbeta/(\sigma N)^{1/3}$, and $t = \ttt/(\sigma N)^{1/3}$. For any $\varepsilon >0$, there exist $s_0 \gg 1$ and $N_0 \in \Z_{>0}$ such that for all $s > s_0$ and $N > N_0$, we have 
    \begin{equation}
    \begin{aligned}
        &\bigg|\Pf\left(J - \widehat{\mathbf{K}}\right)_{L^{2}(\tau,\infty)} -1 \bigg| < \varepsilon,\\
        &\Bigg|\Pf\left(J - \widehat{\bold{K}} - (\sigma N)^{-1/3}\ketbra{\begin{array}{c} \widehat{\phi}_2 \\
        -\widehat{\phi}_1
        \end{array}}{\widehat{B} \quad -\widehat{B}^{\prime}} - (\sigma N)^{-1/3}\ketbra{\begin{array}{c}
            \widehat{B}\\
            -\widehat{B}^{\prime}
        \end{array}}{-\widehat{\phi}_2 \quad \widehat{\phi}_1 }\right)_{L^{2}(\tau,\infty)} -1\Bigg| < \varepsilon.
    \end{aligned}
    \end{equation}
\end{lemma}

\begin{proof}
    We prove the more difficult kernel. The result for $\widehat{\mathbf{K}}$ can be analyzed similarly. Let $\rho = (\sigma N)^{-1/3}$. By Lemma \ref{lem: limit and upper tail of kernel}, Lemma \ref{lem: upper tail for phi and C} and the definition of Fredholm Pfaffian, we get
    \begin{equation}
        \begin{aligned}
            &\Bigg|\Pf\left(J - \widehat{\bold{K}} - \rho\ketbra{\begin{array}{c} \widehat{\phi}_2 \\
        -\widehat{\phi}_1
        \end{array}}{\widehat{B} \quad -\widehat{B}^{\prime}} - \rho\ketbra{\begin{array}{c}
            \widehat{B}\\
            -\widehat{B}^{\prime}
        \end{array}}{-\widehat{\phi}_2 \quad \widehat{\phi}_1 }\right) -1\Bigg|\\
        &\leq \sum_{n=1}^{\infty} \frac{1}{n!} \int_{s}^{\infty}\cdots \int_{s}^{\infty}\\
        &\Bigg|\Pf\left( \left( \begin{pmatrix}
            \widehat{K} & -\rho^{-1}\partial_{u_j}\widehat{K}\\
            -\rho^{-1}\partial_{u_i}\widehat{K} & \rho^{-2}\partial_{u_i}\partial_{u_j}\widehat{K}
        \end{pmatrix} +\ketbra{\!\!\!\begin{array}{c} \rho\widehat{\phi}_2 \\
        -\widehat{\phi}_1
        \end{array}\!\!\!}{\widehat{B} \quad -\!\rho^{-1}\widehat{B}^{\prime}} + \ketbra{\!\!\begin{array}{c}
            \widehat{B}\\
            -\rho^{-1}\widehat{B}^{\prime}
        \end{array}\!\!\!}{-\rho\widehat{\phi}_2 \quad \widehat{\phi}_1 } \right)(u_i,u_j)\right)_{i,j=1}^n\Bigg|  \prod_{i=1}^{n}du_i\\
        &\leq \sum_{n=1}^{\infty} \frac{1}{n!} \int_{s}^{\infty}\cdots \int_{s}^{\infty} (2n)^{\frac{n}{2}}C^n(1+e^{s(\ttt-\delta)})^{n} e^{-2\delta\sum_{i=1}^nu_i}\prod_{i=1}^{n}du_i\\
        &= \sum_{n=1}^{\infty} \frac{1}{n!}(2n)^{\frac{n}{2}}C^n(1+e^{s(\ttt-\delta)})^{n} \frac{e^{-2\delta n s}}{(2\delta)^n}= \sum_{n=1}^{\infty} \frac{1}{n!}(2n)^{\frac{n}{2}}\left(\frac{C}{{2\delta}}\right)^n\left({e^{-2\delta s}} + e^{s(\ttt-3\delta)}\right)^n <\varepsilon,
        \end{aligned}
    \end{equation}
    for sufficiently large $s$ and for some $\delta$ satisfying $0<\delta < \ttt$ and $\ttt-3\delta <0.$ We recall that $\delta$ determines the integral contour for each function.
\end{proof}

\begin{lemma}\label{lem: 1param 2param beta large upper tail Q}
    Fix any $\alpha >0,$ $\ttt>0,$ $\tbeta > \ttt,$ and $s\in \R$. Consider the scaling $\tau = -2N\psi(\alpha) + (\sigma N)^{1/3}s,$ $\beta = \tbeta/(\sigma N)^{1/3}$, and $t = \ttt/(\sigma N)^{1/3}$. Recall the definition of ${Q}_{N,\beta >t}^{High}(\tau)$ in Definition \ref{def: two param beta > t} and the definition of $Q_{N,t}^{Low}(\tau)$ in Definition \ref{def: one_param_finite}. Then for any $\varepsilon >0,$ there exist $s_0 \gg 1$ and $N_0 \in \Z_{>0}$ and $C_1,C_2 >0$ independent of $s$ such that for all $s>s_0$ and $N > N_0$, we have
    \begin{equation}
    \begin{aligned}
        &|(\sigma N)^{-1/3}{Q}_{N,\beta >t}^{High}(-2N\psi(\alpha) + (\sigma N)^{1/3}s)| \leq C_1e^{s(\ttt - \delta)},\\
        &|(\sigma N)^{-1/3}Q_{N,t}^{Low}(-2N\psi(\alpha) + (\sigma N)^{1/3}s)| \leq C_2e^{s(\ttt - \delta)}.\\
    \end{aligned}
    \end{equation}
\end{lemma}

\begin{proof}
    The bound for ${Q}_{N,\beta >t}^{High}(\tau)$ is a direct consequence of Lemma \ref{lem: upper tail for phi and C} and the upper bound of $\widehat{C}$ provided in Lemma \ref{lem:upper tail for modified kernel}. To get the same result for $Q_{N,t}^{Low}(\tau)$, we simply need to prove a very similar version of Lemma \ref{lem: upper tail for phi and C} and Lemma \ref{lem:upper tail for modified kernel} for product stationary formulas. 
\end{proof}

\subsection{Upper tails for $Q^{High}_{N,\beta < t}$}

Previously, in Lemma~\ref{lem: limits of constant terms for 2 param beta small}, we fixed $s\in \R$ and the following functions are constants. Now, we regard the following quantities as functions of $s$ and derive their upper bounds in terms of $s$.
\begin{lemma}\label{lem: upper tail functions 2 param} 
Fix any $\alpha >0,$ $\ttt>0,$ $-\ttt<\tbeta < \ttt,$ and $s\in \R$. Fix any $\max\{\tbeta,0\}<\delta < \ttt.$ Consider the scaling $\tau = -2N\psi(\alpha) + (\sigma N)^{1/3}s,$ $\beta = \tbeta/(\sigma N)^{1/3}$, and $t = \ttt/(\sigma N)^{1/3}$. Then there exist $N_0 \in \Z_{>0}$ and $C(N_0,\delta) >0$ independent of $s$ such that for all $N\geq N_0$, we have
\begin{equation}
    \begin{aligned}
        &|\A_1(s)| \leq Ce^{-(\tbeta + \delta)s}, \quad |\A_2(s)| \leq Ce^{-(\tbeta + 3\delta)s},\quad |\B_1(s)| \leq Ce^{-2\delta s}, \quad |\B_2(s)| \leq Ce^{-4\delta s},\\
        &|(\sigma N)^{-1/3}\C_1(s)| \leq Ce^{-(\tbeta -\ttt)s}, \quad |(\sigma N)^{-1/3}\C_2(s)| \leq Ce^{-(2\delta + \tbeta -\ttt)s},\\
        &|(\sigma N)^{-1/3}\D_1(s)| \leq Ce^{-(\delta -t)s}, \quad |(\sigma N)^{-1/3}\D_2(s)| \leq Ce^{-(3\delta -\ttt)s},\\
        &|(\sigma N)^{-1/3}\M_1(s)| \leq Ce^{-(\delta -\ttt)s}, \quad |(\sigma N)^{-1/3}\M_2(s)| \leq Ce^{-(3\delta -\ttt)s},\\
        &|\NN_1(s)-1| \leq Ce^{-(\delta +\tbeta)s}, \quad |\NN_2(s)| \leq Ce^{-(3\delta +\tbeta)s}.\\
    \end{aligned}
\end{equation}
\end{lemma}

\begin{proof}
    We proceed as in Lemma~\ref{lem: upper bounds for all functions 2 param beta small}  and Lemma~\ref{lem: upper tail for phi and C}. We explain the most involved term $\B_2$ and all remaining terms can be treated by a similar analysis. We first factor out $e^{\tau(Z+U+V+W)}$ which admits the upper bound $|e^{-s(\tZ+\tU+\tV+\tW)}| \leq e^{4s\delta}$ for $\tZ,\tW,\tU,\tV$ on the contour $C(\delta; \pi/3).$ After applying the scaling, the remaining integrand contains $e^{\frac{\tZ^3}{3} + \frac{\tW^3}{3} + \frac{\tU^3}{3} + \frac{\tV^3}{3} + \mathcal{O}((\sigma N)^{-1/3})}$ that has an exponential decay. Consequently, the resulting integrals are uniformly bounded by a constant.
\end{proof}

\begin{lemma}\label{lem:upper tail for modified kernel 2param}
Fix any $\alpha>0,$ $\ttt>0$, $-\ttt<\tbeta < \ttt.$ Consider the scaling $\tau = -2N\psi(\alpha) + (\sigma N)^{1/3}s,$ $\beta = \tbeta/(\sigma N)^{1/3}$, and $t = \ttt/(\sigma N)^{1/3}$.
Then for any $\varepsilon >0$, there exist $s_0 \gg 1$ and $N_0 \in \Z_{>0}$ such that for all $s > s_0$ and $N > N_0$, we have 
    \begin{equation}
    \begin{aligned}
        &\bigg|\Pf\left(J - \boldsymbol{\widehat{\mathcal{K}}}\right)_{L^{2}(\tau,\infty)} -1 \bigg| < \varepsilon,\quad
        \Bigg|\Pf\left(J - \boldsymbol{\widehat{\mathcal{K}}} -\widehat{A} \right)_{L^{2}(\tau, \infty)}-1\Bigg| < \varepsilon,
    \end{aligned}
    \end{equation}
    where $\widehat{A}$ can be any of the matrix kernels in \eqref{eq: all possible prelimit kernels}.
\end{lemma}

\begin{proof}
We follow the same argument as in Lemma~\ref{lem:upper tail for modified kernel} and focus on the kernel with the largest growth rate. Other kernels can be proved similarly. Let $\rho = (\sigma N)^{-1/3}$. Using Lemma~\ref{lem: upper bounds for all functions 2 param beta small} gives
    \begin{equation}
        \begin{aligned}
            &\Bigg|\Pf\left(J - \widetilde{\boldsymbol{\mathcal{K}}} - \rho\ketbra{\begin{array}{c} \widehat{\eta}_2 \\
        -\widehat{\eta}_1
        \end{array}}{\widehat{\psi}_1 \quad \widehat{\psi}_2} - \rho\ketbra{\begin{array}{c}
            \psi_1\\
            \psi_2
        \end{array}}{-\eta_2 \quad \eta_1 }\right) -1\Bigg|\\
        &\leq \sum_{n=1}^{\infty} \frac{1}{n!} \int_{s}^{\infty}\cdots \int_{s}^{\infty}\\
        &\Bigg|\Pf\left(\left(\begin{pmatrix}
            \K & -\rho^{-1}\partial_{u_j}\K\\
            -\rho^{-1}\partial_{u_i}\K & \rho^{-2}\partial_{u_i}\partial_{u_j}\K
        \end{pmatrix} +\ketbra{\begin{array}{c} \eta_2 \\
        -\rho^{-1}\eta_1
        \end{array}}{\rho\psi_1 \quad \psi_2} + \ketbra{\begin{array}{c}
            \rho\psi_1\\
            \psi_2
        \end{array}}{-\eta_2 \quad \rho^{-1}\eta_1 } \right)\Bigg|(u_i,u_j)\right)_{1\leq i,j\leq n}  \prod_{i=1}^{n}du_i\\
        &\leq \sum_{n=1}^{\infty} \frac{1}{n!} \int_{s}^{\infty}\cdots \int_{s}^{\infty} (2n)^{\frac{n}{2}}C^n(1+e^{s(\ttt-\tbeta-2\delta)})^{n} e^{-2\delta\sum_{i=1}^nu_i}\prod_{i=1}^{n}du_i\\
        &= \sum_{n=1}^{\infty} \frac{1}{n!}(2n)^{\frac{n}{2}}C^n(1+e^{s(2\ttt-2\delta)})^{n} \frac{e^{-2\delta n s}}{(2\delta)^n}
        = \sum_{n=1}^{\infty} \frac{1}{n!}(2n)^{\frac{n}{2}}\left(\frac{C}{{2\delta}}\right)^n\left( {e^{-2\delta s} + e^{2s(\ttt-2\delta)}}\right)^n <\varepsilon
        \end{aligned}
    \end{equation}
    for sufficiently large $s$ and for some $\delta$ satisfying $\max\{\tbeta,0\}<\delta < \ttt$ and $\ttt-2\delta <0.$
\end{proof}

\begin{lemma}\label{lem: two param beta small upper tail Q}
    Fix any $\alpha >0,$ $\ttt>0,$ $\tbeta > \ttt,$ and $s\in \R$. Consider the scaling $\tau = -2N\psi(\alpha) + (\sigma N)^{1/3}s,$ $\beta = \tbeta/(\sigma N)^{1/3}$, and $t = \ttt/(\sigma N)^{1/3}$. Recall the definition of ${Q}_{N,\beta <t}^{High}(\tau)$ in Definition \ref{def: two param beta < t finite}. There exist $s_0 \gg 1$, $N_0 \in \Z_{>0}$ and $C>0$ independent of $s$ such that for all $s>s_0$ and $N > N_0$, we have
    \begin{equation}
        \big|(\sigma N)^{-1/3}{Q}_{N,\beta <t}^{High}(-2N\psi(\alpha) + (\sigma N)^{1/3}s)\big| \leq Ce^{s(\ttt - \tbeta)}.
    \end{equation}
\end{lemma}

\begin{proof}
By Lemma \ref{lem: upper tail functions 2 param} and Lemma \ref{lem:upper tail for modified kernel 2param}, we get that for any $\varepsilon >0,$ there exists $s_0 \gg 1$ and $N_0 \in \Z_{>0}$, and $C>0$ independent of $s$ such that for all $s \geq s_0$, 
    \begin{equation}
    \begin{aligned}
        &\bigg|(\sigma N)^{-1/3}{Q}_{N,\beta <t}^{High}(-2N\psi(\alpha) + (\sigma N)^{1/3}s)\bigg|
        \leq \bigg(\left(\big|(\sigma N)^{-1/3}\M_1(s)\big| + |\M_2(s)| + \varepsilon\right)\left(1+|\NN_2(s)| +  \varepsilon\right)\\
        &-(|\A_1(s)| + |\A_2(s)| + \varepsilon)\left(\big|(\sigma N)^{-1/3}\D_1(s)\big| + \big|(\sigma N)^{-1/3}\D_2(s)\big| + \varepsilon \right)\\
        &+ \left(|\B_1(s)| + |\B_2(s)| +  \varepsilon\right)\left(\big|(\sigma N)^{-1/3}\C_1(s)\big| + \big|(\sigma N)^{-1/3}\C_2(s)\big| +  \varepsilon\right)\leq Ce^{-(\tbeta - \ttt)s}
    \end{aligned}
    \end{equation}
    where the final bound follows from the fact that $\big|(\sigma N)^{1-/3}\C_1(s)\big|$ has the largest growth rate among all contributing terms.
\end{proof}

\section{Numerical evaluation of stationary log-gamma polymer and exponential LPP diagonal distribution}
In this section, we take the log-gamma to exponential limit of \eqref{eq: Two_param_formula_positive} and \eqref{Two_param_formula_negative}, see \cite[section 3]{LogGammaStationary} for example, and compute the exact distribution of the Exponential LPP when $N=2$ and $N = 3$. We also compare the resulting Fredholm Pfaffian representation with the formula obtained by \cite{PatrikStatExp}. At the level of finite $N$,
we do not see a direct transformation relating the two formulas. We also provide numerical evaluations of \eqref{eq: One_param_formula} via MATLAB for small $N$.

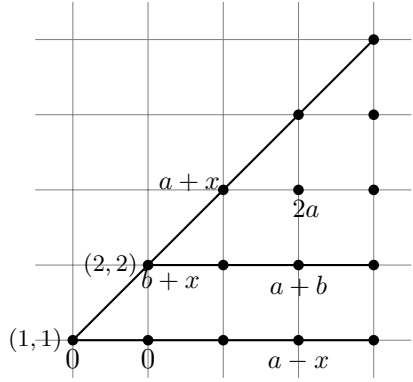
\begin{figure}[H]
\centering
\begin{tikzpicture}
\draw[very thin, gray] (-0.5,-0.5) grid (4.5,4.5);
    \foreach \x in {1,...,4} {
        \foreach \y in {0,1,...,\x} {
            \fill (\x,\y) circle (2pt);
        }
    }
    \fill (0,0) circle (2pt);
    \draw[thick,-] (0,0) -- (4,0);
    \draw[thick,-] (0,0) -- (4,4);
    \draw[thick,-] (1,1) -- (4,1);
    \fill (0,0) circle (1pt);
    \node at (0,-0.25) {$0$};
     
    \node at (1,-0.25) {$0$};
    \node at (-0.5,0) {\footnotesize{$(1,1)$}};
    \node at (0.5,1) {\footnotesize{$(2,2)$}};
    \node at (1.3,0.8) {\small{$b+x$}};
    \fill (1,1) circle (1pt);
    \node at (3,-0.28) {\small{$a-x$}};
    \node at (3,0.72) {\small{$a+b$}};
    \node at  (1.55,2.1) {\small{$a+x$}};
    \node at (3.1,1.75) {\small{$2a$}};
\end{tikzpicture}
\caption{Degeneration from $Z^{t,\beta}$ to stationary Exponential LPP model $G^{stat}_{exp}$ with diagonal parameter $x$, first-row parameter $-x$, second-row parameter $b$ and bulk parameter $a$.}
\label{fig:2param Exp stationary}
\end{figure}

\begin{figure}[H]
\centering
\begin{tikzpicture}
\draw[very thin, gray] (-0.5,-0.5) grid (4.5,4.5);
    \foreach \x in {1,...,4} {
        \foreach \y in {0,1,...,\x} {
            \fill (\x,\y) circle (2pt);
        }
    }
    \fill (0,0) circle (2pt);
    \draw[thick,-] (0,0) -- (4,0);
    \draw[thick,-] (0,0) -- (4,4);
    \draw[thick,-] (1,1) -- (4,1);
    \fill (0,0) circle (1pt);
    \node at (0,-0.25) {$x+y$};
     
    \node at (1,-0.25) {$b+y$};
    \node at (-0.5,0) {\footnotesize{$(1,1)$}};
    \node at (0.5,1) {\footnotesize{$(2,2)$}};
    \node at (1.3,0.8) {\small{$b+x$}};
    \fill (1,1) circle (1pt);
    \node at (3,-0.28) {\small{$a+y$}};
    \node at (3,0.72) {\small{$a+b$}};
    \node at  (1.55,2.1) {\small{$a+x$}};
    \node at (3.1,1.75) {\small{$2a$}};
\end{tikzpicture}
\caption{Inhomogeneous Exponential LPP model $G_{exp}$ with diagonal parameter $x$, first-row parameter $y$, second-row parameter $b$ and bulk parameter $a$.}
\label{fig:2param Exp before stationary}
\end{figure}
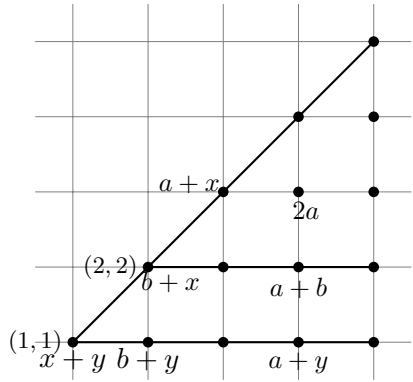

Consider the Gamma inverse random variable $\Gamm(x)$. We know that $\epsilon \Gamm(\epsilon x)$ converges to $\text{Exp}(x)$ in distribution as $\epsilon \rightarrow 0$. We call this type of scaling limit: the log-gamma-to-exponential limit. We first introduce a few notations.
Let $G_{exp}^{\text{stat}}(N,N)$ be the half-space exponential LPP value from $(1,1)$ to $(N,N)$ under the two-parameter exponential stationary initial condition, see \cite[Proposition 3.2]{LogGammaStationary}. The specific weights of this model is described in Figure \ref{fig:2param Exp stationary} where the
weights at $(1,1)$ and $(2,1)$ are set to $0$, and all remaining weights
are independent exponential random variables with the labeled rates. Let $G_{exp}(N,N)$ be the Exponential LPP value with weights described in Figure \ref{fig:2param Exp before stationary} and $G'_{exp}(N,N)$ be the Exponential LPP value with weights described in Figure \ref{fig:2param Exp before stationary} except the weights at $(1,1)$ and $(2,1)$ are replaced by $0$.
We state the shift argument for removing random variables Exp$(x+y)$ at $(1,1)$ and Exp$(b+y)$ at $(2,1)$. The geometric version of this can be found in \cite[section 6.3.1]{zeng2025stationary}. We have
\begin{equation}
    \Pb\left( G'_{exp}(N,N) \leq \gamma \right) = \partial_{\gamma} \frac{\Pb\left( G_{exp}(N,N) \leq \gamma \right)}{x+y} + \frac{1}{b +y}\partial_{\gamma}^2\frac{\Pb\left( G_{exp}(N,N) \leq \gamma \right)}{x+y}.
\end{equation}
We then take $y \rightarrow -x$ on both sides and let $g(s) : =\lim_{y \rightarrow -x}\frac{\Pb\left( G_{exp}(N,N) \leq s \right)}{x+y}.$ The fact that $g(s)$ is well-defined follows from \cite[Theorem 1.5 (1.9)]{zeng2025stationary}. We get
\begin{equation}\label{eq: exp shift}
    \Pb\left( G_{exp}^{\text{stat}}(N,N) \leq \gamma \right) = \partial_{\gamma} g(\gamma) + \frac{1}{b - x}\partial_{\gamma}^2 g(\gamma).
\end{equation}

Consider the following scaling of parameters $\beta,t,\alpha,u,\tau$ for the stationary log-gamma model:
\begin{equation}\label{eq: scalings}
    \beta = \epsilon b, \quad t = \epsilon x, \quad \alpha = \epsilon a, \quad u = \epsilon y, \quad \tau = \gamma/\epsilon.
\end{equation} 
Let $K_{exp}$ denote the log-gamma-to-exponential limit of the kernel $\bold{\check{K}}$ and let $\Pf\left(J - K_{exp}\right)_{\mathbb{L}^2((\gamma,\infty))}$ be the log-gamma-to-exponential limit of $\Pf\left(J - \bold{\check{K}}\right)_{\mathbb{L}^2((\tau,\infty))}$.
One can check that performing the same kernel decomposition and analytic continuation analysis to $\Pf\left(J - K_{exp}\right)$ gives the same formula as taking the log-gamma-to-exponential limit of \eqref{eq: One_param_formula}, \eqref{eq: Two_param_formula_positive}, and \eqref{Two_param_formula_negative}.

We compute the log-gamma-to-exponential limit of \eqref{Two_param_formula_negative} when $N=2$. We use a superscript $f^e$ to denote the log-gamma-to-exponential limit of each function $f$.
\begin{equation}\label{eq: N=2 two param}
\begin{aligned}
    &\frac{1}{\Pf\left(J - \widehat{\boldsymbol{\mathcal{K}}}^e\right)}\left( \widehat{\M}^e\widehat{\NN}^e - \widehat{\A}^e \widehat{\D}^e + \widehat{\B}^e \widehat{\C}^e\right)\\
    &= \gamma - \frac{1}{b - x} - \frac{1}{b + x} + \frac{e^{\gamma(x-b)}(x+b)}{2x(b-x)}+ \frac{(x-b)}{2x(x+b)}e^{-\gamma(b+x)},
\end{aligned}  
\end{equation}
where
\begin{equation}
    \begin{aligned}
    &\widehat{\boldsymbol{\mathcal{K}}}^e= \begin{pmatrix}
        0 & 0 \\
        0 & 0
    \end{pmatrix},\quad 
        \widehat{\M}^{e} = \gamma - \frac{1}{b - x} - \frac{1}{b + x}, \quad \widehat{\NN}^{e} = 1,\\
        &\widehat{\A}^{e} = - e^{-\gamma(b + x)}, \quad \widehat{\B}^{e} = \frac{b+x}{b-x},\quad
        \widehat{\C}^{e} = \frac{e^{\gamma(x-b)}}{2x}, \quad \widehat{\D}^{e} = \frac{(x-b)}{2x(x+b)}.
    \end{aligned}
\end{equation}
Then we can apply the shift argument \eqref{eq: exp shift} to \eqref{eq: N=2 two param} and get
\begin{equation}\label{eq: correct CDF}
\begin{aligned}
    \Pb\left( G_{exp}^{\text{stat}}(2,2) \leq \gamma \right) &=\left(\partial_\gamma + \frac{1}{b-x}\partial_\gamma^2\right)\left(\gamma - \frac{1}{b - x} - \frac{1}{b + x} +\frac{e^{\gamma(x-b)}(x+b)}{2x(b-x)}+ \frac{(x-b)}{2x(x+b)}e^{-\gamma(b+x)}\right)\\
    &= 1- e^{-(x + b)\gamma}.
    \end{aligned}
\end{equation}
The log-gamma-to-exponential limit of \eqref{eq: Two_param_formula_positive} under $N=2$ is
\begin{equation}\label{eq: N=2, beta large finite formula}
\begin{aligned}
    &\Pf\left(J - \widehat{\boldsymbol{{K}}}^e\right)\left( \widehat{C}^e - \brabarket{\widehat{B^e} \quad - \widehat{B^e}'}{\left(\Id - J^{-1}\widehat{\boldsymbol{{K}}}^e\right)^{-1}}{\begin{array}{c}
        \widehat{\phi}_1^e\\
        \widehat{\phi}_2^e
    \end{array}}\right)\\
    &=\gamma - \frac{1}{b-x} - \frac{1}{b+x} + \frac{e^{\gamma(x-b)}(x+b)}{2x(b-x)} + \frac{(x-b)}{2x(x+b)}e^{-s(x+b)},
\end{aligned}  
\end{equation}
where
\begin{equation}
    \begin{aligned}
        &\widehat{\boldsymbol{{K}}}^e= \begin{pmatrix}
        0 & 0 \\
        0 & 0
        \end{pmatrix},\quad \Pf\left(J - \widehat{\boldsymbol{{K}}}^e\right) = 1, \quad \widehat{C}^e = \gamma - \frac{1}{b-x} - \frac{1}{b+x} + \frac{e^{\gamma(x-b)}(x+b)}{2x(b-x)},\\
        &\widehat{\phi}_1^e(X) = \frac{b}{2x}\frac{(b-x)}{(b+x)}e^{-bX}, \quad \widehat{\phi}_2^e(X) = \frac{-1}{2x}\frac{(b-x)}{(b+x)}e^{-bX},\\
        &\widehat{B^e}(Y) = \frac{x+b}{x-b}\left( e^{-bY}- e^{-xY} \right), \quad \widehat{B^e}'(Y) = \frac{x+b}{x-b}\left( -be^{-bY}+x e^{-xY} \right),\\
        &-\Pf\left(J - \widehat{\boldsymbol{{K}}}^e\right)\brabarket{\widehat{B^e} \quad - \widehat{B^e}'}{\left(\Id - J^{-1}\widehat{\boldsymbol{{K}}}^e\right)}{\begin{array}{c}
        \widehat{\phi}_1^e\\
        \widehat{\phi}_2^e
        \end{array}} = \frac{(x-b)}{2x(x+b)}e^{-s(x+b)}.
    \end{aligned}
\end{equation}
We see that \eqref{eq: N=2 two param} and \eqref{eq: N=2, beta large finite formula} match.


Next, we consider the case $N = 3$.
When $b \neq x,-x,$ we have the following log-gamma-to-exponential limit formula for \eqref{Two_param_formula_negative}:
\begin{equation}
    \begin{aligned}
        &\widehat{\M}^e = \gamma - \frac{1}{a -x} -\frac{1}{a +x}-\frac{1}{b -x}-\frac{1}{b +x} - \frac{(b +a)}{2x(b - a)}\left(\frac{(b - x)(x + a)}{(b + x)(x - a)}e^{-\gamma(a-x)}+\frac{(b + x)( a-x)}{(b - x)( a+x)}e^{-\gamma(a+x)}\right) \\
        &\widehat{\NN}^e = 1- e^{-\gamma(a + b)}, \quad \Pf\left(J - \widehat{\boldsymbol{\mathcal{K}}}^e\right) = 1-e^{-\gamma(a + b)}.\\
        &\widehat{\A}^e = \left(e^{-\gamma(b + a)} -e^{-\gamma(b + x)}\right)\frac{(x+a)(b+a)}{(x-a)(b-a)}, \quad
        \widehat{\B}^e = \frac{b + x}{b - x}(1-e^{-\gamma(a + x)}), \\
        &\widehat{\C}^e = \frac{1}{2x} \left(e^{-\gamma(b + a)}- e^{-\gamma(b -x)} \right)\frac{(a-x)(b+a)}{(a+x)(b-a)}, \quad 
        \widehat{\D}^e = -\frac{1}{2x}\frac{b - x}{b + x}(1-e^{-\gamma(a - x)}).
    \end{aligned}
\end{equation}

where $\widehat{\M}^e$ is evaluated as the log-gamma-to-exponential limit of the following $\widehat{\M}$
\begin{equation}
    \widehat{\M} =\Pf\left(J - \widehat{\boldsymbol{\mathcal{K}}}\right) \left(\widehat{\M}_1 + \widehat{\M}_2 \right)- \Pf\left(J - \widehat{\boldsymbol{\mathcal{K}}} \right)\brabarket{\widehat{\aleph}_1 \quad \widehat{\aleph}_2
}{\left(\Id - J^{-1}\widehat{\boldsymbol{\mathcal{K}}}\right)}{\begin{array}{c}
            \widehat{\zeta}_1\\
            \widehat{\zeta}_2
        \end{array}}
\end{equation}
rather using the difference of Pfaffian formula. Similarly, we obtain $\widehat{\NN}^e,$ $\widehat{\A}^e,$ $\widehat{\B}^e,$ $\widehat{\C}^e,$ $\widehat{\D}^e.$

Taking the log-gamma-to-exponential limit of \eqref{eq: Two_param_formula_positive}, we obtain that
\begin{equation}
    \begin{aligned}
        &\Pf\left( J - \widehat{\boldsymbol{{K}}}^e\right) = 1-e^{-(a + b)\gamma},\\
        &\widehat{C}^e = \gamma - \frac{1}{a-x} - \frac{1}{b - x} - \frac{1}{a + x} - \frac{1}{b + x} - \frac{(a + b)(b - x)(a + x)e^{-\gamma(a -x)}}{2x(b - a)(x - a)(x+b)} - \frac{(a  + b)(a - x)(b + x)e^{-\gamma(b - x)}}{2x(a - b)(a + x)(x-b)},\\
        &-\Pf\left( J - \widehat{\boldsymbol{{K}}}^e\right)\brabarket{\widehat{B}^e \quad -\widehat{B^{\prime}}^e}{\left(\Id - J^{-1}\widehat{\boldsymbol{{K}}}^e\right)^{-1}}{\begin{array}{c}
            \widehat{\phi}_1^e\\
            \widehat{\phi}_2^e
        \end{array}} = \\
        &-\frac{a+b}{2(a-b)(a-x)(b-x)x(a+x)(b+x)}
\bigg(-2e^{-2(a+b)\gamma}\bigl(-1+e^{(a+b)\gamma}\bigr)(b-x)^2(a+x)^2 \\
&\quad
+ e^{(-3a-2b+x)\gamma}\bigl(-1+e^{(a+b)\gamma}\bigr)(b-x)^2(a+x)^2
+ e^{-(a+2b+x)\gamma}\bigl(-1+e^{(a+b)\gamma}\bigr)(b-x)^2(a+x)^2 \\
&\quad
+ e^{-(3a+2b+x)\gamma}\bigl(e^{a \gamma}-e^{x \gamma}\bigr)^2(b-x)^2(a+x)^2
+ 2 e^{-2(a+b)\gamma}\bigl(-1+e^{(a+b)\gamma}\bigr)(a-x)^2(b+x)^2 \\
&\quad
- e^{(-2a-3b+x)\gamma}\bigl(-1+e^{(a+b)\gamma}\bigr)(a-x)^2(b+x)^2
- e^{-(2a+b+x)\gamma}\bigl(-1+e^{(a+b)\gamma}\bigr)(a-x)^2(b+x)^2 \\
&\quad
- e^{-(2a+3b+x)\gamma}\bigl(e^{b \gamma}-e^{x \gamma}\bigr)^2(a-x)^2(b+x)^2
\bigg).
    \end{aligned}
\end{equation}

One can easily check using Mathematica that 
\begin{equation}
    \begin{aligned}
        \Pf\left(J - \widehat{\boldsymbol{\mathcal{K}}}^e\right)\bigg(\widehat{\M}^e\widehat{\NN}^e &-\widehat{\A}^e\widehat{\D}^e + \widehat{\B}^e\widehat{\C}^e\bigg) = \\
        &\Pf\left(J - \widehat{\boldsymbol{{K}}}^e\right)\widehat{C}^e-\Pf\left( J - \widehat{\boldsymbol{K}}^{e}\right)\brabarket{\widehat{B}^e \quad -\widehat{B'}^e}{\left(\Id - J^{-1}\widehat{\boldsymbol{K}}^{e}\right)^{-1}}{\begin{array}{c}
            \widehat{\phi}^e_1\\
            \widehat{\phi}^e_2
        \end{array}}.
    \end{aligned}
\end{equation}

We directly compute the $N=3$ formula in \cite[Theorem 2.4]{PatrikStatExp}. We use the half-space exponential LPP kernel for $G^{{stat}}_{exp}$ defined in Figure \ref{fig:2param Exp stationary}. Let 
\begin{equation}
    \mathsf{K}(X,Y) = \begin{pmatrix}
        K_{11} & K_{12}\\
        K_{21} & K_{22}
    \end{pmatrix}(X,Y),
\end{equation}
where $\mathsf{K}(X,Y)$ is defined in \cite[(2.13)]{PatrikStatExp}.
\begin{equation}
    \begin{aligned}
        K_{11}(X,Y) &= \frac{b + a}{b-a}(b - x)(x - a)(e^{-b X - aY} - e^{-aX - b Y }), \quad K_{21}(X,Y) = - K_{12}(Y,X),\\
        K_{12}(X,Y) &= \frac{b + a}{b-a}\left(2a\frac{b - x}{a+x}e^{-b X - aY} - 2b\frac{a-x}{b + x} e^{-aX - b Y } +(a-x)\frac{b - x}{b + x}e^{-aX+x Y } - (b-x)\frac{a - x}{a + x}e^{-b X+x Y }\right),\\
        K_{22}(X,Y) &= \frac{4b a(b + a)}{(b - a)(a+x)(x + b)}\left(e^{-aX-b Y} - e^{-b X - aY}\right) + \frac{(b + a)(b - x)(2a)}{(b + x)(b - a)(x + a)}\left(e^{x X- aY} - e^{-aX + x Y}\right)\\
        &-\frac{(a-x)(a+b)(2b)}{(a+x)(a-b)(b + x)}\left(e^{x X - b Y} - e^{-b X + x Y}\right) + K_{22}^{exp}(X,Y),\\
        &K_{22}^{exp}(X,Y) = \mathrm{sgn}(X-Y)e^{x|X-Y|}.
    \end{aligned}
\end{equation}
To evaluate the Fredholm Pfaffian of above kernel, we make the decomposition:
\begin{equation}
    \begin{aligned}
        &d_2(X) = \frac{b + a}{b-a} \left(\frac{e^{xX}}{-(b + x) } + \frac{2b}{(b + x)(b - x)}e^{-b X}\right), \quad g_1(X) = \frac{b + a}{b-a} e^{-b X},\\
        &f_1(Y) = (b - x)(x - a)e^{-aY}, \quad h_2(Y) = 2a\frac{b - x}{a+x}e^{-aY} - (a-x)\frac{b - x}{a+x}e^{x Y}.
    \end{aligned}
\end{equation}
\begin{equation}
    X_1 = \ket{{\begin{array}{c}
           d_2 \\
           g_1
        \end{array}}}, \quad Y_1 = \bra{f_1 \quad h_2}, \quad X_2 =\ket{{\begin{array}{c}
           -h_2 \\
           f_1
        \end{array}}}, \quad Y_2 = \bra{-g_1 \quad d_2}.
\end{equation}
Then we compute
\begin{equation}
\begin{aligned}
    &\braket{Y_1}{X_1} = \frac{2x(b - x)}{(b - a)(a+x)}e^{-\gamma(b + a)}- \frac{(a-x)(b +a)}{(a+x)(b - a)}e^{-\gamma(b -x)}+ \frac{(b + a)(b - x)}{(b - a)(b+x)}e^{-\gamma(a - x)}+\frac{2b(x - a)}{(b + a)(b+x)}e^{-\gamma(b + a)},\\
    &\brabarket{Y_1}{\begin{pmatrix}
        0 & -K_{22}^{exp}\\
        0 & 0
    \end{pmatrix}}{X_1} = \frac{(b - x)(x - a)(b + a)}{(b - a)}\brabarket{e^{-aX}}{\mathrm{sgn}(X-Y)e^{x|X-Y|}}{e^{-b Y}}
    = -e^{-(a+b)\gamma},
    \end{aligned}
\end{equation}
and
\begin{equation}\label{eq: Pfaffian for exp Patrik}
    \begin{aligned}
        &\Pf\left(J - \mathsf{K}\right) = 1- \braket{Y_1}{X_1} - \brabarket{Y_1}{\begin{pmatrix}
            0 & -K_{22}^{exp}\\
            0 & 0
        \end{pmatrix}}{X_1}\\
        &= 1-\frac{2x(b - x)}{(b - a)(a+x)}e^{-\gamma(b + a)}+\frac{(a-x)(b +a)}{(a+x)(b - a)}e^{-\gamma(b -x)}- \frac{(b + a)(b - x)}{(b - a)(b+x)}e^{-\gamma(a - x)}-\frac{2b(x - a)}{(b + a)(b+x)}e^{-\gamma(b + a)}\\
        &\quad + e^{-(a+b)\gamma}.
    \end{aligned}
\end{equation}
We see that the Fredholm Pfaffian of the central kernel in \eqref{eq: Pfaffian for exp Patrik} is more complicated than $\Pf\left(J - \widehat{\boldsymbol{\mathcal{K}}}^e\right)$
even though both formulas describe the same stationary exponential LPP distribution.

\subsection{Right tail behavior of asymptotic formulas} By KPZ universality, one should expect that taking the critical scaling of exponential LPP pfaffian formula should coincide with the asymptotics in the log-gamma case. Then it is natural to test our formula using the shift argument from ELPP rather than using the fourier transform and inversion. Then we immediately see that those complicated formulas after shift argument have the correct right tail behavior of a CDF function. 

Consider the product stationary asymptotics \eqref{eq: one param asymptotic formula}. Fix $0<\delta < \ttt$. The residue at $W = \ttt$ of the complex integral in $\widetilde{C}_L$ gives $s-\ttt^2$. Assume $s \gg 1$. We have Fredholm Pfaffian is close to $1$ using similar arguments in Lemma \ref{lem:upper tail for modified kernel}. The asymptotic version of the one-parameter shift argument is simply the partial derivative as shown in \cite[Theorem 2.7 and Lemma 3.3]{PatrikStatExp}.
Applying $\partial_s$ yields 
\begin{equation}
    \lim_{N\rightarrow \infty}\Pb\left(\frac{\log Z^{t} +Nf}{(\sigma N)^{1/3}}\leq s\right)= \partial_s \left((s-\ttt^2)\left(1-\mathcal{O}(e^{-\delta s})\right)\right) = 1 - \mathcal{O}(e^{-\delta s}).
\end{equation}
Next, we consider the two-parameter asymptotics \eqref{eq: some equations beta>t 2 param} when $\tbeta >\ttt$. The only difference of it from the product stationary case is the residue at $W=\ttt$ of $\widetilde{C}$ in \eqref{eq: some equations beta>t 2 param}:
\begin{equation}
    \text{Res}(\widetilde{C},\ttt) = s - \ttt^2 - \frac{1}{\tbeta - \ttt}- \frac{1}{\tbeta + \ttt}.
\end{equation}
Applying $\partial_s$ yields 
\begin{equation}
    \lim_{N\rightarrow \infty}\Pb\left(\frac{\log Z^{t,\beta} +Nf}{(\sigma N)^{1/3}}\leq s\right)= \partial_s \left(\left(s - \ttt^2 - \frac{1}{\tbeta - \ttt}- \frac{1}{\tbeta + \ttt}\right)\left(1-\mathcal{O}(e^{-\delta s})\right)\right) = 1 + \mathcal{O}(e^{-\delta s}).
\end{equation}
Finally, we consider the two-parameter asymptotics \eqref{eq: K_0 version 2 param} when $-\ttt<\tbeta < \ttt$. Fix $\max\{\tbeta,0\}<\delta < \ttt$. We have the following estimates of each function:
\begin{equation}
    \begin{aligned}
        &\Pf\left(J - \widetilde{\boldsymbol{\mathcal{K}}}\right) = 1 - \mathcal{O}(e^{-\delta s}),\quad \text{Res}(\widetilde{\A}_1, \ttt) = -e^{\frac{\tbeta^3}{3} + \frac{\ttt^3}{3} - s(\tbeta + \ttt)} + \mathcal{O}(e^{-\delta s}), \quad \text{Res}(\widetilde{\A}_2, \ttt) = \mathcal{O}(e^{-\delta s}),\\
        &\text{Res}(\widetilde{\B}_1, \ttt) = \frac{(\tbeta + \ttt)}{(\tbeta - \ttt)}+ \mathcal{O}(e^{-\delta s}), \quad \text{Res}(\widetilde{\B}_2, \ttt) = \mathcal{O}(e^{-\delta s}),\\
        &\text{Res}(\widetilde{\C}_1, \ttt) = \frac{1}{2\ttt}e^{\frac{\tbeta^3}{3} - \frac{\ttt^3}{3} + s(\ttt-\tbeta)}+ \mathcal{O}(e^{-\delta s}), \quad \text{Res}(\widetilde{\C}_2, \ttt) = \mathcal{O}(e^{-\delta s}).\\
        \end{aligned}
\end{equation}
\begin{equation}
\begin{aligned}
        &\text{Res}(\widetilde{\D}_1, \ttt) = \frac{(\ttt-\tbeta)}{2\ttt(\tbeta+\ttt)}+ \mathcal{O}(e^{-\delta s}), \quad \text{Res}(\widetilde{\D}_2, \ttt) = \mathcal{O}(e^{-\delta s}),\\
        &\text{Res}(\widetilde{\M}_1, \ttt) = s - t^2 - \frac{1}{\tbeta - \ttt} - \frac{1}{\tbeta + \ttt}+ \mathcal{O}(e^{-\delta s}), \quad \text{Res}(\widetilde{\M}_2, \ttt) = \mathcal{O}(e^{-\delta s}),\\
        &\text{Res}(\widetilde{\NN}_1,\ttt) = 1+ \mathcal{O}(e^{-\delta s}), \quad \text{Res}(\widetilde{\NN}_2, \ttt) = \mathcal{O}(e^{-\delta s}).\\
    \end{aligned}
\end{equation}
Applying $\partial_s$ yields
\begin{equation}
    \begin{aligned}
        \lim_{N\rightarrow \infty}\Pb\left(\frac{\log Z^{t,\beta} +Nf}{(\sigma N)^{1/3}}\leq s\right)=  \left(\frac{1}{\tbeta - \ttt}\partial_s + \partial_s^2\right) \left(\widetilde{\M}\widetilde{\NN} - \widetilde{\A}\widetilde{\D} + \widetilde{\B}\widetilde{\C} \right) = 1 - e^{\frac{\tbeta^3}{3} + \frac{\ttt^3}{3}-3s(\ttt+\tbeta)} + \mathcal{O}(e^{-\delta s}).
    \end{aligned}
\end{equation}

\subsection{Matlab Evaluations and Simulation results}
Recall the definition of $Q_{N,t}^{Low}(\tau)$ in Definition \ref{def: one_param_finite}. We compute $Q_{N,t}^{Low}(\tau)$ by direct evaluation of complex integrals and approximation of Fredholm Pfaffian using Matlab. Set $N = 4,$ $t = 0.4,$ $\alpha = 2.1$, and use the contour $-d+\I\R = -0.2 + \I\R.$ Let $S = \E\bigg[ 2K_0\left( 2
            {e^{\left(\log{{\mathcal{Z}}^t}(N,N)-\tau\right)/2} }\right)  \bigg].$ We estimate $S$ by generating $10^5$ samples and approximate expectation via SLLN. The results are
\begin{equation}
\begin{aligned}
\text{When }\tau = 1, \quad & \quad Q_{N,t}^{Low}(\tau) =2.8322, S(\tau) = 2.7992, \\
\text{When }\tau = 2.5, \quad & \quad Q_{N,t}^{Low}(\tau) =4.1138, S(\tau) = 4.1580, \\
\text{When }\tau = 3, \quad & \quad Q_{N,t}^{Low}(\tau) = 4.5987, S(\tau) = 4.6372, \\
\text{When }\tau = 5, \quad & \quad Q_{N,t}^{Low}(\tau) = 6.5765,  S(\tau) = 6.5912,\\
\text{When }\tau = 7, \quad & \quad Q_{N,t}^{Low}(\tau) = 8.5778, S(\tau) = 8.5816,\\
\text{When }\tau = 9, \quad & \quad Q_{N,t}^{Low}(\tau) = 10.5833, S(\tau) = 10.5800.
\end{aligned}
\end{equation}

\bibliographystyle{plain} 
\bibliography{main.bib}
\end{document}